\documentclass{amsart}
\usepackage{amsmath,amssymb}
\usepackage[dvips]{graphics}
\usepackage[all]{xy}
\newtheorem{Theorem}{Theorem}[section]
\newtheorem{Lemma}[Theorem]{Lemma}

\newtheorem{Corollary}[Theorem]{Corollary}

\newtheorem{Remark}[Theorem]{Remark}

\makeatletter
\@addtoreset{figure}{section}
\def\@thmcountersep{-}
\makeatother

\numberwithin{equation}{section}

\begin{document}

\begin{center} {\large\bf
 Almost positive links have negative signature}
\end{center}
\begin{center}
J\'ozef H. Przytycki and Kouki Taniyama 
\end{center}
Draft 1990-1991; Corrections July-September 1999; Completion May-July 2008, February-April 2009\\ 

\begin{center}
\vspace{0.1in}\ \\
\end{center}
        \begin{quotation}\ \\
       {\bf Abstract.}\\
       {\it
We analyze properties of links which have diagrams with a 
small number of negative crossings. 
We show that if a nontrivial link has a diagram with 
all crossings positive except possibly one, then the signature of 
the link is negative. If a link diagram has two 
negative crossings, we show that the signature of the link is 
nonpositive with the exception of the left-handed Hopf link (with 
possible trivial components). We also characterize those links 
which have signature zero and diagrams with two negative 
crossings. In particular, we show that if a nontrivial knot has a diagram 
with two negative crossings then the signature of the knot is negative, unless 
the knot is a twist knot with negative clasp.
We completely determine all trivial link diagrams with two or fewer negative crossings.
For a knot diagram with three negative crossings, the signature of the knot is nonpositive 
except the left-handed trefoil knot. These results generalize those of L.~Rudolph, 
T.~Cochran, E.~Gompf, P.~Traczyk, and J.~H.~Przytycki, solve 
Conjecture 5 of \cite{P-2}, and give a partial answer to Problem 2.8 of 
\cite{Co-G} about knots dominating the trefoil knot or the trivial knot.
We also describe all unknotting number one positive knots.}
\end{quotation}

\section {Introduction}
This paper is a sequel to \cite{P-2} and \cite{T-1,T-2} and uses ideas 
from these papers.  K.~Murasugi, following C.~A.~Giller \cite{Gi} 
showed that nontrivial positive alternating 
links have negative signature. In 1982 L.~Rudolph showed that 
nontrivial positive braids have negative signature (in our 
orientation convention), \cite{R-1}. These  results are generalized in 
\cite{Co-G,Tr,P-2} to the theorem that nontrivial positive links 
have negative signature. 
This theorem was first proven (for knots) by T.~Cochran and 
R.~Gompf in the summer of 1985 (at MSRI) but the paper \cite{Co-G} was 
 written only two years later. The theorem was 
independently proven by P.~Traczyk in summer of 1987 \cite{Tr}, and by 
J.~H.~Przytycki in September of 1987 \cite{P-2}. 
In this paper\footnote{Collaboration on this paper began when Przytycki noticed that 
the affirmative answer to a question he asked in 1987 \cite{P-2} followed from Taniyama's master 
 thesis (see \cite{T-1,T-2}). The question was
whether nontrivial almost positive link has negative signature.
 The first draft of the paper was written in 1990-1991. 
} we generalize the theorem to show 
that nontrivial almost positive links and 2-almost positive 
links (except twist knots, the left-handed Hopf link, and the 
connected or disjoint sum of the left-handed Hopf link, and a 
$(2,2k)$-torus link) have negative signature. 

To describe precisely the content of our paper we start with a few
preliminary definitions.

We consider oriented links in ${\mathbb S}^3$ and oriented link diagrams on 
${\mathbb S}^2$ up to ambient isotopy of ${\mathbb S}^3$ and ${\mathbb S}^2$ respectively. 
A link diagram is {\it $m$-almost positive} if all but $m$ of its crossings are positive 
$ \left(
\begin{minipage}{10pt}
\scalebox{0.1}{\includegraphics{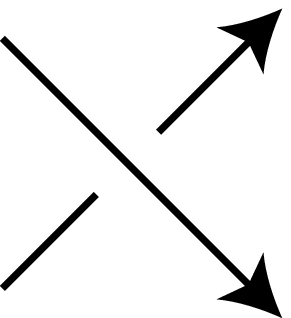}}
\end{minipage}
\right) $.

A $0$-almost positive diagram is called positive and a $1$-almost positive diagram is called almost positive. 
A link is $m$-almost positive if it has an $m$-almost positive diagram. 
In particular a 0-almost positive link is called positive and
a 1-almost positive link is called almost positive. For links $L_1$ and $L_2$, we define the 
relation $\geq$ as follows: 
$L_1$ is greater than or equal to $L_2$, denoted $L_1\ge L_2$, if  $L_2$ can be obtained from $L_1$ 
by changing some positive crossings to negative crossings. 
This relation is weaker than the relation in \cite{Co-G} and different from, 
but related to, the relation of \cite{T-1}.  
We denote by $\preceq$ the relation defined on knots in \cite{Co-G}.\footnote{ According to 
\cite{Co-G}, $K_1 \succeq K_2$ if $K_1$ is concordant 
to $K_2$ inside of a 4-manifold with 
positive definite intersection form whose concordance annulus is homologically trivial. 
See Definition 2.1 in \cite{Co-G}.} 
Here we only need the fact that if $K_1 \geq K_2$ then $K_1 \succeq K_2$.
It was observed by Giller \cite{Gi} and utilized by Murasugi, Cochran and Lickorish \cite{Co-Li}
that if links $L_1$ and $L_2$ are oriented links which differ only by 
one crossing at which $L_1$ is positive and $L_2$ is negative then, 
$$\sigma(L_2)-2\leq\sigma(L_1)\leq\sigma(L_2)$$ where $\sigma(L)$ denotes the signature of $L$.
Thus for links $L_1$ and $L_2$ with $L_1\ge L_2$ we have the inequality $\sigma(L_1)\leq\sigma(L_2)$. 
We call this the Giller inequality in our paper\footnote{Giller was motivated by J.~H.~Conway \cite{Con}, 
and used the inequality to show that the signature of knots is a skein equivalence invariant. 
Conway, in turn, was influenced by Murasugi's papers on signature \cite{M-1,M-2}.}.
This inequality is a basic tool used in the paper.
In Section 2 we work with positive links. There we will prove the following result: 

\begin{Theorem}\label{p1}
 Let $K$ be a nontrivial positive knot.
Then $K\ge$ (2,5)-torus knot or $K$ is a connected sum of some pretzel knots of type 
$L(p_1,p_2,p_3)$, where $p_1$, $p_2$ and $p_3$ are 
positive odd numbers; see Fig. \ref{pretzel-knot}. 
\end{Theorem}

\begin{figure}[htbp]
      \begin{center}
\scalebox{0.58}{\includegraphics*{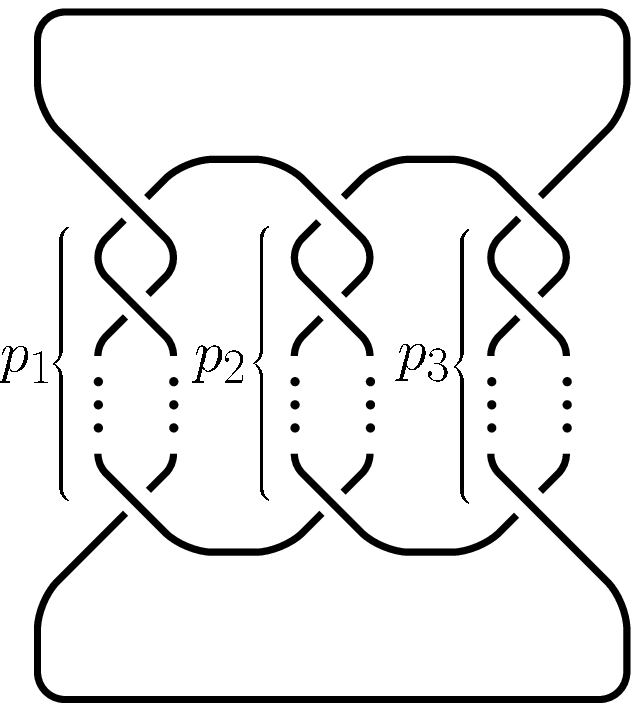}}
      \end{center}
   \caption{}
  \label{pretzel-knot}
\end{figure} 

%

\begin{Corollary}\label{p2}
 Let $K$ be a nontrivial positive knot. Then $K\succeq (2,5)$-torus knot or $K$ is a connected sum 
of some pretzel knots of type $L(p_1,p_2,p_3)$, where $p_1$, $p_2$ and $p_3$ are positive odd numbers.
\end{Corollary}

\begin{Corollary}\label{p3}
 Let $K$ be a nontrivial positive knot.
Then either the signature $\sigma(K)\leq-4$, or $K$ is a pretzel 
knot $L(p_1,p_2,p_3)$, where $p_1$, $p_2$ and $p_3$ are 
positive odd numbers. We have $\sigma(L(p_1,p_2,p_3))=-2$. 
\end{Corollary}

In Section 3 we analyze almost positive links. We prove 
there in particular: 

\begin{Theorem}\label{ap1} Let $L$ be a 
nontrivial almost positive link. Then $L\ge$ right-handed trefoil knot (plus trivial components), 
or $L\ge$ right-handed Hopf link (plus trivial components). 
\end{Theorem}

\begin{Theorem}\label{ap2}
 Let $\tilde {L}$ be an almost positive diagram representing a trivial link. 
 Then $\tilde L$ can be reduced to one of the diagrams of Fig. \ref{almost-positive-trivial} by first 
Reidemeister moves (reducing the number of crossings), and deleting trivial circles. 
\end{Theorem}

\begin{figure}[htbp]
      \begin{center}
\scalebox{0.5}{\includegraphics*{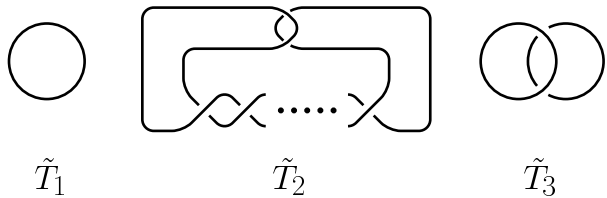}}
      \end{center}
   \caption{Almost positive diagrams representing a trivial link}
  \label{almost-positive-trivial}
\end{figure} 

%

\begin{Corollary}\label{ap3} Let $K$ be a nontrivial almost positive knot. Then $K\succeq$ right-handed trefoil knot. 
\end{Corollary}

\begin{Corollary}\label{ap4}
Nontrivial almost positive links have negative signature.
\end{Corollary}

In Section 4 we analyze 2-almost positive links. In  particular, we prove: 

\begin{Theorem}\label{2ap1}
Let $\tilde L$ be a nontrivial 2-almost positive link.
Then

\begin{enumerate}
\item[(1)]
$L\geq$ right-handed trefoil knot (plus trivial components), or

\item[(2)]
$L\geq$ $6_2$ (Fig. \ref{6_2-Whitehead-3-comp} (a)) (plus trivial components), or

\item[(3)]
$L\geq$ right-handed Hopf link (plus trivial components), or

\item[(4)]
$L\geq$ disjoint or connected sum of right-handed trefoil knot and left handed Hopf link
(plus trivial components), or

\item[(5)]
$L\geq$ Whitehead link (Fig. \ref{6_2-Whitehead-3-comp} (b)) (plus trivial components), or

\item[(6)]
$L\geq$ disjoint and/or connected sum of two right-handed Hopf links and a left-handed
Hopf link (plus trivial components), or

\item[(7)]
$L\geq$ disjoint or connected sum of $(2,4)$-torus link (Fig. \ref{6_2-Whitehead-3-comp} (c)) and a left-handed Hopf link (plus trivial components), or

\item[(8)]
$L\geq$ the link of Fig. \ref{6_2-Whitehead-3-comp} (d) (plus trivial components), or

\item[(9)]
$\tilde L$ is a twist knot with negative clasp (Fig. \ref{negative-twist-knot}) (plus trivial components), or

\item[(10)]
$L$ is a disjoint or connected sum of left-handed Hopf link and $(2,n)$-torus link with anti-parallel orientation of components (Fig. \ref{torus-link2}) (plus trivial components), or

\item[(11)]
$L$ is a left-handed Hopf link (plus trivial components).
\end{enumerate}
\end{Theorem}

\begin{figure}[htbp]
      \begin{center}
\scalebox{0.5}{\includegraphics*{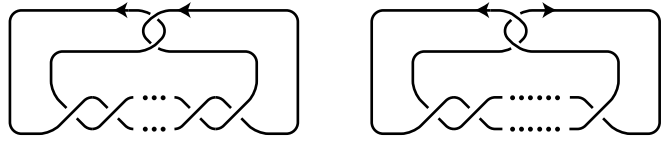}}
      \end{center}
   \caption{Twist knots with negative clasp}
  \label{negative-twist-knot}
\end{figure} 

%

%
\begin{figure}[htbp]
      \begin{center}
\scalebox{0.5}{\includegraphics*{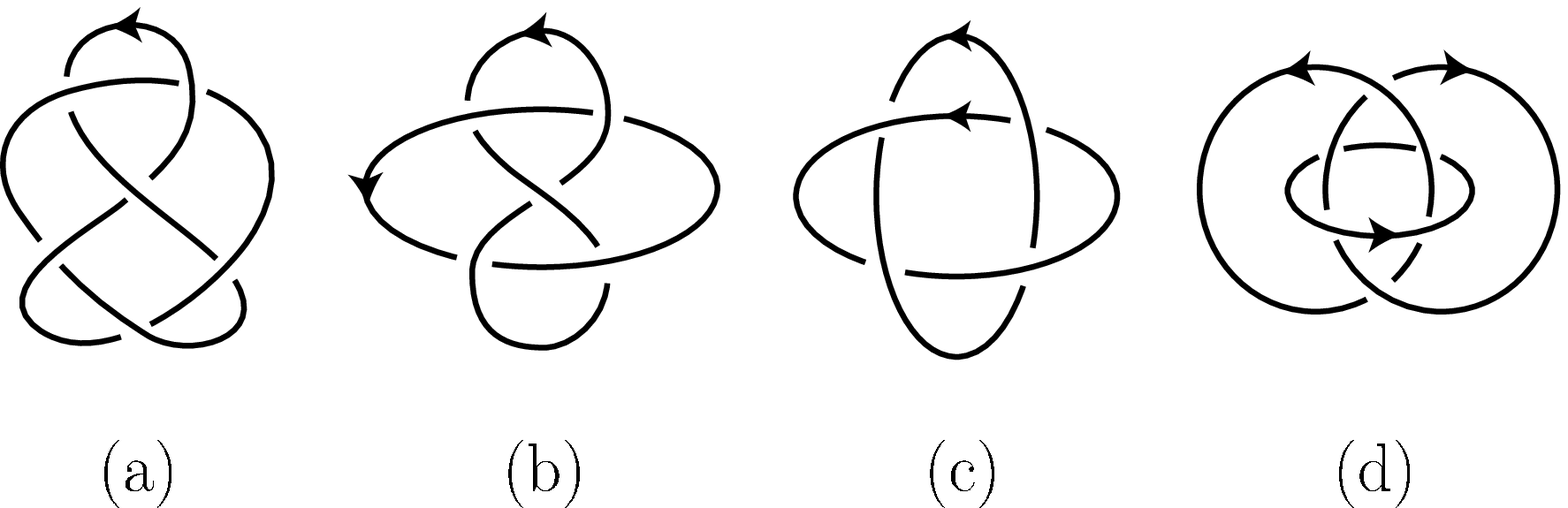}}
      \end{center}
   \caption{}
  \label{6_2-Whitehead-3-comp}
\end{figure} 

%

%
\begin{figure}[htbp]
      \begin{center}
\scalebox{0.5}{\includegraphics*{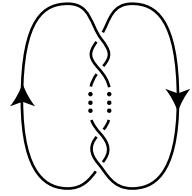}}
      \end{center}
   \caption{$(2,n)$-torus link with anti-parallel orientation of components}
  \label{torus-link2}
\end{figure} 

%

\begin{Theorem}\label{2ap2}
Let $\tilde {L}$ be a 2-almost 
positive diagram (drawn on $S^2=R^2\cup \infty$) representing a trivial link. Then $\tilde L$ 
can be obtained from the diagrams in Fig. \ref{almost-positive-trivial} and Fig. \ref{2ap-trivial},
or their $\pi$-rotation along $y$ axis, by performing some combination of diagram disjoint sum operation, 
diagram connected sum operation, and first and second 
Reidemeister moves which increase the number of crossings.  
\end{Theorem}

\begin{figure}[htbp]
      \begin{center}
\scalebox{0.6}{\includegraphics*{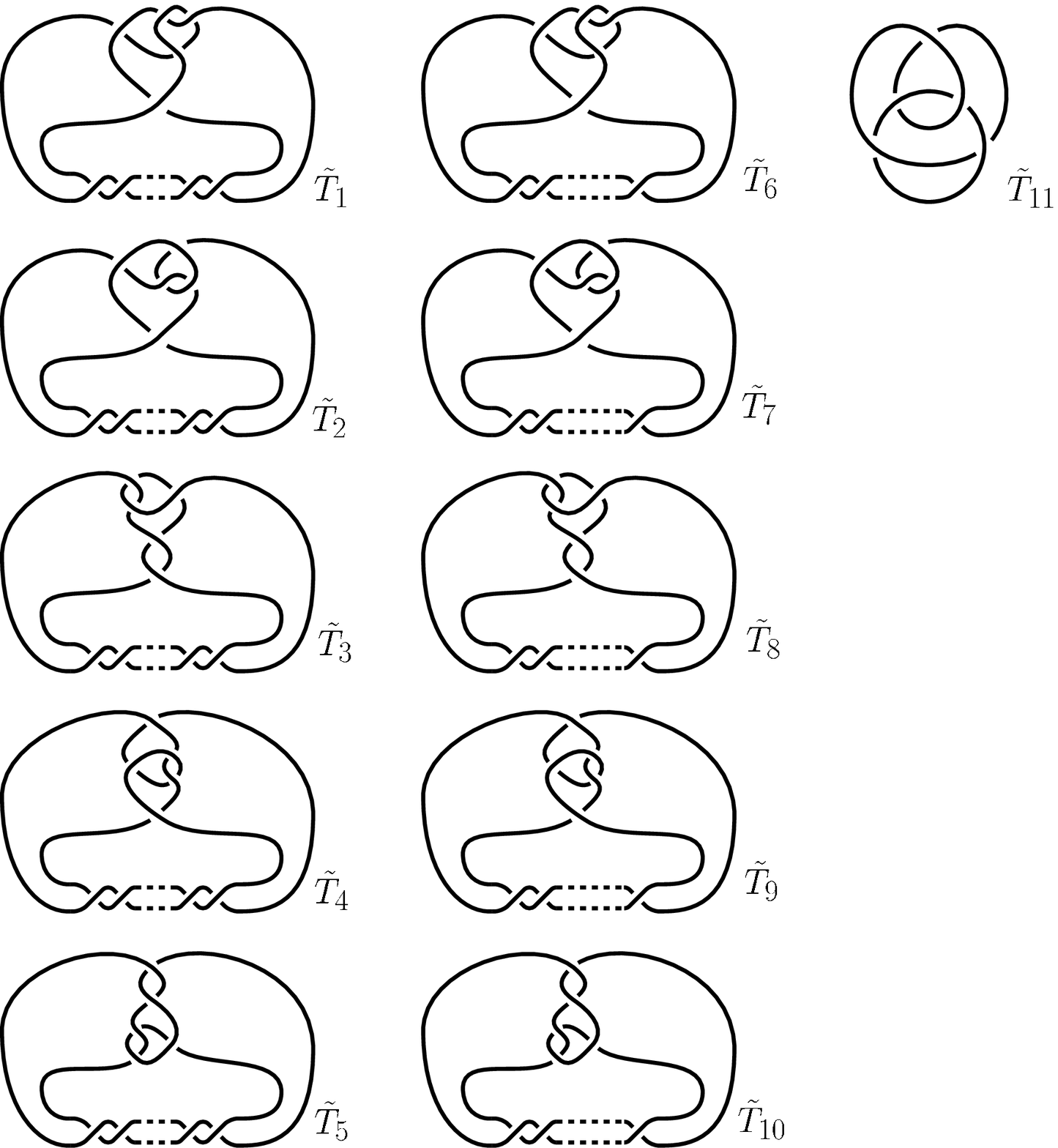}}
      \end{center}
   \caption{}
  \label{2ap-trivial}
\end{figure} 

%

\begin{Corollary}\label{2ap3}
 A nontrivial 2-almost positive link $L$ has nonnegative signature if and only if 
\begin{enumerate}
\item[(1)]
 $L$ is a twist knot with negative clasp (with, possibly, trivial components), in which case  
$\sigma(L)=0$, or 

\item[(2)]
 $L$ is a left-handed Hopf link (with, possibly, trivial 
components),  in which case $\sigma(L)=1$, or 

\item[(3)]
 $L$ is a disjoint or connected sum of left-handed Hopf link 
and $(2,n)$-torus link with anti-parallel orientation of components 
(Fig. \ref{torus-link2}) (with, possibly, trivial components), in which case $\sigma(L)=0$. 
\end{enumerate}
\end{Corollary}

In Section 5 we consider 3-almost positive knots and prove: 

\begin{Theorem}\label{3ap1}
Let $K$ be a 3-almost positive knot. 
Then either $K\geq$ trivial knot or $K$ is the left-handed trefoil knot (plus positive knots as connected summands). 
\end{Theorem}

\begin{Corollary}\label{3ap2}
Let $K$ be a 3-almost positive knot, then 
\begin{enumerate}
\item[(a)] either $\sigma(K) \leq 0$ or $K$ is the left-handed trefoil knot. 
\item[(b)] either $K\succeq$ trivial knot or $K$ is the left-handed trefoil knot (plus positive knots as connected summands). 
\end{enumerate}
\end{Corollary}

In Section 6 we use our previous results to characterize the unknotting number one positive links, 2-almost positive 
amphicheiral links, and 2-almost positive slice links. We also analyze homology 3-spheres not 
bounding a (rational) homology 4-ball, the Tristram-Levine signatures and the Jones polynomial. 
In particular, we prove:

\begin{Theorem}\label{appl1}
\begin{enumerate}
\item[(a)] If a positive knot has unknotting number one then it is a positive twist knot.
\item[(b)] If $K$ is a $2$-almost positive knot different from a twist knot with a negative clasp 
then $K(1/n)$ (i.e. $1/n$ surgery on
  $K$, $n>0$) is a homology 3-sphere that does not bound a compact, smooth homology $4$-ball. 
Furthermore, $K(1/n)$ has a nontrivial Floer homology.
\item[(c)] If $K$ is a non-trivial 2-almost positive knot different from the stevedore's knot then $K$ is not a slice knot.
\item[(d)] If $K$ is a non-trivial 2-almost positive knot different from the figure eight knot then $K$ is not amphicheiral.
\end{enumerate}
\end{Theorem}


\section{Positive links }\label{Section 1}

It has been proved in \cite{Co-G} and \cite{P-1} that if $K$ is a nontrivial 
positive knot then $K\geq$ right-handed trefoil knot. 
We generalize this result in this and the next sections.

\subsection{Preliminary terminology}

If we ignore over/under crossing information of a link diagram $\tilde L$, 
then we call it the underlying {\it projection} (or {\it universe}; 
\cite{Ka-1}) of $\tilde L$, and denote it by $\hat L$.
%
In general $\tilde X$ denotes some part of a diagram and $\hat X$ denotes its underlying projection.

 A tangle is composed of properly embedded (oriented) system of arcs in 
a 3-ball ${\mathbb B}^3$. We consider tangles up to ambient isotopy of ${\mathbb B}^3$ 
which is fixed on $\partial  {\mathbb B}^3$. For oriented tangles 
$T_{1}$ and $T_{2}$, we define the relation $\geq$ where
 $T_{1}\geq T_{2}$ if  $T_{2}$ 
can be obtained from $T_{1}$ by changing some positive crossings to 
negative crossings. We consider tangle diagrams
$\tilde T$  and its underlying projection $\hat T$ on the unit disk ${\mathbb B}^2$ up to ambient isotopy which is fixed on 
$\partial {\mathbb B}^2$.
Note that no tangles, tangle diagrams and tangle projections in this 
paper contain closed components unless otherwise stated.
For a link projection $\hat L$ (resp. tangle projection $\hat T$) we denote by ${\rm LINK}(\hat L)$ (resp. ${\rm TANGLE}(\hat T)$) the set of all links (resp. tangles) that has a link diagram (resp. tangle diagram) whose underlying projection is $\hat L$ (resp. $\hat T$).
For two link (resp. tangle) projections $\hat L_1$ and $\hat L_2$ (resp. $\hat T_1$ and $\hat T_2$), 
we define the relation $\hat L_{1}\geq \hat L_{2}$ (resp. $\hat T_{1}\geq \hat T_{2}$) if 
${\rm LINK}(\hat L_1)\supset {\rm LINK}(\hat L_2)$ (resp. ${\rm TANGLE}(\hat T_1)\supset{\rm TANGLE}(\hat T_2)$). 
Then we say that $\hat L_2$ (resp. $\hat T_2$) is a {\it minor of} $\hat L_1$ (resp. $\hat T_1$).

An oriented tangle diagram is {\it m-almost positive} if all its crossings, but exactly $m$, are positive.

A crossing $P$ of $\tilde L$ (resp.\footnote{
Most of the definition we give below apply to the case of a link 
diagram $\tilde L$, its projection
$\hat L$, and a tangle diagram $\tilde T$ and its projection $\hat T$. 
To omit cumbersome notation/repetitions we
will list only one or two of them in definitions, 
unless description differs, in which case we list the difference.}
 $\hat L$, $\tilde T$, 
$\hat T$) is called {\it nugatory} if
$\hat L-P$
(resp. $(\partial {\mathbb B}^2\cup \hat T)-P$) has more components 
than $\hat L$ (resp. $\partial {\mathbb B}^2\cup \hat T$).
A link (resp. tangle) diagram without nugatory 
crossings is called {\it reduced}.

A link (resp. tangle) diagram is called {\it R2-reduced} if it is 
reduced and it does not contain a 2-gon; as illustrated in Fig. \ref{R2}

\begin{figure}[htbp]
      \begin{center}
\scalebox{0.48}{\includegraphics*{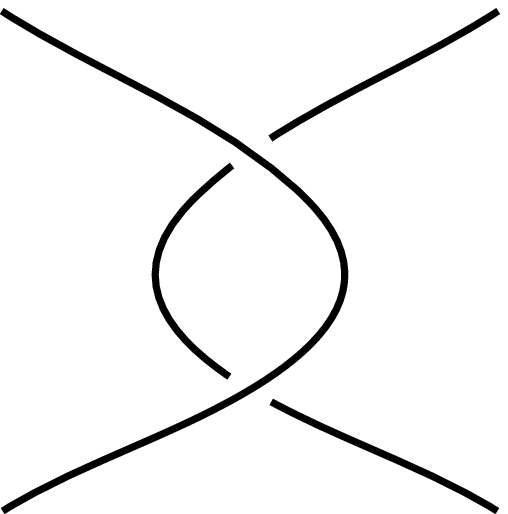}}
      \end{center}
   \caption{}
  \label{R2}
\end{figure} 

%

We say that a link  (resp. tangle) diagram is {\it prime} if it is connected and any simple closed 
curve in ${\mathbb S}^2$ (resp. ${\mathbb B}^2$)
meeting it transversally in two points bounds a trivial disk pair. 
We remark that prime link projections are reduced except a projection of a trivial knot with one crossing.

Let $\hat L$ be a link projection that is not prime. 
Let $\gamma$ be a simple closed curve on ${\mathbb S}^2$ meeting it transversally in two points, 
say $P$ and $Q$, that does not bound trivial disk pair. Let $D_1$ and $D_2$ 
be disks in ${\mathbb S}^2$ bounded by $\gamma$. Let $\alpha$ be a simple arc in 
$\gamma$ joining $P$ and $Q$. Let $\hat L_1=(\hat L\cap D_1)\cup\alpha$ and  $\hat L_2=(\hat L\cap D_2)\cup\alpha$. 
Then we say that $\hat L$ decomposes to $\hat L_1$ and $\hat L_2$. We also say that $\hat L$ is 
a {\it connected sum of $\hat L_1$ and $\hat L_2$}. See Fig. \ref{composite-projection}.
\begin{figure}[htbp]
      \begin{center}
\scalebox{0.58}{\includegraphics*{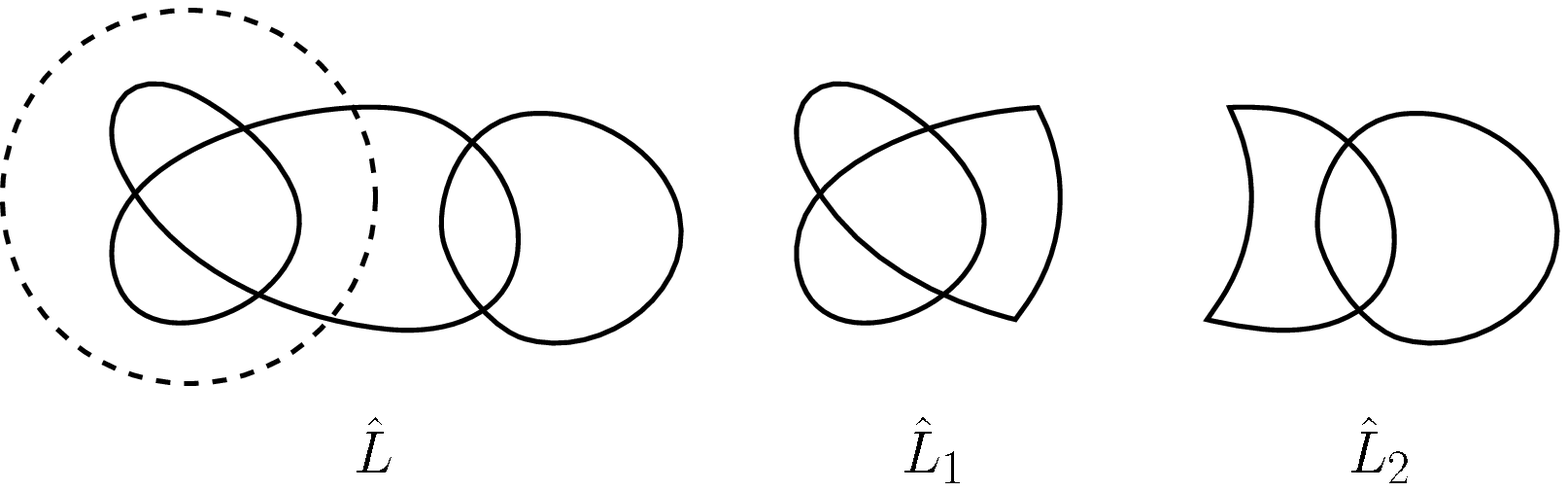}}
      \end{center}
   \caption{}
  \label{composite-projection}
\end{figure} 

%

A {\it self-crossing} of a link (resp. tangle) diagram 
is a crossing involving a single component. 
A {\it mixed crossing} is a crossing between two different components.

We say that two mixed crossings on a component are {\it successive} on 
the component if there is an arc in the component connecting the crossings without mixed crossings on the arc.

A self-crossing $P$ of an arc  component $\tilde t$ of a 
tangle diagram $\tilde T$ divides $\tilde t$ into three
immersed arcs. Let $r(\tilde t,P)$ be the diagram 
obtained from $\tilde t$ by removing the middle immersed arc.\footnote{That is, we smooth 
the crossing $P$ and delete the closed component.}
Let $s(\tilde t,P)$  be the middle part so that 
$\tilde t=r(\tilde t,P)\cup s(\tilde t,P)$.
Let $r(\tilde T,P)$ be the diagram 
obtained from $\tilde T$ by replacing $\tilde t$ by $r(\tilde t,P)$.
Similarly, a self-crossing $P$ of a 
component $\tilde\ell$ of a link diagram $\tilde L$ 
divides $\tilde\ell$  into two immersed arcs. 
We denote one of them by $r(\tilde\ell,P)$ and the other by $s(\tilde\ell,P)$. 
In a similar manner we define $r(\tilde L,P)$ and $s(\tilde L,P)$.

A component $\tilde\ell$ (resp. $\tilde t$) of a 
link (resp. tangle) diagram is said to be {\it almost trivial} if 
 $\hat\ell-P$ (resp. $\hat t-P$) is not connected
for any self-crossing $P$ of $\hat\ell$ 
(resp. $\hat t$).\footnote{That is, a knot $\tilde \ell$ (resp. an arc $\tilde t$)  has only nugatory crossings.}

For a component $\tilde t$ of a tangle 
diagram $\tilde T$ and a component (arc)  
$\alpha$ of $\partial {\mathbb B}^2-\partial\tilde t$, we say that a point $P$
on $\tilde t$ is {\it outermost with respect to} $\alpha$ if
 $\alpha$ and $P$ belong to the same closure of a component of
${\mathbb B}^2-\tilde t$. 
When $\alpha$ contains the point $(1,0)$ (resp. $(-1,0)$) we say that $\tilde t$ is rightmost (resp. leftmost).

For a point $A$ on $\tilde t$  and a component 
$\alpha$ of $\partial {\mathbb B}^2- \partial\tilde t$, we 
define the {\it depth} of $A$ on $\tilde t$, 
$d(A)=d(A,\alpha)$, to be the minimal number
of the transverse intersection points of $\hat t$ and an arc 
joining $\alpha$ and $A$ (avoiding crossing points) which is in general position with respect 
to $\hat t$. In particular $d(A)=0$ if $A$ is outermost.
When $\alpha$ contains the point $(1,0)$ (resp. $(-1,0)$) we denote $d(A,\alpha)=d(A,(1,0))$ (resp. $d(A,\alpha)=d(A,(-1,0))$) and call it the {\it right depth} (resp. {\it left depth}) of $A$ on $\tilde t$.
%

For an oriented arc component  $\hat t$ of a tangle diagram $\hat T$, we introduce inductively 
the following terminology. Let $\Gamma_{0}$ be the set of all self-crossings of $\hat t$. Set $\hat t_0=\hat t$. 
We trace $\hat t_0$ along its orientation starting from its first end point. Let $P_1$ be the first self-crossing 
of $\hat t_0$ we encounter. Let $\hat t_1=r(\hat t_0,P_1)$. 
Then we trace $\hat t_1$ starting from $P_1$ along its orientation. 
Let $P_2$ be the first self-crossing of $\hat t_1$ we encounter. 
Let $\hat t_2=r(\hat t_1,P_2)$. We continue this and finally have $\hat t_{n_1}$ that has no self-crossings. 
We set $\hat t'=\hat t_{n_1}$ and call it the {\it spine} of $\hat t$.

Let $\Gamma_{1}$ be  $\Gamma_{0}- \hat t'$. Namely  $\Gamma_{1}$ is the set of self-crossings of 
$\hat t$ that are not on $\hat t'$.

Set $\hat t'_0=\hat t$ again. We trace again $\hat t'_0$ along its orientation starting from its first end point. 
Let $P'_1$ be the first self-crossing in $\Gamma_{1}$ we encounter. Let $\hat t'_1=r(\hat t'_0,P'_1)$. 
Then we trace $\hat t'_1$ starting from $P'_1$ along its orientation. Let $P'_2$ be the first self-crossing of $\hat t'_1$ in $\Gamma_{1}$ we encounter. Let $\hat t'_2=r(\hat t'_1,P'_2)$. We continue this and finally have $\hat t'_{n_2}$ whose all 
self-crossings are on $\hat t'$. We set $\hat t''=\hat t_{n_2}$ and call it the {\it second spine of} $\hat t$.

We define inductively the $n$-th spine $\hat t^{(n)}$ in a similar manner. 
In particular we define $\Gamma_{n}= \Gamma_{n-1} - \hat t^{(n)}$. 
Of course $\hat t^{(n-1)} \subset \hat t^{(n)}$.

See for example Fig. \ref{spine}.

\begin{figure}[htbp]
      \begin{center}
\scalebox{0.58}{\includegraphics*{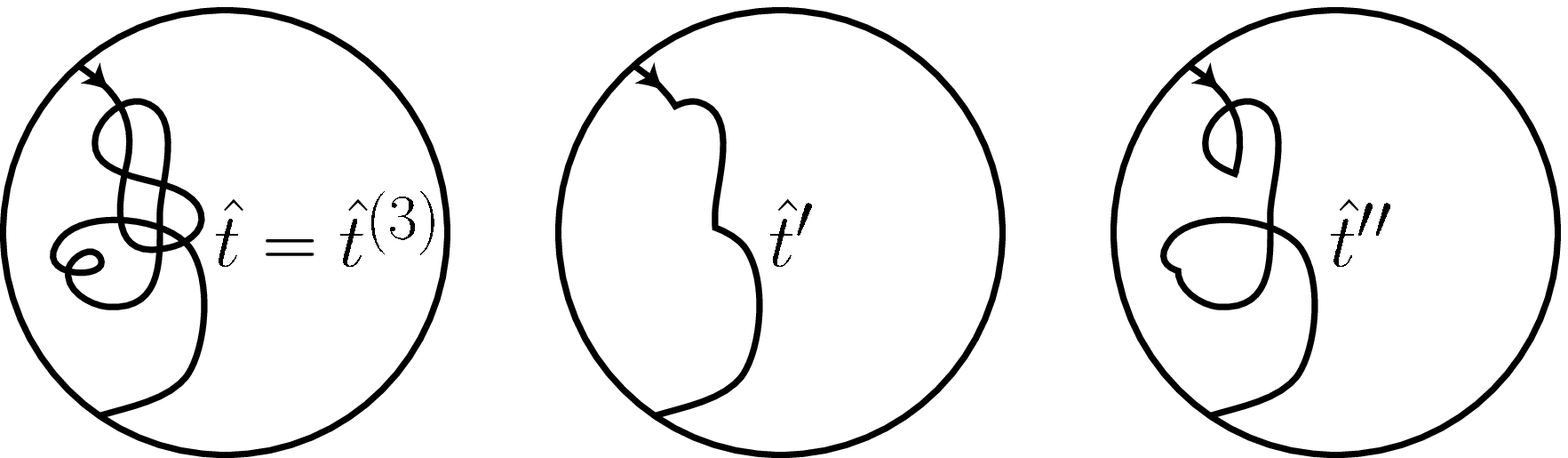}}
      \end{center}
   \caption{}
  \label{spine}
\end{figure} 

%

%
For a point $A$ in an almost trivial arc component $\hat t$, we define the {\it multiplicity} of
$A$ on $\hat t$, $m(A)=m(A,\hat t)$ to be
the minimal number $n$ such that $A$ is contained in $\hat t^{(n+1)}$.
A self-crossing $P$ of $\hat t$ is called the {\it root} of $A$ if $P$ is on $\hat t^{(m(A))}$ and $A$ is on $s(\hat t,P)$. Note that every $A$ with $m(A)\geq1$ has a unique root.
For points $A$ and $B$ in $\hat t$, we say that  $A$ and $B$ 
are {\it related on}
 $\hat t$ if $m(A)=m(B)$ and $A$ and $B$ belong to the same 
component of $\hat t-(\Gamma_0-\Gamma_{m(A)})$. Note that if $m(A)=m(B)\geq1$ then they have the common root $P$. See for example Fig. \ref{related-points}.
\begin{figure}[htbp]
      \begin{center}
\scalebox{0.58}{\includegraphics*{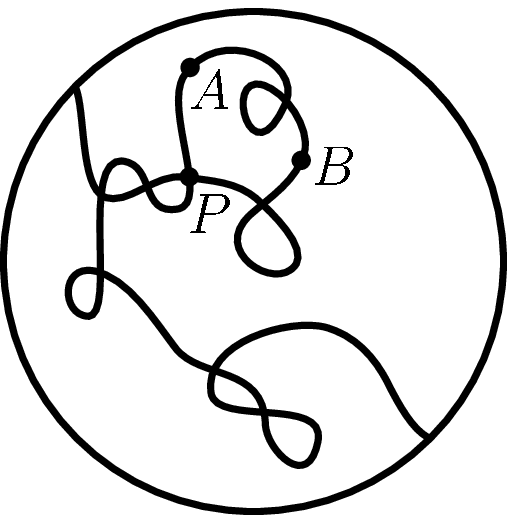}}
      \end{center}
   \caption{}
  \label{related-points}
\end{figure} 

%


We say that a link diagram $\tilde L_{1}$ (resp. $\tilde T_{1}$) is an 
{\it R1 augmentation} of a link diagram $\tilde L_{0}$ (resp. $\tilde T_{0}$) if
$\hat L_{1}$ (resp $\hat T_{1}$) can be obtained from 
$\hat L_{0}$ (resp $\hat T_{0}$) by a finite sequence of the first Reidemeister moves increasing the number of crossings as illustrated in Fig. \ref{R1-move}.

\begin{figure}[htbp]
      \begin{center}
\scalebox{0.58}{\includegraphics*{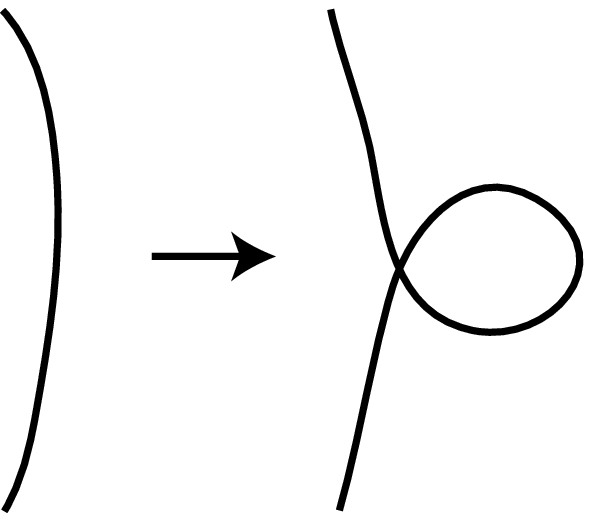}}
      \end{center}
   \caption{}
  \label{R1-move}
\end{figure} 

%

If $\tilde L_{1}$ is an R1 augmentation of
$\tilde L_{0}$ then $\hat L_{0}$ is a subspace of  $\hat L_{1}$
and we say that $\hat L_{1}-\hat L_{0}$ is 
an {\it R1 residual} of $\hat L_{1}$ with respect to $\hat L_{0}$, 
and $\hat L_{0}$ is a {\it core} of $\hat L_{1}$.

Throughout the paper we use the notation introduced in Fig. \ref{tangle-connection} 
for a 2-string tangle diagram (resp. projection). 
Namely the endpoints of the strings are 
denoted by $A_{0}, A_{\infty}, B_{0}$ and $B_{\infty}$, 
and the arc component $A_{0}A_{\infty}$ is denoted
by $\tilde a$ (resp $\hat a$), and the arc component
$B_{0}B_{\infty}$ is denoted
by $\tilde b$ (resp $\hat b$).
Recall that $\tilde x$ denotes a diagram and $\hat x$ denotes a projection.
Let $\tilde a'$ be the spine of $\tilde a$ and $\tilde b'$ the spine of $\tilde b$.
We call $\tilde T'=\tilde a'\cup\tilde b'$ the {\it spine of $\tilde T=\tilde a\cup\tilde b$}.

\begin{figure}[htbp]
      \begin{center}
\scalebox{0.58}{\includegraphics*{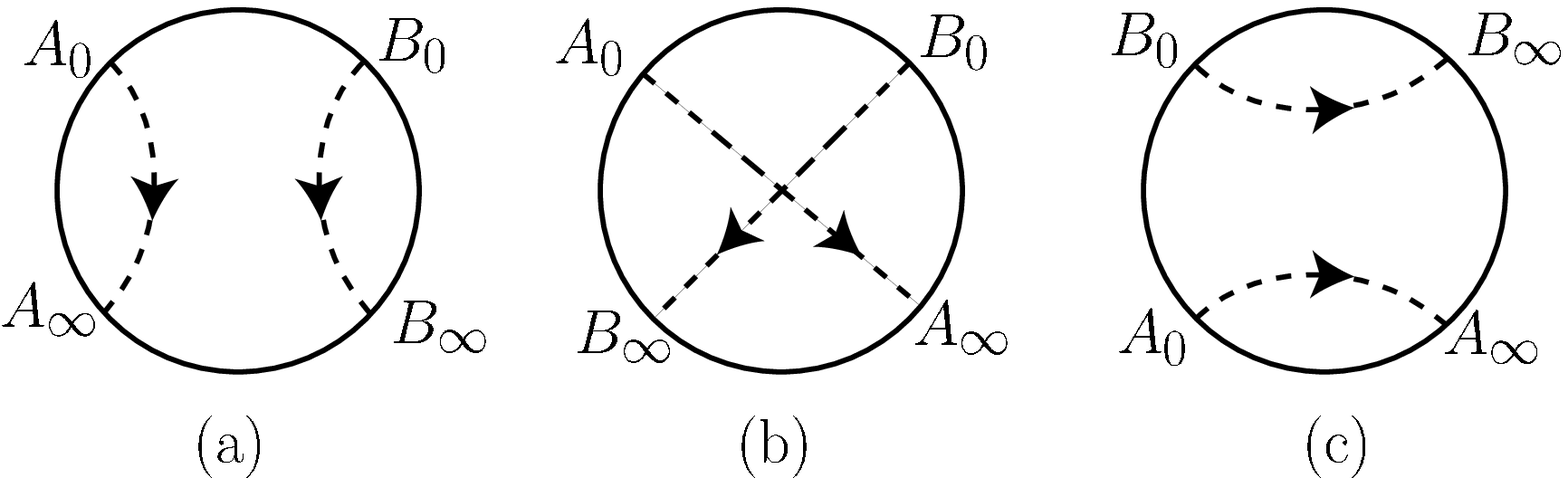}}
      \end{center}
   \caption{}
  \label{tangle-connection}
\end{figure} 

%

We say that the tangle diagram (resp. projection) of Fig. \ref{tangle-connection} (a) 
has a {\it vertical connection} and of Fig. \ref{tangle-connection} (b) has an {\it X-connection}. 
We also sometimes use the tangle diagram (resp. projection) of Fig. \ref{tangle-connection} (c) and say it has a {\it horizontal connection}.
The string is always oriented from $A_{0}$ to $A_{\infty}$ and from $B_{0}$ to $B_{\infty}$ unless otherwise stated.

We denote by $A_{1},A_{2},
\ldots ,A_{n}$ the crossings of $\tilde a$ with $\tilde b$. We order crossings according to orientation of $\tilde a$. 
Similarly, we denote by $B_{1},B_{2},\ldots ,B_{n}$ the crossings of $\tilde b$ with $\tilde a$  
when travelling from $B_{0}$ to $B_{\infty}$. 
The permutation $(\sigma (1), \sigma (2),\ldots ,\sigma (n))$ of 
$(1,2,\ldots ,n)$ is defined by $A_{i}= B_{\sigma (i)}$ for each $i$.
By $A_i^-$ (resp. $A_i^+$) we denote a point on $\hat a$ or $\tilde a$ just before (resp. after) the crossing $A_i$ with respect to the orientation of $\hat a$ or $\tilde a$. Note that both $A_i^-$ and $A_i^+$ are not crossings of the projection or diagram. The points $B_i^-$ and $B_i^+$ on $\hat b$ or $\tilde b$ are defined analogously.
For points $P$ and $Q$ on a tangle projection $\hat T$ that are possibly but
 not necessarily crossings we denote by $PQ$ the immersed arc in $\hat T$ starting 
from $P$ and ending at $Q$ with respect to the orientation of $\hat T$. 
In the case that both of them are mixed crossings there are two such possibilities. 
Then we specify the component. In case that both of them are self-crossings of the same component 
there may be two such possibilities. Then we take another point $R$ and denote it by $PRQ$. 
See Fig. \ref{arc-notation}.

\begin{figure}[htbp]
      \begin{center}
\scalebox{0.58}{\includegraphics*{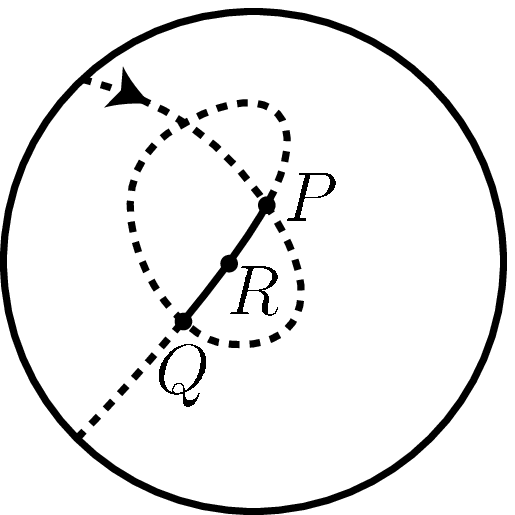}}
      \end{center}
   \caption{}
  \label{arc-notation}
\end{figure} 

%

%
The {\it vertical symmetry} of a 2-string tangle 
diagram is obtained by $\pi$-rotation along the vertical $y$-axis through the center of ${\mathbb B}^2$. Similarly a {\it horizontal symmetry} is obtained by 
$\pi$-rotation along the horizontal $x$-axis through the center of ${\mathbb B}^2$. 
We also use $\pi/2$- (counter-clockwise) rotation.
We denote the image of the vertical symmetry (resp. horizontal symmetry, $\pi/2$-rotation) of a tangle $T$ by $V(T)$ (resp. $H(T)$, $R(T)$). See Fig. \ref{tangle-symmetry}.
\begin{figure}[htbp]
      \begin{center}
\scalebox{0.48}{\includegraphics*{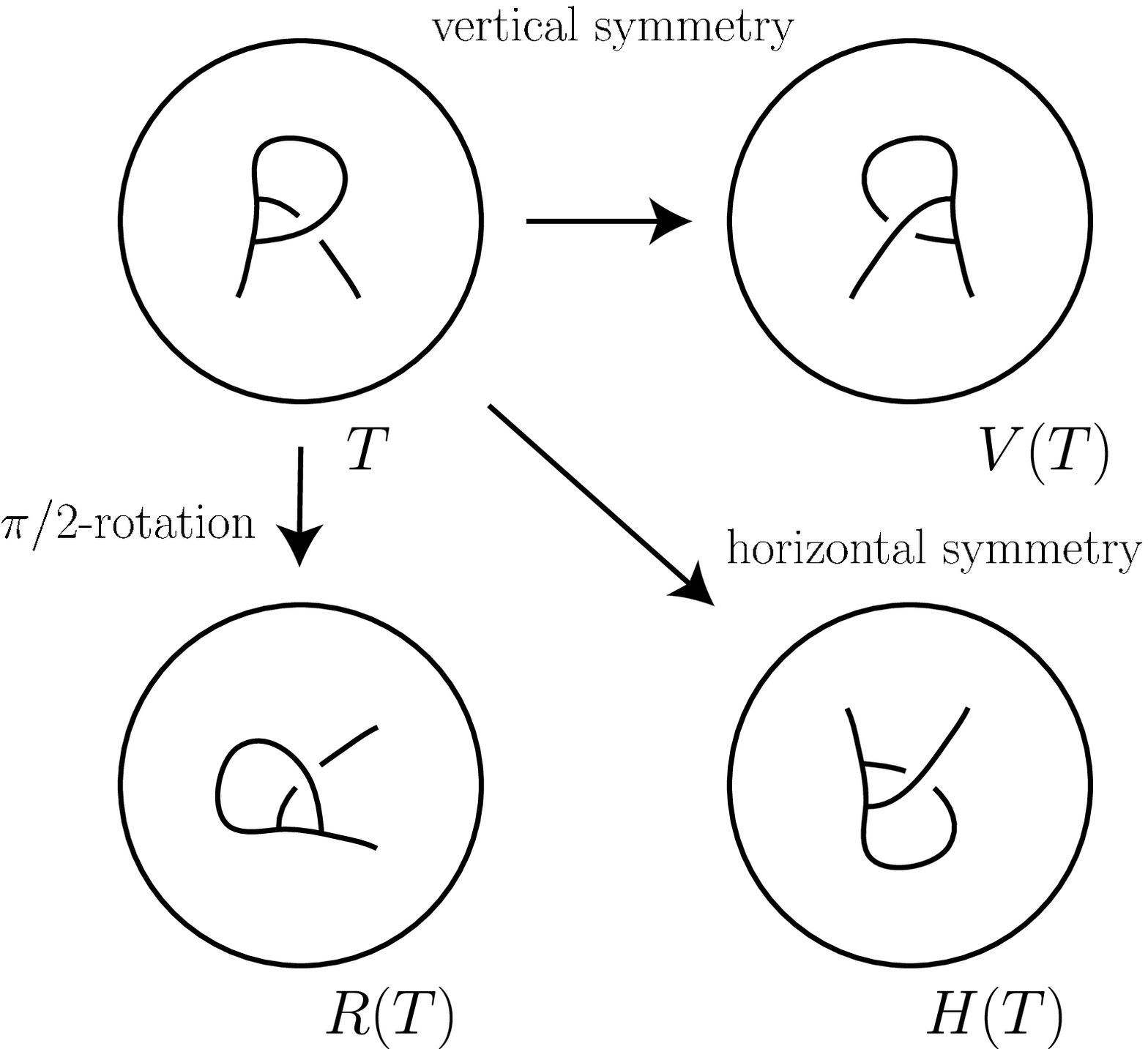}}
      \end{center}
   \caption{}
  \label{tangle-symmetry}
\end{figure} 

%

A {\it flyping} (or {\it Tait flyping}) is a local move of a link or tangle diagram (resp. projection) illustrated in Fig. \ref{flyping} 
where $\tilde S$ (resp. $\hat S$) is any subdiagram (resp. subprojection).

\begin{figure}[htbp]
      \begin{center}
\scalebox{0.48}{\includegraphics*{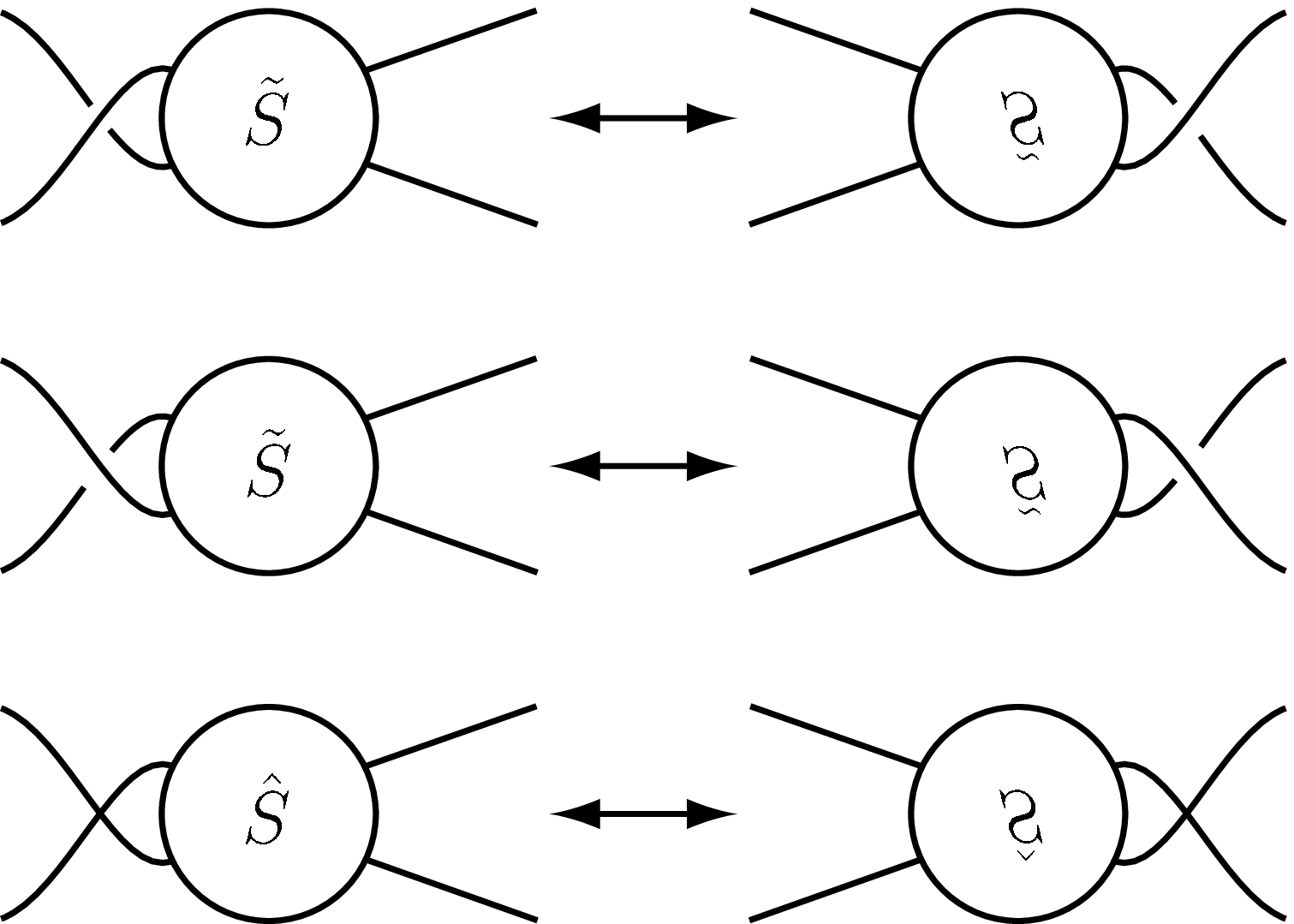}}
      \end{center}
   \caption{}
  \label{flyping}
\end{figure} 

%

If $\tilde L$ (resp. $\tilde T$) is obtained 
from $\tilde L_{0}$ (resp. $\tilde T_{0}$)
by a series of flypings, then we say that 
$\tilde L$ (resp. $\tilde T$) is a {\it flype} of $\tilde L_{0}$
(resp. $\tilde T_{0}$).

Throughout the paper we use notation describing various kinds of twists performed on a diagram 
as illustrated in Fig. \ref{twist-box-diagram1} and Fig. \ref{twist-box-diagram2} 
and on a projection as illustrated in Fig. \ref{twist-box-projection} 
(we code also in the notation the number or character of twists).
Note that the direction of twists in this paper is opposite to the standard Conway's convention. 
We also remark that $n$ is a variable so that even in the same figure $n$ may be different. 
See for example Fig. \ref{twist-box-example}, where a box can denote any positive even number.
\begin{figure}[htbp]
      \begin{center}
\scalebox{0.9}{\includegraphics*{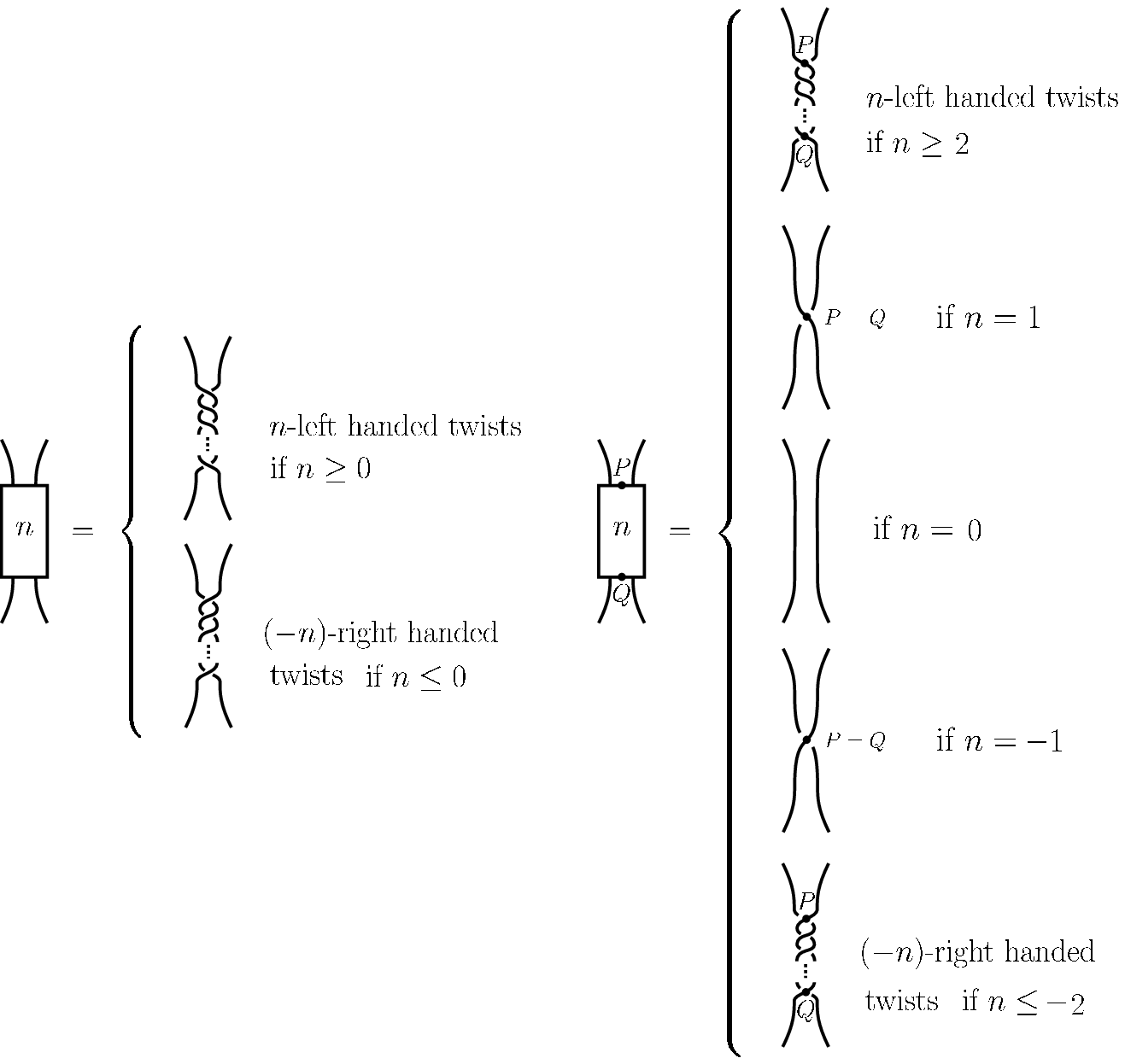}}
      \end{center}
   \caption{}
  \label{twist-box-diagram1}
\end{figure} 

%

%
\begin{figure}[htbp]
      \begin{center}
\scalebox{0.9}{\includegraphics*{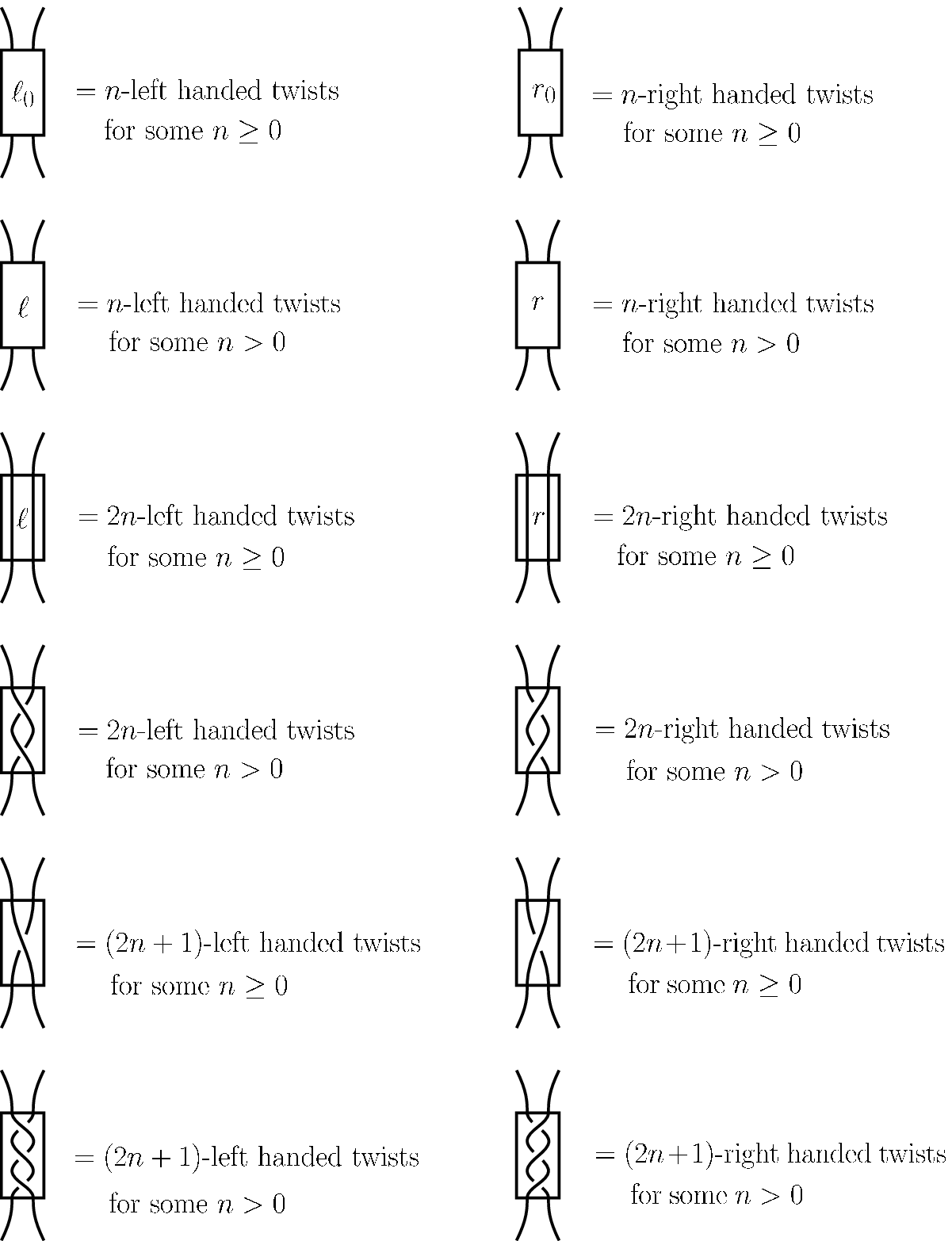}}
      \end{center}
   \caption{}
  \label{twist-box-diagram2}
\end{figure} 

%

%
\begin{figure}[htbp]
      \begin{center}
\scalebox{0.7}{\includegraphics*{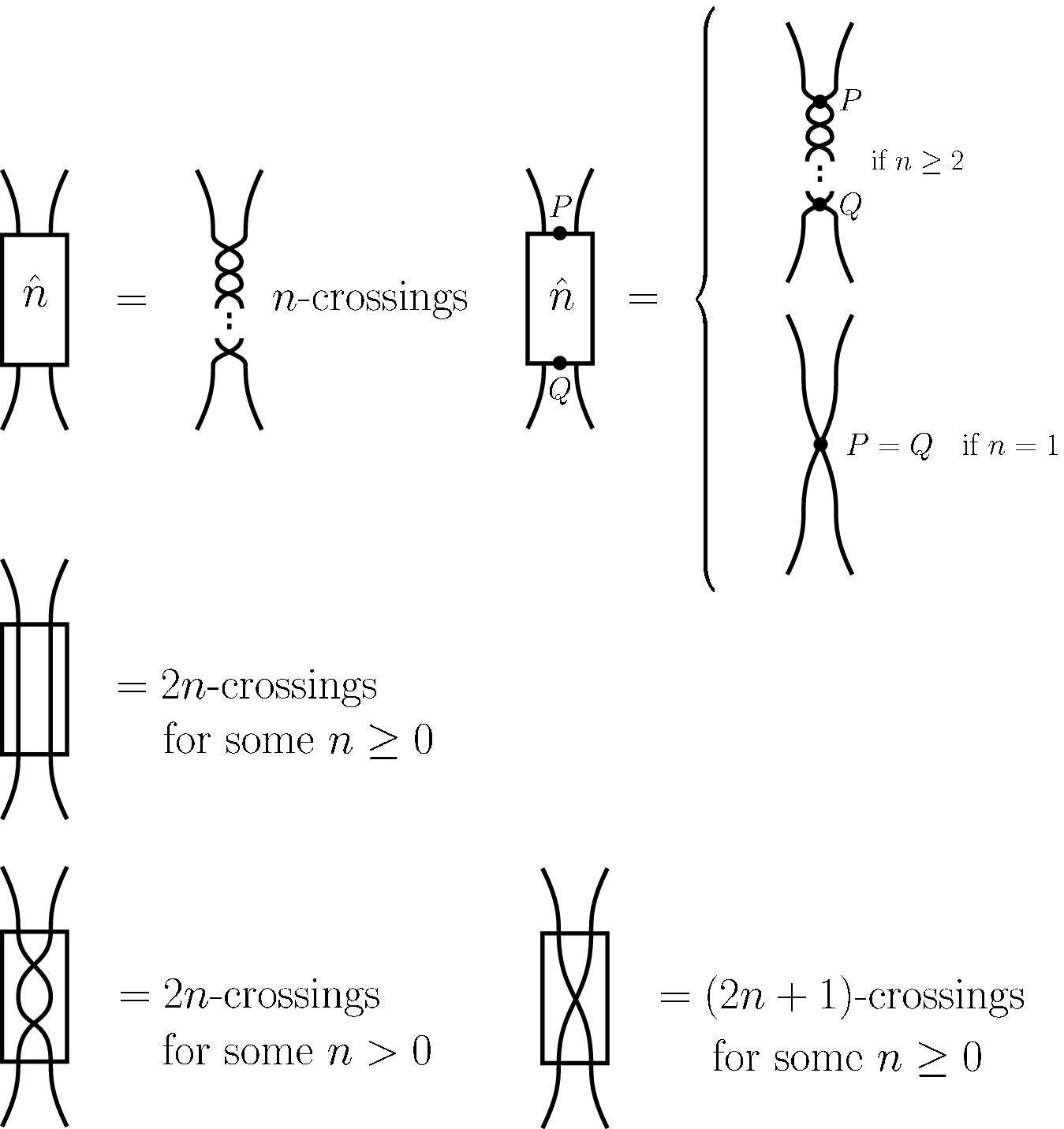}}
      \end{center}
   \caption{}
  \label{twist-box-projection}
\end{figure} 

%

%
\begin{figure}[htbp]
      \begin{center}
\scalebox{0.6}{\includegraphics*{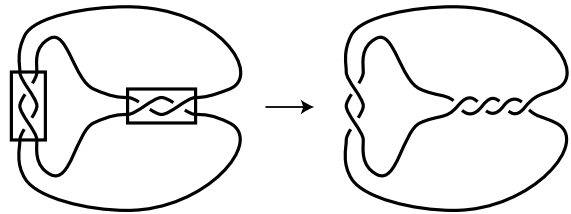}}
      \end{center}
   \caption{\ Every box can be replaced by an even positive number of twists}
  \label{twist-box-example}
\end{figure} 

%

%
By $T_1+T_2$ we denote the sum of tangles $T_1$ and $T_2$ as illustrated in Figure \ref{tangle-notation}. 
Sum of tangle diagrams and sum of tangle projections are similarly defined. 
We also use notations $T(n)$, $T(1/n)$, $\tilde T(n)$, $\tilde T(1/n)$, $\hat T(n)$ and $\hat T(1/n)$ 
where $n$ is an integer as illustrated in Figure \ref{tangle-notation}. 
Note that the notation $T(1/n)$ and $\tilde T(1/n)$ is used only for $n\geq2$ to avoid the confusion 
that $T(1/1)\neq T(1)$ and $\tilde T(1/1)\neq \tilde T(1)$. By $T(a,b)$ we denote the sum $T(a)+T(b)$. 
This notation is based on Conway's tangle notation, but differs from it slightly. 
See Figure \ref{tangle-notation}.
\begin{figure}[htbp]
      \begin{center}
\scalebox{0.58}{\includegraphics*{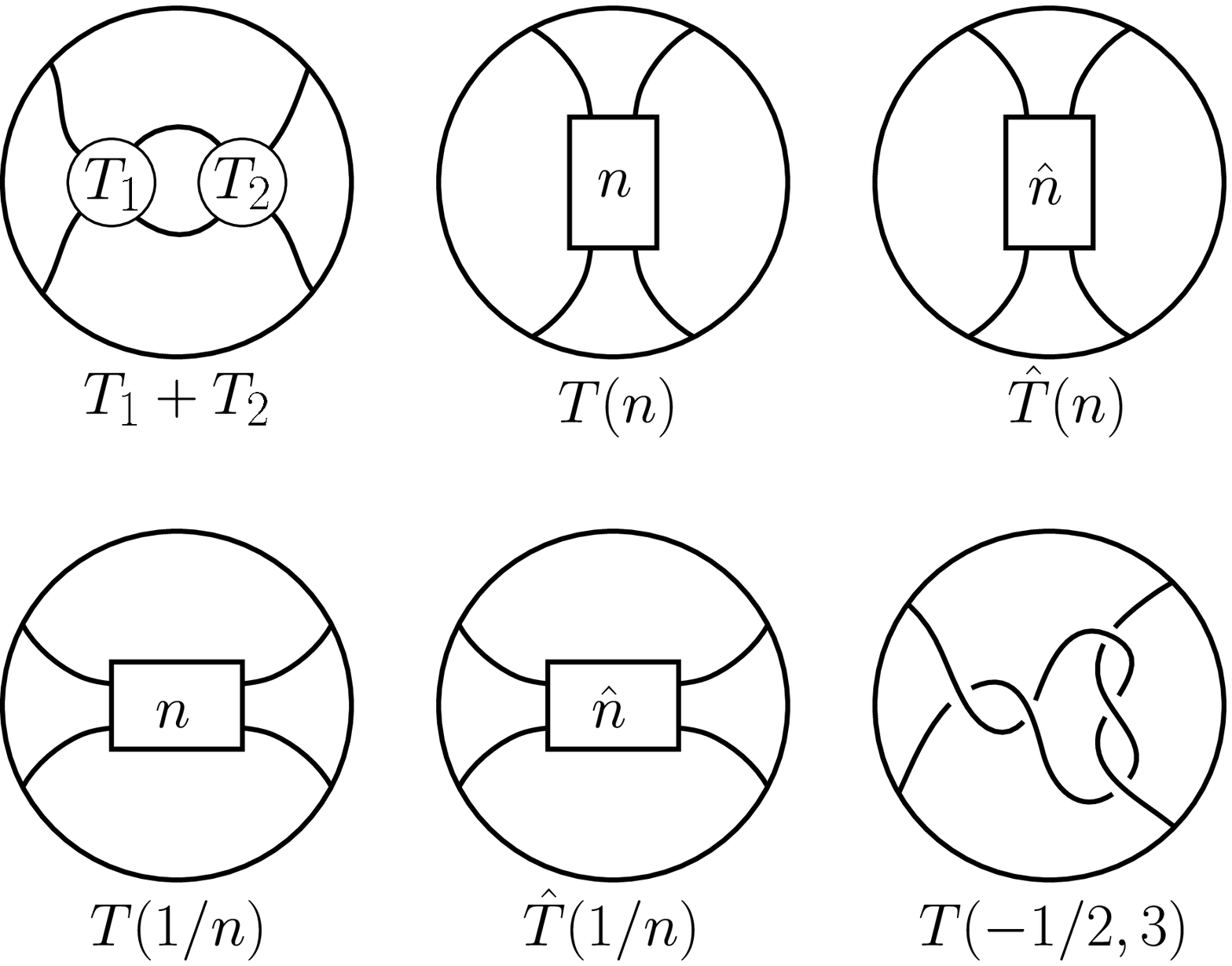}}
      \end{center}
   \caption{}
  \label{tangle-notation}
\end{figure} 

%

For a crossing $P$ of a link diagram $\tilde L$, we say that 
the tangle diagram obtained by deleting a small disk centered at $P$ 
from ${\mathbb S}^2$ is the {\it complementary tangle diagram} 
of $\tilde L$ at $P$.

For a 2-string tangle $T$ we define the {\it denominator} $D(T)$, {\it numerator} $N(T)$, $X_+$-closure $X_+(T)$ and $X_-$-closure $X_-(T)$ as illustrated in Fig. \ref{tangle-closure}.
\begin{figure}[htbp]
      \begin{center}
\scalebox{0.65}{\includegraphics*{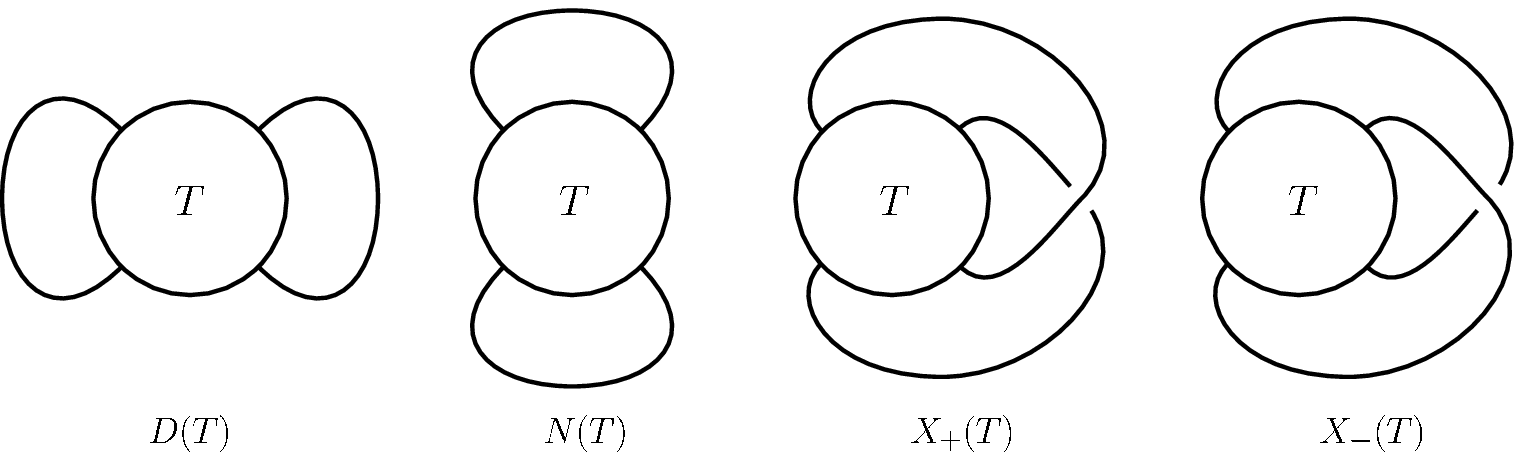}}
      \end{center}
   \caption{}
  \label{tangle-closure}
\end{figure} 

%

Let $\tilde L$ (resp. $\tilde T$) be a positive diagram representing 
a link $L$ (resp. tangle $T$). If the underlying projection  
$\hat L$ (resp. $\hat T$) is an underlying projection of another 
link $L_{0}$ (resp. tangle $T_{0}$), 
then clearly $L\geq L_{0}$ (resp. $T\geq T_{0}$). 
We use this fact throughout the paper without mentioning it explicitly. 
For the convenience we sometimes do not distinguish a link
(resp. tangle) from its diagram and we will write for example that
$\tilde L \geq L_{0}$ meaning that the link represented by $\tilde L$ 
is greater than or equal to $L_{0}$. 

\subsection{Fundamental technique.}\ 
The fundamental technique used in this paper is to apply descending (or ascending) algorithm 
that simplify or trivialize the link or tangle diagrams.
As is well known in knot theory any link projection becomes a trivial link diagram 
if we give over/under crossing information according to descending algorithm. 
Namely we choose an order of components and choose a base point at each component and trace the link 
projection starting from the base point in chosen order. 
We then give over/under crossing information so that we meet each crossing as an over-crossing first time. 
We use descending algorithm to trivialize a part of the diagram often without explicitly mentioning it.

Many of the lemmas in this section follow closely that of \cite{T-1}.
We review them for  completeness as their proofs illustrate our main method.

\begin{Lemma}\label{reducing-lemma}
Let $\hat L$ (resp. $\hat T$) be a link (resp. tangle) projection. 
Let $P$ be a self-crossing of a component $\hat\ell$(resp. $\hat t$). 
Then $\hat L \geq r(\hat L ,P)$
(resp. $\hat T \geq r(\hat T ,P)$).
\end{Lemma}

\begin{proof} 
Let $\tilde L$ be a diagram of a link $L$ whose underlying projection is $r(\hat L ,P)$. Let $\tilde L'$ 
be a diagram whose underlying projection is $\hat L$ with the following crossing information. 
For the crossings of $r(\hat L ,P)$ it is same as $\tilde L$, and for the crossings between 
$s(\hat L,P)$ and  $r(\hat L ,P)$, $s(\hat L,P)$ is always over $r(\hat L,P)$, and for the crossings 
of $s(\hat L,P)$, it is determined by the descending algorithm so that $s(\tilde L,P)$ itself 
is unknotted. Then it is clear that $\tilde L'$ is also a diagram of $L$ as desired. 
The tangle case is entirely the same and we omit it.
\end{proof} 

The following lemma is a diagram version of Lemma 2.1.

\begin{Lemma}\label{diagram-reducing-lemma}
Let $\tilde L$ (resp. $\tilde T$) be a link (resp. tangle) diagram. 
Let $P$ be a self-crossing of a component $\tilde\ell$ (resp. $\tilde t$). 
Suppose that the following conditions (1) and (2) hold.

(1) The knot represented by $s(\tilde \ell,P)$ (resp. $s(\tilde t,P)$) is greater than or equal to the trivial knot.

(2) Either $s(\tilde \ell,P)$ (resp. $s(\tilde t,P)$) is over $r(\tilde \ell,P)$ (resp. $r(\tilde t,P)$) 
at every negative crossings between them, or $s(\tilde \ell,P)$ (resp. $s(\tilde t,P)$) is under 
$r(\tilde \ell,P)$ (resp. $r(\tilde t,P)$) at every negative crossings between them.

Then $\tilde L \geq r(\tilde L ,P)$ (resp. $\tilde T \geq r(\tilde T ,P)$).
\end{Lemma}

\begin{proof} Suppose that $s(\tilde \ell,P)$ is over $r(\tilde \ell,P)$ at every negative crossings between them. Then by changing some positive crossings between them to negative crossings we have a diagram $\tilde L'$ in which $s(\tilde \ell,P)$ is over $r(\tilde \ell,P)$ at every crossings. Note that $\tilde L\geq \tilde L'$. Then we have that $\tilde L'$ is a diagram of a link that is a connected sum of a link represented by $r(\tilde L,P)$ and a knot represented by $s(\tilde \ell,P)$. Then by the condition (1) we have that $\tilde L'\geq r(\tilde L,P)$. The other case is entirely analogous.
\end{proof}

\begin{Lemma}\label{nugatory-erasing}
Let $\hat L_1$ and $\hat L_2$ (resp. $\hat T_1$ and $\hat T_2$) be a link (resp. tangle) projections that differ locally as illustrated in Fig \ref{nugatory} where $P$ is a nugatory crossing of 
$\hat L_1$ (resp. $\hat T_1$) and $R$ is a tangle (resp. subtangle) possibly with some closed components.
Then ${\rm LINK}(\hat L_1)={\rm LINK}(\hat L_2)$ (resp. ${\rm TANGLE}(\hat T_1)={\rm TANGLE}(\hat T_2)$).
\end{Lemma}

\begin{proof}
Let $\tilde L_1$ (resp. $\tilde T_1$) be a diagram whose underlying projection is $\hat L_1$ (resp. $\hat T_1$). 
Then by rotating the tangle (resp. subtangle) $\tilde R$ we eliminate the nugatory crossing $P$ and 
obtain the diagram $\tilde L_2$ (resp. $\tilde T_2$) of the same link (resp. tangle) 
whose underlying projection is $\hat L_2$ (resp. $\hat T_2$). 
Thus we have ${\rm LINK}(\hat L_1)\subset{\rm LINK}(\hat L_2)$ 
(resp. ${\rm TANGLE}(\hat T_1)\subset{\rm TANGLE}(\hat T_2)$). The converse holds in the same way.\ 
We can summarize the proof succinctly: $P$ is a nugatory crossing for any $\tilde L_1$ (resp. $\tilde T_1$) 
with underlying projection $\hat L_1$ (resp. $\hat T_1$).
\end{proof} 

%
\begin{figure}[htbp]
      \begin{center}
\scalebox{0.58}{\includegraphics*{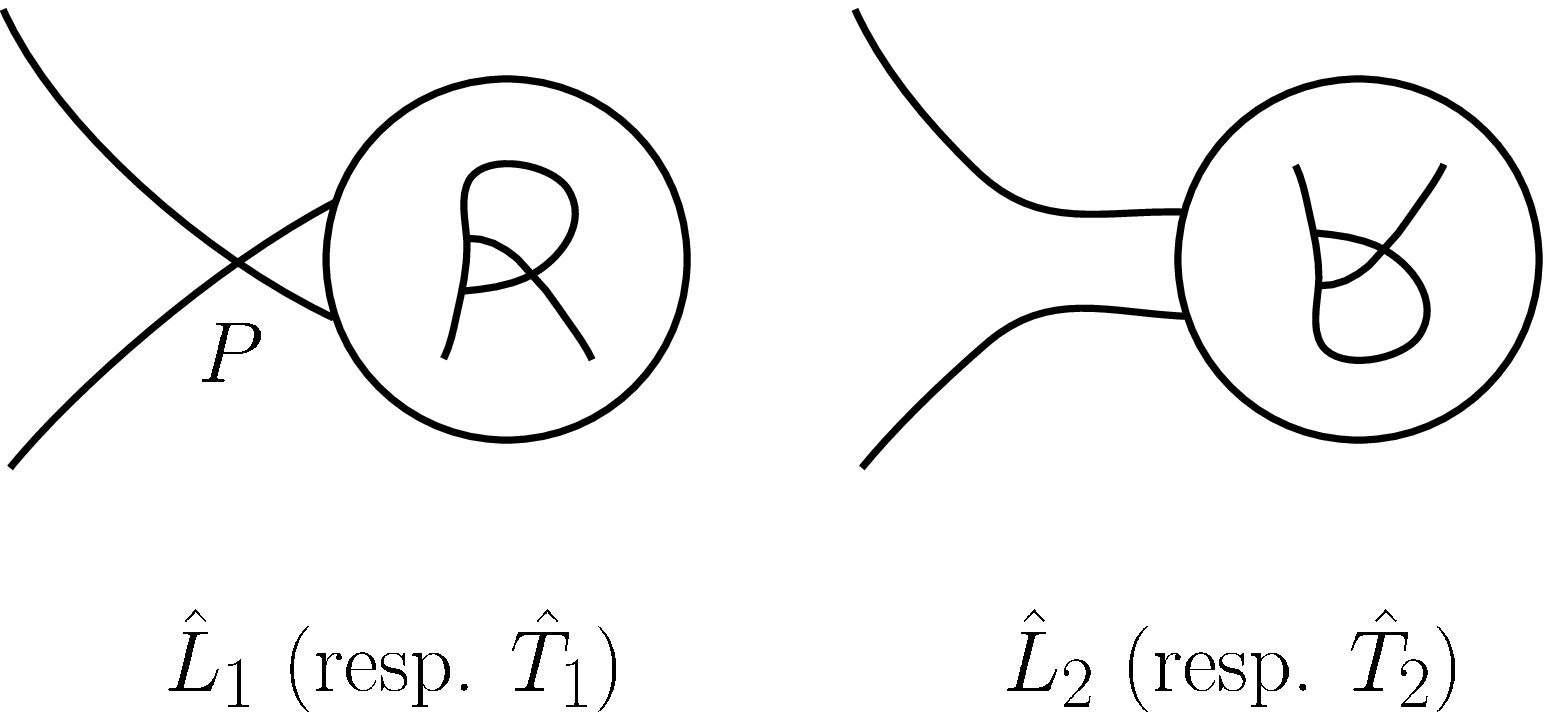}}
      \end{center}
   \caption{}
  \label{nugatory}
\end{figure} 

%


\begin{Lemma}\label{sublemma1}
Let $\hat T_1$ and $\hat T_2$ be tangle projections that differs locally as illustrated in 
Fig. \ref{R2-projections}. Then $\hat T_1$ is greater than or equal to $\hat T_2$.
\end{Lemma}

\begin{figure}[htbp]
      \begin{center}
\scalebox{0.58}{\includegraphics*{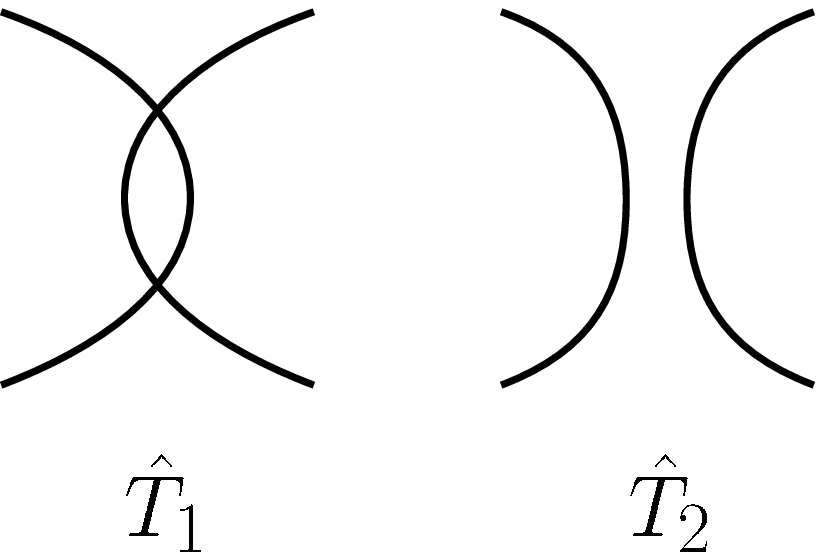}}
      \end{center}
   \caption{}
  \label{R2-projections}
\end{figure} 

%

\begin{proof}
Let $\tilde T_2$ be a diagram whose underlying projection is $\hat T_2$. 
Let $\tilde T_1$ be a diagram whose underlying projection is $\hat T_1$ such that the left string 
is over the right string in Fig. \ref{R2-projections} and other crossings have the same over/under 
crossing information as $\tilde T_2$. Then $\tilde T_1$ and $\tilde T_2$ are transformed into each other 
by a second Reidemeister move. Therefore the result is proved.
\end{proof}

Similarly we have the following lemma.

\begin{Lemma}\label{sublemma1-2}
Let $\hat T_1$ and $\hat T_2$ be tangle projections that differs locally as illustrated in Fig. \ref{R2-projections-2}. Then $\hat T_1$ is greater than or equal to $\hat T_2$.
\end{Lemma}

\begin{figure}[htbp]
      \begin{center}
\scalebox{0.58}{\includegraphics*{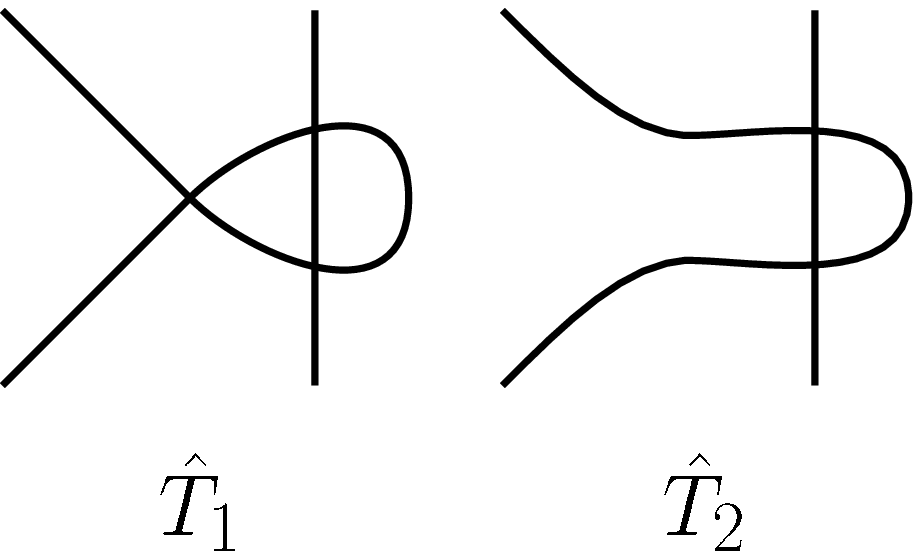}}
      \end{center}
   \caption{}
  \label{R2-projections-2}
\end{figure} 

%

The proof is quite similar and we omit it.


In \cite{P-1} and \cite{T-1}
it is shown that a knot projection $\hat K$
is a projection of a trefoil knot if and only if $\hat K$ is not 
almost trivial, and the link projection $\hat L$ is a projection 
of a Hopf link (plus, possibly, trivial components) if and only if 
$\hat L$ has mixed crossings (compare also \cite{Co-G}). For tangle projections
the analogous results hold:

\begin{Lemma}\label{local-trefoil-lemma}
 A 1-string tangle projection $\hat T$ is a projection of the tangle
in Fig. \ref{local-trefoil} (we call it a local trefoil) if and only if $\hat T$ is not 
an almost trivial tangle.
\end{Lemma}

\begin{figure}[htbp]
      \begin{center}
\scalebox{0.58}{\includegraphics*{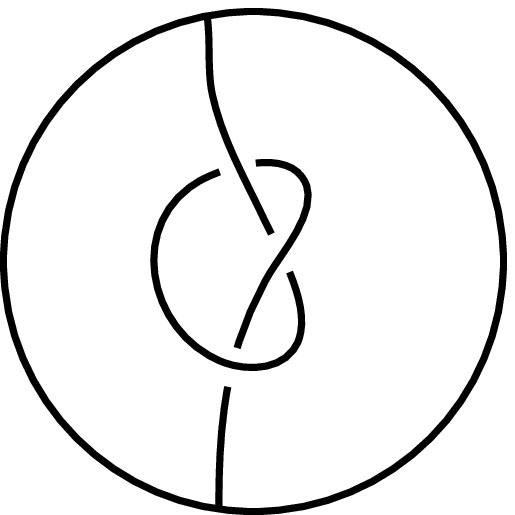}}
      \end{center}
   \caption{}
  \label{local-trefoil}
\end{figure} 

%

%
\begin{proof} 
By Lemma \ref{nugatory-erasing} we may assume that $\hat T$ has no nugatory crossings. 
We trace $\hat T$ starting from its end point and then find a tear drop disk (1-gon) $\delta$ with vertex $P$,  
as illustrated in Fig. \ref{local-trefoil-proof} (a). Then we continue to trace $\hat T$ and since $P$ is not 
a nugatory crossing we return to $\delta$ at a crossing, say $Q$. See Fig. \ref{local-trefoil-proof} (b). 
Then we add over/under crossing information to $\hat T$ by descending algorithm (except $Q$), 
as illustrated in Fig. \ref{local-trefoil-proof} (c).
\end{proof} 
\begin{figure}[htbp]
      \begin{center}
\scalebox{0.58}{\includegraphics*{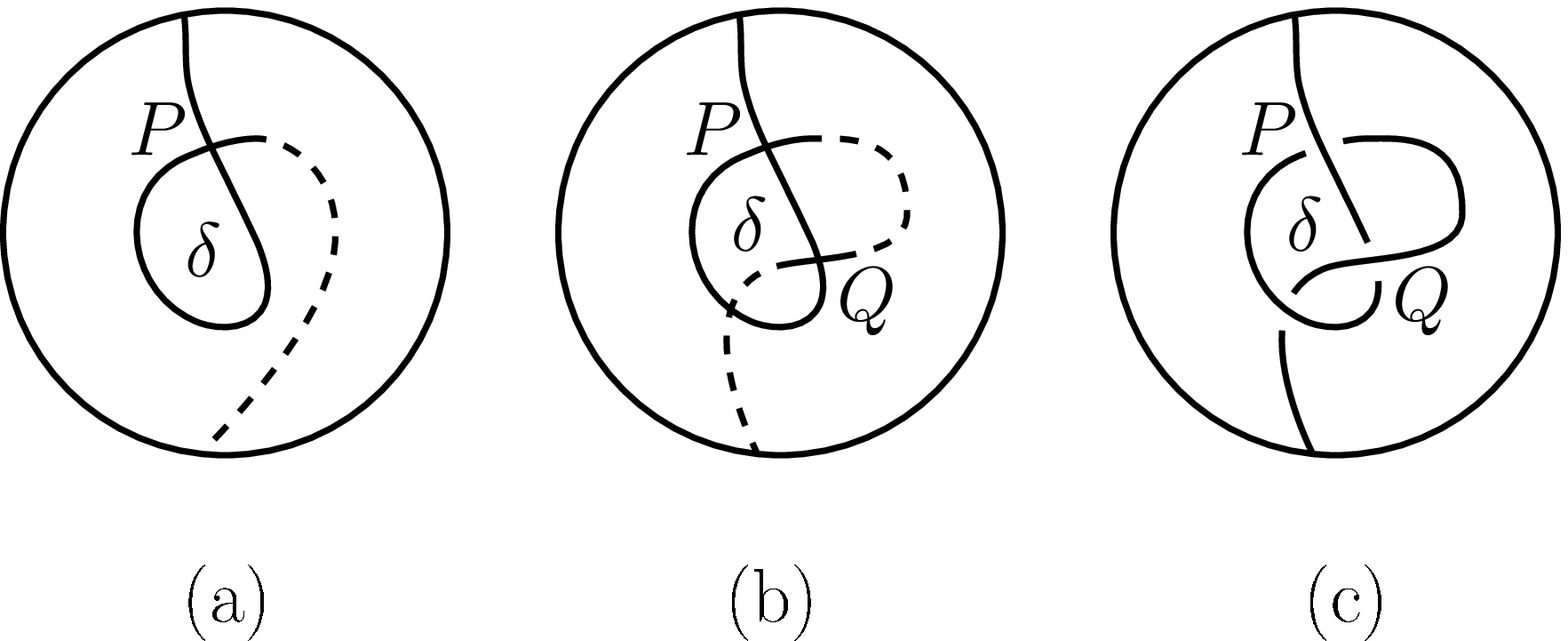}}
      \end{center}
   \caption{}
  \label{local-trefoil-proof}
\end{figure} 

%

\begin{Lemma}\label{hook-tangle}
A 2-string tangle projection $\hat T$ with vertical connection
is a projection of the tangle $T(-2)$ (Fig. \ref{vertical-hook}) (and it mirror image $T(2)$) if and only if 
$\hat T$ has mixed crossings.
\end{Lemma}

\begin{figure}[htbp]
      \begin{center}
\scalebox{0.58}{\includegraphics*{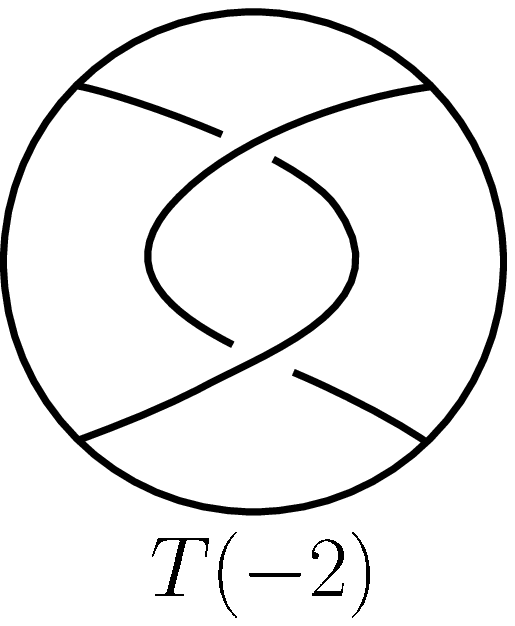}}
      \end{center}
   \caption{}
  \label{vertical-hook}
\end{figure} 

%

%
\begin{proof} Let $A_1$ be the first crossing on $\hat a$ between arcs $\hat a$ and $\hat b$. 
In particular the interior of $A_0A_1$ is disjoint from $\hat b$. We give over/under crossing information to $A_1$ so that 
it becomes a positive crossing. There are two cases depending on which string is over the other at the 
positive crossing $A_1$. We use the notation $A_i^-$ for a point on the arc $\hat a$ just before
the crossing $A_i$ and $A_i^+$ for a point just after the crossing. 
Neither $A_i^-$ nor $A_i^+$ are crossings of $\hat T$. 
We give over/under crossing information to all other crossings of the tangle projection as follows.
\begin{enumerate}
\item[(i)] 
If $\hat a$ is made to be under $\hat b$ at $A_1$ then we add over/under crossing information to all 
other crossings between arcs $\hat a$ and $\hat b$ in such a way that $\tilde a$ is over $\tilde b$ 
furthermore $\tilde a$ and $\tilde b$ are descending. 
The key point here is that  $a$ and $b$ can be now deformed in two steps: first we deform $a$ so it looks like on 
Fig. \ref{vertical-hook-proof-1}(b), in particular,  $a$ has only two crossings with $b$, 
one is $A_1$ and the other is near $A_1$. 
Then we deform $b$ to obtain the tangle $T(-2)$.  See Fig. \ref{vertical-hook-proof-1}.

\item[(ii)] If $\hat a$ is over $\hat b$ at $A_1$ then we add over/under crossing information to all 
other crossings so that $A_1^+A_\infty$ is under $\hat b$ and both $\tilde a$ and $\tilde b$ are descending. 
The key point here is that we first deform the string $a$ Fig. \ref{vertical-hook-proof-1}(b),
  so that $a$ has only two crossings with $b$, one is $A_1$ and the other is near $A_1$. 
Then we deform $b$ to obtain the tangle $T(-2)$. See Fig. \ref{vertical-hook-proof-2}.
\end{enumerate}
\end{proof} 
\begin{figure}[htbp]
      \begin{center}
\scalebox{0.58}{\includegraphics*{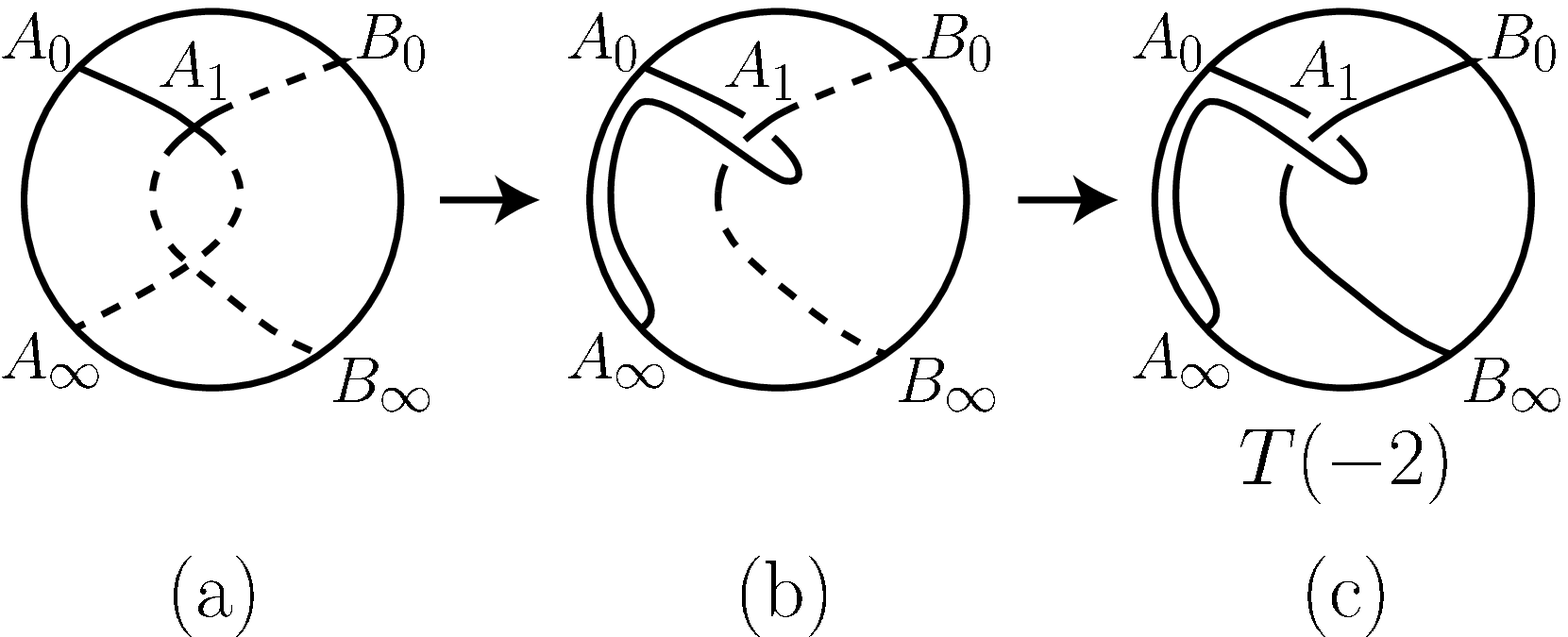}}
      \end{center}
   \caption{}
  \label{vertical-hook-proof-1}
\end{figure} 

%

%
\begin{figure}[htbp]
      \begin{center}
\scalebox{0.58}{\includegraphics*{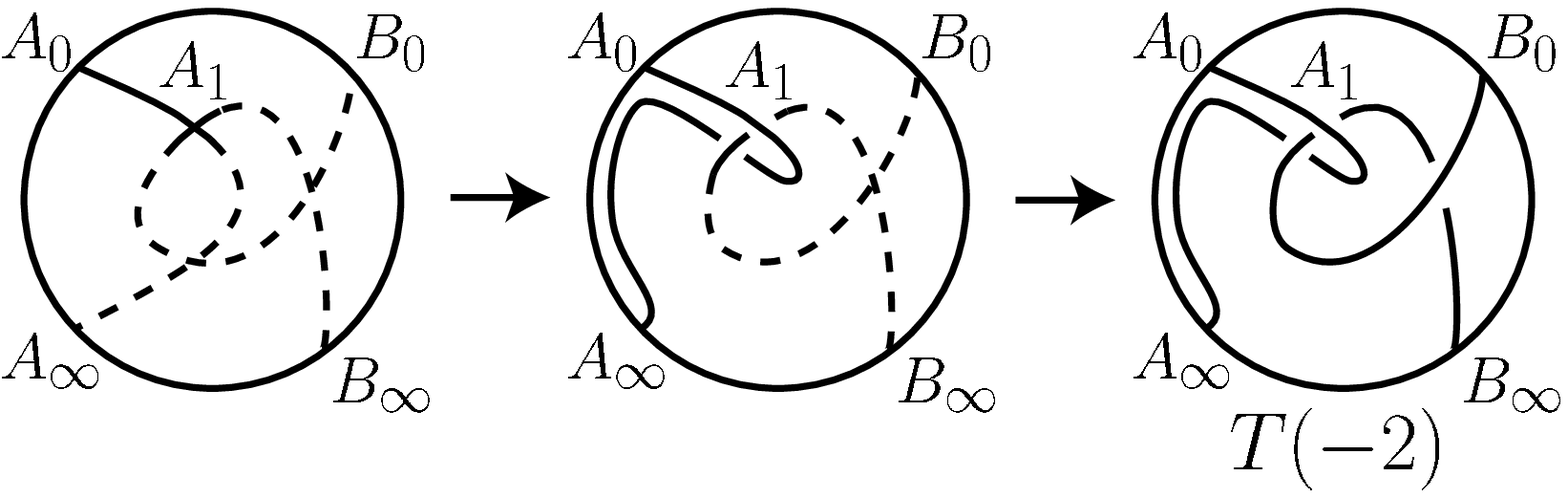}}
      \end{center}
   \caption{}
  \label{vertical-hook-proof-2}
\end{figure} 

%

\begin{Lemma}\label{sublemma}
Let $\hat T=\hat a\cup\hat b$ be a 2-string tangle projection with vertical connection. 
Suppose that $\hat T$ has mixed crossings and the spine $\hat a'$ does not intersect $\hat b$. 
Then $\hat T$ is a projection of the tangle $T(1,2)$ (Fig. \ref{vertical-trefoil-tangle}).
\end{Lemma}

\begin{figure}[htbp]
      \begin{center}
\scalebox{0.58}{\includegraphics*{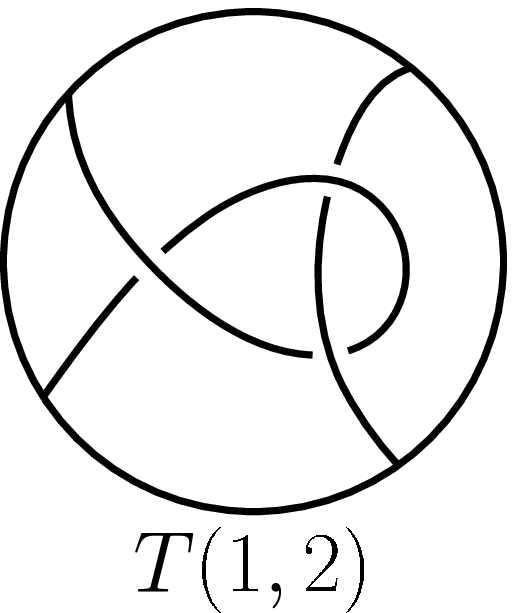}}
      \end{center}
   \caption{}
  \label{vertical-trefoil-tangle}
\end{figure} 

%

\begin{proof}
Let $A_1$ be the first mixed crossing of $\hat T$ on $\hat a$. By virtue of Lemma \ref{reducing-lemma} we may assume that $A_0A_1$ has no self-crossings. 
Since $A_1$ is not on the spine $\hat a'$ we find, again by using Lemma \ref{reducing-lemma}, a tear drop disk $\delta$ as illustrated in Fig. \ref{sublemma-proof} (a) or (b). Let $B_i$ be the first mixed crossing in $\partial\delta$ on $\hat b$. Note that $B_i$ may or may not equal to $A_1$. Again by Lemma \ref{reducing-lemma} we may assume that $B_0B_i$ is simple. 
Thus we have the situation illustrated in Fig. \ref{sublemma-proof} (a) or (b). In case (a) it is easy to see that $\hat T$ is a projection of $T(1,2)$. In case (b) we add over/under crossing information so that the dotted line part of $\hat a$ is under everything and the dotted line part of $\hat b$ is under the real line parts of $\hat a$ and $\hat b$. Let $P$ be the vertex of the 1-gon $\delta$. Since $\hat b$ does not intersect $A_0P$ we can deform $\hat b$ after deforming $\hat a$ into a neighborhood of $\partial B\cup A_0P$ as illustrated in Fig. \ref{sublemma-proof} (b).
\end{proof}

\begin{figure}[htbp]
      \begin{center}
\scalebox{0.65}{\includegraphics*{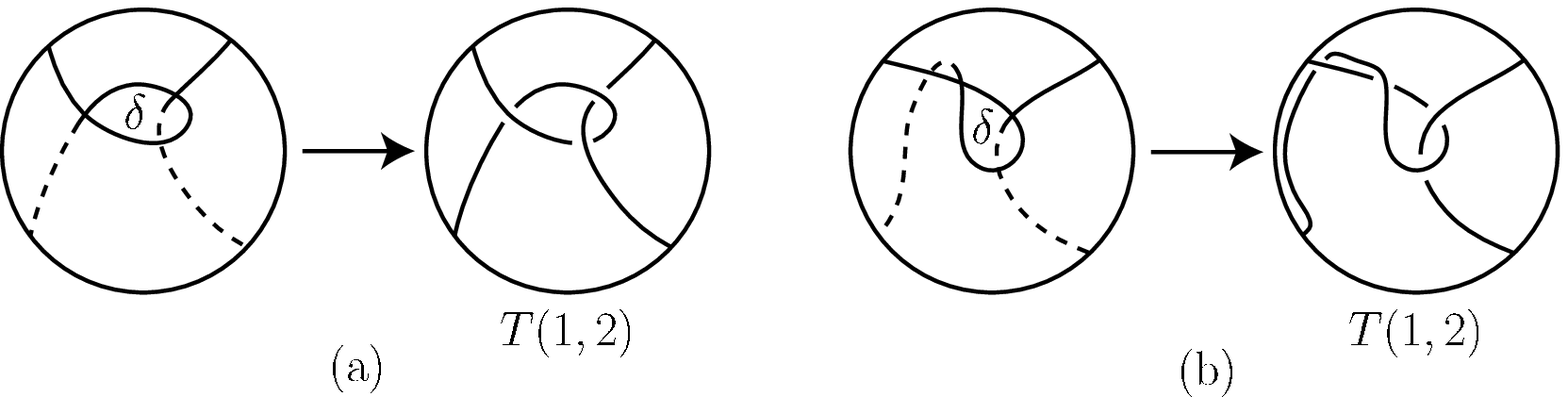}}
      \end{center}
   \caption{}
  \label{sublemma-proof}
\end{figure} 

%

\begin{Lemma}\label{vertical-trefoil-tangle-lemma}
A prime 2-string tangle projection $\hat T$ with a vertical connection
is a projection of the tangle $T(1,2)$ (Fig. \ref{vertical-trefoil-tangle}) if and only if 
$\hat T$ is not the projection $\hat T(n)$ (see Fig \ref{tangle-notation}) for any even $n$.
\end{Lemma}

\begin{proof} Only tangles $T(m)$ with $|m|\leq n$ has $\hat T(n)$ as a projection. 
It is well known that the tangle $T(1,2)$ is not $T(m)$ for any $m$. i
Thus we have proved the \lq\lq only if'' part. We will show the \lq\lq if'' part. 
First we note that the tangle $T(1,2)$ has the vertical symmetry $T(1,2)=V(T(1,2))=T(2,1)$. Let $\hat T=\hat a\cup\hat b$ 
be a prime 2-string tangle projection with vertical connection that is not $\hat T(n)$ for any $n$.
Let $\hat b'$ be the spine of $\hat b$. By Lemma \ref{reducing-lemma} we have that $\hat T'=\hat a\cup\hat b'$ is a minor of $\hat T$. Suppose that $\hat T'$ has mixed crossings. Let $A_{1},A_{2},\cdots ,A_{2n}$ be the mixed crossings of $\hat T'$ that appear in this order on $\hat a$. Let $B_{1},B_{2},\cdots ,B_{2n}$ be the order of them on $\hat b'$ and $\sigma$ the permutation $(\sigma (1), \sigma (2),\cdots ,\sigma (2n))$ of $(1,2,\cdots ,2n)$ defined by $A_{i}= B_{\sigma (i)}$ for each $i$ as before. We will show that $\hat T'$ is a projection of $T(1,2)$ unless $\sigma$ is an identical permutation. First suppose that $\sigma(2i-1)>\sigma(2i)$ for some $i\in\{1,2,\cdots,n\}$. Then we add over/under crossing information to $\hat T'$ as follows.
\begin{enumerate}
\item[(i)] $A_0A_{2i-1}^-$ is over everything.
\item[(ii)] $A_{2i}^+A_\infty$ is under everything.
\item[(iii)] $\hat a$ is descending.
\item[(iv)] $\hat a$ is under $\hat b'$ at $A_{2i-1}$ and over $\hat b'$ at $A_{2i}$.
\end{enumerate}
Then we have the tangle $T(1,2)$. See Fig. \ref{vertical-trefoil-proof1} (a).

Next suppose that $\sigma(2j-1)<\sigma(2j)$ for each $j\in\{1,2,\cdots,\}$ and $\sigma(2i)>\sigma(2i+1)$ for some $i\in\{1,2,\cdots,n-1\}$. We take smallest such $i$. We further divide this case into the following two cases.

Case 1. $\sigma(2i+1)<\sigma(1)$. In this case we give over/under crossing information to $\hat T'$ as follows.
\begin{enumerate}
\item[(i)] $A_1^+A_{2i+1}^-$ is over everything.
\item[(ii)] $A_{2i+1}^-A_\infty$ is under everything.
\item[(iii)] Each of $A_0A_1$, $A_1^+A_{2i+1}^-$ and $A_{2i+1}^-A_\infty$ are descending.
\item[(iv)] $\hat a$ is under $\hat b'$ at $A_1$.
\end{enumerate}
Then we have the tangle $T(1,2)$. See Fig. \ref{vertical-trefoil-proof1} (b).

Case 2. $\sigma(2i+1)>\sigma(1)$. In this case we give over/under crossing information to $\hat T'$ as follows.
\begin{enumerate}
\item[(i)] $A_1^+A_{2i}^-$ is under everything.
\item[(ii)] $A_{2i+1}^+A_\infty$ is over everything.
\item[(iii)] $A_0A_1^-$ is under $A_{2i}^+A_{2i+1}^-$.
\item[(iv)] Each of $A_0A_1$, $A_1^+A_{2i}^-$, $A_{2i}^+A_{2i+1}^-$ and $A_{2i+1}^+A_\infty$ are descending.
\item[(v)] $\hat a$ is over $\hat b'$ at $A_1$ and $A_{2i}$, and under $\hat b'$ at $A_{2i+1}$.
\end{enumerate}
Then we have the tangle $T(1,2)$. See Fig. \ref{vertical-trefoil-proof1} (c).

\begin{figure}[htbp]
      \begin{center}
\scalebox{0.65}{\includegraphics*{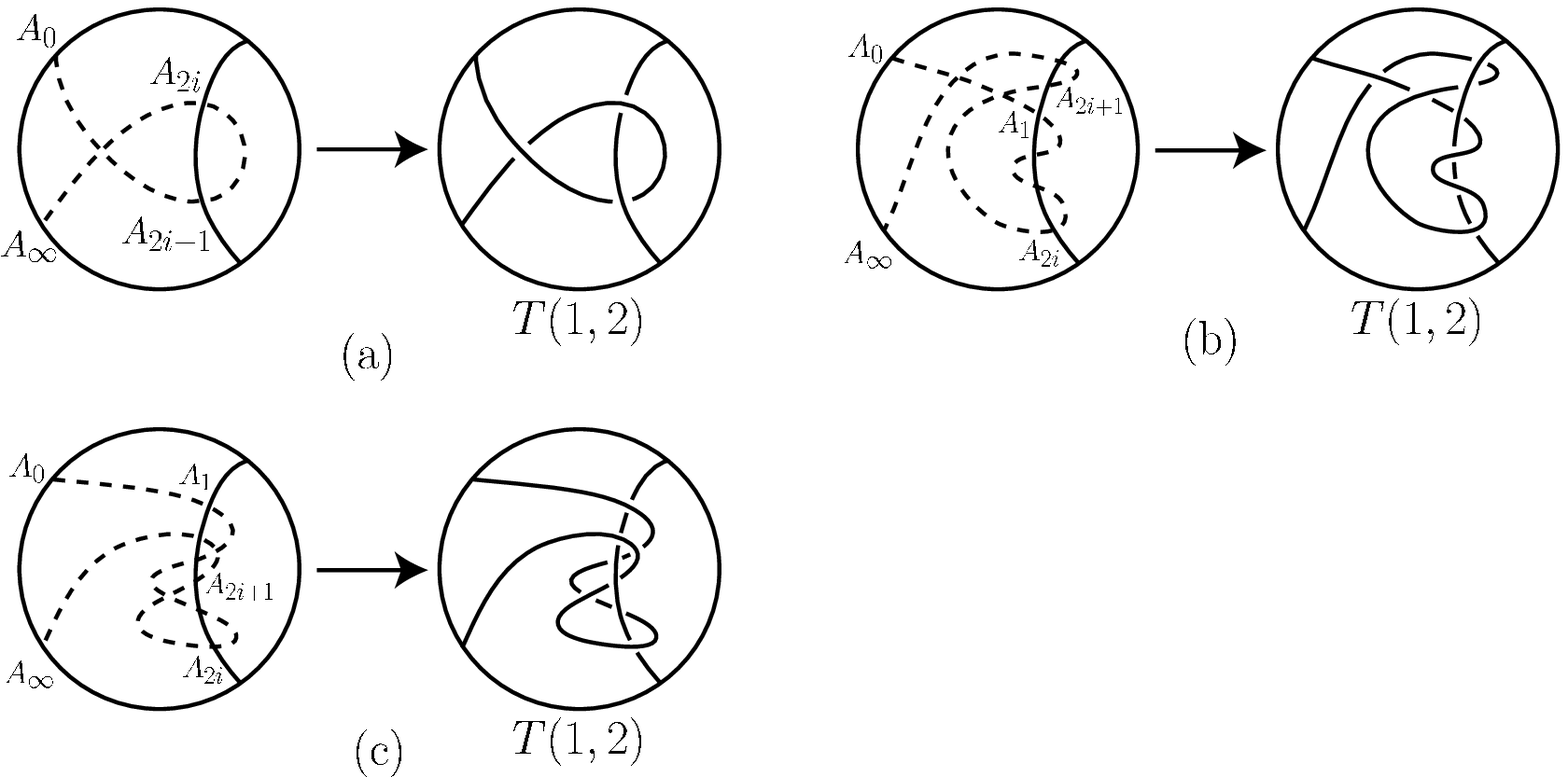}}
      \end{center}
   \caption{}
  \label{vertical-trefoil-proof1}
\end{figure} 

%

Thus the rest is the case that $\sigma(i)<\sigma(i+1)$ holds for every $i\in\{1,2,\cdots,2n-1\}$. 
This implies that $\sigma(i)=i$ for every $i$. Suppose that $A_iA_{i+1}$ intersects $A_{i+2j}A_{i+2j+1}$ 
for some $i$ and $j>0$ where we consider $A_{2n+1}=A_\infty$. Then by the use of Lemma \ref{hook-tangle} 
we have that $\hat T'$ is a projection of the tangle $T(1,2)$. See for example Fig. \ref{vertical-trefoil-proof2}.

\begin{figure}[htbp]
      \begin{center}
\scalebox{0.65}{\includegraphics*{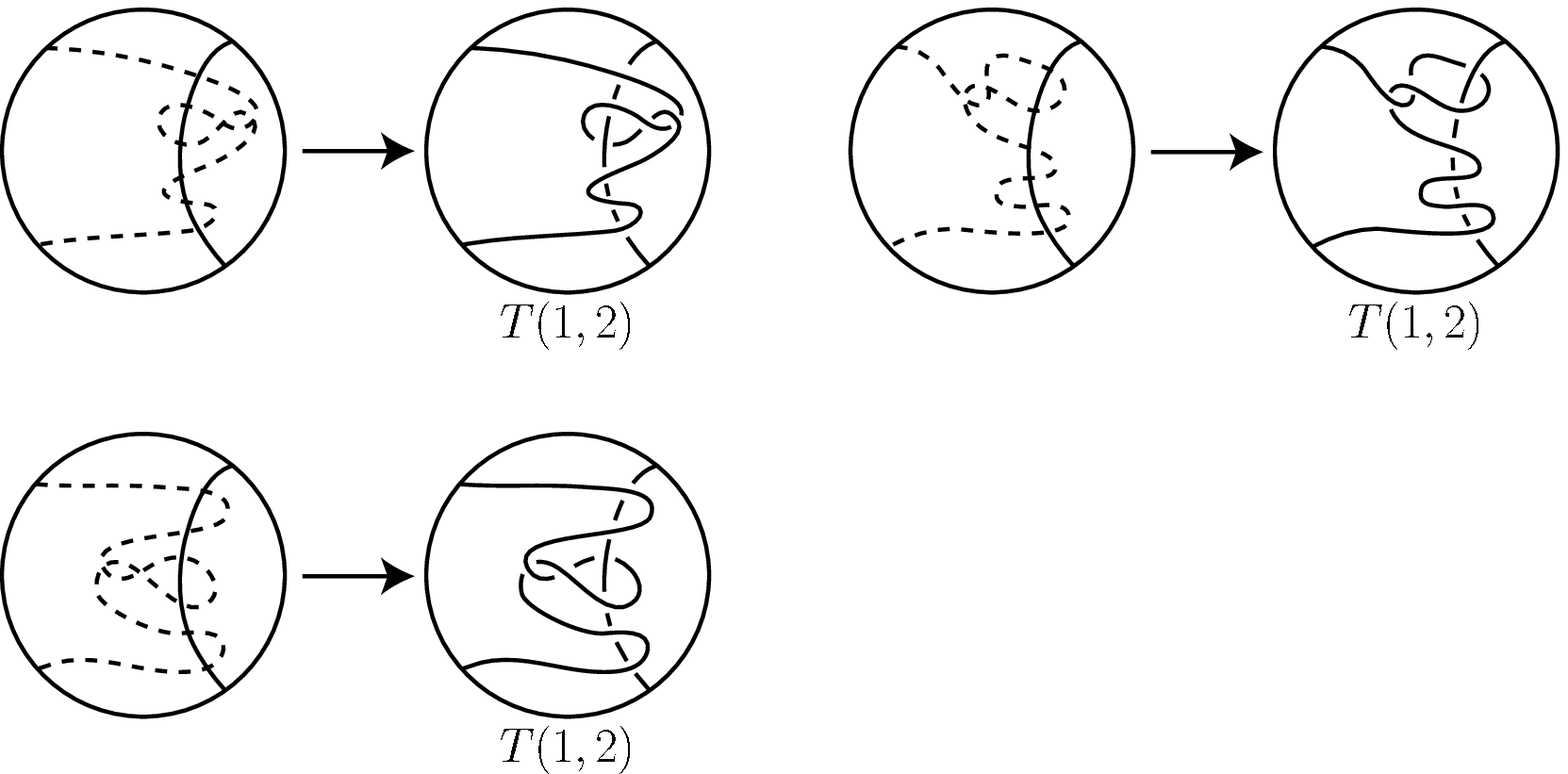}}
      \end{center}
   \caption{}
  \label{vertical-trefoil-proof2}
\end{figure} 

%

Thus we have the situation that there are mutually disjoint and possibly trivial 1-string subtangle 
projections $\hat S_1,\hat S_2,\cdots,\hat S_{2n+1}$ such that the tangle $\hat T'$ is as illustrated 
in Fig. \ref{vertical-trefoil-proof3} (a). Note that in this situation it holds that for any self 
crossing $P$ of $\hat a$ $s(\hat a,P)$ does not intersects $\hat b'$. 
Now we consider the spine $\hat a'$ and consider the tangle $\hat T''=\hat a'\cup \hat b$. 
Since the tangle $T(1,2)$ has the vertical symmetry as we have remarked at the beginning of the proof 
the same argument works for $\hat T''$ and we have the situation that there are mutually disjoint 
and possibly trivial 1-string subtangle projections $\hat U_1,\hat U_2,\cdots,\hat U_{2n+1}$ 
such that the tangle$\hat T''$ is as illustrated in Fig. \ref{vertical-trefoil-proof3} (b). 
Since for any self-crossing $P$ of $\hat b$ $s(\hat b,P)$ does not intersect $\hat a'$  
we have that the set of mixed crossings of $\hat T''$ is exactly the same as the set of 
mixed crossings of $\hat T'$. Since $\hat T$ is prime and not equal to $\hat T(n)$ 
we have that at least one of $\hat S_1,\hat S_2,\cdots,\hat S_{2n+1}$ has a crossing 
and then it must intersect at least one of $\hat U_1,\hat U_2,\cdots,\hat U_{2n+1}$. 
Then by the use of Lemma \ref{hook-tangle} or Lemma \ref{sublemma} we obtain the tangle 
$T(1,2)$ as illustrated in Fig. \ref{vertical-trefoil-proof4}.
\end{proof}

\begin{figure}[htbp]
      \begin{center}
\scalebox{0.65}{\includegraphics*{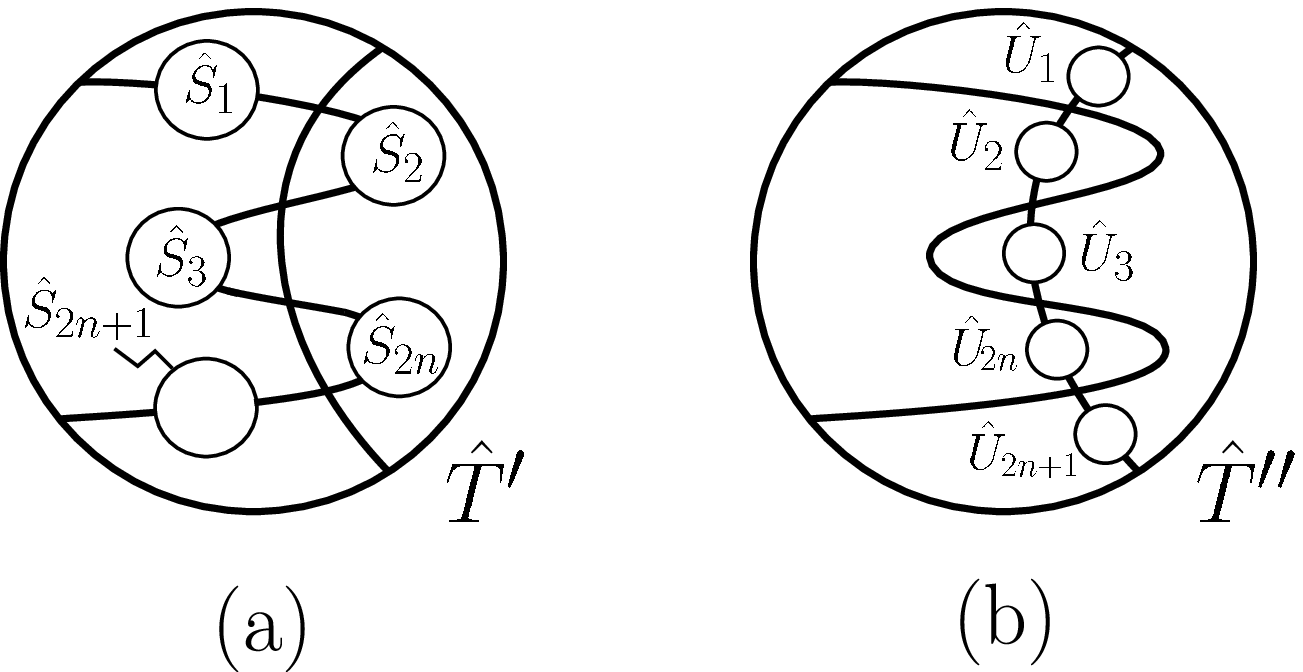}}
      \end{center}
   \caption{}
  \label{vertical-trefoil-proof3}
\end{figure} 

%

%
\begin{figure}[htbp]
      \begin{center}
\scalebox{0.65}{\includegraphics*{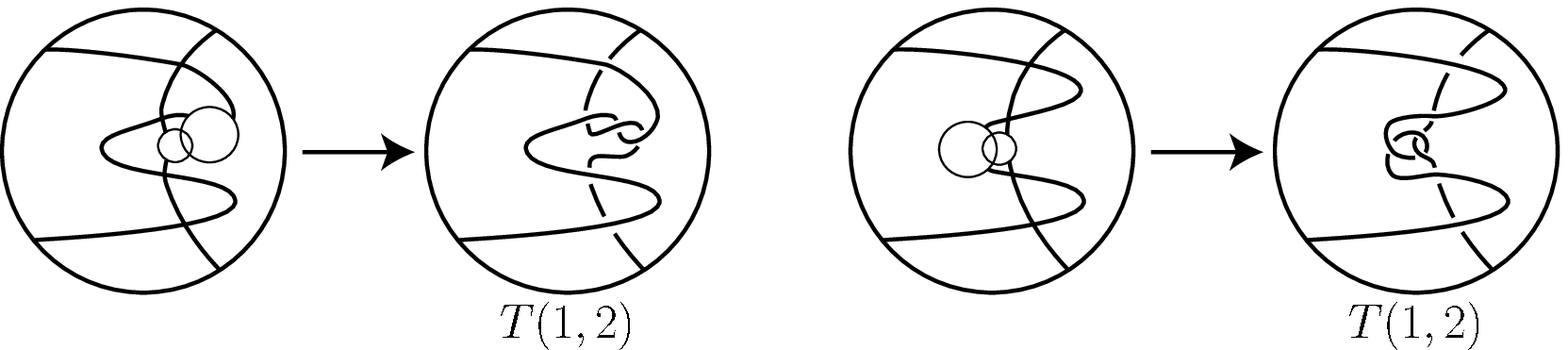}}
      \end{center}
   \caption{}
  \label{vertical-trefoil-proof4}
\end{figure} 

%

\begin{Lemma}\label{one-three-lemma}
A prime 2-string tangle projection $\hat T$ with a X-connection is a projection of the tangle $T(1/3)$ (Fig. \ref{one-three-tangle}) if and only if 
$\hat T$ is not the projection $\hat T(n)$ (see Fig \ref{tangle-notation}) for any odd $n$.
\end{Lemma}

\begin{figure}[htbp]
      \begin{center}
\scalebox{0.58}{\includegraphics*{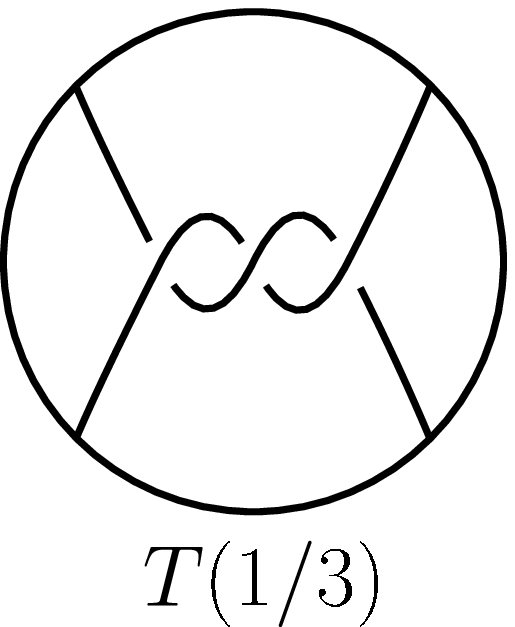}}
      \end{center}
   \caption{}
  \label{one-three-tangle}
\end{figure} 

%

\begin{proof}
The proof is quite similar to that of Lemma \ref{vertical-trefoil-tangle-lemma}.
Since the tangle $T(1/3)$ is different from any of the tangle $T(m)$, the   \lq\lq only if'' part is proved. 
We will show the \lq\lq if'' part. 
Let $\hat T=\hat a\cup\hat b$ be a prime 2-string tangle projection with X-connection that is not $\hat T(n)$ for any $n$.
Let $\hat b'$ be the spine of $\hat b$. Then $\hat T'=\hat a\cup\hat b'$ is a minor of $\hat T$. 
As in the proof of Lemma \ref{vertical-trefoil-tangle-lemma} we have that $\hat T'$ is a projection of $T(1/3)$ 
unless the permutation $\sigma$ on $\{1,2,\cdots,2n+1\}$ is the identity. See Fig. \ref{one-three-proof1}.

\begin{figure}[htbp]
      \begin{center}
\scalebox{0.65}{\includegraphics*{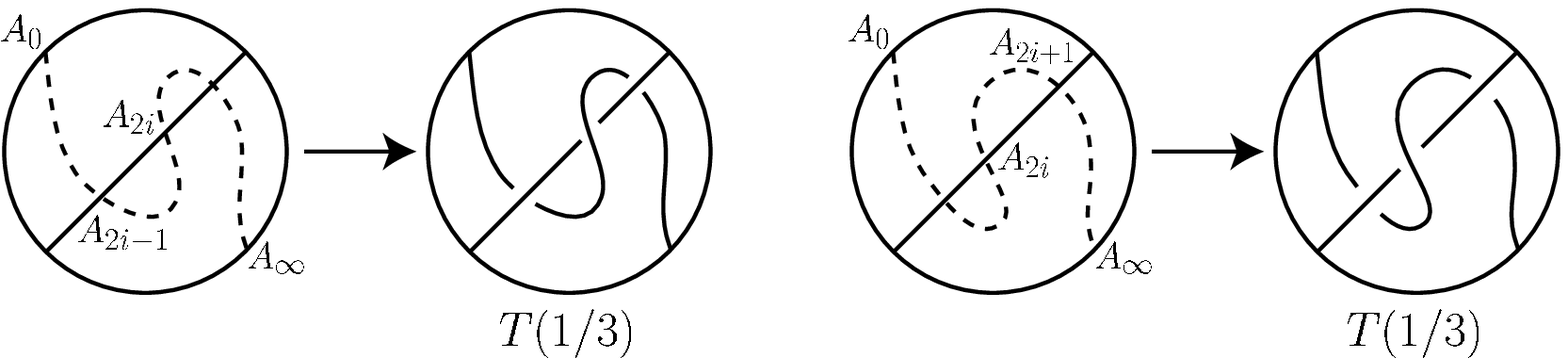}}
      \end{center}
   \caption{}
  \label{one-three-proof1}
\end{figure} 

%

Thus the rest of the proof is the analysis of the case that $\sigma(i)=i$ for every $i$. 
Suppose that $A_iA_{i+1}$ intersects $A_{i+2j}A_{i+2j+1}$ for some $i$ and $j>0$ 
where we consider $A_{2n+1}=A_\infty$. Then by the use of Lemma \ref{hook-tangle} 
we have that $\hat T'$ is a projection of the tangle $T(1/3)$. See for example Fig. \ref{one-three-proof2}.

\begin{figure}[htbp]
      \begin{center}
\scalebox{0.58}{\includegraphics*{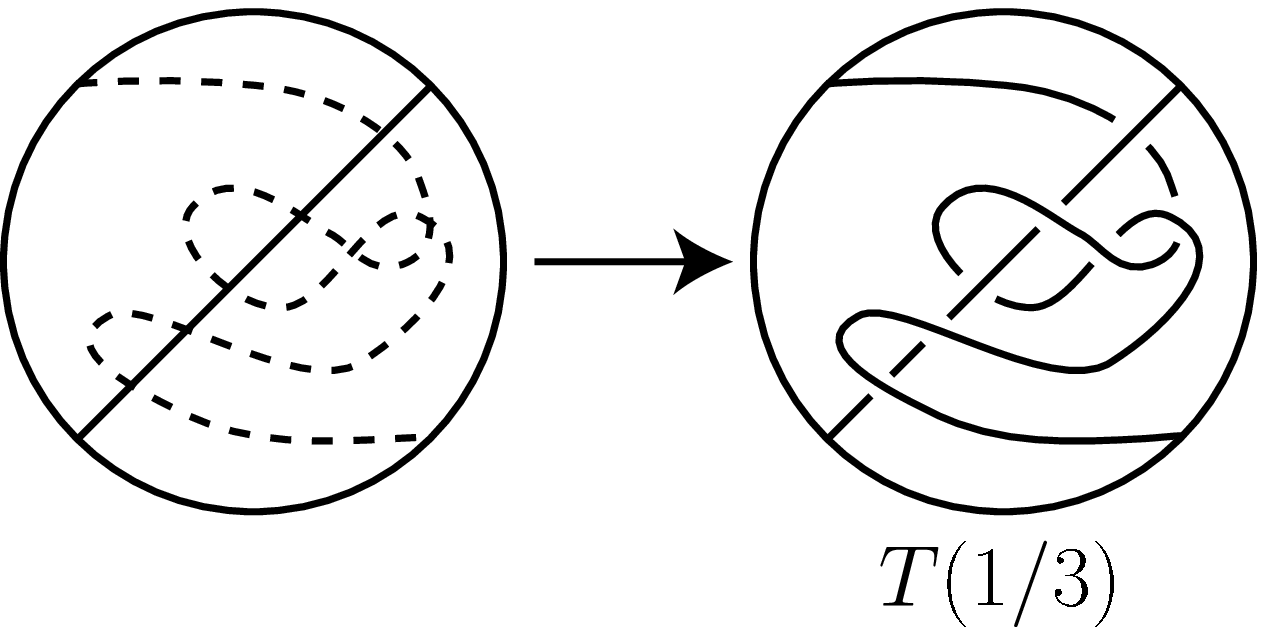}}
      \end{center}
   \caption{}
  \label{one-three-proof2}
\end{figure} 

%

Thus we have the situation that $\hat T'$ and $\hat T''=\hat a'\cup \hat b$ is as illustrated 
in Fig. \ref{one-three-proof3}.

\begin{figure}[htbp]
      \begin{center}
\scalebox{0.58}{\includegraphics*{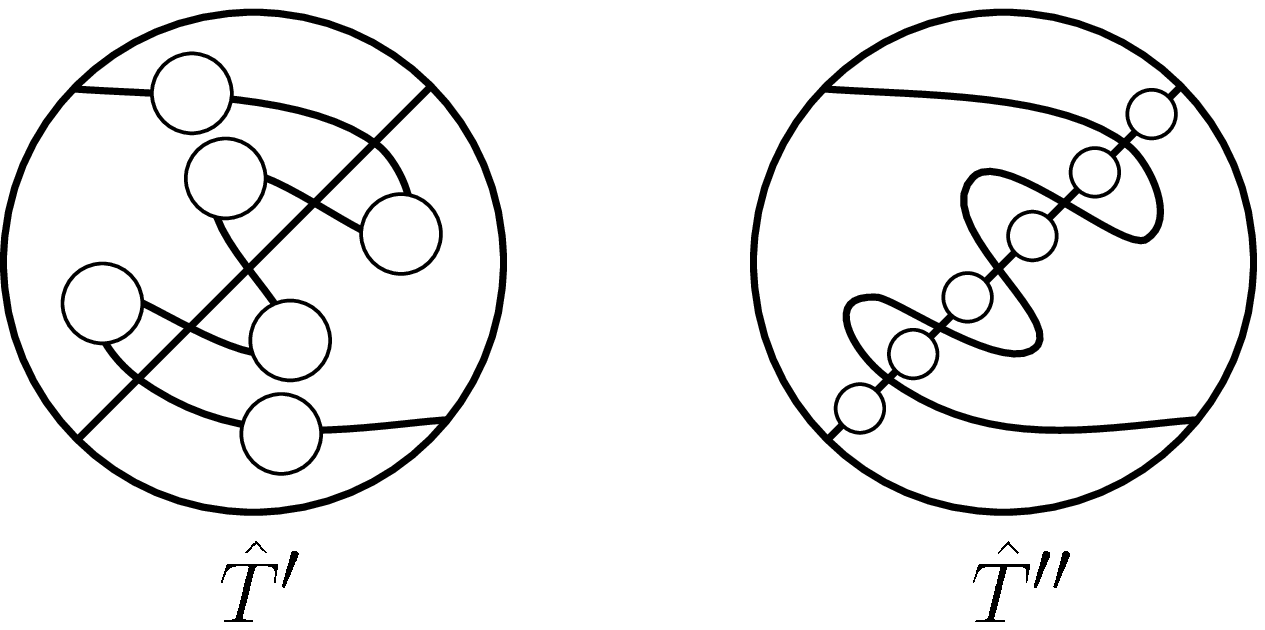}}
      \end{center}
   \caption{}
  \label{one-three-proof3}
\end{figure} 

%

Then by the use of Lemma \ref{hook-tangle} or Lemma \ref{sublemma},
 we have the tangle $T(1/3)$ as illustrated in Fig. \ref{one-three-proof4}.
\end{proof}

\begin{figure}[htbp]
      \begin{center}
\scalebox{0.65}{\includegraphics*{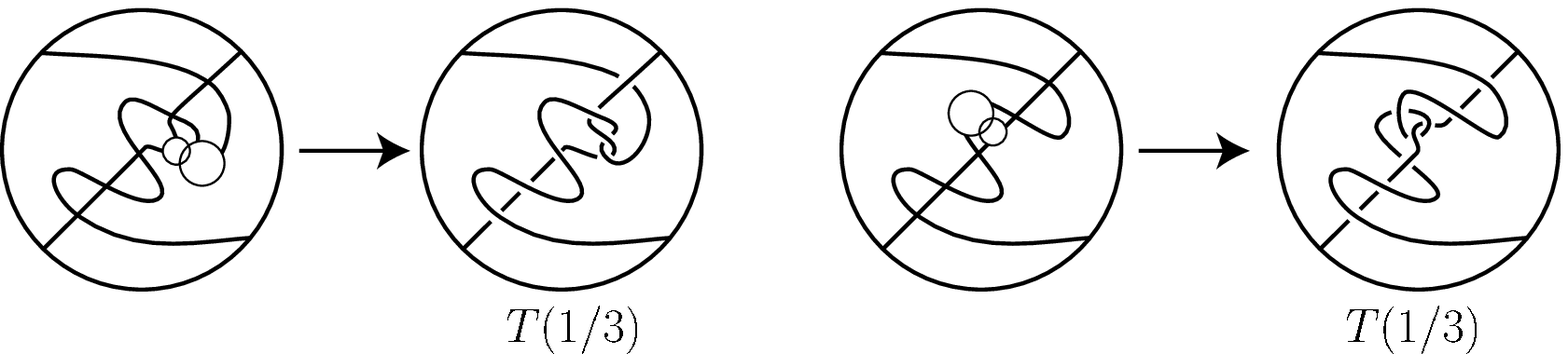}}
      \end{center}
   \caption{}
  \label{one-three-proof4}
\end{figure} 

%

\begin{Lemma}\label{-4-lemma}
Let $\hat T$ be a prime 2-string tangle projection with 
vertical connection.  Suppose that the right string $\tilde b$ has 
no self-crossings and $\hat T$ is not a
projection of the tangle $T(-4)$ (Fig. \ref{-4-tangle}). 
Then one of the following (1) or (2) holds.

(1) There are a tangle projection $\hat S$ with horizontal connection, and natural numbers $m$ and $n$ 
such that $\hat T=\hat S+R(\hat T(2m-1,2n-1))$ (Fig. \ref{pre-pretzel} (a)).

(2) There is a tangle projection $\hat S$ with X-connection and a natural number $m$ such 
that $\hat T=\hat S+\hat T(2m)$ (Fig. \ref{pre-pretzel} (b)).
\end{Lemma}

\begin{figure}[htbp]
      \begin{center}
\scalebox{0.58}{\includegraphics*{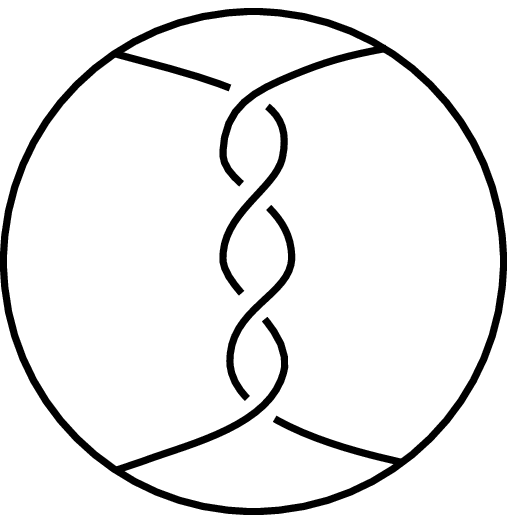}}
      \end{center}
   \caption{}
  \label{-4-tangle}
\end{figure} 

%

%
\begin{figure}[htbp]
      \begin{center}
\scalebox{0.58}{\includegraphics*{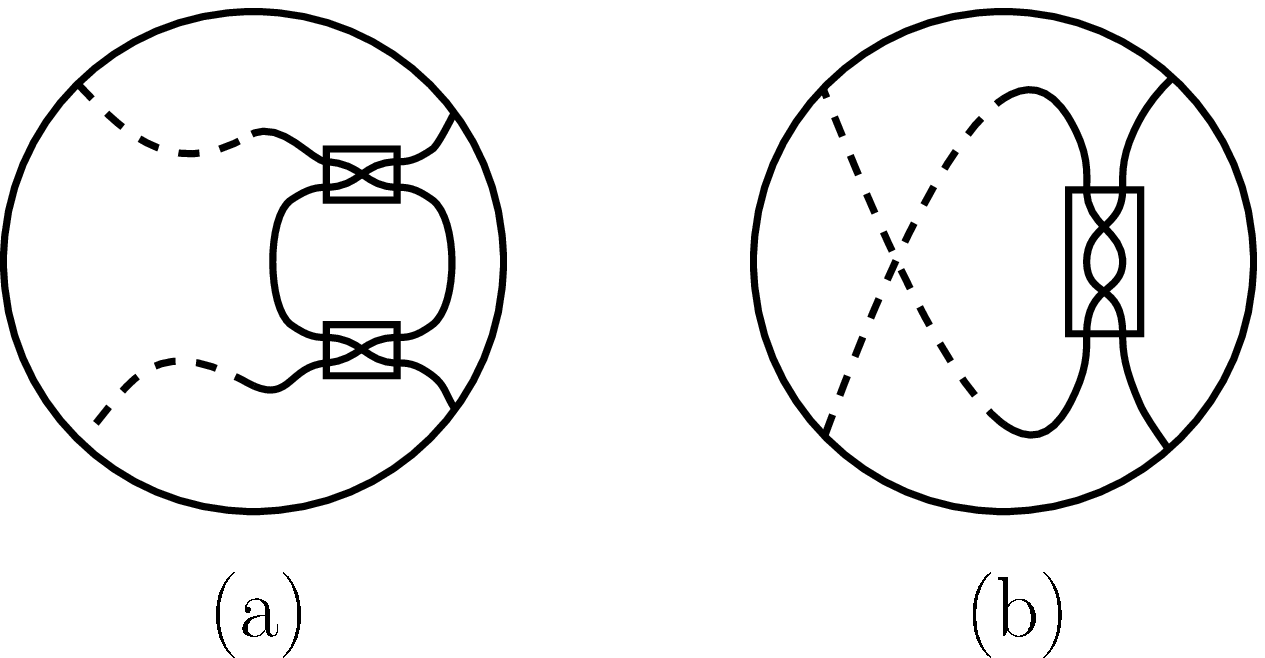}}
      \end{center}
   \caption{}
  \label{pre-pretzel}
\end{figure} 

%

\begin{proof}
We analyze a permutation $(\sigma (1),\sigma (2),\ldots ,\sigma (2k))$ of $(1,2,\ldots ,2k)$ 
defined by $A_{i}=B_{\sigma (i)}$. Note that the tangle projection in Fig. \ref{pre-pretzel} 
(b) defines the \lq\lq backward'' permutation $\sigma (2m)< \sigma (2m-1)< \cdots < \sigma (2) < \sigma (1)$. Namely $\sigma(i)=2m+1-i$ for every $i$. The tangle projection in Fig. \ref{pre-pretzel} (a) defines the permutation composed of two \lq\lq backward'' pieces, that is,
  $\sigma(2n-1) < \sigma (2n-2) < \cdots < \sigma (1) < \sigma (2m+2n-2)< \sigma (2m+2n-3)< \cdots < \sigma (2n+1) < \sigma (2n)$. 
We will show that $\hat T$ is a projection of $T(-4)$ unless 
the permutation $\sigma$ is one of these forms. 
For that purpose we watch the \lq\lq forward'' part of $\sigma$. 
First suppose that $\sigma(i) > \sigma(i+1)$ 
for every $i$. Then we have $\sigma(i)=2k+1-i$ for each $i$. Next suppose that $\sigma (2i)<\sigma (2i+1)$ for some $i$. In this case we give over/under crossing information to $\hat T$ as follows.
\begin{enumerate}
\item[(i)] $A_0A_{2i}^-$ is under everything.
\item[(ii)] $A_{2i+1}^+A_\infty$ is over everything.
\item[(iii)] $\hat a$ is descending.
\item[(iv)] $\hat a$ is over $\hat b$ at $A_{2i}$ and under $\hat b$ at $A_{2i+1}$.
\end{enumerate}
Then we have $T(-4)$ as illustrated in Fig. \ref{-4-tangle-proof1}

\begin{figure}[htbp]
      \begin{center}
\scalebox{0.58}{\includegraphics*{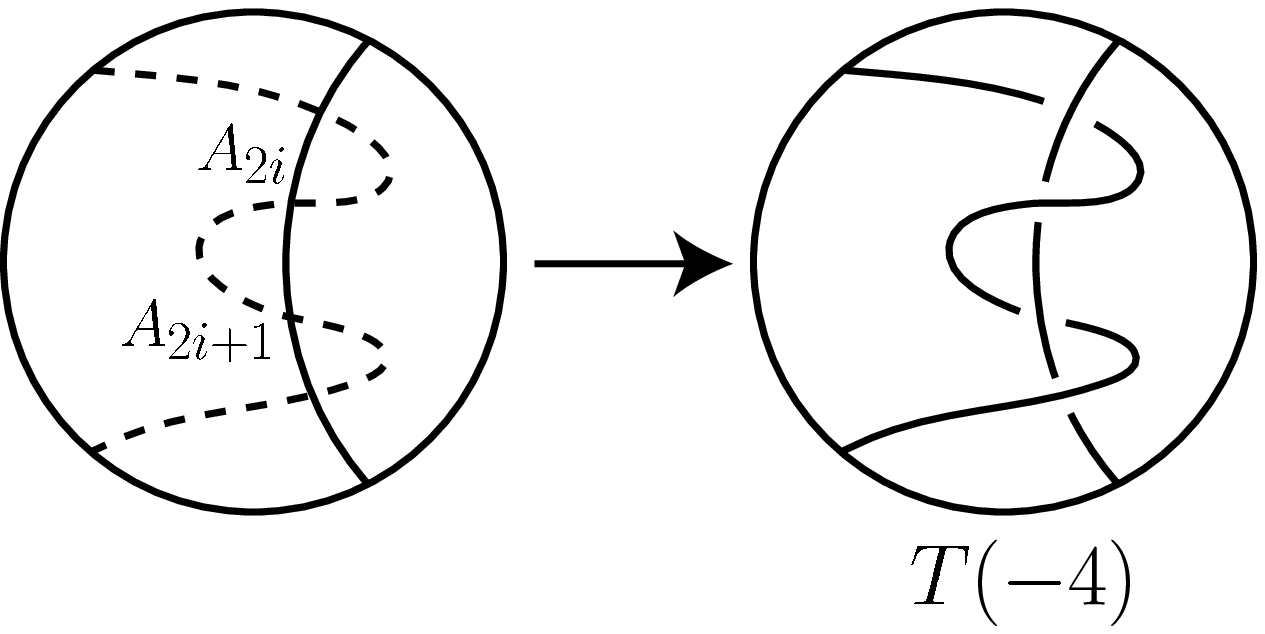}}
      \end{center}
   \caption{}
  \label{-4-tangle-proof1}
\end{figure} 

%

Next suppose that $\sigma (2j)>\sigma (2j+1)$ for every $j$ and there is $i$ such that $\sigma (2i-1)<\sigma (2i)$. We fix this $i$ and analyze how it goes after $A_{2i}$ (Case 1, Case 2) and how it goes back before $A_{2i-1}$ (Case 1$'$, Case 2$'$) .

\vskip 3mm
Case 1. There exists $j>i$ such that $\sigma (2i-1)<\sigma (2j-1)<\sigma (2j)$.

We take the smallest such $j$. Then by the assumption we have $\sigma(2j-1)<\sigma(2j-2)<\cdots<\sigma(2i)$. In this case we give over/under crossing information to $\hat T$ as follows.
\begin{enumerate}
\item[(i)] $A_0A_{2i-1}^-$ is over everything.
\item[(ii)] $A_{2j}^+A_\infty$ is under everything.
\item[(iii)] $\hat a$ is descending.
\item[(iv)] $\hat a$ is under $\hat b$ at $A_{2i-1}$ and $A_{2j-1}$ and over $\hat b$ at $A_{2i},A_{2i+1},\cdots,A_{2j-2}$ and $A_{2j}$.
\end{enumerate}
Then we have $T(-4)$ as illustrated in Fig. \ref{-4-tangle-proof2}
\begin{figure}[htbp]
      \begin{center}
\scalebox{0.58}{\includegraphics*{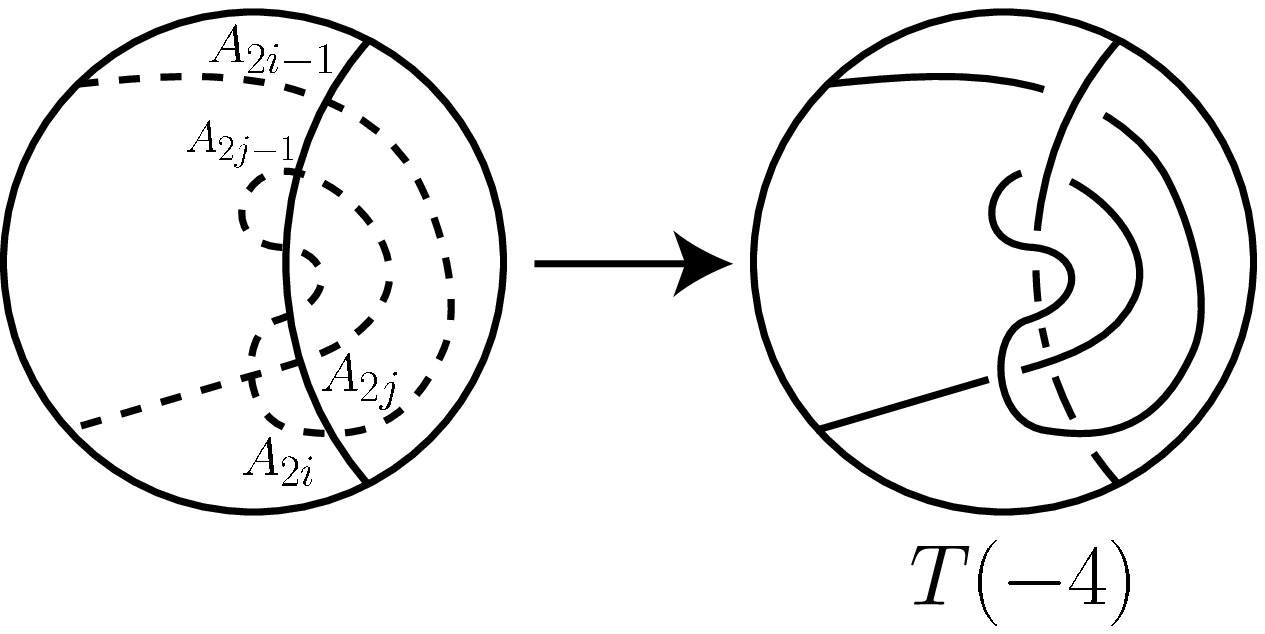}}
      \end{center}
   \caption{}
  \label{-4-tangle-proof2}
\end{figure} 

%

Case 1$'$. There exists $l<i$ such that $\sigma (2i)>\sigma (2l)>\sigma (2l-1)$.

We take the largest such $l$. Then by the assumption we have $\sigma(2l)>\sigma(2l+1)>\cdots>\sigma(2i-1)$. Note that this case is obtained from Case 1 by reversing the order. Namely $2i-1,2i,2j-1$ and $2j$ correspond to $2i,2i-1,2l$ and $2l-1$ respectively. Since the tangle $T(-4)$ has horizontal symmetry $H(T(-4))=T(-4)$ we have that $\hat T$ is a projection of $T(-4)$.

\vskip 3mm

Case 2. Case 1 does not occur and there exist $j>i$ such that $\sigma (2j)<\sigma (2i-1)<\sigma (2j-1)$.

We take the smallest such $j$. Then by the assumption we have $\sigma(2j-1)<\sigma(2j-2)<\cdots<\sigma(2i)$. In this case we give over/under crossing information to $\hat T$ as follows.
\begin{enumerate}
\item[(i)] $A_0A_{2i-1}^-$ is under everything.
\item[(ii)] $A_{2j}^+A_\infty$ is over everything.
\item[(iii)] $A_0A_{2j}$ and $A_{2j}A_\infty$ are descending.
\item[(iv)] $\hat a$ is under $\hat b$ at $A_{2i-1}$ and $A_{2j-1}$ and over $\hat b$ at $A_{2i},A_{2i+1},\cdots,A_{2j-2}$ and $A_{2j}$.
\end{enumerate}
Then we have $T(-4)$ as illustrated in Fig. \ref{-4-tangle-proof3}
\begin{figure}[htbp]
      \begin{center}
\scalebox{0.58}{\includegraphics*{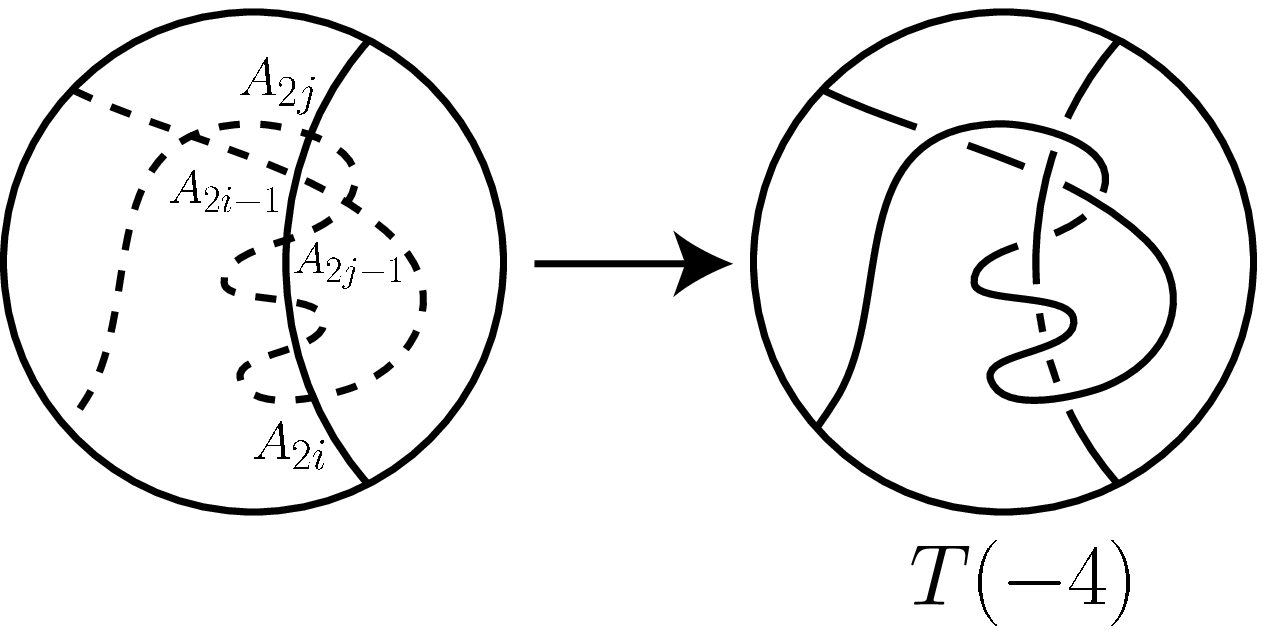}}
      \end{center}
   \caption{}
  \label{-4-tangle-proof3}
\end{figure} 

%

Case 2$'$.  Case 1$'$ does not occur and there exist $l<i$ such that $\sigma(2l-1)>\sigma(2i)>\sigma(2l)$.

By the horizontal symmetry of $T(-4)$ and by Case 2 we have that $\hat T$ is a projection of $T(-4)$.

In the following cases we assume that neither  Case 1 nor Case 1$'$, Case 2 or Case 2$'$ occurs.

Then after $A_{2i}$ we have the following two possibilities.

Case P1.  There exists $j>i$ such that $\sigma (2j-1)<\sigma (2i-1)<\sigma (2j-2)$.

In this case we take the smallest such $j$ and therefore we further have $\sigma(2j-1)<\sigma(2j-2)<\cdots<\sigma(2i)$
(or $\sigma(2j-1)<\sigma(2j-2)=\sigma(2i)$ as a special case).

Case P2. Case P1 does not occurs. Namely $\sigma(2i-1)<\sigma(2k)<\sigma(2k-1)<\cdots<\sigma(2i)$
(or $2k=2i$ as a special case).

Before $A_{2i-1}$ we have the following two possibilities.

Case Q1.  There exists $l<i$ such that $\sigma (2l)>\sigma (2i)>\sigma (2l-2)$.

In this case we take the largest such $l$ and therefore we further have $\sigma(2l)>\sigma(2l+1)>\cdots>\sigma(2i-1)$
(or $\sigma(2l)>\sigma(2l+1)=\sigma(2i-1)$ as a special case).

Case Q2. Case Q1 does not occurs. Namely $\sigma(2i)>\sigma(1)>\sigma(2)>\cdots>\sigma(2i-1)$
(or $1=2i-1$ as a special case).

Thus we have the following four cases to be considered.

Case P1Q1. In this case we give over/under crossing information to $\hat T$ as follows.
\begin{enumerate}
\item[(i)] $A_0A_{2l}^-$ is under everything.
\item[(ii)] $A_{2j-1}^+A_\infty$ is over everything.
\item[(iii)] $A_{2l}^+A_{2i-1}^-$ is over everything except $A_{2j-1}^+A_\infty$.
\item[(iv)] $A_{2i}^+A_{2j-1}^-$ is under everything except $A_0A_{2l}^-$.
\item[(v)] Each of $A_0A_{2l}$, $A_{2l}A_{2i-1}$, $A_{2i-1}A_{2i}$, $A_{2i}A_{2j-1}$ and $A_{2j-1}A_\infty$ is descending.
\item[(vi)] $\hat a$ is under $\hat b$ at $A_{2l},A_{2i-1},A_{2i+1}A_{2i+2},\cdots,A_{2j-1}$ and over $\hat b$ at $A_{2l+1},A_{2l+2},\cdots,A_{2i-2}$ and  $A_{2i}$.
\end{enumerate}
Then we have $T(-4)$ as illustrated in Fig. \ref{-4-tangle-proof4}
\begin{figure}[htbp]
      \begin{center}
\scalebox{0.58}{\includegraphics*{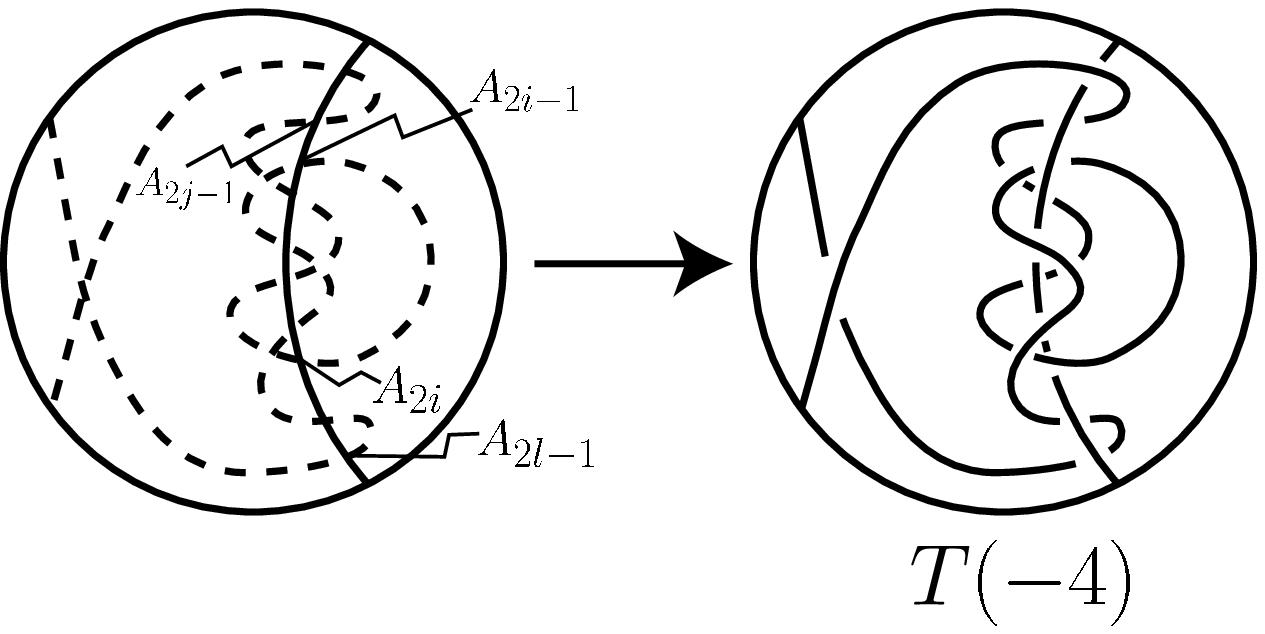}}
      \end{center}
   \caption{}
  \label{-4-tangle-proof4}
\end{figure} 

%
Case P1Q2. In this case we give over/under crossing information to $\hat T$ as follows.
\begin{enumerate}
\item[(i)] $A_{2j-1}^+A_\infty$ is over everything.
\item[(ii)] Each of $A_0A_{2j-1}$ and $A_{2j-1}A_\infty$ is descending.
\item[(iii)] $\hat a$ is over $\hat b$ at $A_{1},A_{2},\cdots,A_{2i-2}$ and $A_{2i}$ and under $\hat b$ at $A_{2i+1},A_{2i+2},\cdots,A_{2j-1}$.
\end{enumerate}
Then we have $T(-4)$ as illustrated in Fig. \ref{-4-tangle-proof5}
\begin{figure}[htbp]
      \begin{center}
\scalebox{0.58}{\includegraphics*{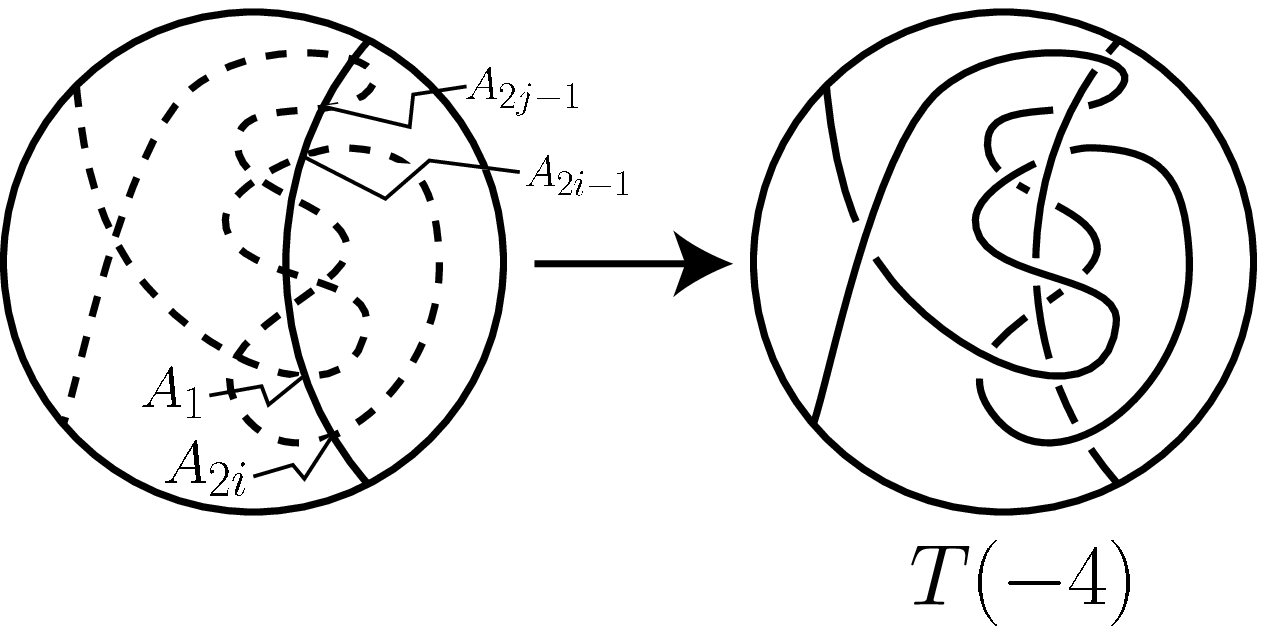}}
      \end{center}
   \caption{}
  \label{-4-tangle-proof5}
\end{figure} 

%

Case P2Q1. This case follows from Case P1Q2 and the horizontal symmetry of $T(-4)$.

Case P2Q2. Suppose that $\sigma(1)>\sigma(2k)$. Then we give over/under crossing information to $\hat T$ as follows.
\begin{enumerate}
\item[(i)] $A_{1}^+A_{2i}^-$ is over everything.
\item[(ii)] $A_{2i}^+A_{2k}^-$ is under everything.
\item[(iii)] Each of $A_0A_{1}$, $A_1A_{2i}$, $A_{2i}A_{2k}$ and $A_{2k}A_\infty$ is descending.
\item[(iv)] $A_0A_{1}$ is under $A_{2k}A_\infty$.
\item[(v)] $\hat a$ is over $\hat b$ at $A_{2i}$ and $A_{2k}$ and under $\hat b$ at $A_{1}$.
\end{enumerate}
Then we have $T(-4)$ as illustrated in Fig. \ref{-4-tangle-proof6}
\begin{figure}[htbp]
      \begin{center}
\scalebox{0.58}{\includegraphics*{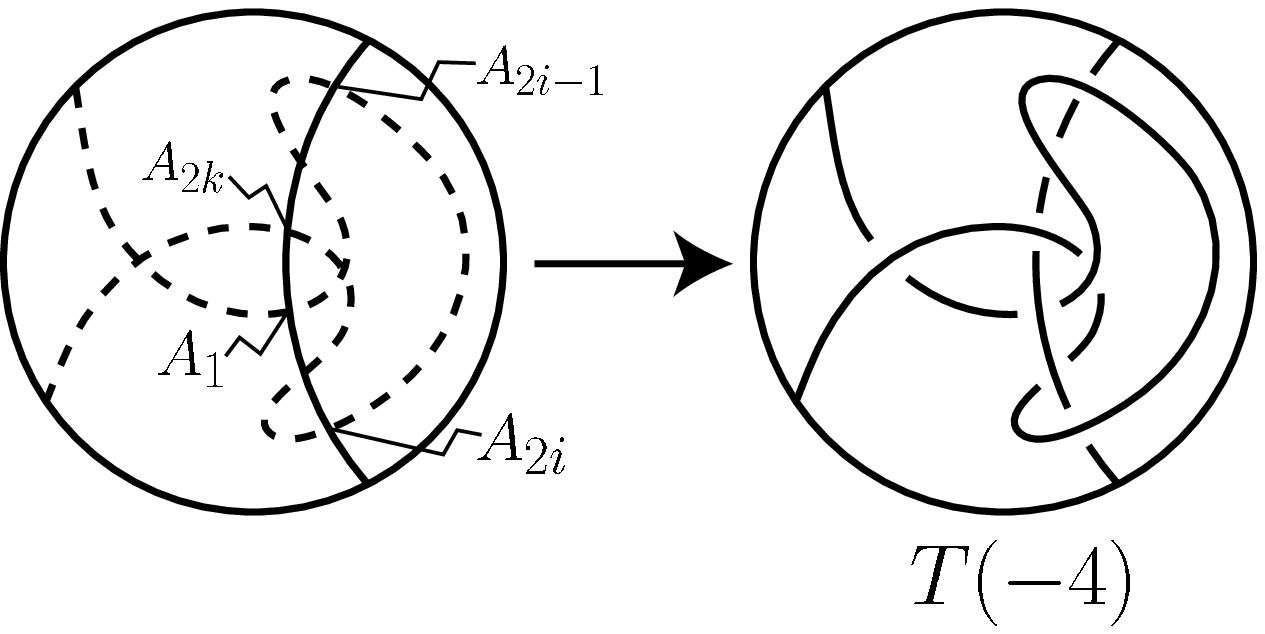}}
      \end{center}
   \caption{}
  \label{-4-tangle-proof6}
\end{figure} 

%

Thus we finally have that $\sigma(2i-1) < \sigma (2i-2) < \cdots < \sigma (1) < \sigma (2k)< \sigma (2k-1)< \cdots < \sigma (2i+1) < \sigma (2i)$. This is one of the permutations described at the beginning of the proof.

Now suppose that $\sigma (2i-1)>\sigma (2i)$ for every $i$. Namely there are no forward part in $\sigma$ 
and so we have $\sigma(2k)<\sigma(2k-1)<\cdots<\sigma(2)<\sigma(1)$. This is the second permutation 
described at the beginning of the proof. 
Suppose that in these cases there are $i$ and $j$ with $i<j$ such that $A_{i}^+A_{i+1}^-$ 
intersects $A_{j}^+A_{j+1}^-$. Then by the use of Lemma \ref{hook-tangle} we have $T(-4)$ 
as illustrated in Fig. \ref{-4-tangle-proof7}. Then by the primeness of $\hat T$ 
we reach the desired conclusion.

\end{proof}

\begin{figure}[htbp]
      \begin{center}
\scalebox{0.58}{\includegraphics*{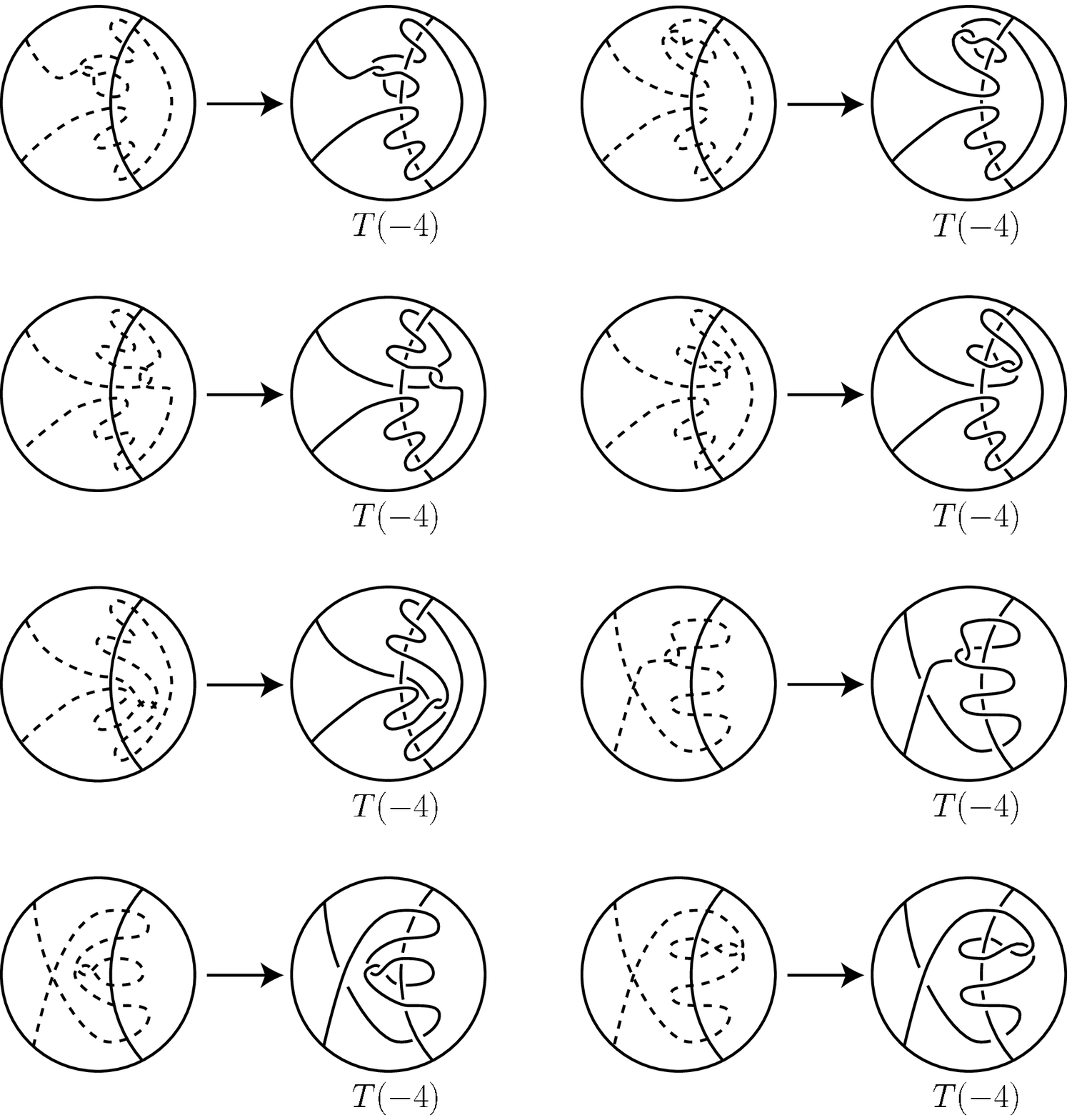}}
      \end{center}
   \caption{}
  \label{-4-tangle-proof7}
\end{figure} 

%

The following corollaries of Lemma \ref{-4-lemma} are the main tools of this section.

\begin{Corollary}\label{-4-lemma-corollary1}
Let $\hat T$ be a prime 2-string tangle projection with 
vertical connection. Suppose that the right string $\hat b$ has no self-crossings and $\hat T$ is not a  
projection of the tangle $T(-4)$, $R(T(-1,-2))+T(-2)$, or $T(-3,2)$ (Fig. \ref{tangles-for-five-one}).
Then one of the followings holds.

(1) $\hat T=\hat T(1/2l)+R(\hat T(2m-1,2n-1))$ (Fig. \ref{pretzel-tangles} (a)) for some $l\geq0$ and $m,n>0$.

(2) $\hat T=\hat T(1/(2l-1),2m)$ (Fig. \ref{pretzel-tangles} (b)) for some $l,m>0$.

\end{Corollary}

\begin{figure}[htbp]
      \begin{center}
\scalebox{0.58}{\includegraphics*{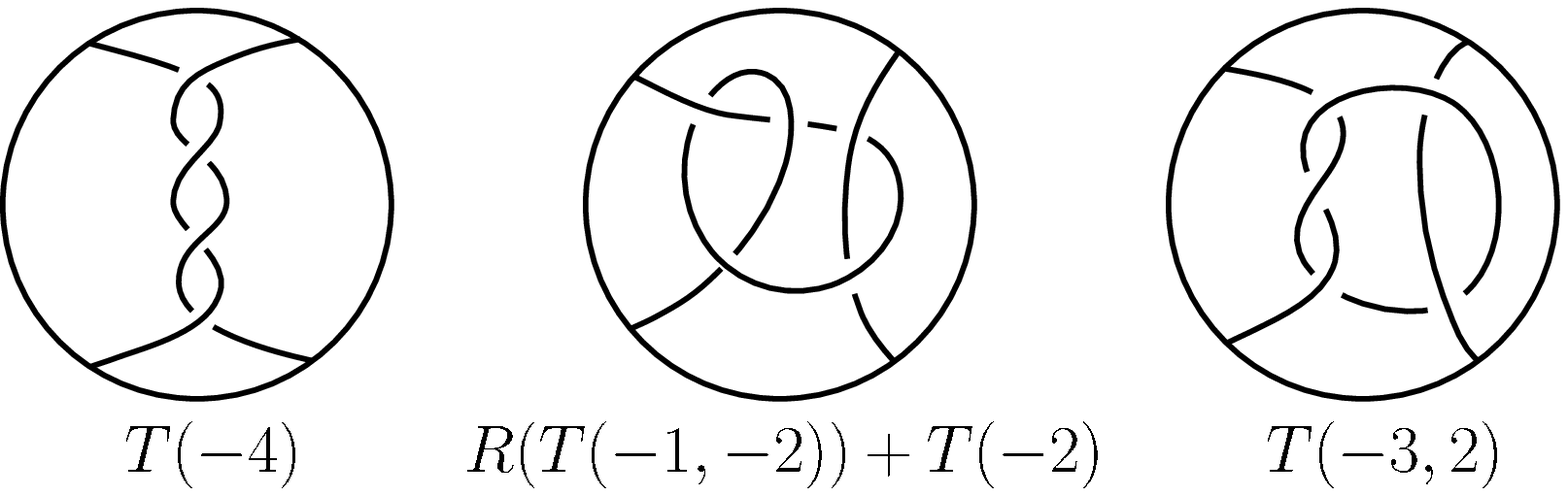}}
      \end{center}
   \caption{}
  \label{tangles-for-five-one}
\end{figure} 

%

%
\begin{figure}[htbp]
      \begin{center}
\scalebox{0.58}{\includegraphics*{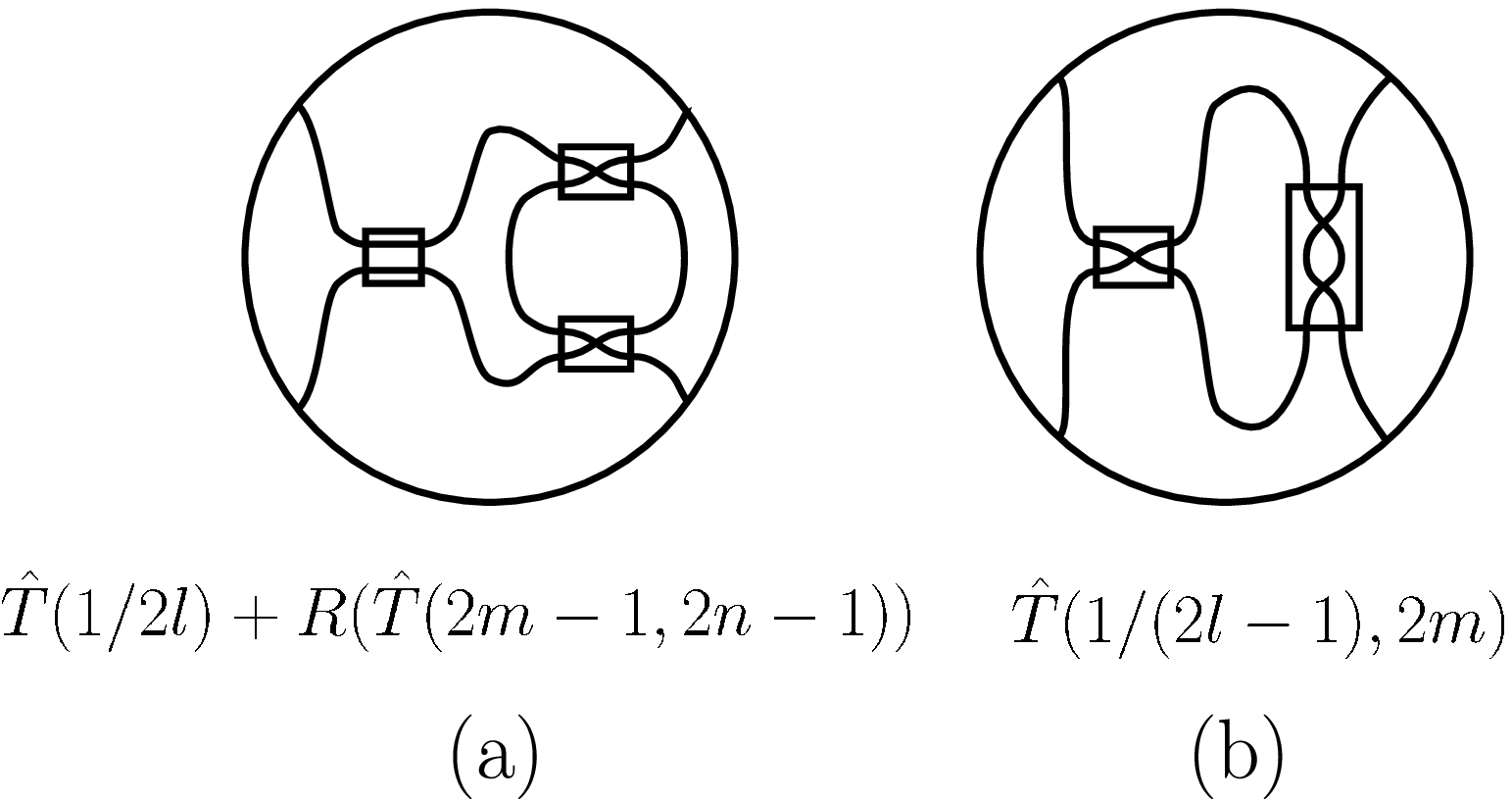}}
      \end{center}
   \caption{}
  \label{pretzel-tangles}
\end{figure} 

%

\begin{proof}
Assume that $\hat T$ is not a projection of the tangle $T(-4)$ then (1) or (2) of Lemma \ref{-4-lemma} holds.

First suppose that the condition (1) of Lemma \ref{-4-lemma} holds. Note that the tangle projection $\hat S+R(\hat T(1,1))=\hat S+\hat T(2)$ is a minor of $\hat S+R(\hat T(2m-1,2n-1))$. We apply Lemma \ref{vertical-trefoil-tangle-lemma} to $\hat S$ and have that $\hat S+\hat T(2)$ is a projection of $R(T(-1,-2))+T(-2)$ unless $\hat S=\hat T(1/2l)$ for some $l$. Thus we have the conclusion (1).

Next suppose that the condition (2) of Lemma \ref{-4-lemma} holds. Note that the tangle projection $\hat S+\hat T(2)$ is a minor of $\hat S+\hat T(2m)$. We apply Lemma \ref{one-three-lemma} to $\hat S$ and have that $\hat S+\hat T(2)$ is a projection of $T(-3,2)$ unless $\hat S=\hat T(1/(2l-1))$ for some $l$. Thus we have the conclusion (2).
\end{proof}

\begin{Lemma}\label{five-one-projection-lemma}
Let $\hat K$ be a knot projection. Suppose that 
$\hat K$ is not an underlying projection of the $(2,5)$-torus knot. 
Then $\hat K$ is a connected sum of some knot projections each of which is one of the projections in 
Fig. \ref{pretzel-knot-projection}. 
\end{Lemma}
\begin{figure}[htbp]
      \begin{center}
\scalebox{0.58}{\includegraphics*{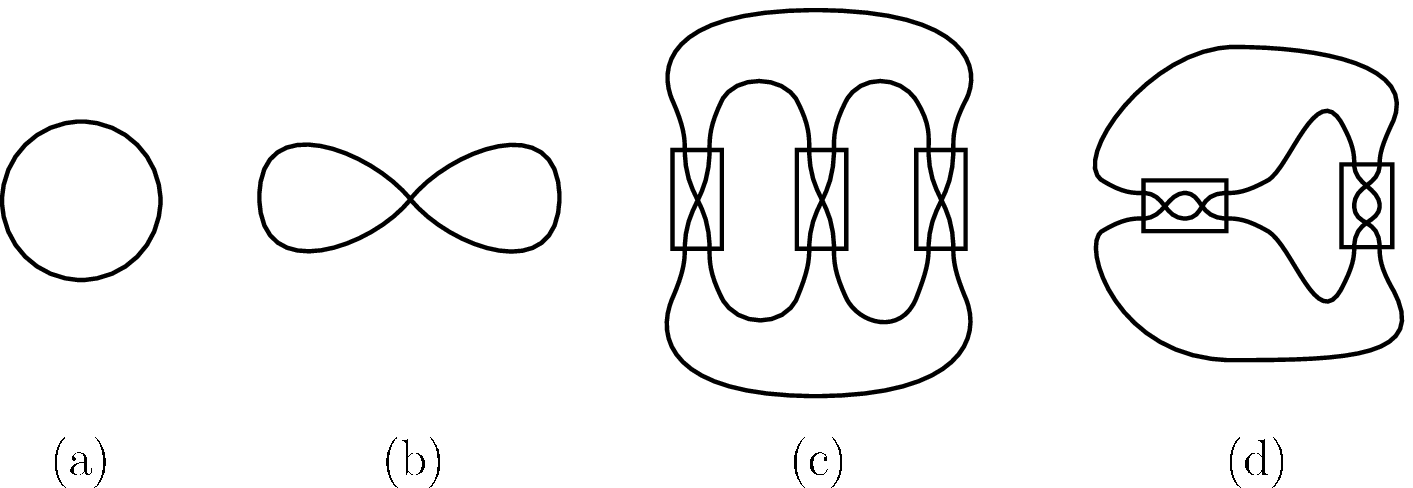}}
      \end{center}
   \caption{}
  \label{pretzel-knot-projection}
\end{figure} 

%

%
\begin{proof}
Suppose that $\hat K$ is a connected sum of prime projections $\hat K_1,\cdots,\hat K_k$. Then each $\hat K_i$ is not an underlying projection of the $(2,5)$-torus knot. We will show that each $\hat K_i$ is one of the projections in Fig. \ref{pretzel-knot-projection}.
Suppose that $\hat K_i$ has two or more crossings. Since $\hat K_i$ is prime we have that $\hat K_i$ is reduced.
Choose a 1-gon disk (tear drop disk) $\delta$ in $\hat K_i$ as follows:
orient $\hat K_i$ and choose a basepoint $P$ on $\hat K_i$ that is not a crossing point. Follow
$\hat K_i$, starting from $P$, and let $P_0$ be 
the first point (crossing) at which one  meets his trace. 
The simple closed curve which was traced divides the sphere of the 
projection into two disks and $\delta$ is the disk which 
does not contain $P$.  Let $\hat T$  be the complementary 
tangle projection of the vertex $P_0$. If $\hat T$ is a projection 
of the tangle $T(-4)$, $R(T(-1,-2))+T(-2)$, or $T(-3,2)$ then $\hat K_i$ is a projection of 
the $(2,5)$-torus knot as illustrated in Fig. \ref{five-one}.
Then we have by Lemma \ref{-4-lemma-corollary1} that
$\hat T=\hat T(1/2l)+R(\hat T(2m-1,2n-1))$ or $\hat T=\hat T(1/(2l-1),2m)$.
Thus we have that $\hat K_i$ is one of the projections in Fig. \ref{pretzel-knot-projection}.
\end{proof}
\begin{figure}[htbp]
      \begin{center}
\scalebox{0.5}{\includegraphics*{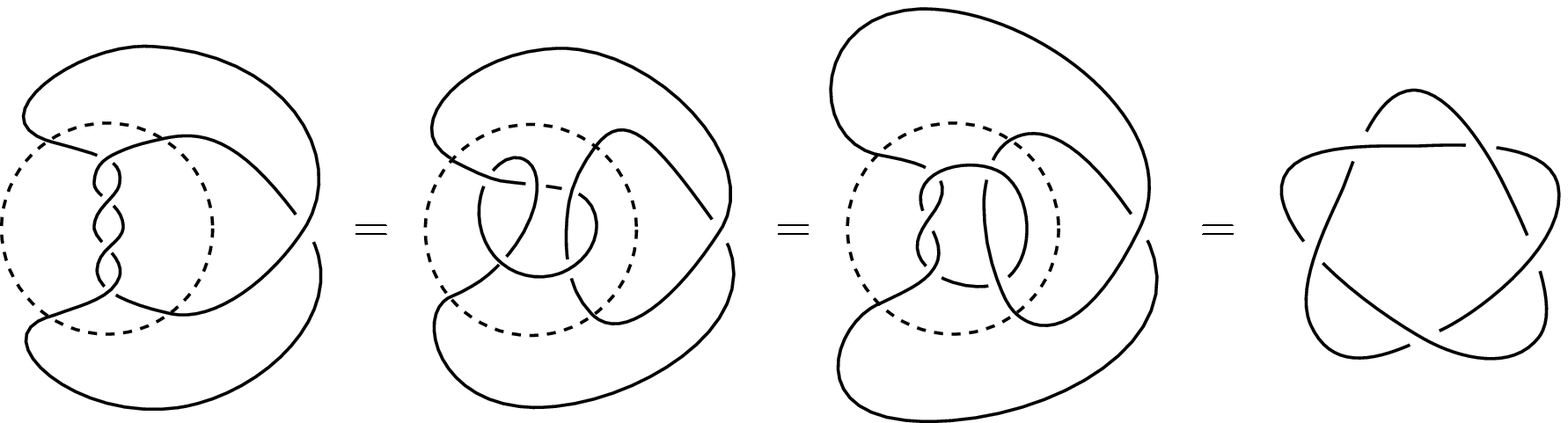}}
      \end{center}
   \caption{}
  \label{five-one}
\end{figure} 

%

\subsection{The main results on Positive Links.} \ \\

\begin{Theorem}\label{projection-theorem}
Let $\hat L=\hat K_1\cup\hat K_2\cup\cdots\cup\hat K_\mu$ be an oriented link projection. Then either

\item[(A)]
$\hat L$ is an underlying projection of the $(2,5)$-torus knot (Fig. \ref{five-one}) (with possible trivial components) or

\item[(B)]
Each $\hat K_i$ is a knot projection described in Lemma \ref{five-one-projection-lemma} and

\item[(1)]
$\hat L$ is an underlying projection of the $(2,4)$-torus link (Fig. \ref{6_2-Whitehead-3-comp} (c)) (with possible trivial components) or

\item[(2)]
$\hat L$ is an underlying projection of the $(3,3)$-torus link (Fig. \ref{three-three} (a)) (with possible trivial components) or

\item[(3)]
$\hat L$ is an underlying projection of $8^3_{10}$ (Fig. \ref{three-three} (b)) (with possible trivial components) or

\item[(4)]
$\hat L$ is an underlying projection of a connected or disjoint sum of the right-handed trefoil
knot and the right-handed Hopf link (with possible trivial components) or

\item[(5)]
$\hat L$ is an underlying projection of a disjoint and/or connected sum of three right-handed Hopf links 
 (with possible trivial components) or

\item[(6)]
$\hat L$ has no mixed crossings or

\item[(7)]
$\hat L$ is an R1 augmentation of $N(\hat T(p_1,p_2,p_3))$ with anti-parallel 
orientations of components for some positive even numbers $p_1,p_2$ and $p_3$ 
(Fig. \ref{pretzel-link-projection}) with possible trivial circles or

\item[(8)]
$\hat L$ is an R1 augmentation of $D(\hat T(2k))$ with anti-parallel orientation of 
components for some positive number $k$ (Fig. \ref{torus-link-projection}) with possible 
trivial circles or a connected or disjoint sum of two copies 
of such link projections.

\end{Theorem}

\begin{figure}[htbp]
      \begin{center}
\scalebox{0.58}{\includegraphics*{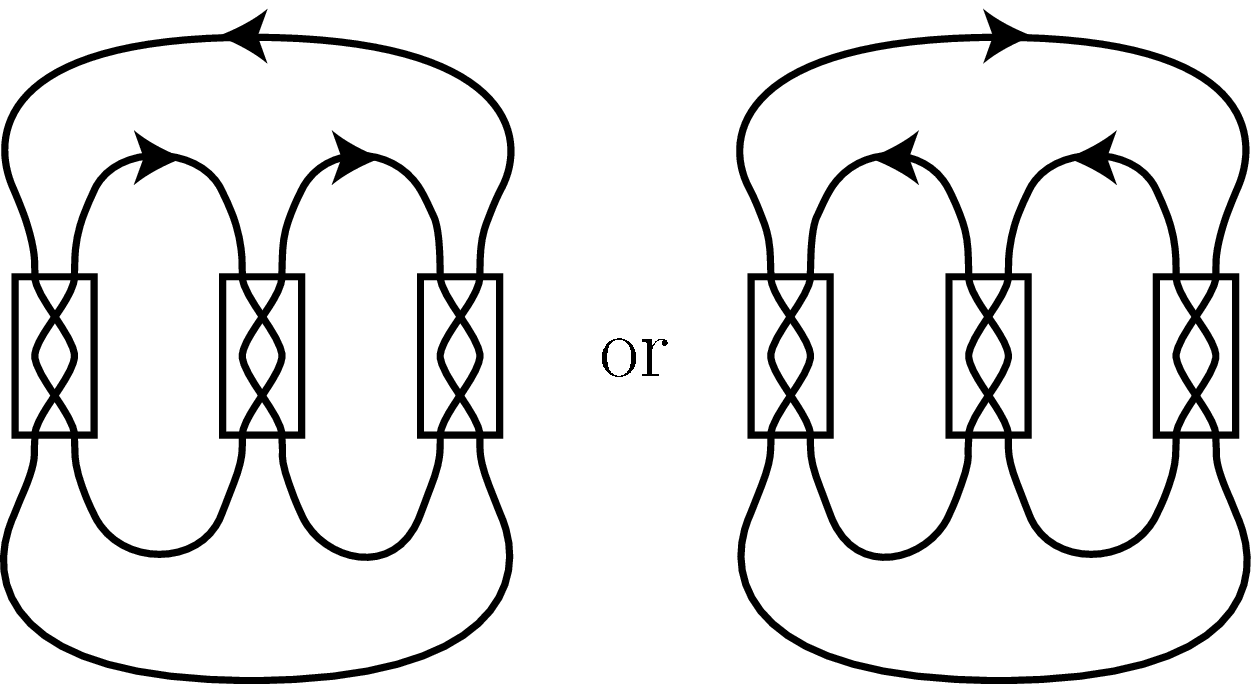}}
      \end{center}
   \caption{}
  \label{pretzel-link-projection}
\end{figure} 

%

%
\begin{figure}[htbp]
      \begin{center}
\scalebox{0.58}{\includegraphics*{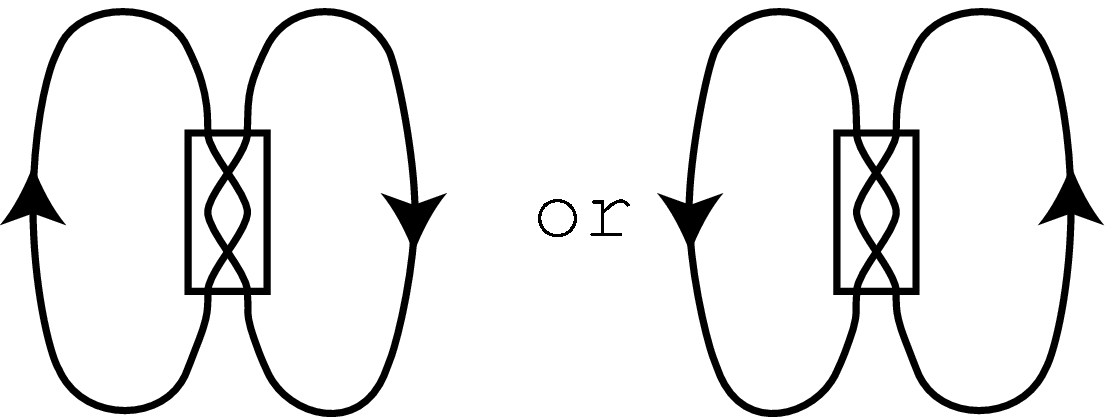}}
      \end{center}
   \caption{}
  \label{torus-link-projection}
\end{figure} 

%


\begin{proof} Suppose that some $\hat K_i$ is a projection of the $(2,5)$-torus knot. 
Then we are dealing with Case (A). Suppose that each $\hat K_i$ is not a projection of the $(2,5)$-torus knot. 
Then by Lemma \ref{five-one-projection-lemma} we have that $\hat K_i$ is a connected 
sum of some knot projections each of which is a projection in Fig. \ref{pretzel-knot-projection}. 

Suppose that $\hat K_i$ is not almost trivial and $\hat K_i$ has some mixed crossings. 
Let $P$ be a mixed crossing on $\hat K_i$ with $\hat K_j$ where $\hat K_i$ crosses from left to right of $\hat K_j$. Then we add over/under crossing information to $\hat K_i\cup\hat K_j$ so that $\hat K_j$ produces a trivial knot, $\hat K_i$ produces a right-handed trefoil knot and $\hat K_j$ is over $\hat K_i$ except $P$. Then the result is a connected sum of a right-handed trefoil knot and a right-handed Hopf link. Thus we have case (B) (4). Suppose that $\hat K_i$ has no mixed crossings but some $\hat K_j$ and $\hat K_l$ has a common mixed crossing. Then by a similar argument we have that $\hat L$ is an underlying projection of a disjoint sum of a right-handed trefoil knot and a right-handed Hopf link. Namely we again have case (B) (5). Suppose that $\hat L$ has no mixed crossings. Then we have case (B) (6). Thus in the following we assume that each $\hat K_i$ is an almost trivial projection.

Now suppose that $\hat K_i$ and $\hat K_j$ has a common mixed crossing, say $P$. We see the complementary tangle projection $\hat T$ of $\hat K_i\cup\hat K_j$ at $P$. We apply Lemma \ref{one-three-lemma} to $\hat T$ and have that $\hat K_i\cup\hat K_j$ is a projection of the $(2,4)$-torus link (Fig. \ref{6_2-Whitehead-3-comp} (c)) unless $\hat K_i\cup\hat K_j$ is an R1 augmentation of the link projection $D(\hat T(2n))$ with anti-parallel orientation for some natural number $n$ (Fig. \ref{torus-link-projection}). If $\hat K_i\cup\hat K_j$ is a projection of the $(2,4)$-torus link then we have the case (B) (1). Thus we consider the case that for any $\hat K_i$ and $\hat K_j$ that has a common mixed crossing the union $\hat K_i\cup\hat K_j$ is an R1 augmentation of the link projection $D(\hat T(2n))$ with anti-parallel orientation for some natural number $n$.

Let $\hat L=\hat L_1\cup\hat L_2\cup\cdots\cup\hat L_u$ be the connected components of $\hat L$. Namely each $\hat L_i$ is a connected projection and $\hat L_i$ and $\hat L_j$ are disjoint for $i\neq j$.

Let $\mu(i)$ be the number of knot projections that compose $\hat L_i$. We may suppose without loss of generality that $\mu(1)\geq\mu(2)\geq\cdots\geq\mu(u)$. Suppose that $\mu(1)\geq\mu(2)\geq\mu(3)\geq2$. Then for each $i=1,2,3$, choose two knot projections, say $\hat K_a$ and $\hat K_b$ in $\hat L_i$ that has a mixed crossing, say $P$. We may suppose without loss of generality that $\hat K_b$ crosses from left to right of $\hat K_a$ at $P$. Then we add over/under crossing information to $\hat L_i$ such that $\hat K_a$ is over $\hat K_b$ except at $P$. We trivialize other parts and we have that each $\hat L_i$ is a projection of a right-handed Hopf link with possible trivial components. Thus we have case (B) (5). Next we consider the case that $\mu(3)=1$ or $u\leq2$. Suppose that $\mu(1)\geq4$. Let $\hat K_{a_1},\hat K_{a_2},\cdots,\hat K_{a_{\mu(1)}}$ be the knot projections in $\hat L_1$. Consider an abstract graph $G$ with vertices $\hat K_{a_1},\hat K_{a_2},\cdots,\hat K_{a_{\mu(1)}}$ and edges $\hat K_{a_i}\hat K_{a_j}$ where $\hat K_{a_i}$ and $\hat K_{a_j}$ have mixed crossings. Let $H_0$ be a spanning tree of $G$. 
By re-ordering the vertices if necessary we may suppose that $K_{a_1}$ is a free vertex (degree one vertex) of $H_0$. Let $P_1$ be a mixed crossing on $\hat K_{a_1}$ at which $\hat K_{a_1}$ crosses from right to left of the other component. Let $H_1$ be the graphs obtained from $H_0$ by removing the vertex $K_{a_1}$ together with the edge incident to it. We may again suppose without loss of generality that $K_{a_2}$ is a free vertex of $H_1$. Let $P_2$ be a mixed crossing on $\hat K_{a_2}$ at which $\hat K_{a_2}$ crosses from right to left of the other component. Let $H_2$ be the graphs obtained from $H_1$ by removing the vertex $K_{a_2}$ together with the edge incident to it. 
We continue this process and finally reach a graph $H_{\mu(1)-1}$ on a single vertex $K_{a_{\mu(1)}}$. 
Now we give over/under crossing information to $\hat L_1$ so that $\hat K_{a_i}$ is over $\hat K_{a_j}$ if $i<j$ 
except at $P_1,P_2,\cdots,P_{\mu(1)-2}$ and $P_{\mu(1)-1}$. Then we have that the result is a connected sum of $\mu(1)-1$ right-handed Hopf links whose linking number structure respects the tree $H_0$. Thus the rest is the case that $2\leq\mu(1)\leq3$. Suppose that $\mu(1)=3$ and $\mu(2)\geq2$. Then by giving over/under crossing information to $\hat L$ in a similar way we have that $\hat L$ is a projection of the disjoint sum of two links, one is a connected sum of two right-handed Hopf links, and the other is a right-handed Hopf link. Namely we again have the case (B) (5). 

Suppose that $\mu(1)=3$ and $\mu(2)=1$ or $u=1$. We may suppose that $\hat K_1,\hat K_2$ and $\hat K_3$ are the components of $\hat L_1$. 
We may suppose without loss of generality that $\hat K_1$ and $\hat K_2$ have common mixed crossings and $\hat K_2$ and $\hat K_3$ have common mixed crossings. Since $\hat K_1\cup \hat K_2$ is an R1 augmentation of the link projection $D(\hat T(2n))$ with anti-parallel orientation it is as illustrated in Fig. \ref{three-three-proof1} where the tangles are almost trivial 1-tangles. Let $A_1,A_2,\cdots,A_{2n}$ be the common mixed crossings of $\hat K_1$ and $\hat K_2$ that appears in this order on $\hat K_2$ along the orientation of $\hat K_2$. 

\begin{figure}[htbp]
      \begin{center}
\scalebox{0.48}{\includegraphics*{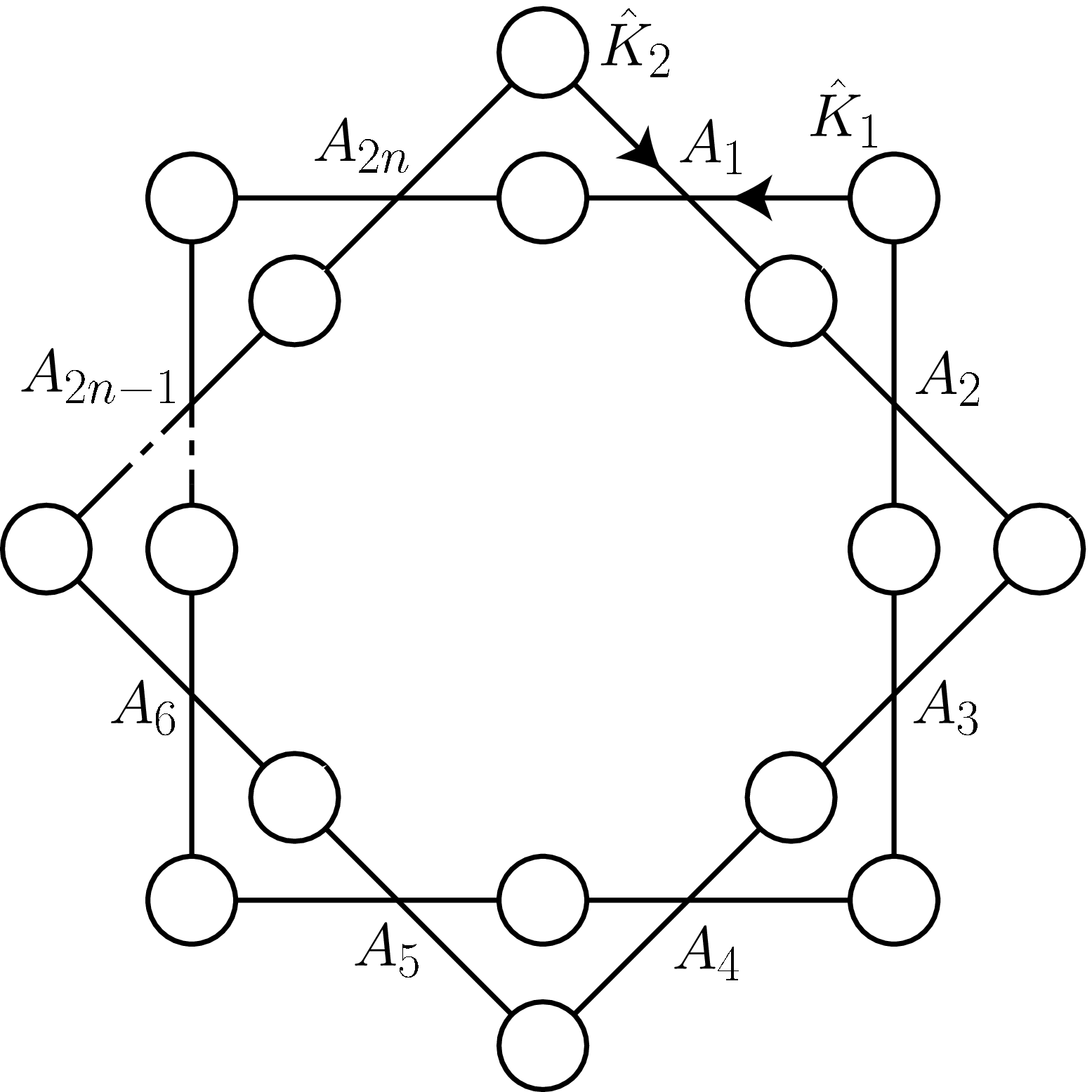}}
      \end{center}
   \caption{}
  \label{three-three-proof1}
\end{figure} 

%

Case 1. $\hat K_1$ and $\hat K_3$ have no common mixed crossings.

First suppose that $\hat K_3$ has intersection with only one of $A_1A_2,A_2A_3,\cdots,A_{2n-1}A_{2n}$. 
Then we have that the projection $\hat L_1$ is a connected sum of two link projections and we have the case (B) (8). Next suppose that  $\hat K_3$ has intersection with at least two of $A_1A_2,A_2A_3,\cdots,A_{2n-1}A_{2n}$, say $A_1A_2$ and $A_{2i-1}A_{2i}$. Since $\hat K_2\cup \hat K_3$ is also an R1 augmentation of the link projection $D(\hat T(2m))$ with anti-parallel orientation for some natural number $m$ it has the property that if we do 
smoothing at a self-crossing of $\hat K_2$ only one of the two components of the result has intersection with $\hat K_3$. This implies that $\hat K_3$ do not intersects the R1 residual of $\hat K_2$. Thus in this case we may assume that each of $\hat K_1,\hat K_2$ and $\hat K_3$ is a simple closed curve on ${\mathbb S}^2$. Then we have that $\hat L_1$ is as illustrated in Fig. \ref{three-three-proof2} (a) and it is a projection of the link $8^3_{10}$. See for example Fig. \ref{three-three-proof2} (b). Thus we have case (B) (3).

\begin{figure}[htbp]
      \begin{center}
\scalebox{0.58}{\includegraphics*{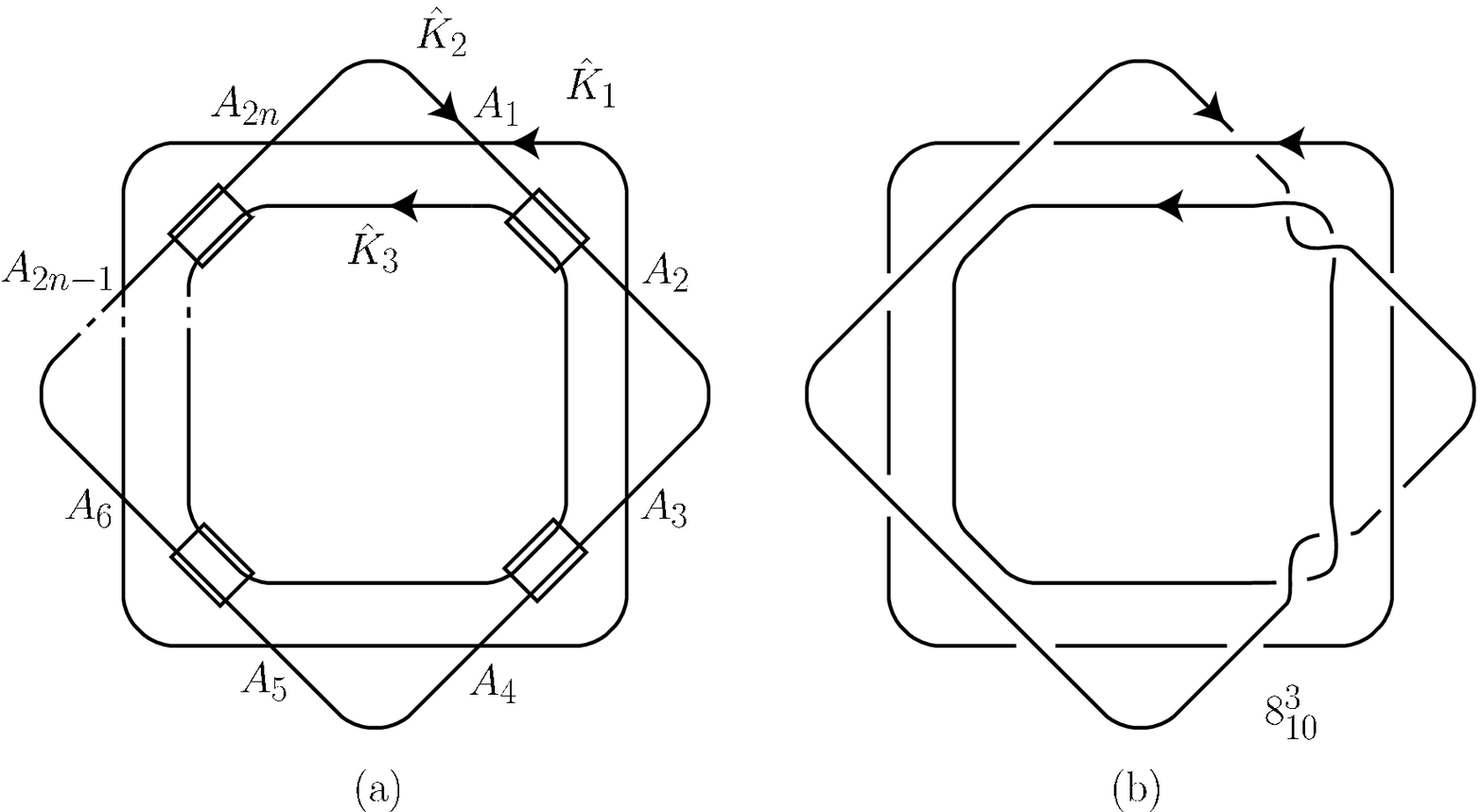}}
      \end{center}
   \caption{}
  \label{three-three-proof2}
\end{figure} 

%

Case 2. $\hat K_1$ and $\hat K_3$ have common mixed crossings.

First suppose that $\hat K_3$ intersects R1 residual of $\hat K_1$ or $\hat K_2$, say $\hat K_2$. 
Then by taking a minor of $\hat L_1$ using Lemma \ref{reducing-lemma} if necessary 
we have one of the situations illustrated in Fig. \ref{three-three-proof3} and 
Fig. \ref{three-three-proof4}. Then we have that $\hat L_1$ is a projection of the $(3,3)$-torus link 
as is also illustrated in \ref{three-three-proof3} and Fig. \ref{three-three-proof4}.

\begin{figure}[htbp]
      \begin{center}
\scalebox{0.68}{\includegraphics*{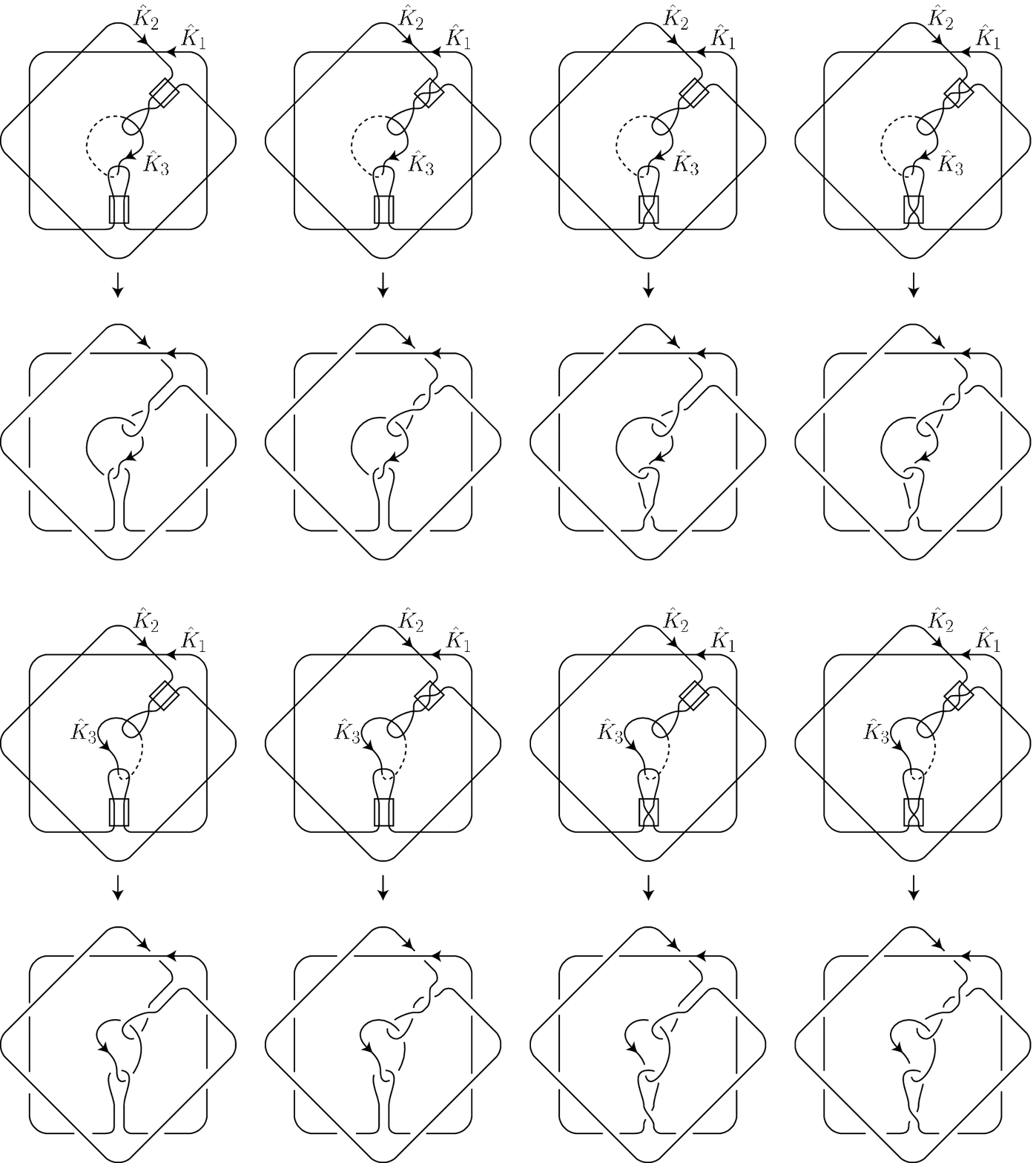}}
      \end{center}
   \caption{}
  \label{three-three-proof3}
\end{figure} 

%

%
\begin{figure}[htbp]
      \begin{center}
\scalebox{0.68}{\includegraphics*{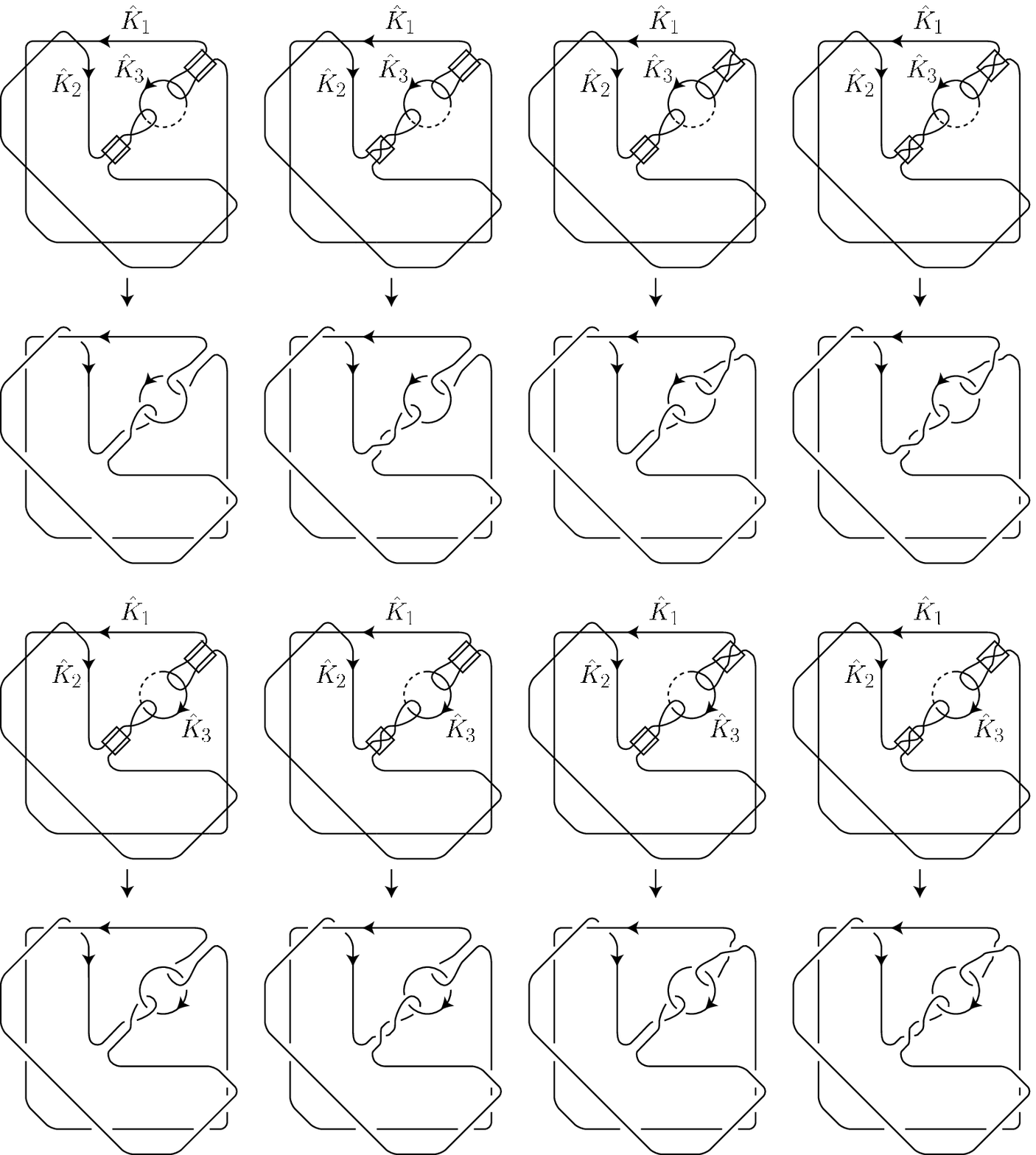}}
      \end{center}
   \caption{}
  \label{three-three-proof4}
\end{figure} 

%

Thus we may assume that $\hat K_3$ has no intersection with R1 residuals of $\hat K_1$ and $\hat K_2$. 
Therefore we may assume that both $\hat K_1$ and $\hat K_2$ are simple closed curves on ${\mathbb S}^2$. 
By changing the role of $\hat K_2$ and $\hat K_3$ and doing the same argument we may further assume 
that $\hat K_1$ is also a simple closed curve on ${\mathbb S}^2$. Suppose that there are consecutive 
crossings on $\hat K_3$ each of which $\hat K_3$ crosses from left to right, or from right to left. 
Then we have that $\hat L_1$ is a projection of $(3,3)$-torus link as illustrated in Fig. \ref{three-three-proof5} (a). 
Thus we have that two consecutive crossings $P_1$ and $P_2$ on $\hat K_3$ are placed as illustrated in 
Fig \ref{three-three-proof5} (b). (Since $(3,3)$-torus link is invertible and $\hat K_1$ and $\hat K_2$ 
are exchangeable omit another orientation possibility.) If the next crossing $P_3$ on $\hat K_3$ is 
before $P_2$ with respect to the orientation of the other curve then we have the  $(3,3)$-torus link 
as illustrated in Fig. \ref{three-three-proof5} (c). Therefore we have the situation illustrated 
in Fig. \ref{three-three-proof5} (d). 
Then we have that only possible projection is a projection in Fig. \ref{pretzel-link-projection}. 
See Fig. \ref{three-three-proof5} (e). Thus we have case (B) (7).

\begin{figure}[htbp]
      \begin{center}
\scalebox{0.51}{\includegraphics*{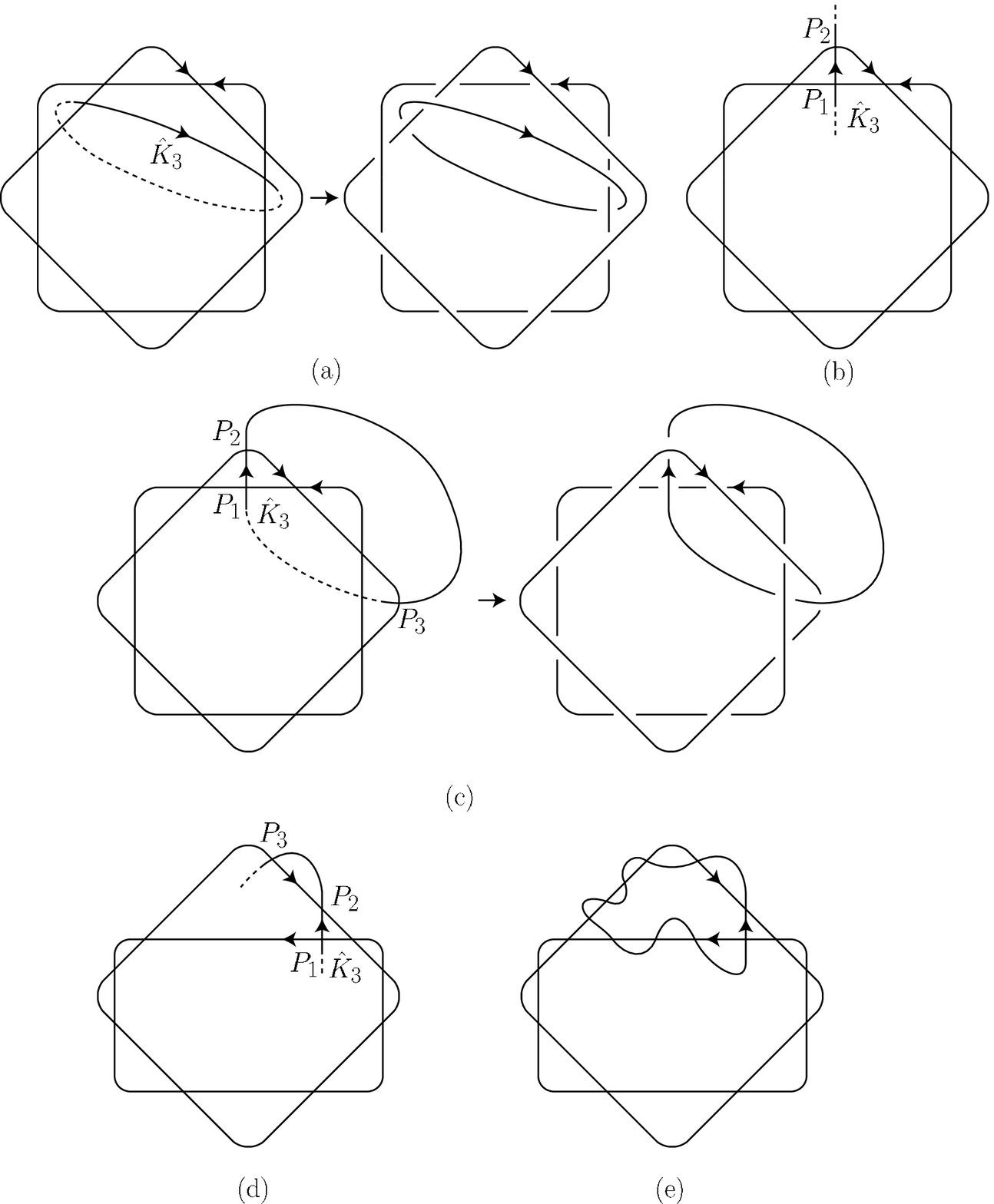}}
      \end{center}
   \caption{}
  \label{three-three-proof5}
\end{figure} 

%

Finally suppose that $\mu(1)=2$. Then we have the case (B) (8).
This completes the proof.
\end{proof}


\begin{Theorem}\label{positive-theorem}
Let $L$ be a nontrivial positive link. Then either

\item[(A)]
$L\geq$ $(2,5)$-torus knot (Fig. \ref{five-one}) (with possible trivial components) or

\item[(B)]
Each component of $L$ is a connected sum of pretzel knots 
$L(p_{1},p_{2},p_{3})$ where $p_1,p_2$ and $p_3$ are positive odd numbers (Fig. \ref{pretzel-knot}) or the unknot and

\item[(1)]
$L\geq$ $(2,4)$-torus link (Fig. \ref{6_2-Whitehead-3-comp} (c)) (with possible trivial components) or

\item[(2)]
$L\geq$ $(3,3)$-torus link (Fig. \ref{three-three} (a)) (with possible trivial components) or

\item[(3)]
$L\geq$ $8^3_{10}$ (Fig. \ref{three-three} (b)) (with possible trivial components) or

\item[(4)]
$L\geq$ connected or disjoint sum of the right-handed trefoil
knot and the right-handed Hopf link (with possible trivial components) or

\item[(5)]
$L\geq$ disjoint and/or connected sum of three right-handed Hopf links 
 (with possible trivial components) or

\item[(6)]
$L$ is a (connected and/or disjoint sum of) pretzel knot(s) 
$L(p_{1},p_{2},p_{3})$ (Fig. \ref{pretzel-knot})  (with possible trivial components) or

\item[(7)]
$L$ is a three-component pretzel link  $L(p_{1},p_{2},p_{3})$ (Fig. \ref{three-three} (c))
(with possible trivial components) or

\item[(8)]
$L$ is a $(2,2k)$-torus link with anti-parallel orientation of 
components (Fig. \ref{three-three} (d)) or the connected or disjoint sum of two copies 
of such links (with possible trivial components).

\end{Theorem}

\begin{figure}[htbp]
      \begin{center}
\scalebox{0.58}{\includegraphics*{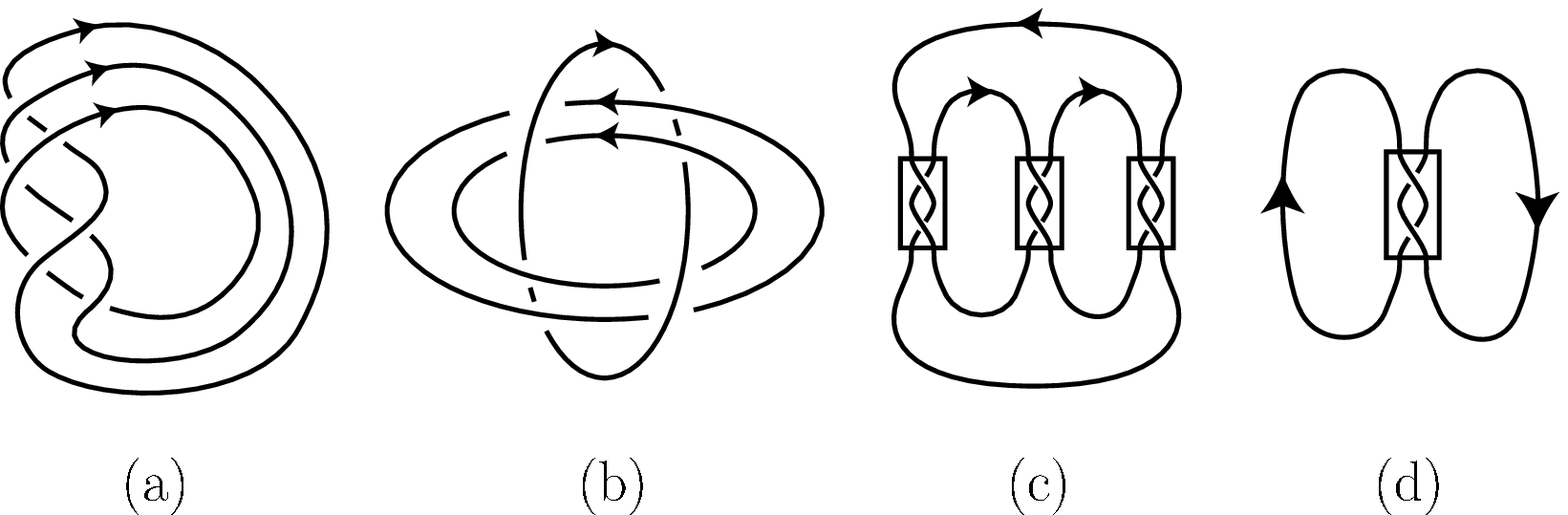}}
      \end{center}
   \caption{}
  \label{three-three}
\end{figure} 

%

%
\begin{proof}
First note that a positive diagram whose underlying projection is the projection illustrated in 
Fig. \ref{pretzel-knot-projection} (d) represents a pretzel knot $L(1,p_2,p_3)$ for some positive 
odd numbers $p_2$ and $p_3$. See Fig. \ref{pretzel-knot2}. Also note that a pretzel knot 
$L(p_1,p_2,p_3)$ where $p_1,p_2$ and $p_3$ are positive odd numbers is greater than or 
equal to the right-handed trefoil knot $L(1,1,1)$.
Let $\tilde L$ be a positive diagram of $L$ and $\hat L$ its underlying projection. 
Then by Theorem \ref{projection-theorem} and the facts noted above we have the result.
\end{proof}
\begin{figure}[htbp]
      \begin{center}
\scalebox{0.58}{\includegraphics*{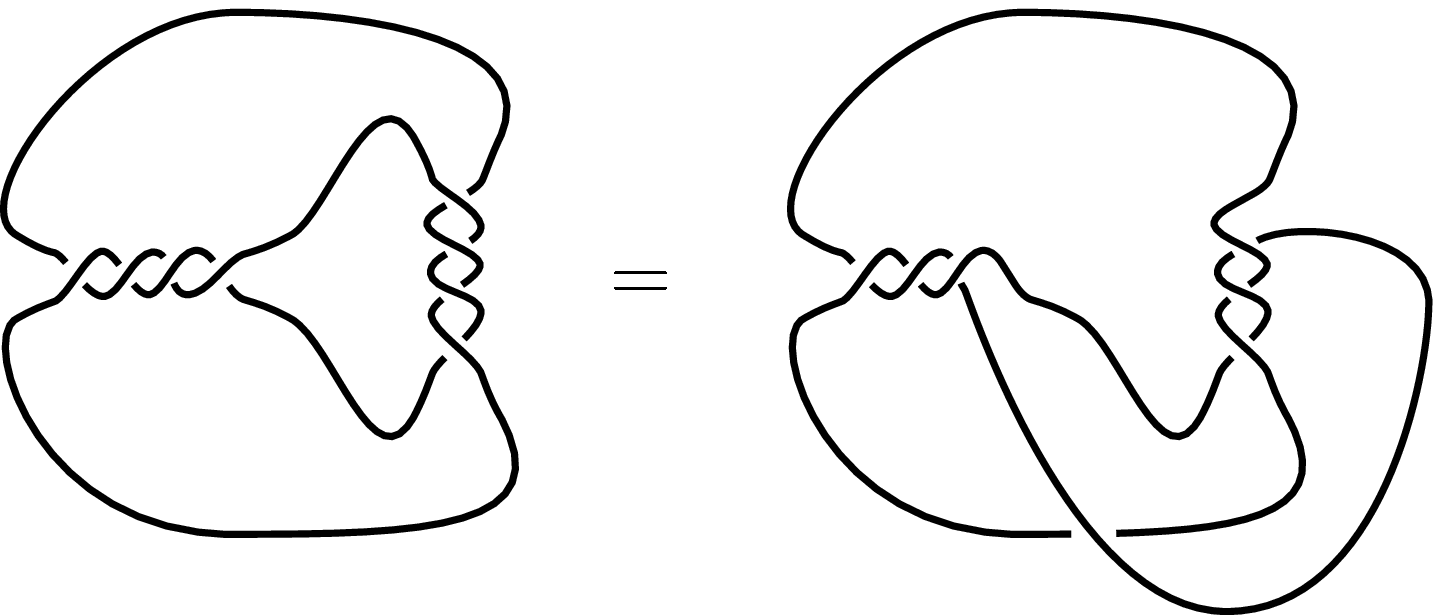}}
      \end{center}
   \caption{}
  \label{pretzel-knot2}
\end{figure} 

%

\begin{Corollary}\label{positive-corollary}
Let $L$ be a nontrivial positive link, then either $\sigma (L)\leq -3$ or:

\begin{enumerate}
\item[(a)]
$L$ is a  pretzel knot  $L(p_{1},p_{2},p_{3})$ for some positive odd numbers $p_1,p_2$ and $p_3$ (Fig. \ref{pretzel-knot}) 
(with possible trivial components); then $\sigma (L)=-2$ or

\item[(b)]
$L$ is a  three component pretzel link  
$L(p_{1},p_{2},p_{3})$ for some positive even numbers $p_1,p_2$ and $p_3$ (Fig. \ref{three-three} (c))  (with possible trivial components);
then $\sigma (L)=-2$ or

\item[(c)]
$L$ is a  $(2,2k)$-torus link with anti-parallel orientation of components (Fig. \ref{torus-link2}) (with possible trivial components); then $\sigma (L)=-1$ or

\item[(d)]
$L$ is a connected or disjoint sum of two copies of links from (c) 
(with possible trivial components); then $\sigma (L)=-2$.
\end{enumerate}
In particular if $K$ is a nontrivial positive knot, then either 
$\sigma (L)\leq -4$ or (a) holds.
\end{Corollary}

\begin{proof}
It follows immediately from the fact that the signatures of 
the $(2,5)$-torus knot, $(2,4)$-torus link, $(3,3)$-torus link, and the link 
$8^3_{10}$ (Fig. \ref{three-three} (b)) are -4,-3, -4 and -3 
respectively, and that if $L_{1}\geq L_{2}$ then, by Giller inequality, $\sigma (L_{1})\leq\sigma (L_{2})$.
\end{proof}

\begin{Remark}\label{Remark 2.17}
We work, in the next section, with almost positive links and demonstrate there, 
in particular, the title result of the paper that a nontrivial {\it Almost Positive Links 
have negative Signature} (Corollary 1.7). However Corollary 1.7 can be also derived 
from Corollary 2.16.
The outline of this derivation is as follows.\ Let $\tilde L$ be an almost positive diagram of a link $L$. 
Let $\tilde L_+$ be the positive diagram obtained from $\tilde L$ by changing the negative 
crossing to a positive crossing. Let $L_+$ be the link represented by $\tilde L_+$. 
If $\sigma(L_+)\leq-3$ then we have $\sigma(L)\leq-1<0$ as desired. 
Then by checking the exceptional cases (a)-(d) we obtain the result. 
In the next section we give a different proof of Corollary 1.7. 
It applies Theorem 3.2 which is based on Lemma \ref{almost-positive-tangle-lemma}. 
This lemma plays an important role also in Section 4.

\end{Remark}


\section{Almost positive links}\label{Almost positive links}

\begin{Lemma}\label{almost-positive-tangle-lemma}
Let $\tilde T$ be a reduced almost positive 2-string tangle diagram 
with vertical connection.
Suppose that the negative crossing is a mixed crossing. 
If $\tilde T$ is not greater than
or equal to any of the tangles $T(-2),T(3_1,0_1)$ and $T(0_1,3_1)$ in Fig. \ref{three-tangles}, then $\tilde T$ is one of 
the diagrams $\tilde T_{1+},\tilde T_{1-},\tilde T_{2+}$ and $\tilde T_{2-}$ in Fig. \ref{almost-positive-tangles}.
\end{Lemma}

\begin{figure}[htbp]
      \begin{center}
\scalebox{0.58}{\includegraphics*{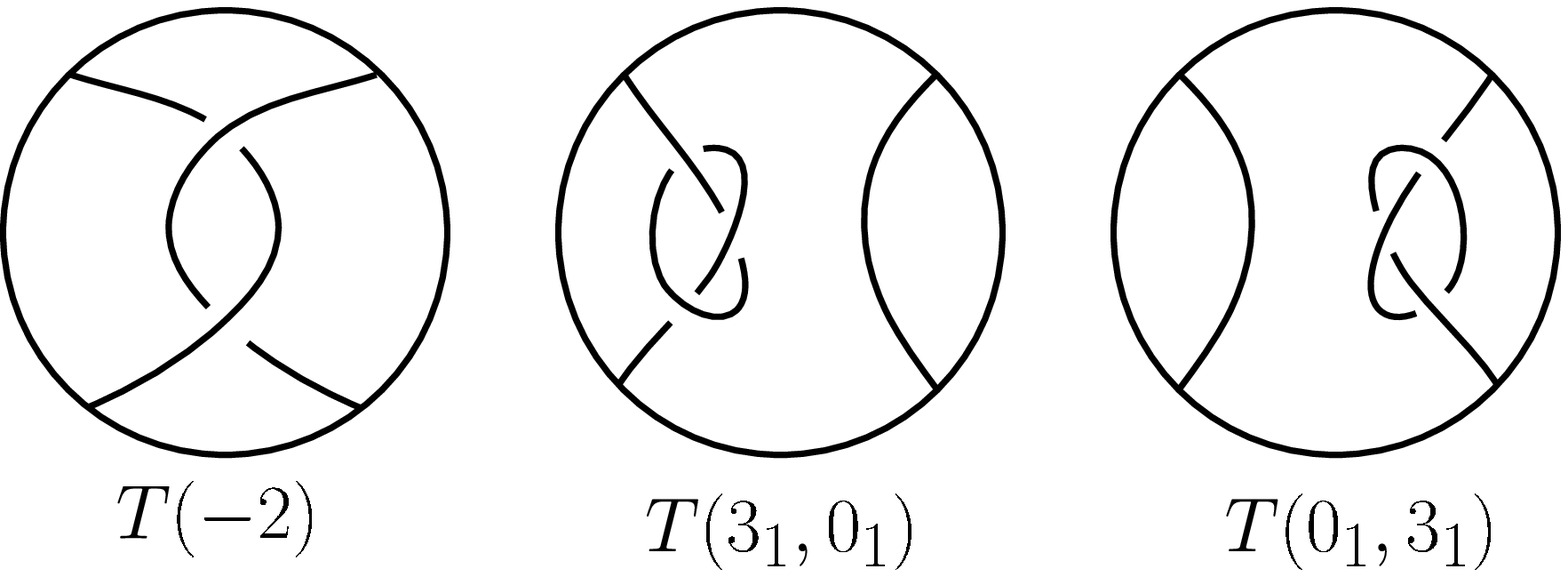}}
      \end{center}
   \caption{}
  \label{three-tangles}
\end{figure} 

%

%
\begin{figure}[htbp]
      \begin{center}
\scalebox{0.65}{\includegraphics*{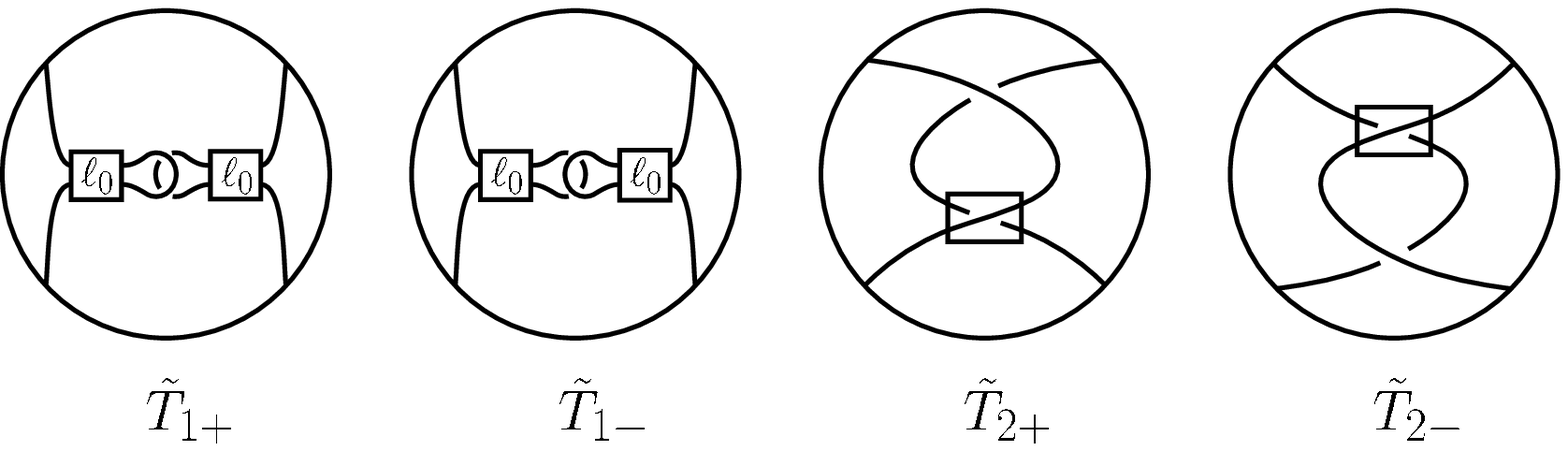}}
      \end{center}
   \caption{}
  \label{almost-positive-tangles}
\end{figure} 

%

\begin{proof}
Let $N$ be the negative crossing of $\tilde T$. We have by Lemma \ref{local-trefoil-lemma} that both components 
of $\tilde T$ are almost trivial.

\begin{enumerate}
\item[Case 1.]
$\tilde T$ has just two mixed crossings.

Then we have that these crossings are related on both strings and we have that $\tilde T$ is one of $\tilde T_{1\pm}$.

\item[Case 2.]
$\tilde T$ has four or more mixed crossings.

\item[Case 2.1.] $\tilde T$ has a self-crossing.
We may suppose without loss of generality that $\tilde a$ has a self-crossing.

First suppose that $m(N,\tilde a)\geq1$ and $P$ is the root of $N$ on $\tilde a$. 
Suppose that the positive tangle $r(\tilde T,P)$ has mutual crossings. Then by Lemma \ref{hook-tangle} we have that $r(\tilde T,P)\geq T(-2)$. Since we have $\tilde T\geq r(\tilde T,P)$ by Lemma \ref{diagram-reducing-lemma} we have $\tilde T\geq T(-2)$. Thus we may suppose that if $m(N,\tilde a)\geq1$ and $P$ is the root of $N$ on $\tilde a$, then $r(\tilde T,P)$ has no mutual crossings. Then we have that $N$ is rightmost.
We have one of the following:

\item[Case 2.1.1.]
There are two positive mixed crossings that are related on $\tilde a$ which are rightmost and have multiplicity greater than zero on $\tilde a$.

\item[Case 2.1.2.]
Case 2.1.1 does not hold and there are two positive crossings that are related on $\tilde a$
which are not rightmost.

In Case 2.1.1 we can choose such crossings $B_{i}$ and $B_{j}$ $(i<j)$ so that $i$ is odd
and $B_{k}$ is not related to $B_{i}$ on $\tilde\alpha$ for $i<k<j$.

Then we can operate, according to the position of $N$, one of the following two crossing
changes without changing the negative crossing $N$.

\item[(1)]
$B_{0}B_{i}^-$ is over the other parts and $B_{j}^+B_{\infty}$ is under the other parts.

\item[(2)]
$B_{0}B_{i}^-$ is under the other parts and $B_{j}^+B_{\infty}$ is over the other parts.
\end{enumerate}

Then we have $\tilde T\geq T(-2)$ as illustrated in Fig. \ref{almost-positive-proof1}.

\begin{figure}[htbp]
      \begin{center}
\scalebox{0.65}{\includegraphics*{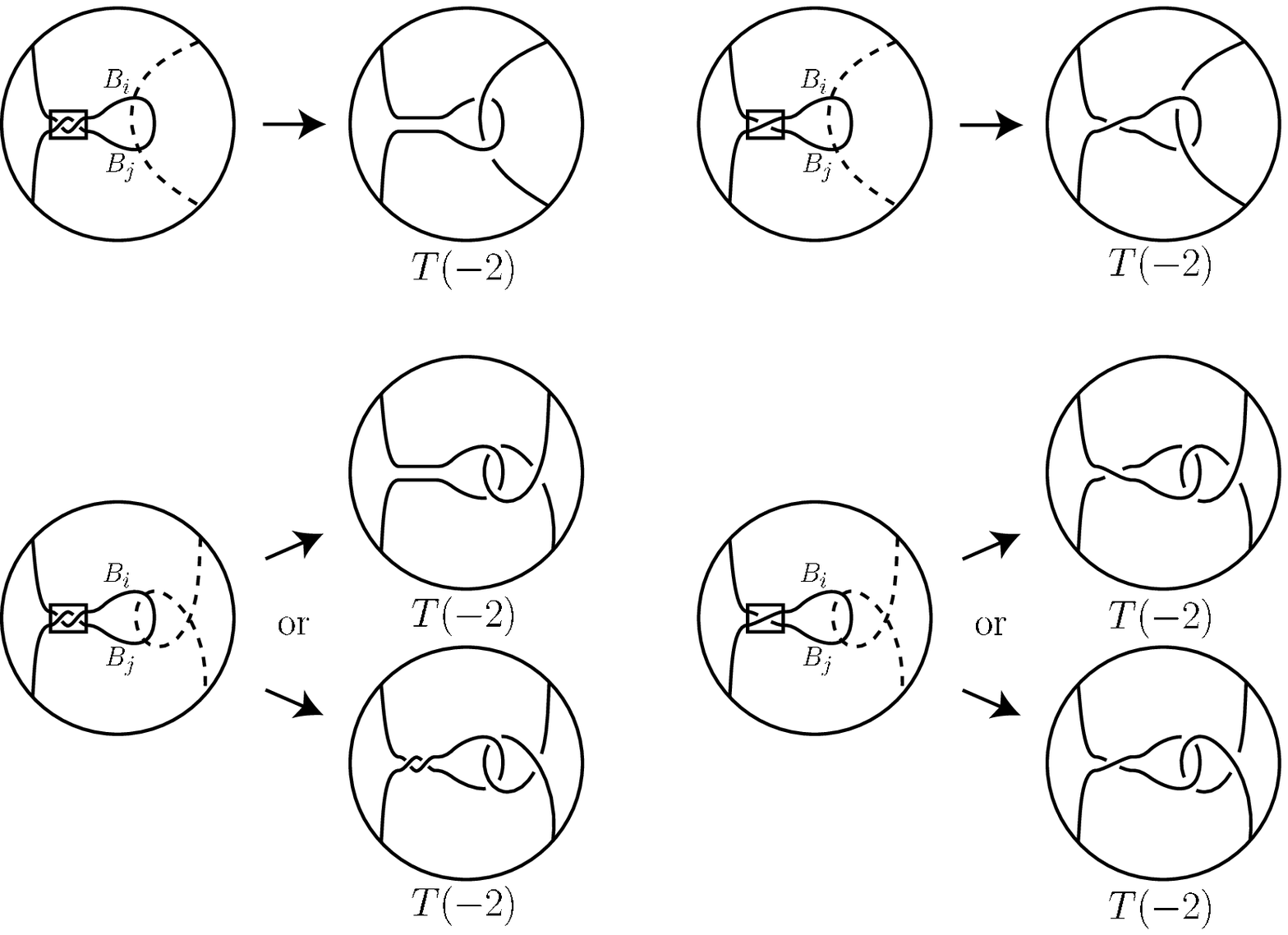}}
      \end{center}
   \caption{}
  \label{almost-positive-proof1}
\end{figure} 

%

In Case 2.1.2, using also Lemma \ref{hook-tangle}, we have $\tilde T\geq T(-2)$ as illustrated (up to horizontal symmetry) in Fig. \ref{almost-positive-proof2}.

\begin{figure}[htbp]
      \begin{center}
\scalebox{0.65}{\includegraphics*{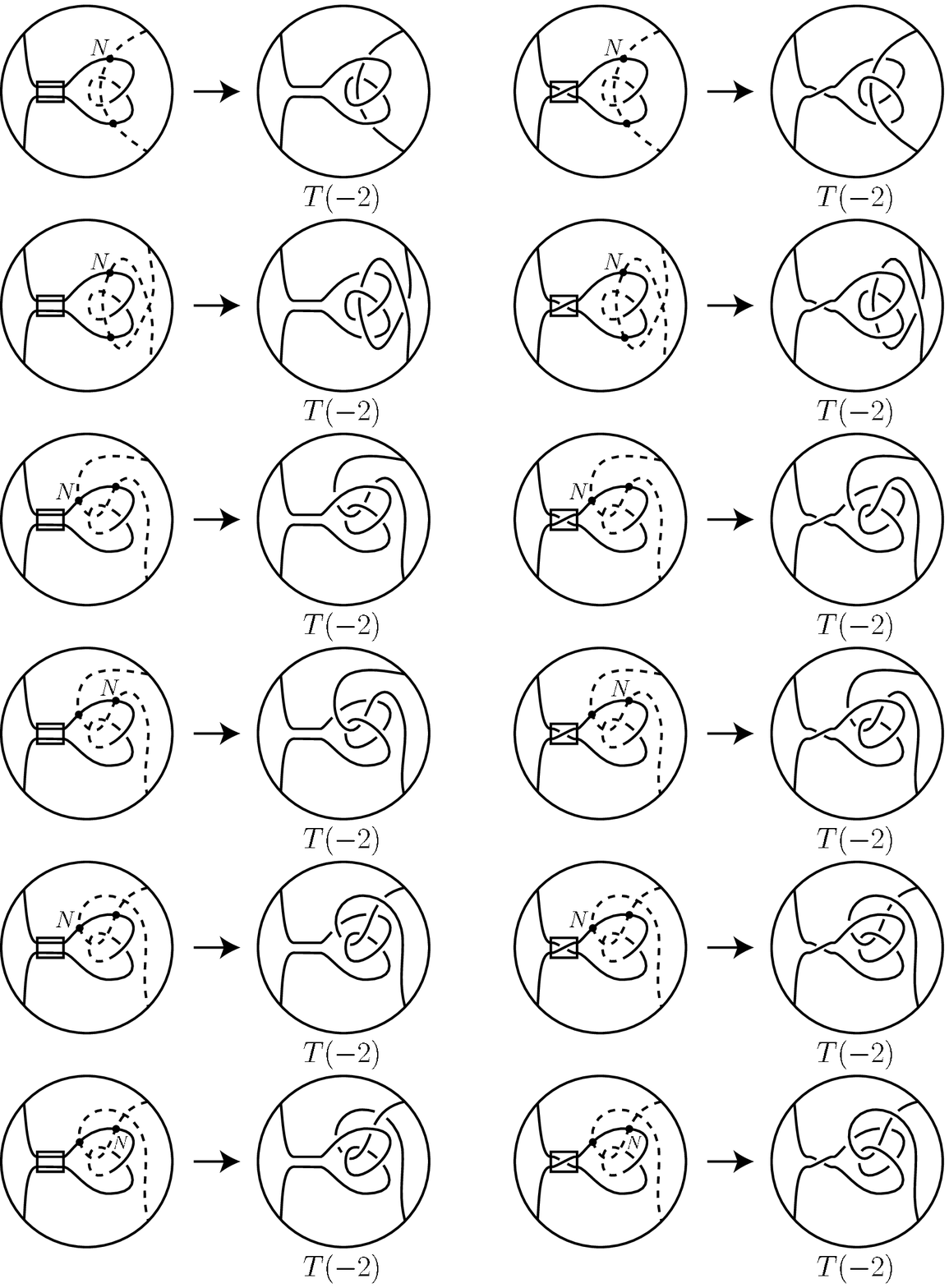}}
      \end{center}
   \caption{}
  \label{almost-positive-proof2}
\end{figure} 

%

\begin{enumerate}
\item[Case 2.2.]
$\tilde T$ has no self-crossings.
\end{enumerate}

Suppose that $N$ is neither the first crossing nor the last crossing on $\tilde b$. 
Then we have one of the situations illustrated in Fig. \ref{almost-positive-proof3} and then have  
$\tilde T\geq T(-2)$ as illustrated.

\begin{figure}[htbp]
      \begin{center}
\scalebox{0.65}{\includegraphics*{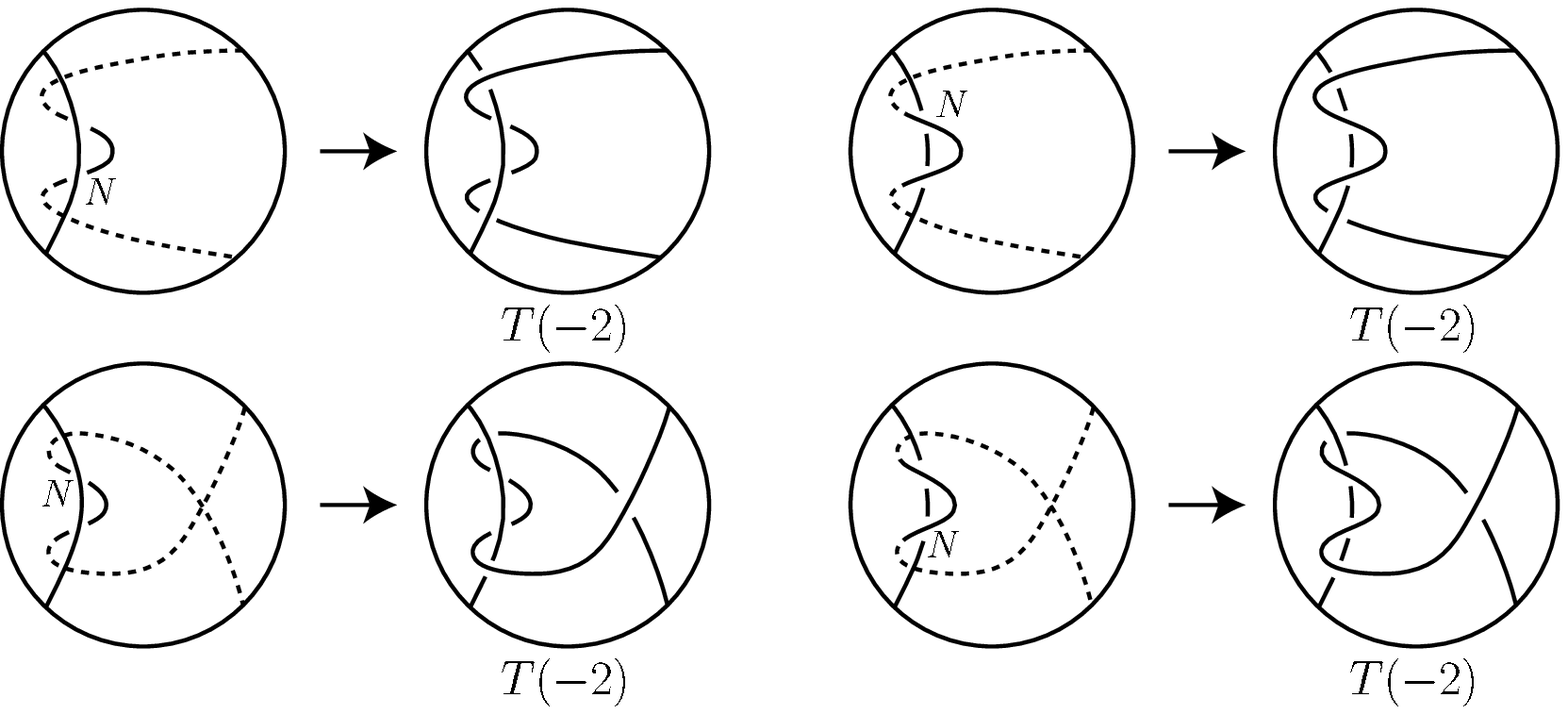}}
      \end{center}
   \caption{}
  \label{almost-positive-proof3}
\end{figure} 

%

Therefore $N$ must be the first or the last crossing on $\tilde b$. 
In addition if there is a part in $\tilde T$ illustrated in Fig. \ref{almost-positive-proof4} then we have $\tilde T\geq T(-2)$.
Then an easy analysis forces $\tilde T$ to be $\tilde T_{\pm}$. This completes the proof.
\end{proof}

\begin{figure}[htbp]
      \begin{center}
\scalebox{0.65}{\includegraphics*{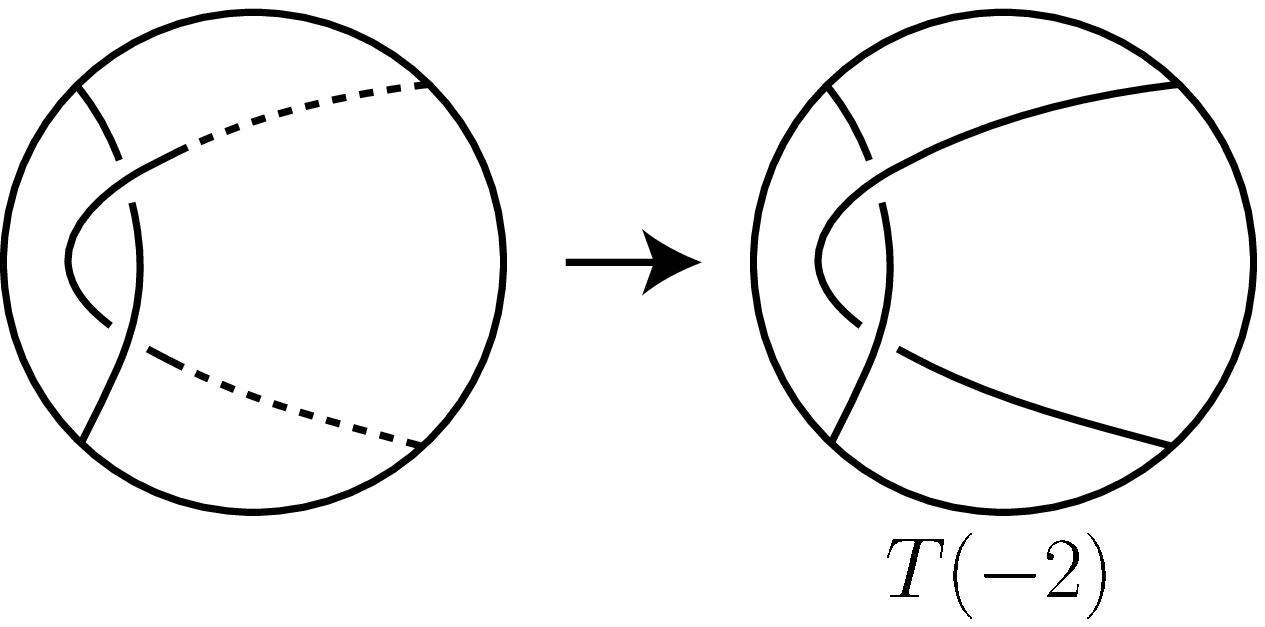}}
      \end{center}
   \caption{}
  \label{almost-positive-proof4}
\end{figure} 

%

In the following we call the technique used in Case 2.1.1 as \lq\lq over and under technique''.

\begin{Theorem}\label{almost-positive-theorem}
Let $\tilde L$ be a reduced almost positive diagram of a link $L$. If $\tilde L$ is not a 
diagram in Fig. \ref{almost-positive-trivial} (plus trivial circles), then

$L\geq$ right-handed trefoil knot (plus trivial components) or

$L\geq$ right-handed Hopf link (plus trivial components).
\end{Theorem}

\begin{proof}
First suppose that the negative crossing $N$ of $\tilde L$ is a self-crossing of a component
say $\tilde\ell$. If $\tilde L$ has a mixed crossing then we have $L\geq$ right-handed Hopf
link (plus trivial components) by Lemma 2.2. Suppose that $\tilde\ell$ has no mixed crossings with other 
components. The negative crossing $N$ divides $\tilde\ell$ into two parts. If the two parts 
have no common crossings except $N$, then either we have the right-handed trefoil knot or
$\tilde\ell$ is almost trivial. Otherwise, choose such a crossing (which is positive)
and look at the complementary tangle diagram of the crossing. Note that $N$ is a mixed crossing of that tangle diagram. 
Then by Lemma \ref{almost-positive-tangle-lemma} either we have $\tilde\ell\geq$ 
right-handed trefoil knot or $\tilde\ell$ is the diagram $\tilde T_2$ of Fig. \ref{almost-positive-trivial}.

Next suppose that $N$ is a mixed crossing of the components, say $\tilde\ell_{1}$
and $\tilde\ell_{2}$. If there are components $\tilde\ell_{i}$ and $\tilde\ell_{j}$
with $\tilde\ell_{i}\cap\tilde\ell_{j}\neq\emptyset$ and $\{i,j\}\neq\{1,2\}$, then we have the
right-handed Hopf link. If not, then applying Lemma \ref{one-three-lemma} to the complementary tangle projection
of $N$, we have $\tilde\ell_{1}\cup\tilde\ell_{2}\geq$right-handed Hopf link or
$\tilde\ell_{2}\cup\tilde\ell_{2}$ is the diagram $\tilde T_3$ of Fig. \ref{almost-positive-trivial}. This 
completes the proof.
\end{proof}

Theorem \ref{ap1}, Theorem \ref{ap2}, Corollary \ref{ap3} and Corollary \ref{ap4} are immediate corollaries of 
Theorem \ref{almost-positive-theorem}.

\section{2-almost positive links}

We start from  some preparatory lemmas that are used to analyze 2-almost positive diagrams. 
In the following proofs of these lemmas we use sometimes basic facts from 
Lemma \ref{reducing-lemma}, Lemma \ref{local-trefoil-lemma}, Lemma \ref{hook-tangle} etc. from Section 1 without 
explicitly mentioning it. When we give over/under crossing information to a projection, we divide a projection 
into some arcs and circles and give over/under crossing information to the self-crossings of these arcs and 
circles by descending or ascending algorithm without explicitly mentioning it.

\begin{Lemma}\label{lemma1}
Let $\hat T$ be a prime 2-string tangle projection with vertical connection. If $\hat T$
is not an underlying projection of any of the tangles in Fig. \ref{tangles1}, then $\hat T$ is one of 
the projections illustrated in Fig. \ref{projections1}
\end{Lemma}

\begin{figure}[htbp]
      \begin{center}
\scalebox{0.65}{\includegraphics*{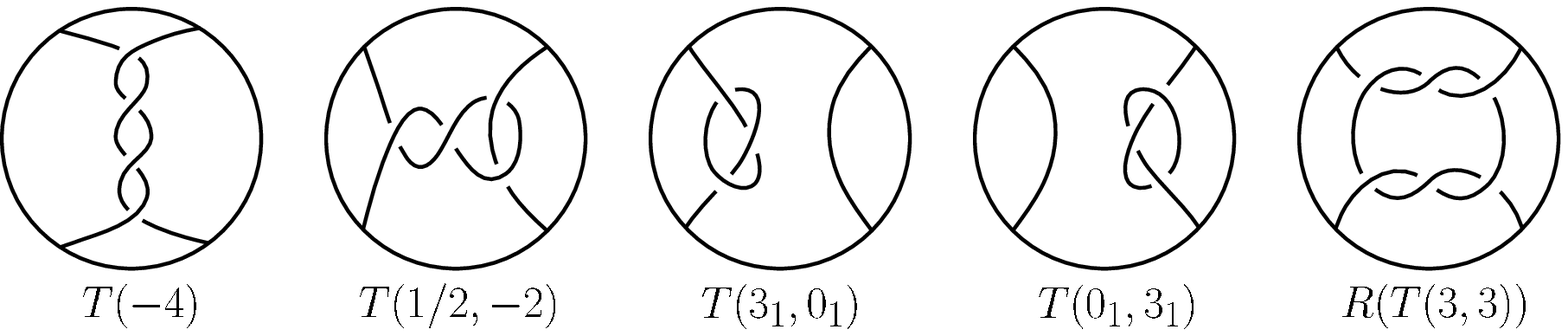}}
      \end{center}
   \caption{}
  \label{tangles1}
\end{figure} 

%

%
\begin{figure}[htbp]
      \begin{center}
\scalebox{0.65}{\includegraphics*{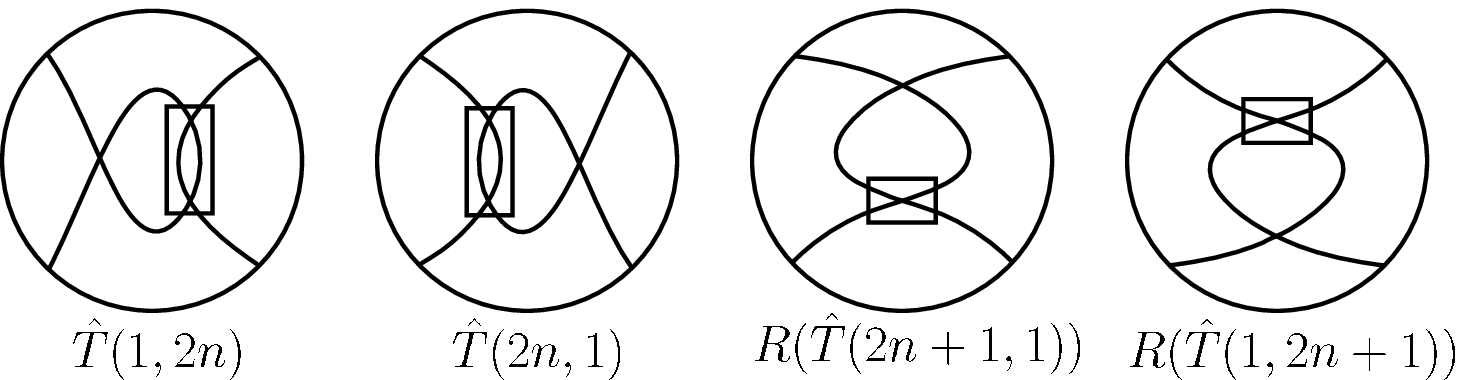}}
      \end{center}
   \caption{}
  \label{projections1}
\end{figure} 

%

\begin{proof}
We have that both $\hat a$ and $\hat b$ are almost trivial by Lemma \ref{local-trefoil-lemma}. 
Suppose that there is a rightmost mixed crossing on $\hat a$ with its multiplicity greater than one, 
or a leftmost mixed crossing on $\hat b$ with its multiplicity greater than one. 
Suppose that we have the former case. Then possibly after taking a minor of $\hat T$, still denoted by $\hat T$ using Lemma \ref{reducing-lemma}, we find a tear drop disk $\delta$ and two rightmost mixed crossings $B_i$ and $B_{i+1}$ on $\partial\delta$ with multiplicity greater than one and $B_iB_{i+1}\subset\delta$. Then we add over/under crossing information to $\hat T$ respecting the multiplicity and the position of $B_i$ and $B_{i+1}$ and have the tangle $T(1/2,-2)$ as illustrated in Fig. \ref{lemma1-proof1}.

\begin{figure}[htbp]
      \begin{center}
\scalebox{0.65}{\includegraphics*{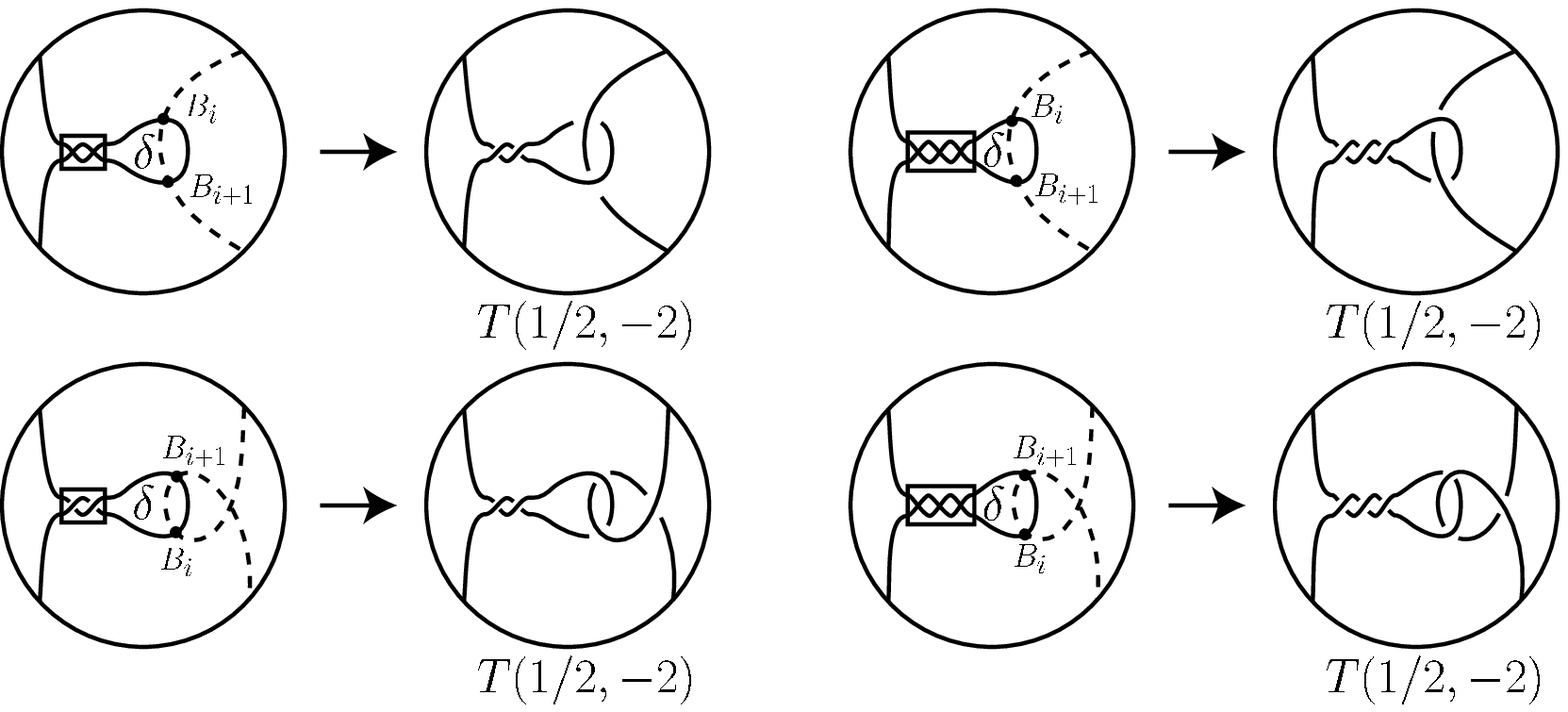}}
      \end{center}
   \caption{}
  \label{lemma1-proof1}
\end{figure} 

%

Suppose that there is a mixed crossing with its depth greater than zero. Then there is a mixed crossing, say $B_i$ on $\hat a$ with depth one. Then using Lemma \ref{reducing-lemma} we may assume that there is a tear drop disk $\delta$ with $B_i\in\partial \delta$. Then by re-choosing $B_i$ if necessary we may further assume that $B_{i+1}\in\partial\delta$ and $B_iB_{i+1}\subset\delta$. We divide $\hat b$ into $B_0B_i,B_iB_{i+1}$ and $B_{i+1}B_\infty$ and give over/under crossing information to $\hat T$ so that $B_0B_i^-$ is over everything and $B_{i+1}^+B_\infty$ is under everything and other crossings are as illustrated in Fig. \ref{lemma1-proof2} and we have the tangle $T(-4)$.

\begin{figure}[htbp]
      \begin{center}
\scalebox{0.65}{\includegraphics*{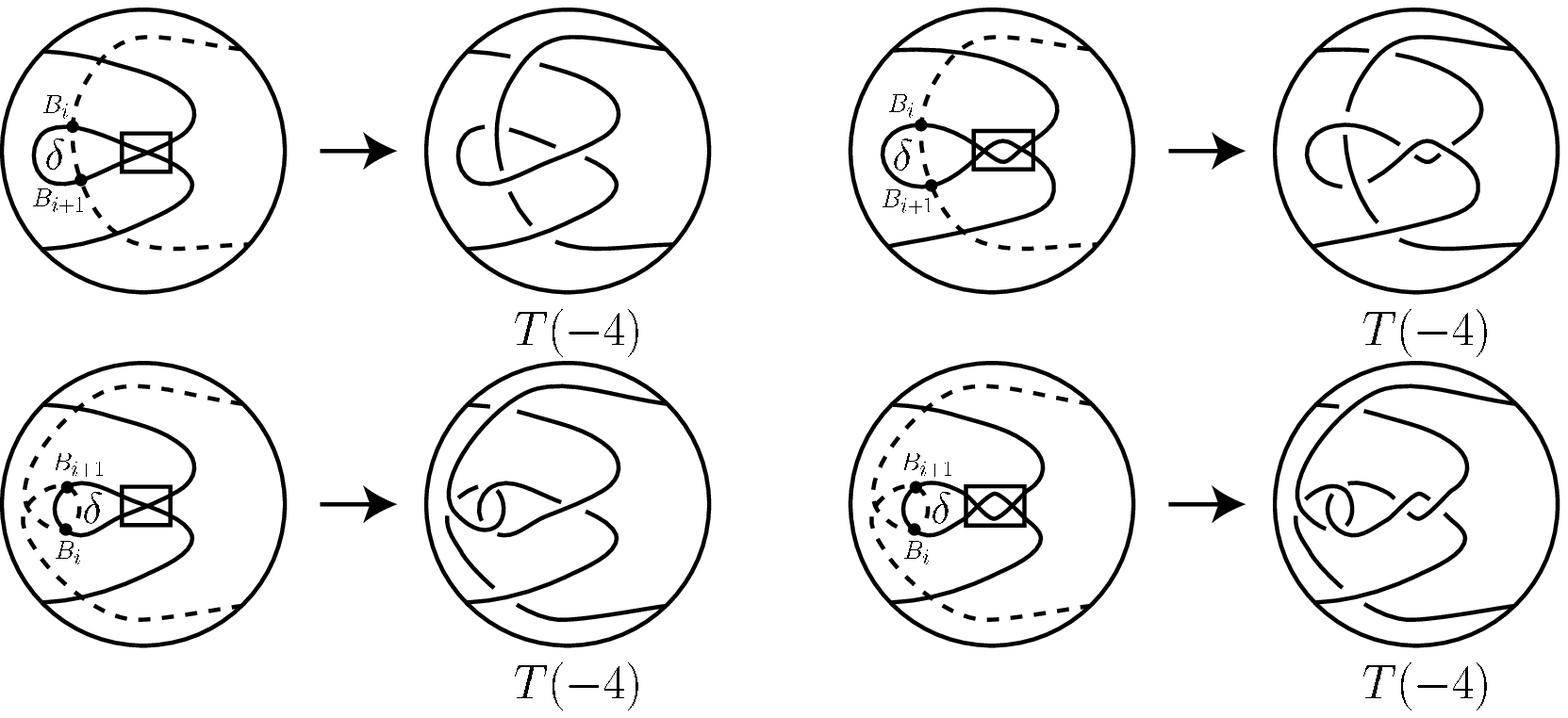}}
      \end{center}
   \caption{}
  \label{lemma1-proof2}
\end{figure} 

%

Therefore we have that $\hat a$ is of the form illustrated in Fig. \ref{lemma1-proof3}

\begin{figure}[htbp]
      \begin{center}
\scalebox{0.65}{\includegraphics*{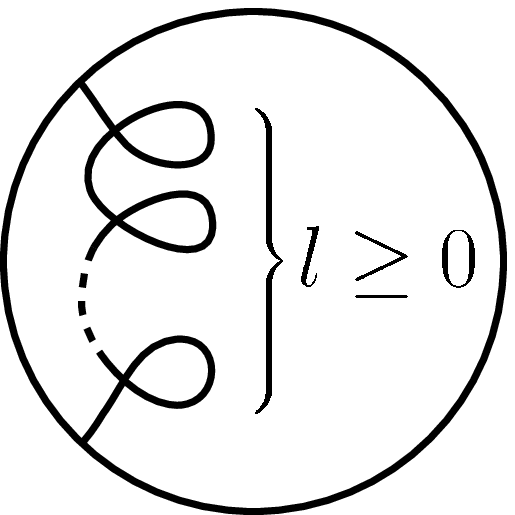}}
      \end{center}
   \caption{}
  \label{lemma1-proof3}
\end{figure} 

%

Suppose $l>0$ and the spine $\hat a'$ of $\hat a$ has mixed crossings. 
Using Lemma \ref{reducing-lemma} we reduce $\hat a$ so that $\hat a$ bounds just one 1-gon $\delta$. Let $V$ be the vertex of $\delta$.

\begin{enumerate}
\item[Case 1]
The first mixed crossing $B_{1}$ of $\hat b$ is on $A_{0}V$.

Then we have the tangle $T(-4)$ as illustrated in Fig. \ref{lemma1-proof4} (a) where $B_i$ denotes the first crossing of $\hat b$ at which $\hat b$ meets $\delta$. 

\item[Case 2.]
$B_{1}$ is on $VA_{\infty}$.

\item[Case 2.1.]
The last crossing of $\hat b$ on $B_{0}B_{1}\cup B_{1}A_{\infty}$ is on $B_{0}B_{1}$.

We have the tangle $T(1/2,-2)$ as illustrated in Fig. \ref{lemma1-proof4} (b).

\item[Case 2.2.]
The last crossing of $\hat b$ on $B_{0}B_{1}\cup B_{1}A_{\infty}$ is on $B_{1}A_{\infty}$.

We have the tangle $T(-4)$ as illustrated in Fig. \ref{lemma1-proof4} (c).

\item[Case 3.]
$B_{1}$ is on $\partial\delta$.

\item[Case 3.1.]
The first crossing of $\hat b$ on $\hat a'=A_{0}V\cup VA_{\infty}$ is on $A_{0}V$.

We have the tangle $T(1/2,-2)$. See Fig. \ref{lemma1-proof4} (d).

\item[Case 3.2.]
The first crossing of $\hat b$ on $\hat a'=A_{0}V\cup VA_{\infty}$ is on $VA_{\infty}$.

We have the tangle $T(-4)$ as illustrated in Fig. \ref{lemma1-proof4} (e).

\end{enumerate}

Therefore we have that $\hat a'$ has no mixed crossings of $\hat T$.
Similarly, if $\hat b$ is not simple, then $\hat b'$ has no mixed crossings
of $\hat T$. If both $\hat a$ and $\hat b$ have self-crossings, then we
have the tangle $T(1/2,-2)$ by Lemma \ref{hook-tangle}.
Then we have that at least one of $\hat a$ and $\hat b$ has no self
crossings. Then the result follows from Lemma \ref{-4-lemma} together with Lemma \ref{hook-tangle} and Lemma \ref{one-three-lemma}.
\end{proof}

\begin{figure}[htbp]
      \begin{center}
\scalebox{0.65}{\includegraphics*{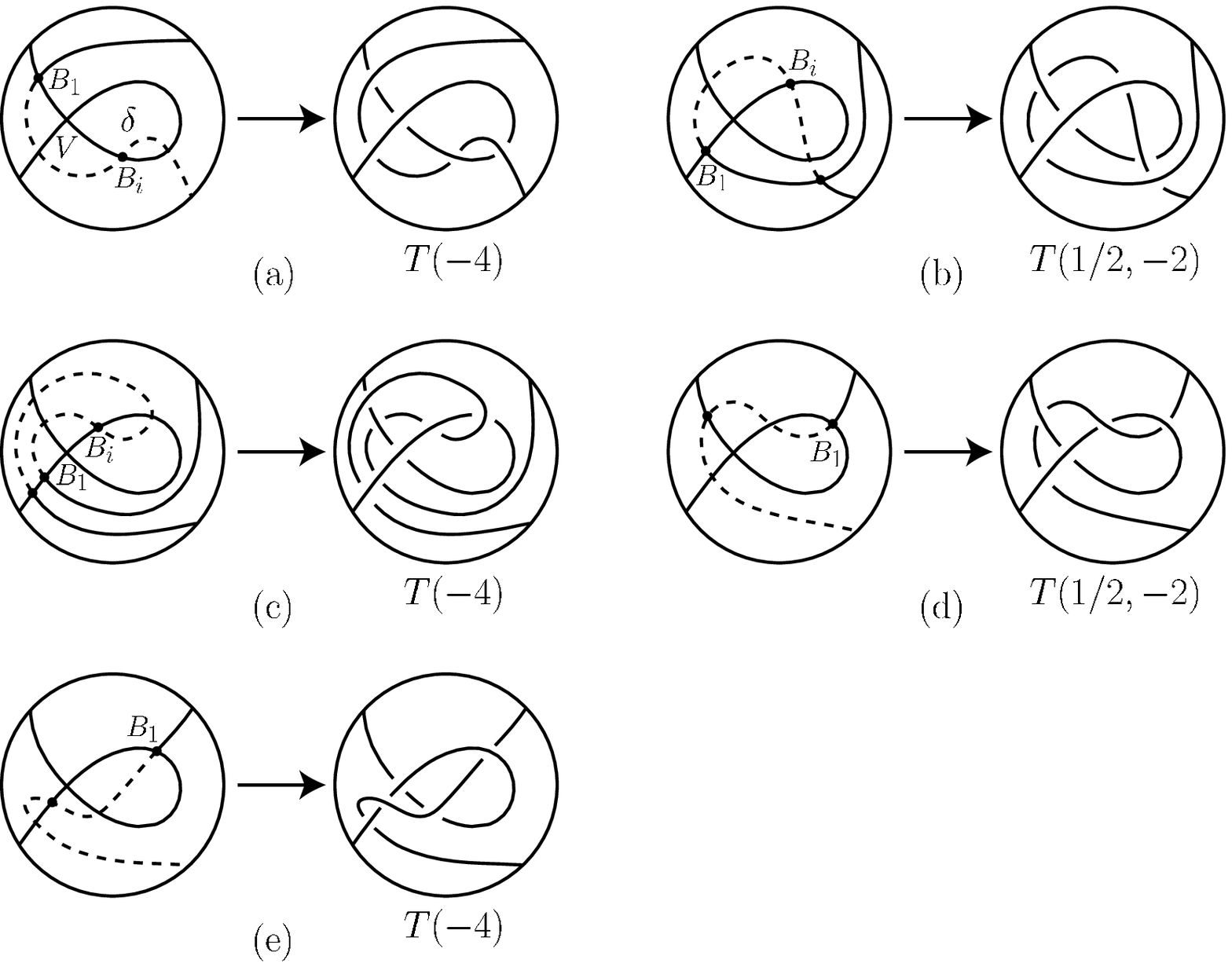}}
      \end{center}
   \caption{}
  \label{lemma1-proof4}
\end{figure} 

%

\begin{Lemma}\label{lemma2}
Let $\tilde T$ be a prime R2-reduced 2-almost positive 2-string 
tangle diagram with vertical connection. 
Suppose that $\tilde T$ has a negative self-crossing on the left 
string $\tilde a$ and a negative mixed crossing. 
If $\tilde T$ is not greater than or equal to any of the tangles in
Fig. \ref{tangles2}, then $\tilde T$ is one of the forms illustrated 
in Fig. \ref{diagrams2} where dotted lines
have no crossings with the real lines.
\end{Lemma}

\begin{figure}[htbp]
      \begin{center}
\scalebox{0.65}{\includegraphics*{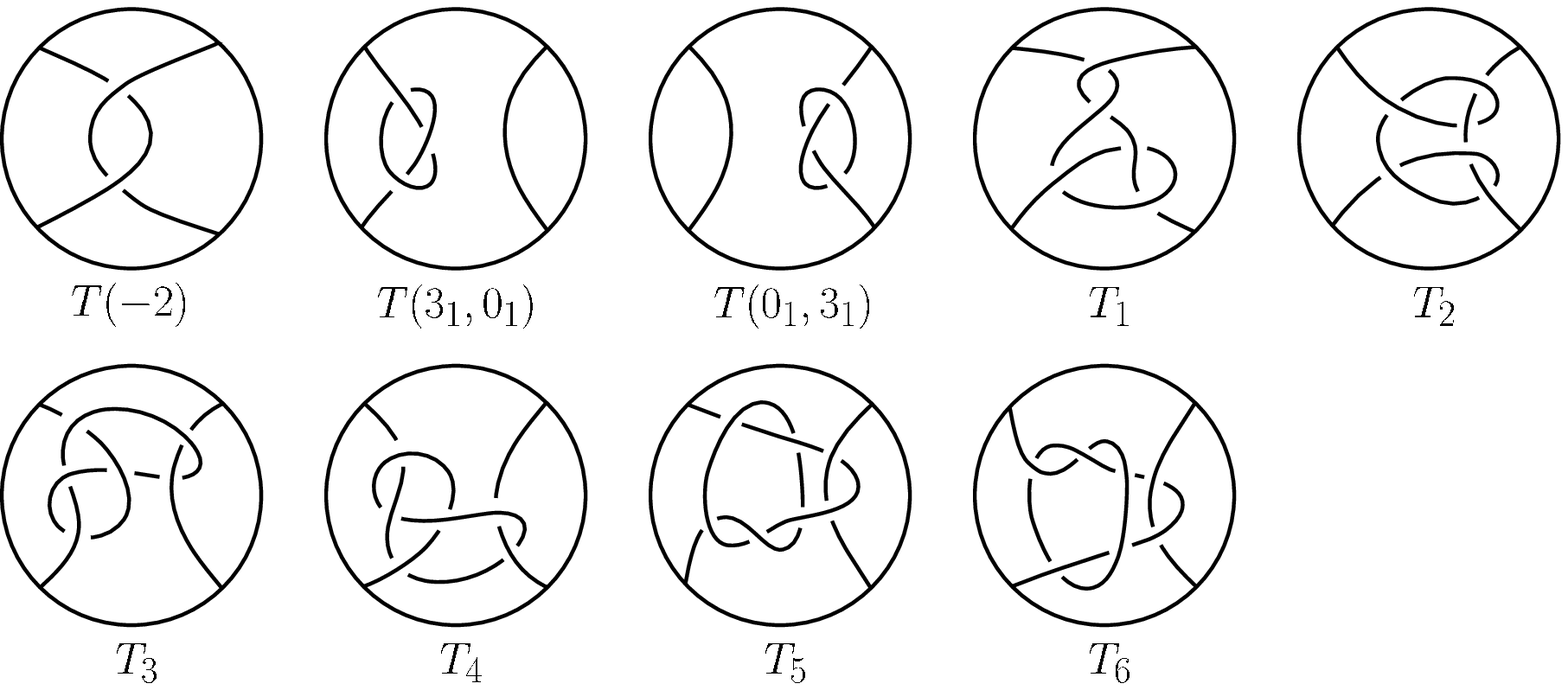}}
      \end{center}
   \caption{}
  \label{tangles2}
\end{figure} 

%

%
\begin{figure}[htbp]
      \begin{center}
\scalebox{0.65}{\includegraphics*{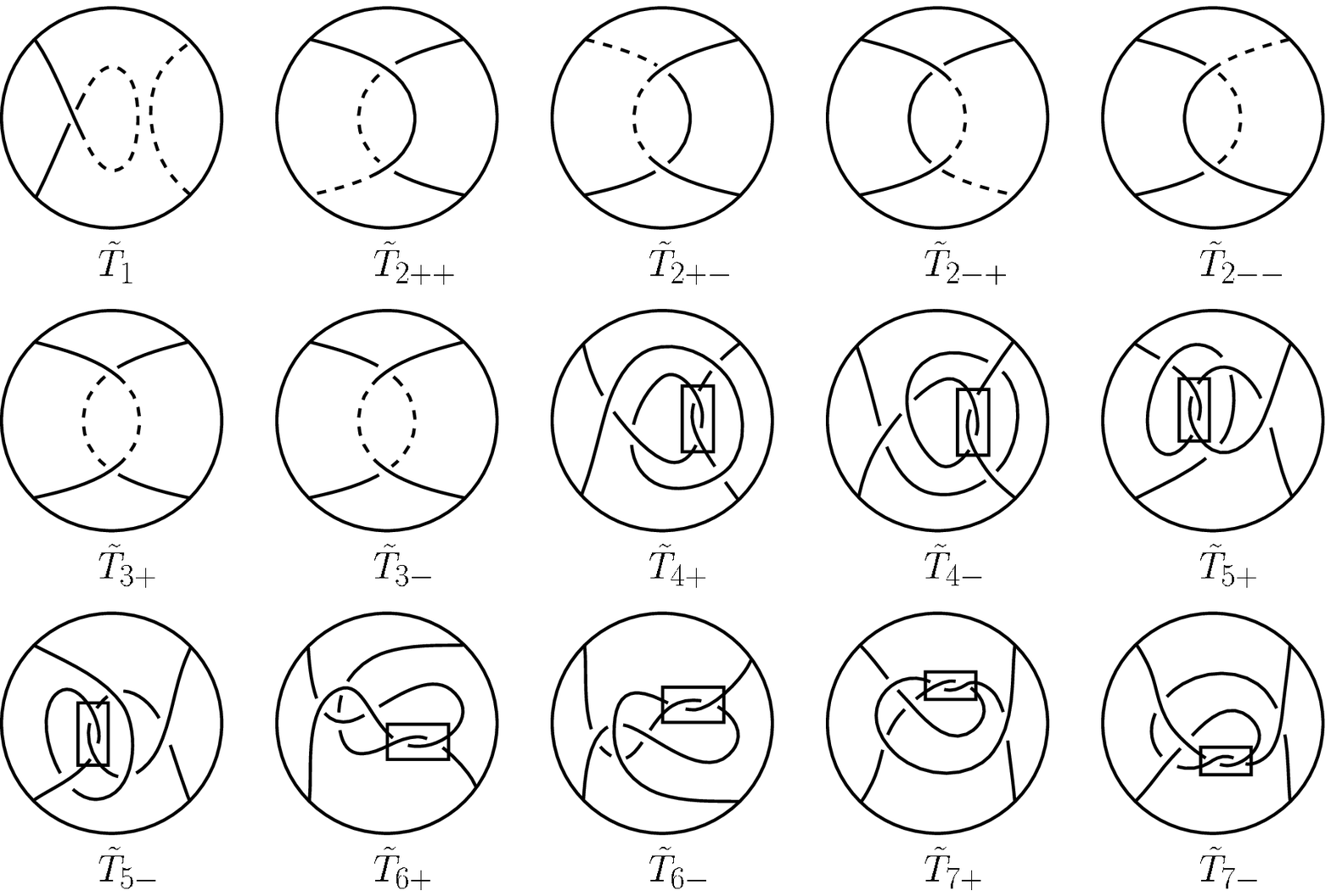}}
      \end{center}
   \caption{}
  \label{diagrams2}
\end{figure} 

%

This lemma is proved by a steady consideration of various cases. 

\begin{proof}
First we note that the tangles in Fig. \ref{tangles2} are closed under horizontal symmetry. In particular $H(T_1)=T_1$.
The diagrams in Fig. \ref{diagrams2} are also closed under horizontal symmetry.
We also note here that $\tilde T_{4\pm}$ and $\tilde T_{5\pm}$, $\tilde T_{6\pm}$ and $\tilde T_{7\pm}$ are flypes each other respectively.
We denote the negative mixed crossing by $N$ and the negative self-crossing by $M$.
As $\tilde T$ is not greater than or equal to $T(3_1,0_1)$, we have by Theorem \ref{almost-positive-theorem} that $\tilde a$ is 
obtained from an R1-augmentation of the diagram $\tilde T_{1}$ or $\tilde T_{2}$ in Fig. \ref{almost-positive-trivial} by deleting a trivial disk pair $(D^{2},D^{1})$. As $\tilde T$ is not greater than or equal to $T(0_1,3_1)$, we have by Lemma \ref{local-trefoil-lemma} that $\tilde b$ is almost trivial.

\begin{enumerate}
\item[Case 1.]
$\tilde a$ is obtained from the diagram $\tilde T_{2}$ in Fig. \ref{almost-positive-trivial}.

Suppose $\tilde a$ is obtained from an R1-augmentation $\tilde S_{2}$ of $\tilde T_{2}$ by deleting a trivial disk pair.

Suppose that a component of 
$(\tilde S_{2}-\tilde T_{2})\backslash {\rm int}D^{2}$ is incident to
$\partial D^{2}$ and has a mixed crossing of $\tilde T$.
Since $\tilde T$ is R2-reduced we can apply Lemma \ref{almost-positive-tangle-lemma} after erasing the negative 
self-crossing $M$ using Lemma \ref{diagram-reducing-lemma} and have the tangle $T(-2)$. See Fig. \ref{lemma2-proof1}.

\begin{figure}[htbp]
      \begin{center}
\scalebox{0.65}{\includegraphics*{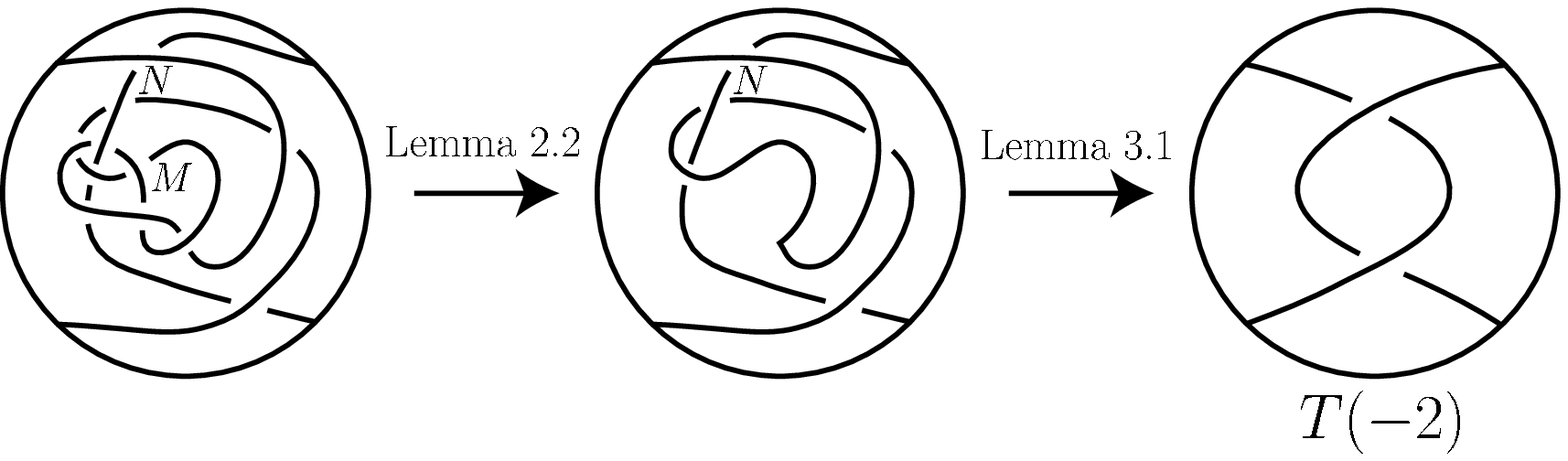}}
      \end{center}
   \caption{}
  \label{lemma2-proof1}
\end{figure} 

%

By similar arguments based on Lemma \ref{diagram-reducing-lemma} and Lemma \ref{almost-positive-tangle-lemma} we may suppose that $\tilde a$ is an R1-augmentation 
of one of the diagrams illustrated in Fig. \ref{lemma2-proof2} and their horizontal symmetries, and that both
$A_{0}P_{+}$ and $P_{+}A_{\infty}$ have no mixed crossings for $p>0$.

\begin{figure}[htbp]
      \begin{center}
\scalebox{0.65}{\includegraphics*{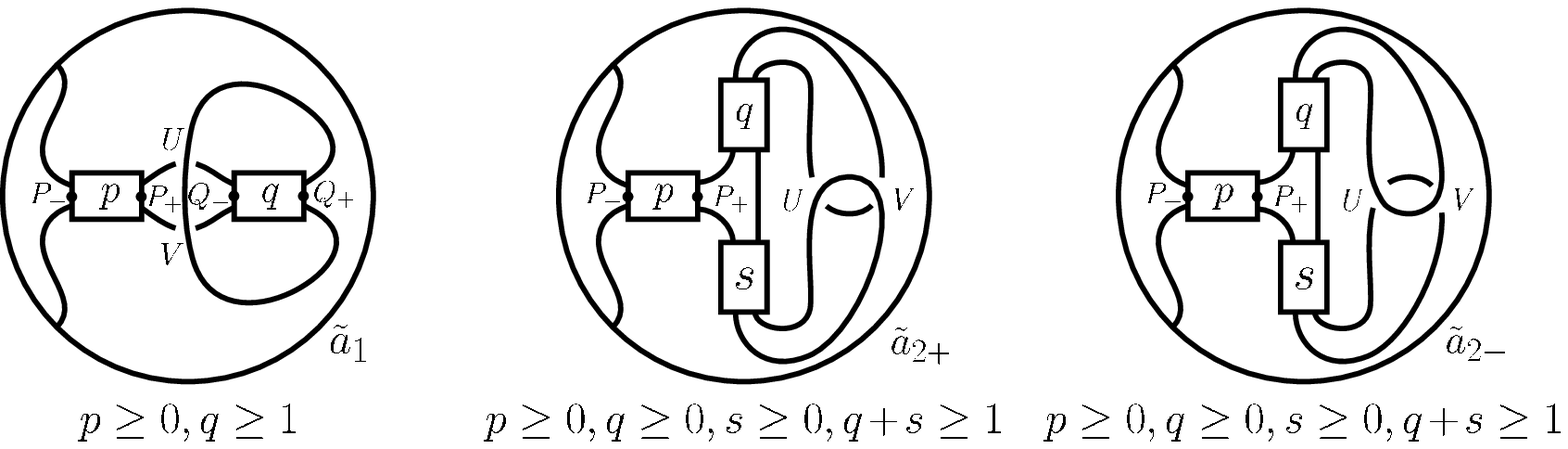}}
      \end{center}
   \caption{}
  \label{lemma2-proof2}
\end{figure} 

%

We remark that as $\tilde T$ is prime R2-reduced, both $s(\tilde a_{1},Q_{+})$ and the 2-gon $UVU$ of $\tilde a_{2\pm}$ have mixed crossings. 
We will show that in Case 1 $\tilde T\geq T(-2),T_1,T_3,T_4,T_5$ or $T_6$.

As $A_{0}P_{+}$ and $P_{+}A_{\infty}$ have no mixed crossings we may suppose, by taking a flype of $\tilde T$, still denoted by $\tilde T$ if necessary  that $p=0$.

\item[Case 1.1.]
$\tilde a$ is an R1-augmentation of the diagram $\tilde a_{1}$ of Fig. \ref{lemma2-proof2} with $p=0$ and $q$ odd. Then we have $U=M$. 
See Fig. \ref{lemma2-proof3}.

\begin{figure}[htbp]
      \begin{center}
\scalebox{0.65}{\includegraphics*{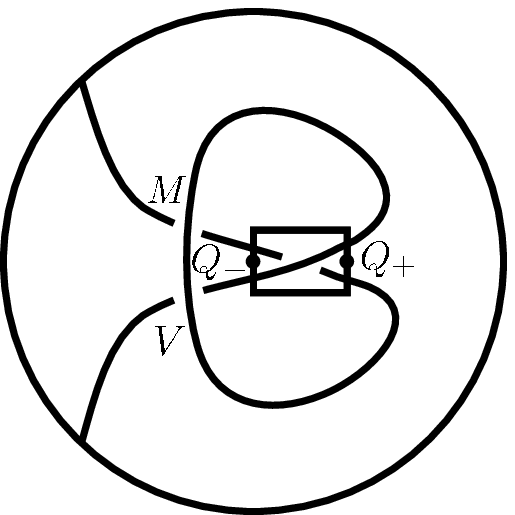}}
      \end{center}
   \caption{}
  \label{lemma2-proof3}
\end{figure} 

%

Suppose that $N$ is on $s(\tilde a  ,M)$. If $r(\tilde a  ,M)$ have mixed crossings then we 
have the tangle $T(-2)$ by Lemma \ref{diagram-reducing-lemma} and Lemma \ref{hook-tangle}. If not then we have $q\geq 3$ by the R2-reducibility.
Note that by Lemma \ref{diagram-reducing-lemma} we have that $r(\tilde T,V)$ is a minor of $\tilde T$. Then by applying Lemma \ref{almost-positive-tangle-lemma} to $r(\tilde T,V)$ we have the tangle $T(-2)$, $T_3$ or $T_4$. See Fig. \ref{lemma2-proof4}.

\begin{figure}[htbp]
      \begin{center}
\scalebox{0.65}{\includegraphics*{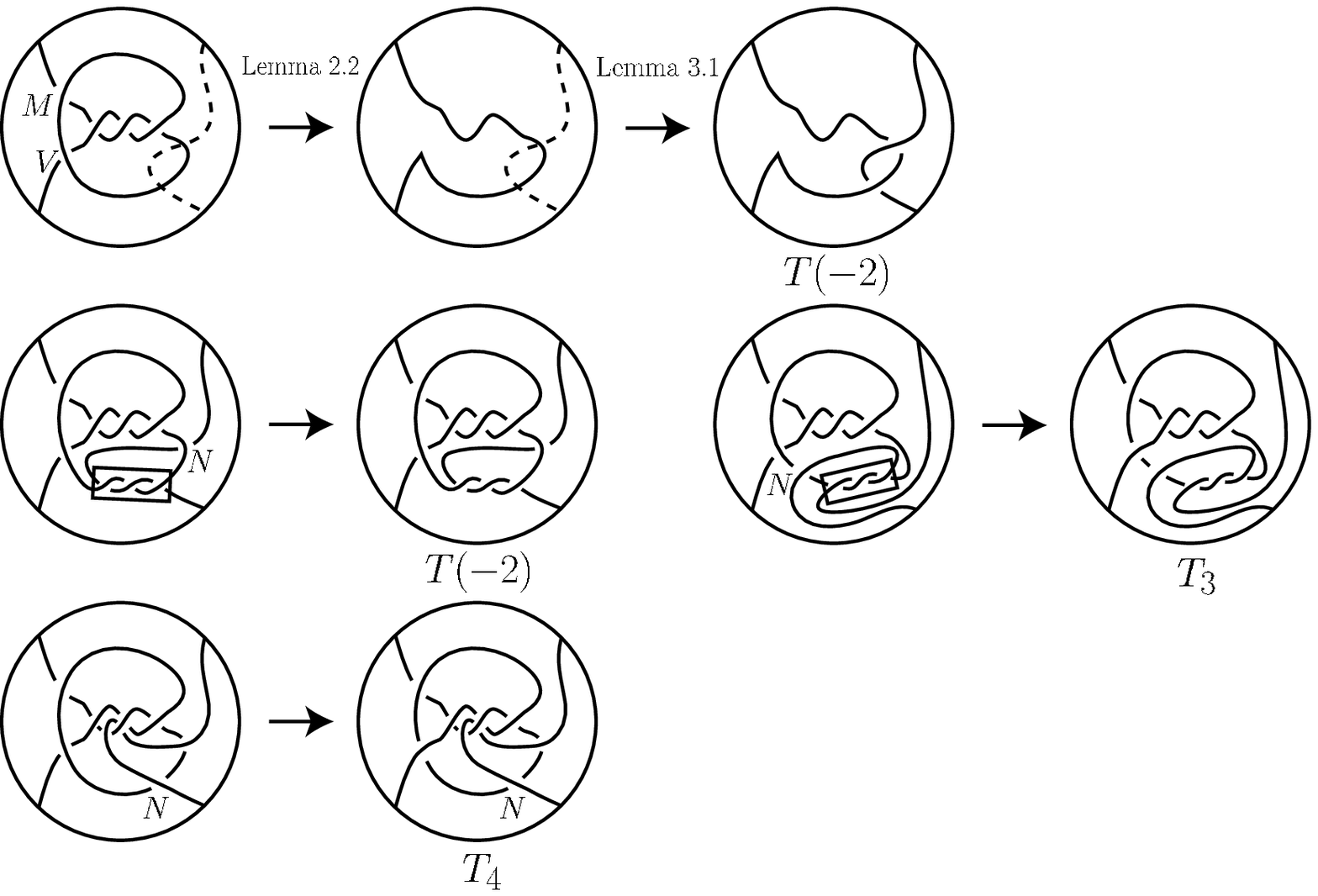}}
      \end{center}
   \caption{}
  \label{lemma2-proof4}
\end{figure} 

%

Therefore we may suppose that $N$ is on $r(\tilde a ,M)$. Since $r(\tilde a ,M)$ is a minor of $\tilde T$ by Lemma \ref{diagram-reducing-lemma} we may assume by Lemma \ref{almost-positive-tangle-lemma} that $r(\tilde T,M)$ is an R1 augmentation of one of the diagrams $\tilde T_{1+},\tilde T_{1-},\tilde T_{2+}$ and $\tilde T_{2-}$ in Fig. \ref{almost-positive-tangles}. We will show that $\tilde T$ is greater than or equal to $T(-2)$ or $T_1$ in all cases as follows.

\item[Case 1.1.1.]
The first crossing of $\tilde b$ with $r(\tilde a,M)$ is $N$ and is on $A_{0}M$.

We have $\tilde T\geq T(-2)$ as illustrated in Fig. \ref{lemma2-proof5}. Here Fig. \ref{lemma2-proof5} (a) describes 
the case that $N$ is not the first mixed crossing of $\tilde b$. The case $N$ is the first mixed crossing 
is illustrated in other figures. 
Because $\tilde T$ is prime and R2-reduced, we have one of 
the situations in Fig. \ref{lemma2-proof5}. Note that there are at most two dotted lines in 
each figure so that we can trivialize them by adding over/under crossing information as one is over 
everything and the other is under everything. 
We also note that the crossings may be on the R1 residual of $\tilde a$ with respect to $\tilde a_{1}$ of Fig. \ref{lemma2-proof2}. That case is not illustrated in Fig. \ref{lemma2-proof5} for the simplicity. However it reduces to the case that the crossings are on the core. 
See Fig. \ref{lemma2-proof6} how to reduce it. In Fig. \ref{lemma2-proof6} (a) the dotted part is supposed to be given over/under crossing information so that it is under everything. In Fig. \ref{lemma2-proof6} (b), (c), (d), (e) the dotted line bounded by two marked crossings are contained in the tear drop disk. Then we have the same result as (a) keeping the conditions that another dotted line is given over/under crossing information so that it is under everything.

\begin{figure}[htbp]
      \begin{center}
\scalebox{0.65}{\includegraphics*{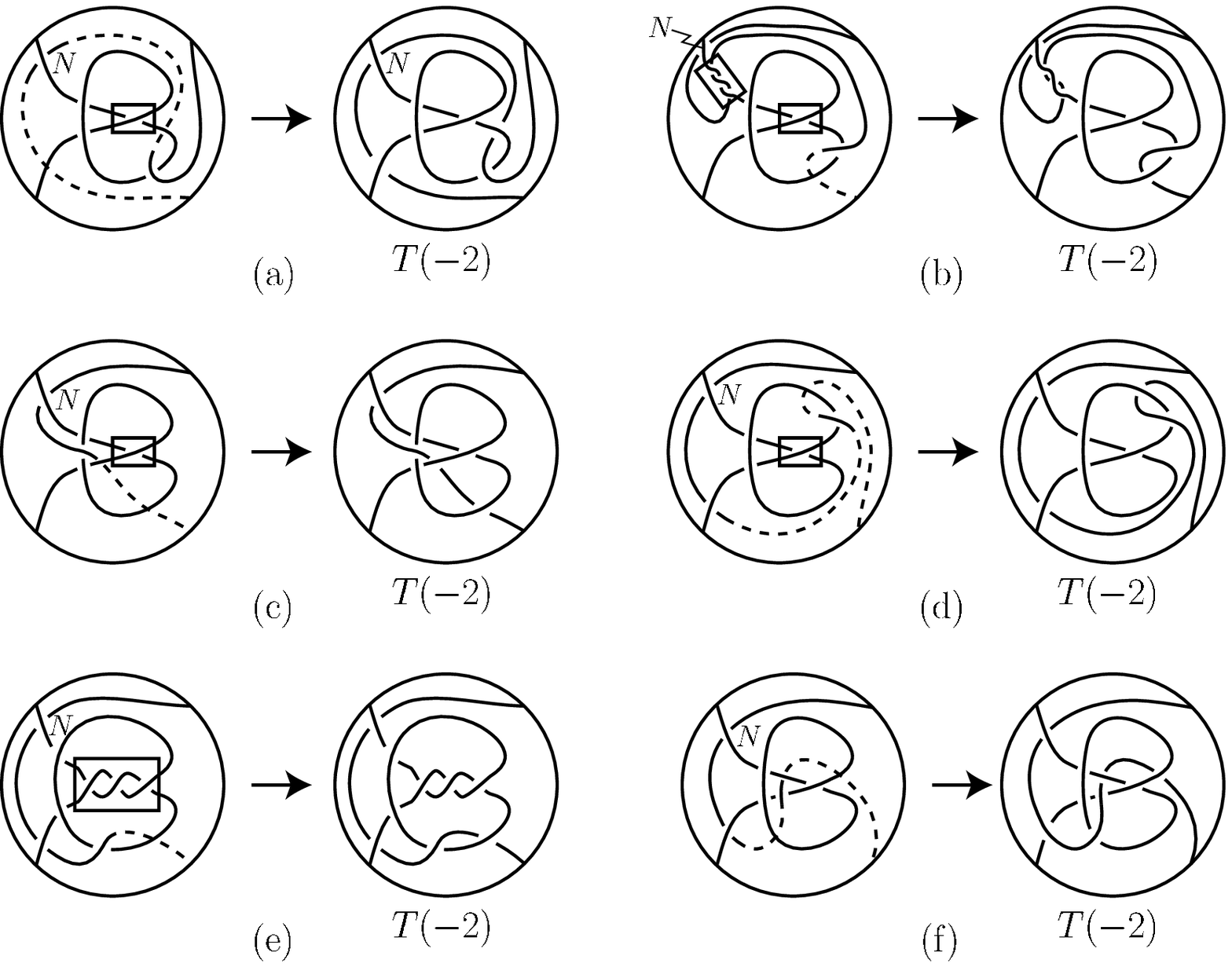}}
      \end{center}
   \caption{}
  \label{lemma2-proof5}
\end{figure} 

%

%
\begin{figure}[htbp]
      \begin{center}
\scalebox{0.65}{\includegraphics*{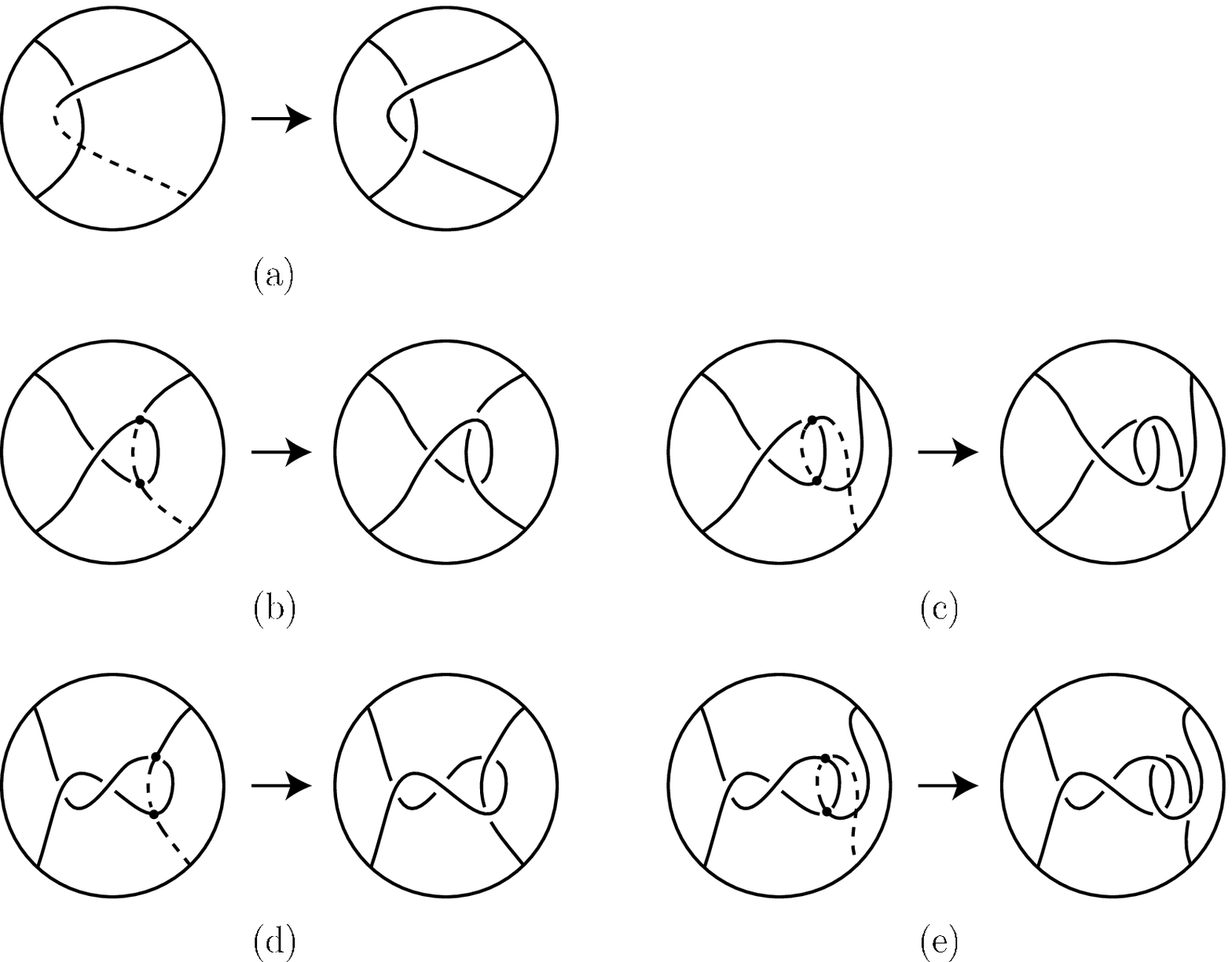}}
      \end{center}
   \caption{}
  \label{lemma2-proof6}
\end{figure} 

%

\item[Case 1.1.2.]
The first crossing of $\tilde b$ with $r(\tilde a , M)$ is $N$ and on $MQ_{+}$.

By applying Lemma \ref{almost-positive-tangle-lemma} to $r(\tilde T,M)$ we may suppose that $\tilde a$ is over $\tilde b$ at $N$. 
Then we may suppose that $r(\tilde a ,Q_{+})\cap\tilde b =\emptyset$, and $r(\tilde a ,V)\cap\tilde b=\emptyset$
otherwise we have $\tilde T\geq T(-2)$ by Lemma \ref{diagram-reducing-lemma} and Lemma \ref{hook-tangle}. Then we have 
$\tilde T\geq T(-2)$ as illustrated in Fig. \ref{lemma2-proof7}

\begin{figure}[htbp]
      \begin{center}
\scalebox{0.65}{\includegraphics*{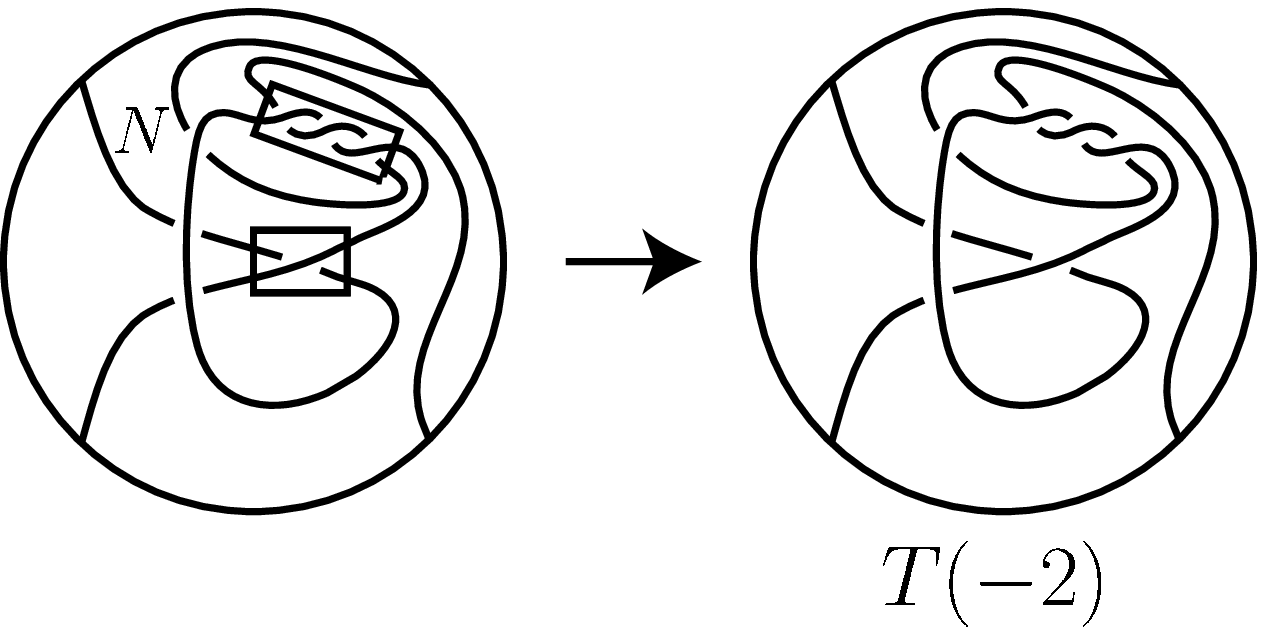}}
      \end{center}
   \caption{}
  \label{lemma2-proof7}
\end{figure} 

%

\item[Case 1.1.3.]
The first crossing of $\tilde b$ with $r(\tilde a,M)$ is $N$ and on $Q_{+}Q_{-}A_{\infty}$. 

If the first mixed crossing of $\tilde b$ is not $N$ then we have $T(-2)$ as illustrated in Fig. \ref{lemma2-proof8} (a). If $N$ is the first mixed crossing of $\tilde b$ then we have the situation illustrated in Fig. \ref{lemma2-proof8} (b) and then also have $T(-2)$.

\begin{figure}[htbp]
      \begin{center}
\scalebox{0.65}{\includegraphics*{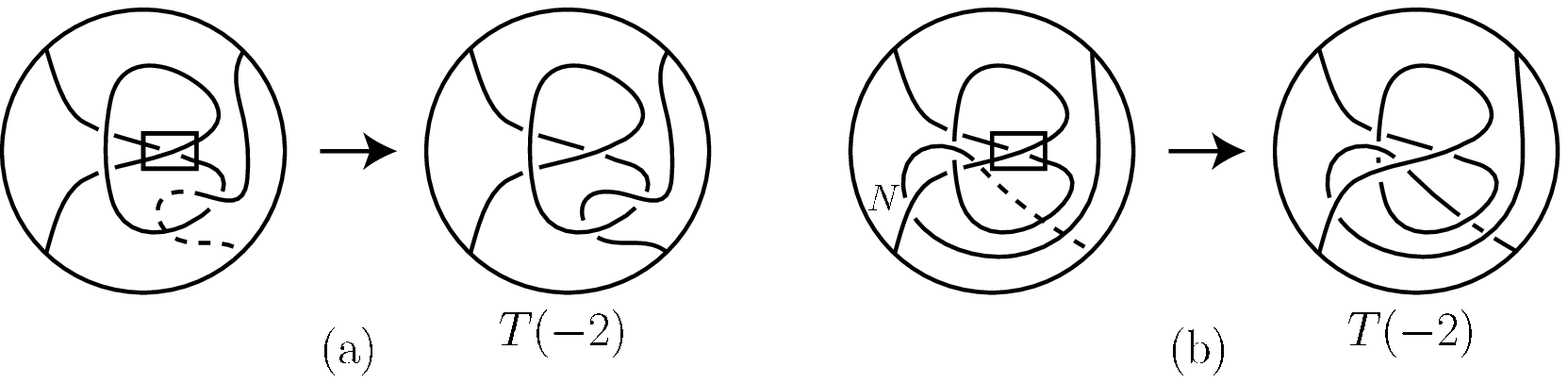}}
      \end{center}
   \caption{}
  \label{lemma2-proof8}
\end{figure} 

%

\item[Case 1.1.4.]
The last crossing of $\tilde b$ with $r(\tilde a , M)$ is $N$ and on $VA_{\infty}$. 
See Fig. \ref{lemma2-proof9}.

\begin{figure}[htbp]
      \begin{center}
\scalebox{0.65}{\includegraphics*{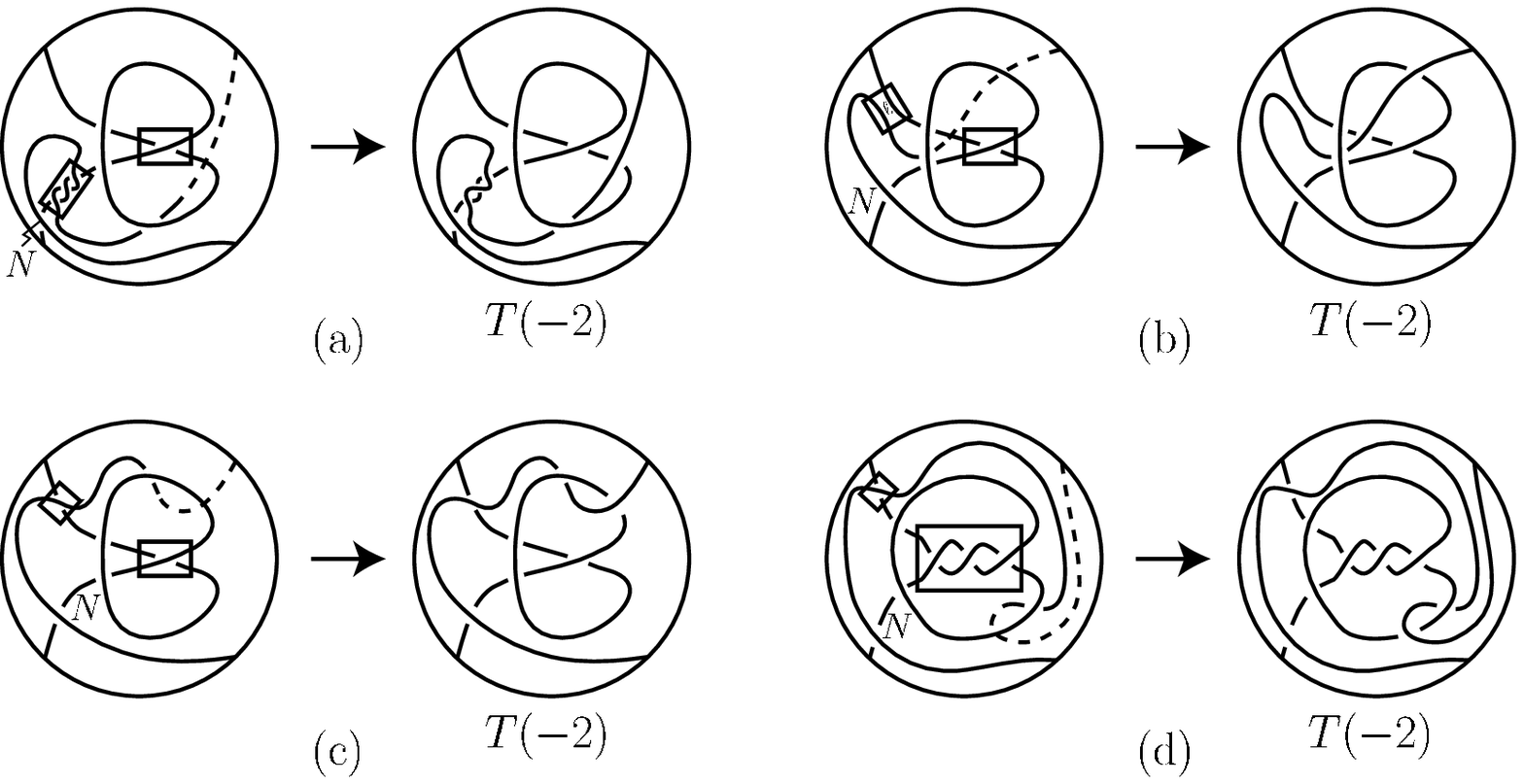}}
      \end{center}
   \caption{}
  \label{lemma2-proof9}
\end{figure} 

%

\item[Case 1.1.5.]
The last crossing of $\tilde b$ with $r(\tilde a , M)$ is $N$ and on $A_{0}M\cup MQ_{+}Q_{-}V$.
See Fig. \ref{lemma2-proof10}.

\begin{figure}[htbp]
      \begin{center}
\scalebox{0.65}{\includegraphics*{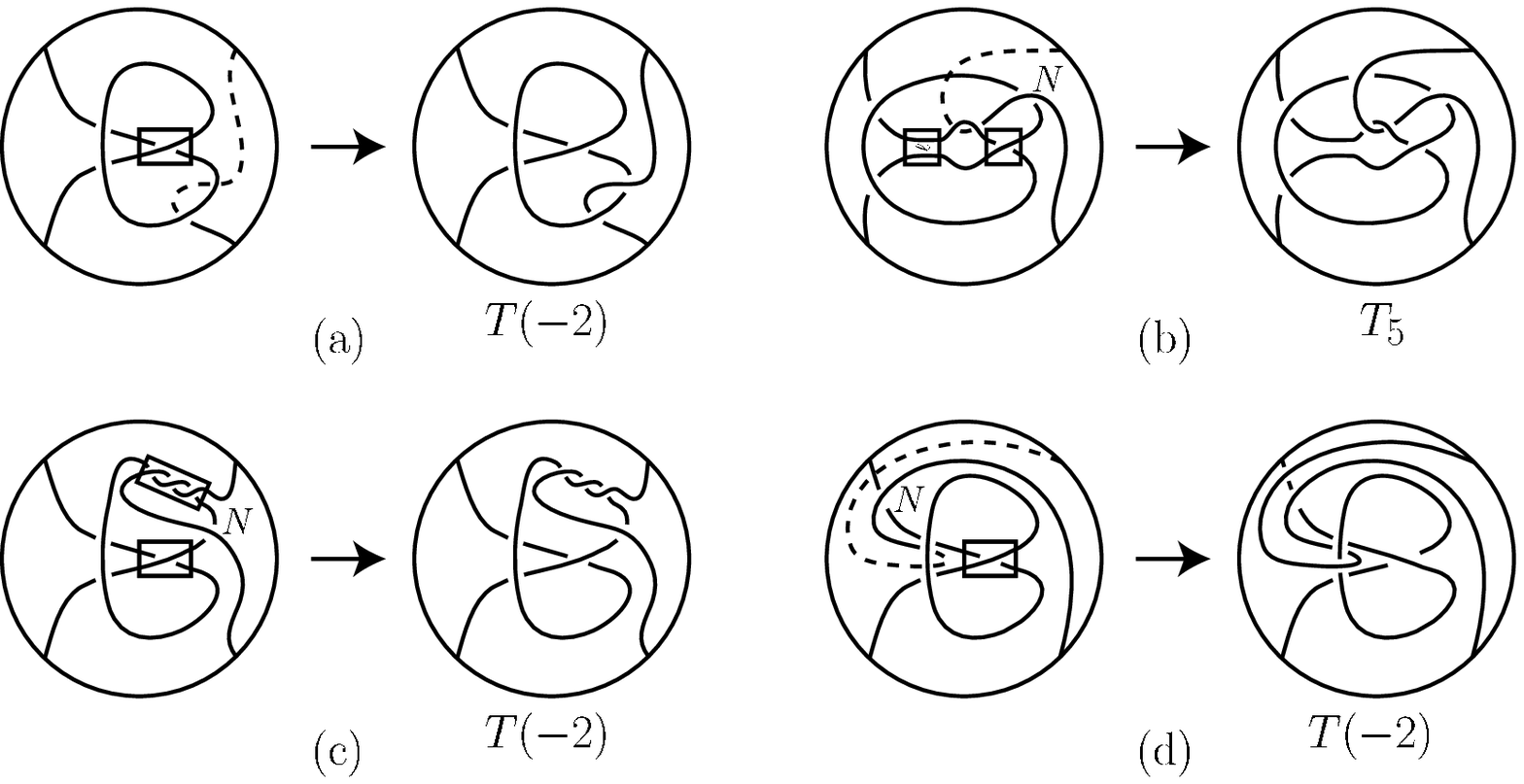}}
      \end{center}
   \caption{}
  \label{lemma2-proof10}
\end{figure} 

%

\item[Case 1.2.]
$\tilde a$ is an R1-augmentation of the diagram $\tilde a_{1}$ of Fig. \ref{lemma2-proof2} with $p=0$ and
$q$ even, $q\geq 2$. Then we have $V=M$. See Fig. \ref{lemma2-proof11}.

\begin{figure}[htbp]
      \begin{center}
\scalebox{0.65}{\includegraphics*{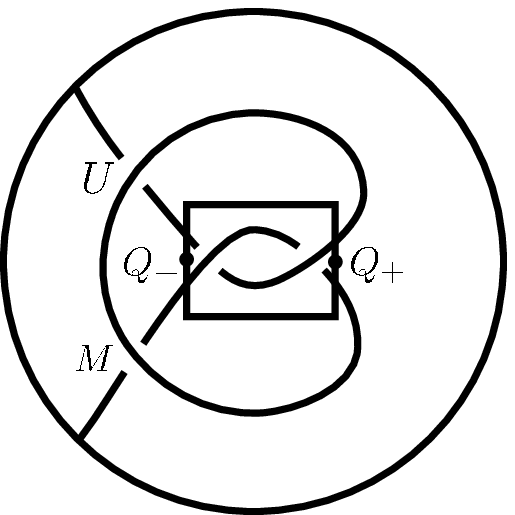}}
      \end{center}
   \caption{}
  \label{lemma2-proof11}
\end{figure} 

%

Suppose $s(\tilde a ,U)\cup s(\tilde a ,M)$ has no mixed crossings. Then using Lemma \ref{almost-positive-tangle-lemma} to
$r(r(\tilde T ,U),M)$ we have the tangle $T(-2)$ as illustrated in Fig. \ref{lemma2-proof12}.

\begin{figure}[htbp]
      \begin{center}
\scalebox{0.65}{\includegraphics*{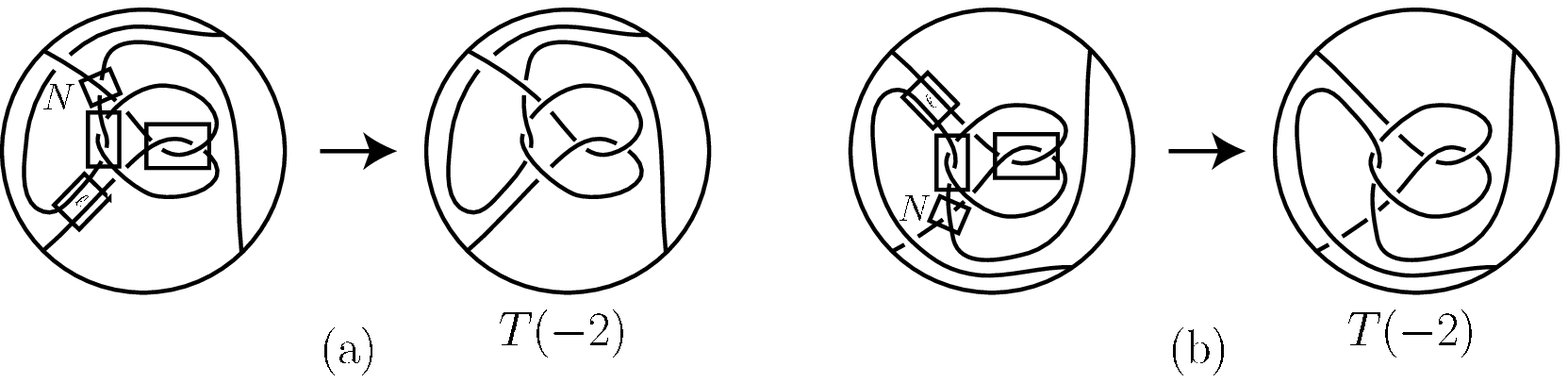}}
      \end{center}
   \caption{}
  \label{lemma2-proof12}
\end{figure} 

%

Next suppose $s(\tilde a ,U)$ has two positive crossings. Then we have $\tilde T\geq
r(\tilde T ,M)\geq T(-2)$ by Lemma \ref{diagram-reducing-lemma} and Lemma \ref{almost-positive-tangle-lemma}. 
Thus we may suppose that $s(\tilde a ,U)$ does not have two positive crossings.
If $N$ is on $s(\tilde a ,U)\cup s(\tilde a ,M)$ and  $r(r(\tilde a ,U),M)$ has mixed crossings then we also have $T(-2)$. 
If  $r(r(\tilde a ,U),M)$ has  no mixed crossings then because $\tilde T$ is R2-reduced,
 we have $s(\tilde a ,U)\cup s(\tilde a ,M)$ has four mixed crossings. 
Then $s(\tilde a ,M)$ has two positive mixed crossings.
If $N$ is on  $r(r(\tilde a ,U),M)$ then by the primeness of $\tilde T$ we have that $s(\tilde a ,M)$ 
has two positive mixed crossings.
Therefore we may suppose that $s(\tilde a ,M)$ has two positive mixed crossings. Then we may
suppose that they are successive on $\tilde b$. Namely the arc of $\tilde b$ bounded by them are contained in a region bounded by $s(\tilde a,M)$. We may also assume by taking a minor of $\tilde T$, still denoted by $\tilde T$ using Lemma \ref{diagram-reducing-lemma} that these two positive mixed crossings considered on $r(\tilde a,U)$ are rightmost.
Note that we may also suppose that either $s(\tilde a ,U)$ has no mixed crossings or it has just two mixed crossings, one is positive the other is $N$.
Then we have $T(-2)$ or $T_1$ by \lq\lq over and
under'' or \lq\lq under and over" technique used in the proof of Lemma \ref{almost-positive-tangle-lemma} in the cases illustrated in Fig. \ref{lemma2-proof13} and Fig. \ref{lemma2-proof13-2}.

\begin{figure}[htbp]
      \begin{center}
\scalebox{0.65}{\includegraphics*{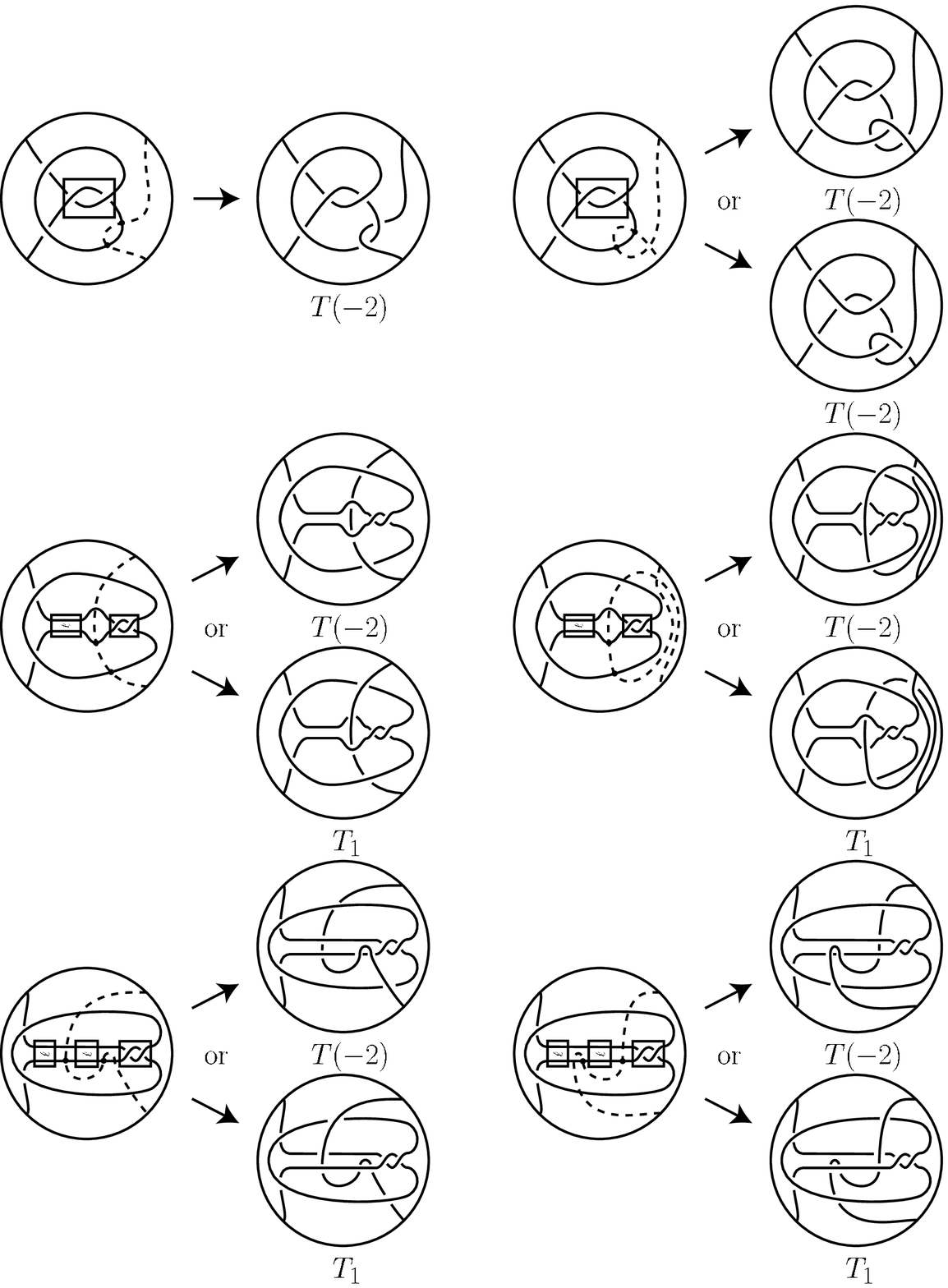}}
      \end{center}
   \caption{}
  \label{lemma2-proof13}
\end{figure} 

%

%
\begin{figure}[htbp]
      \begin{center}
\scalebox{0.65}{\includegraphics*{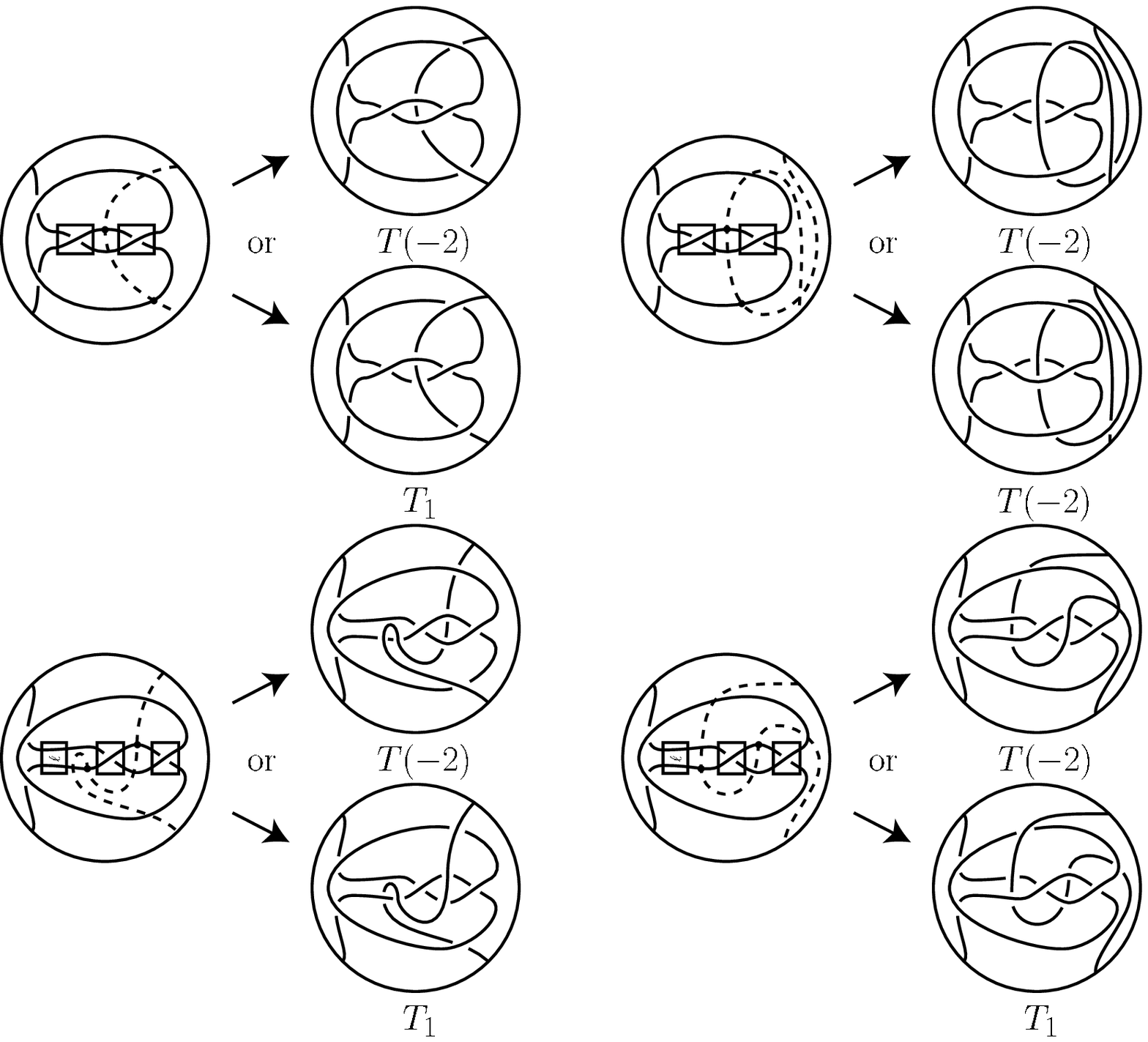}}
      \end{center}
   \caption{}
  \label{lemma2-proof13-2}
\end{figure} 

%

The rest is the case of Fig. \ref{lemma2-proof14}.

\begin{figure}[htbp]
      \begin{center}
\scalebox{0.65}{\includegraphics*{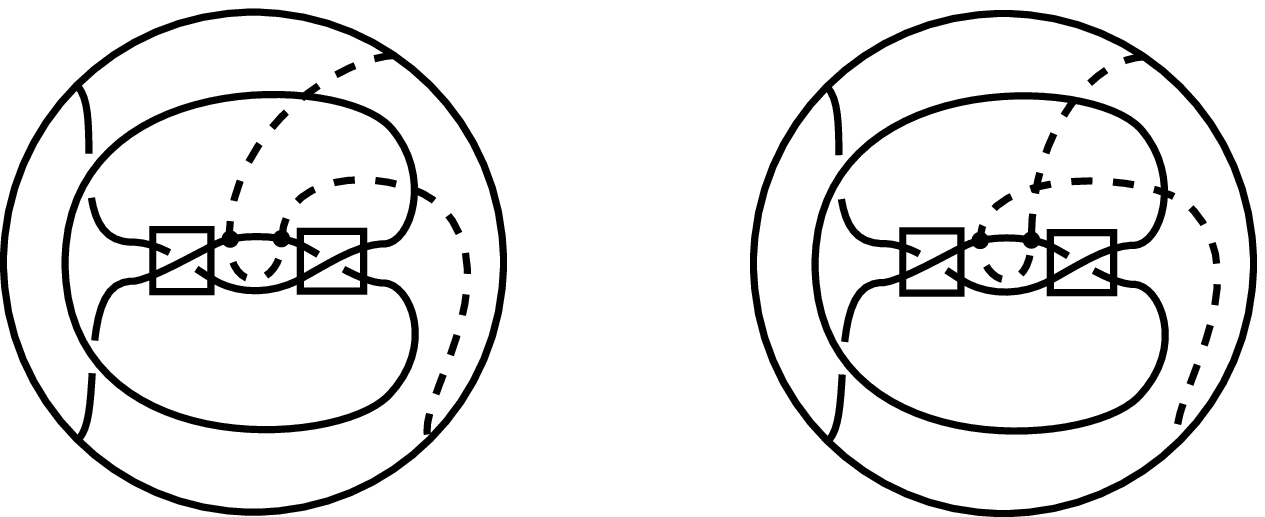}}
      \end{center}
   \caption{}
  \label{lemma2-proof14}
\end{figure} 

%

Note that in this case $s(\tilde a ,U)$ has two mixed crossings and therefore $N$ is on it.
Then either we can apply Lemma \ref{vertical-trefoil-tangle-lemma} for a sub-tangle diagram of $r(\tilde T ,U)$ to have the tangle $T(-2)$ as 
illustrated in Fig. \ref{lemma2-proof15} or we have the tangle $T(-2)$ or $T_3$ as illustrated in Fig. \ref{lemma2-proof16}.

\begin{figure}[htbp]
      \begin{center}
\scalebox{0.65}{\includegraphics*{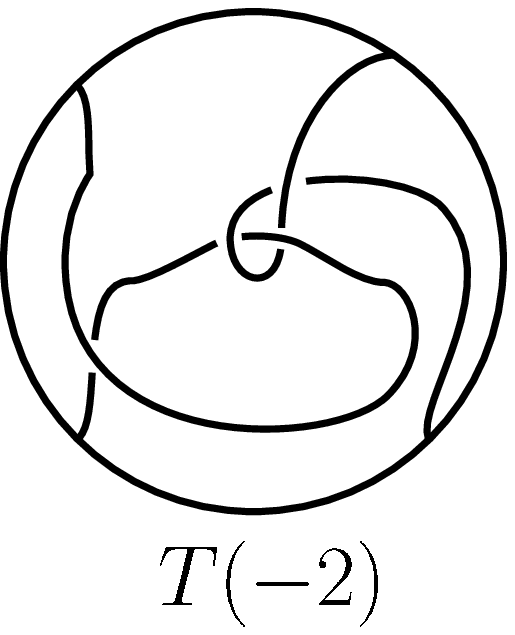}}
      \end{center}
   \caption{}
  \label{lemma2-proof15}
\end{figure} 

%

%
\begin{figure}[htbp]
      \begin{center}
\scalebox{0.65}{\includegraphics*{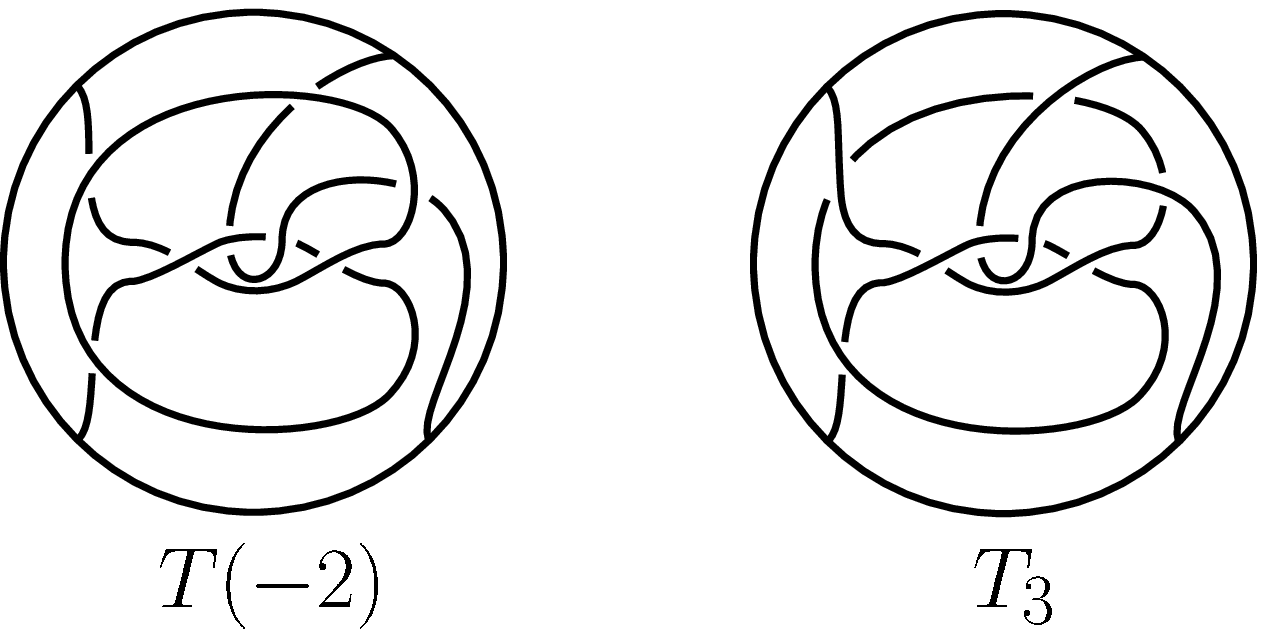}}
      \end{center}
   \caption{}
  \label{lemma2-proof16}
\end{figure} 

%

\item[Case 1.3.]
$\tilde a$ is an R1-augmentation of the diagram $\tilde a_{2\pm}$ in Fig. \ref{lemma2-proof2} with $p=0$
or its horizontal symmetry. See Fig. \ref{lemma2-proof17}.

\begin{figure}[htbp]
      \begin{center}
\scalebox{0.65}{\includegraphics*{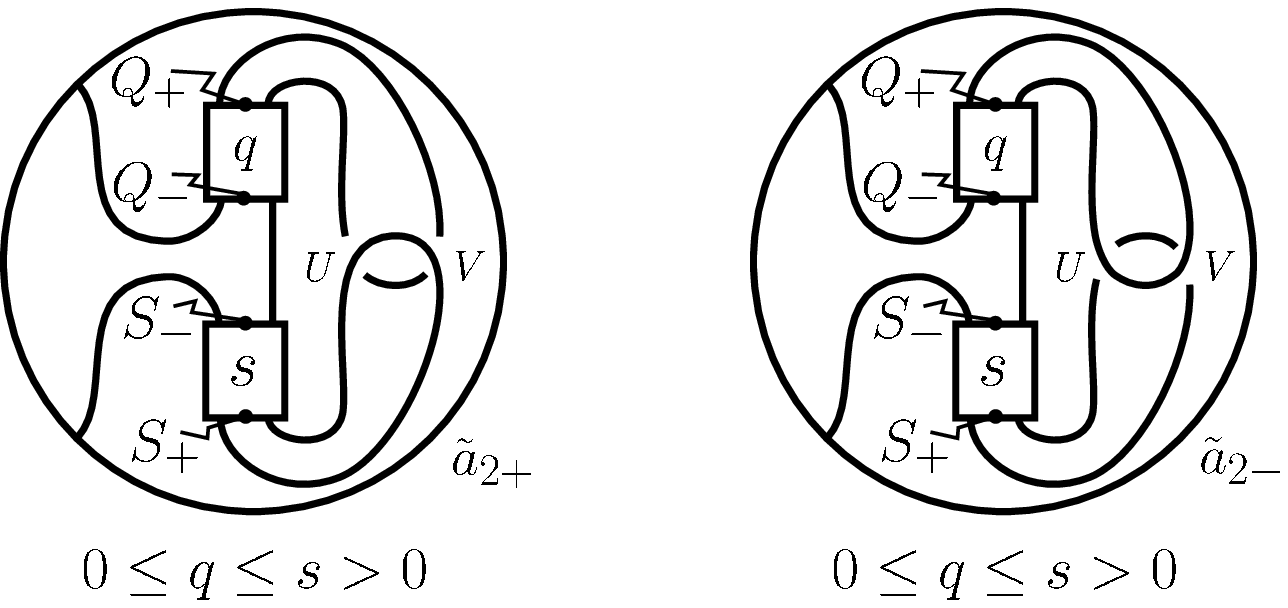}}
      \end{center}
   \caption{}
  \label{lemma2-proof17}
\end{figure} 

%

We may suppose without loss of generality that $q\geq r$ by horizontal symmetry. We will
show that we have the tangle $T(-2),T_1$ or $T_5$ in this case.
Suppose that $Q_-Q_+\cup Q_+Q_-\cup S_-S_+\cup S_+S_-$ has mixed crossings. Suppose for example that there are two successive mixed crossings $P_1$ and $P_2$. If one of them is $N$ then by Lemma \ref{diagram-reducing-lemma} we have $\tilde T\geq r(\tilde T,S_+)$ and taking the fact that the 2-gon $UVU$ has mixed crossings then because $\tilde T$ is R2-reduced,
 we have $r(\tilde T,S_+)\geq T(-2)$ by Lemma \ref{almost-positive-tangle-lemma}. Suppose that 
both $P_1$ and $P_2$ are positive crossings. Then by \lq\lq over and under technique'' we have $\tilde T\geq T(-2)$.
Thus we have that $Q_-Q_+\cup Q_+Q_-\cup S_-S_+\cup S_+S_-$ has no mixed crossings.
Therefore we have that any $\tilde T$ is greater than or equal to a tangle, still denoted by $\tilde T$ of the following:

$q=0$ and $s=2$, or $q=0$ and $s=1$, or $q=s=1$.

\item[Case 1.3.1]
$q=0$ and $s=2$.

For the diagram $\tilde a_{2-}$ of Fig. \ref{lemma2-proof17}, we have $U=M$ and have the tangle $T(-2)$
from $r(\tilde T ,M)$ by Lemma \ref{almost-positive-tangle-lemma}. Therefore we consider the diagram of Fig. \ref{lemma2-proof18}.

\begin{figure}[htbp]
      \begin{center}
\scalebox{0.65}{\includegraphics*{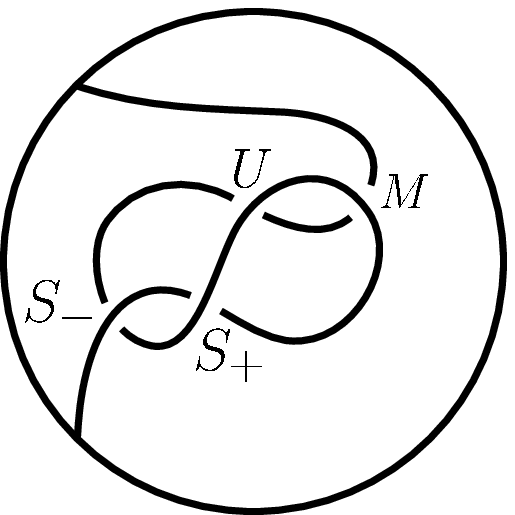}}
      \end{center}
   \caption{}
  \label{lemma2-proof18}
\end{figure} 

%

If $N$ is on $s(\tilde a ,M)$ then we have $T(-2)$. If $N$ is not on $s(\tilde a ,M)$ but on $s(\tilde a,S_+)$ and $\tilde a$ is 
over $\tilde b$ at $N$ then we have $T(-2)$ by applying Lemma \ref{hook-tangle} to $r(\tilde T,S_-)$.
Thus we have that $\tilde a$ is under $\tilde b$ at $N$.
Then applying Lemma \ref{almost-positive-tangle-lemma} to $r(\tilde T ,M)$ we have the result as 
illustrated in Fig. \ref{lemma2-proof19}. Note that Fig. \ref{lemma2-proof19} (a), (b) and (c) do not occur since $\tilde T$ is R2-reduced. Note that Fig. \ref{lemma2-proof19} (b) and (c)  are the cases that $r(\tilde T,M)$ has just two mixed crossings.

\begin{figure}[htbp]
      \begin{center}
\scalebox{0.65}{\includegraphics*{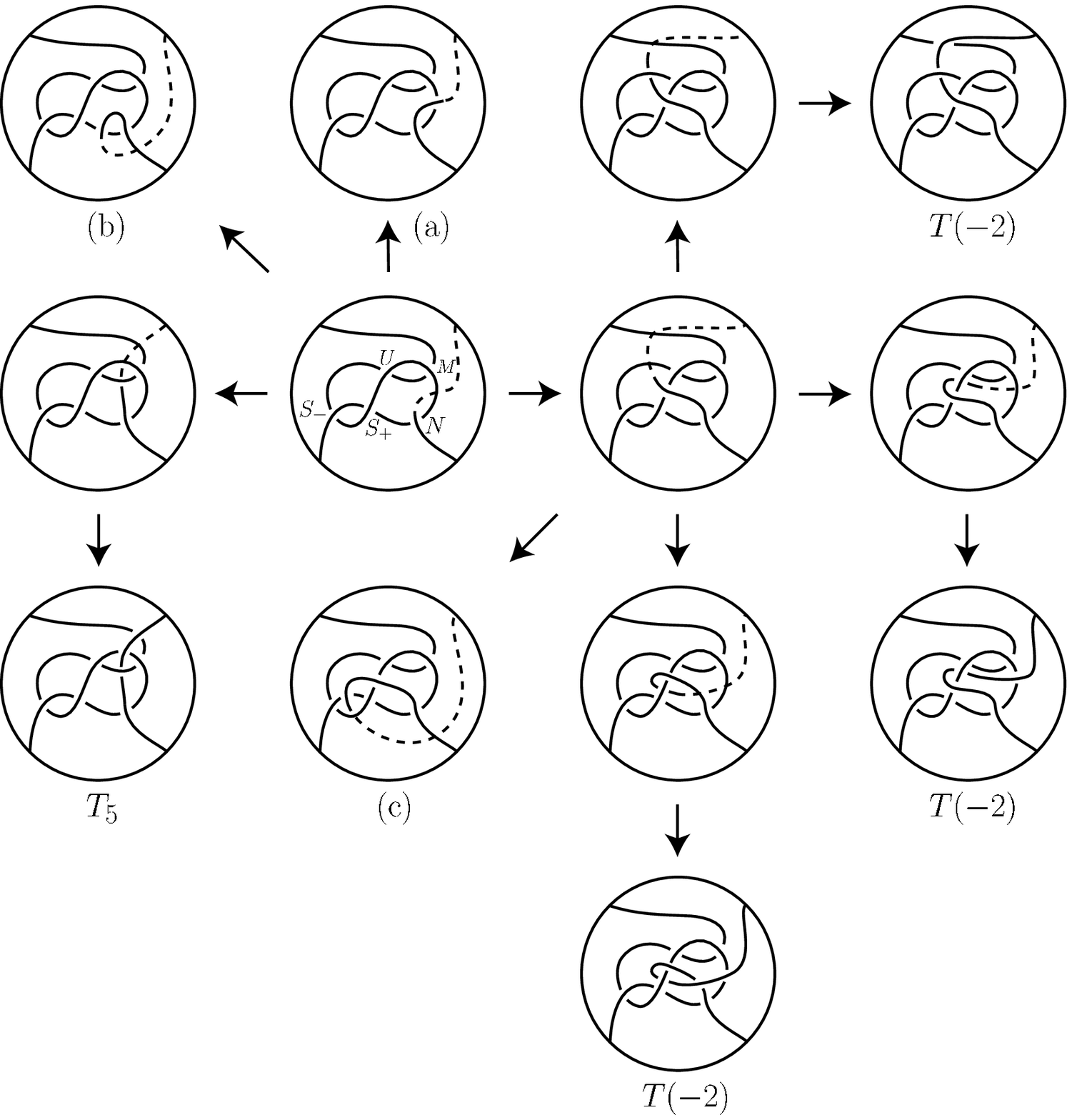}}
      \end{center}
   \caption{}
  \label{lemma2-proof19}
\end{figure} 

%

Suppose $N$ is on otherwise. Then there are at least two positive mixed crossings on $s(\tilde a,S_+)$. Then we have the result by \lq\lq over and under technique'' as illustrated in Fig. \ref{lemma2-proof20} or Fig. \ref{lemma2-proof21}. Note that in the case illustrated in Fig. \ref{lemma2-proof20} (a) we have that $\tilde T\geq r(\tilde T,S_-)\geq T(-2)$ and the case illustrated in Fig. \ref{lemma2-proof20} (b) reduces to other cases. In Fig. \ref{lemma2-proof21} we describe the situation that $s(\tilde a,S_+)\setminus s(\tilde a,M)$ has no mixed crossings. Then, because $\tilde T$ is R2-reduced,
 we have that the middle dotted line of $\tilde b$ illustrated in Fig. \ref{lemma2-proof21} (a) and (b) must 
intersects $MU$.  Then applying Lemma \ref{almost-positive-tangle-lemma} 
to $r(\tilde T ,S_-)$ we have the situation illustrated in Fig. \ref{lemma2-proof21} (c) or (d) and 
we have that $\tilde T$ is greater than or equal to $T(-2)$.

\begin{figure}[htbp]
      \begin{center}
\scalebox{0.65}{\includegraphics*{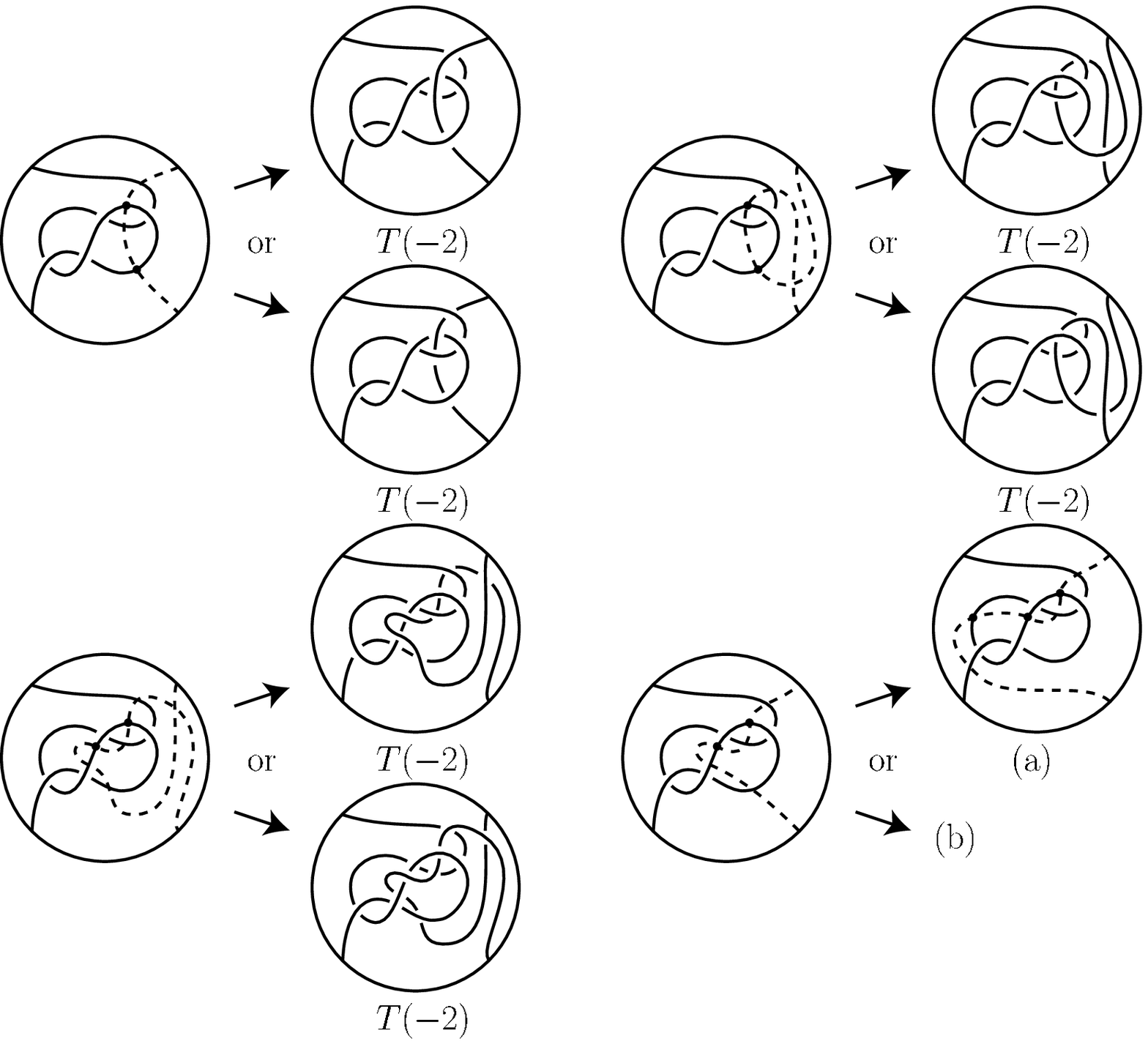}}
      \end{center}
   \caption{}
  \label{lemma2-proof20}
\end{figure} 

%

%
\begin{figure}[htbp]
      \begin{center}
\scalebox{0.65}{\includegraphics*{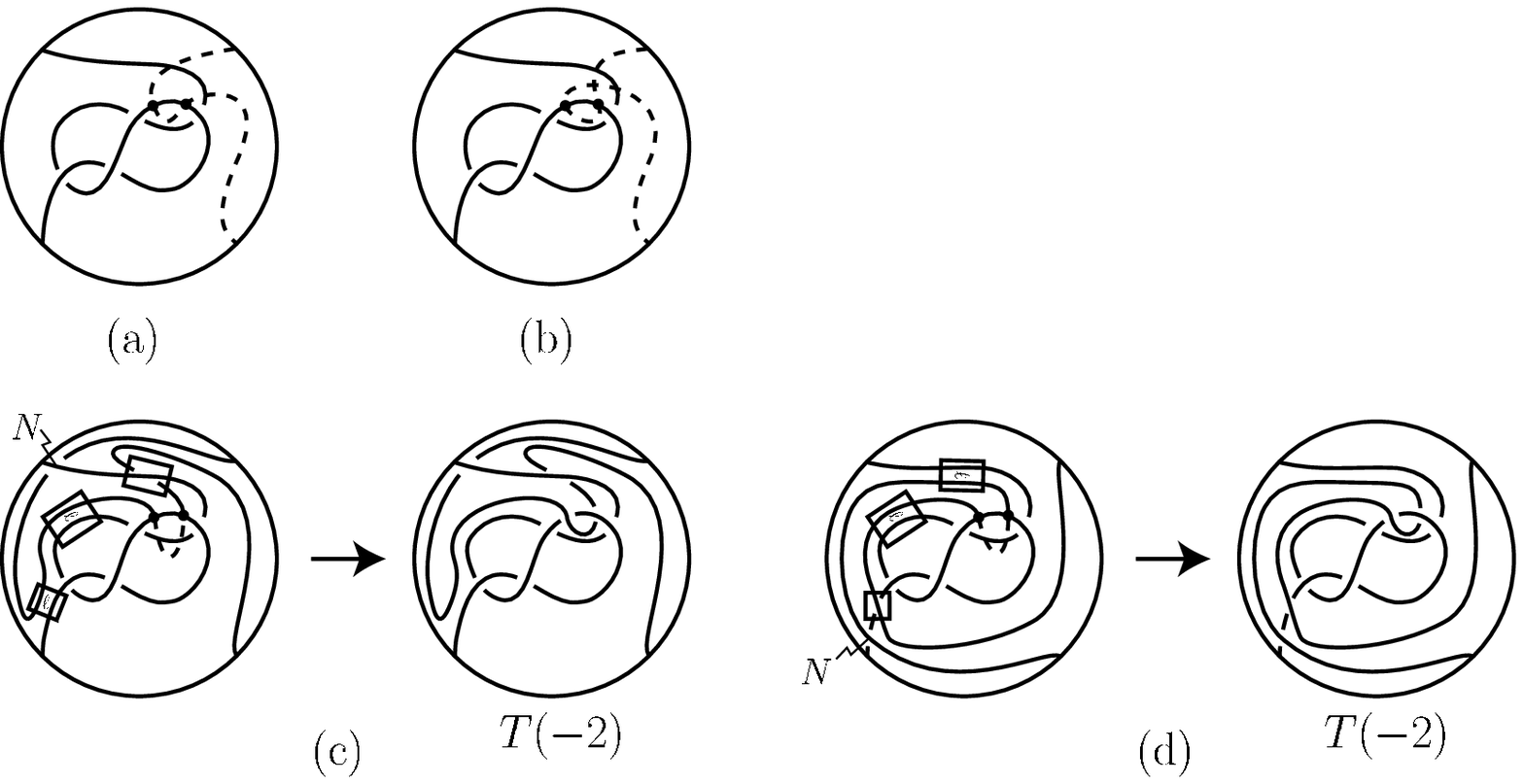}}
      \end{center}
   \caption{}
  \label{lemma2-proof21}
\end{figure} 

%

\item[Case 1.3.2.]
$q=0$ and $s=1$.

First we consider the case that $\tilde a=\tilde a_{2+}$ in Fig. \ref{lemma2-proof17}.
Then we have $M=U$. See Fig. \ref{lemma2-proof22}.

\begin{figure}[htbp]
      \begin{center}
\scalebox{0.65}{\includegraphics*{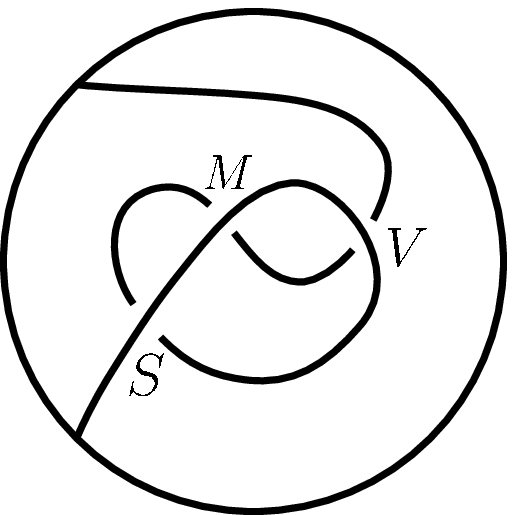}}
      \end{center}
   \caption{}
  \label{lemma2-proof22}
\end{figure} 

%

Suppose that $N$ is on $s(\tilde a,M)$. Since $\tilde T$ is R2-reduced we have that $r(\tilde a,M)$ has mixed crossings. Then by Lemma \ref{diagram-reducing-lemma} and Lemma \ref{hook-tangle} we have $T(-2)$. Thus we have that $N$ is not on $s(\tilde a,M)$. 
Suppose that $N$ is on $MS$ in $s(\tilde a,S)$. If $\tilde a$ is over $\tilde b$ at $N$ then we have $\tilde T\geq r(\tilde T,S)$ by Lemma \ref{diagram-reducing-lemma} and we have $r(\tilde T,S)\geq T(-2)$ by Lemma \ref{hook-tangle}. Thus we have that $\tilde a$ is under $\tilde b$ at $N$.
Then by applying Lemma \ref{almost-positive-tangle-lemma} to $r(\tilde T,M)$ we have that $N$ is the last mixed crossing of $\tilde b$ with $r(\tilde T,M)$. Therefore we have the situation illustrated in Fig. \ref{lemma2-proof23} and have the tangle $T(-2)$.
\begin{figure}[htbp]
      \begin{center}
\scalebox{0.65}{\includegraphics*{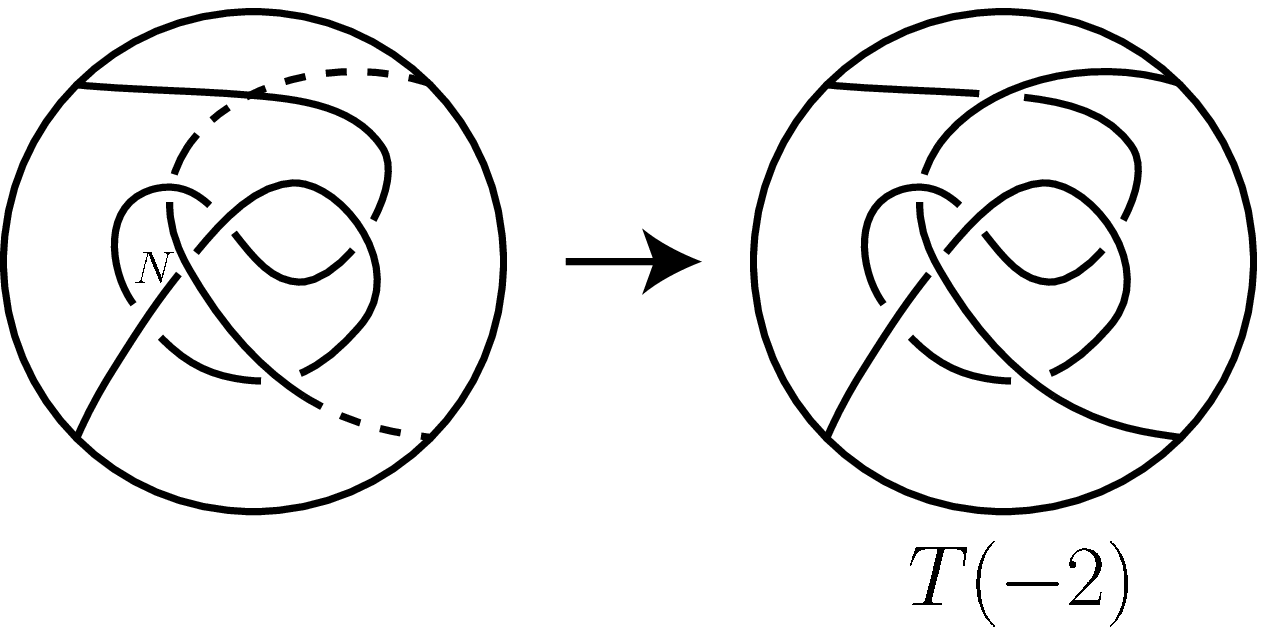}}
      \end{center}
   \caption{}
  \label{lemma2-proof23}
\end{figure} 

%

Thus we have that $N$ is not on $s(\tilde a,S)$. Then we have that there are two successive positive mixed crossings on $s(\tilde a,S)$. If they are not on $VM$ of $s(\tilde a,S)$ then we have $T(-2)$ by \lq\lq over and under technique''. We also have $T(-2)$ in the cases illustrated in Fig. \ref{lemma2-proof24} (a) and (b). For the case illustrated in Fig. \ref{lemma2-proof24} (c) we do not have the situation illustrated in Fig. \ref{lemma2-proof24} (d) because in this situation the second dotted line has a mixed crossing with $r(\tilde a,S)$ and then we have $r(\tilde a,S)\geq T(-2)$ by Lemma \ref{almost-positive-tangle-lemma}. Thus we have that the case illustrated in Fig. \ref{lemma2-proof24} (c) reduces to other cases. In the case illustrated in Fig. \ref{lemma2-proof24} (e) we have by taking the position of $N$ on $r(\tilde a,M)$ into account using Lemma \ref{almost-positive-tangle-lemma} we have that the first dotted line of $\tilde b$ can be under everything, the third dotted line can be over everything, and the second dotted line can be over everything except the third dotted line. Thus we have the tangle $T(-2)$ as illustrated.  The final case is the case that the two successive positive mixed crossings are on $VM$ of $s(\tilde a,S)$. 
Then because  $\tilde T$ is R2-reduced, we have that the arc of $\tilde b$ bounded by them 
intersects $VM$ of $r(\tilde a,M)$. Then by applying Lemma \ref{almost-positive-tangle-lemma} 
to $r(\tilde T,M)$ and $r(\tilde T,S)$, and because $\tilde T$ is R2-reduced, we have the situation 
illustrated in Fig. \ref{lemma2-proof24} (f) and then have the tangle $T(-2)$ as illustrated.

\begin{figure}[htbp]
      \begin{center}
\scalebox{0.65}{\includegraphics*{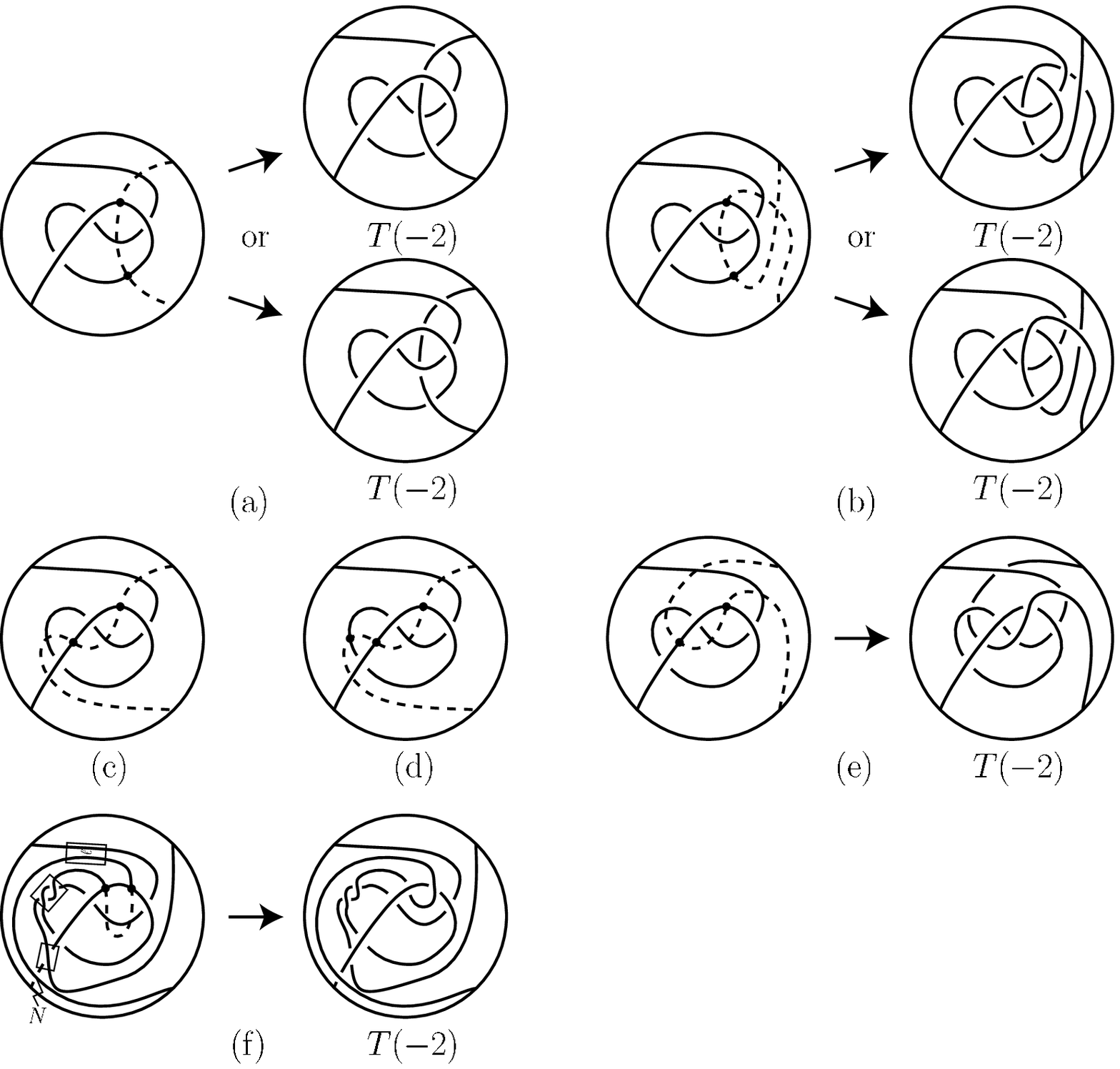}}
      \end{center}
   \caption{}
  \label{lemma2-proof24}
\end{figure} 

%

Next we consider the case that $\tilde a=\tilde a_{2-}$ in Fig. \ref{lemma2-proof17}.
Then we have $M=V$. See Fig. \ref{lemma2-proof25}.

\begin{figure}[htbp]
      \begin{center}
\scalebox{0.65}{\includegraphics*{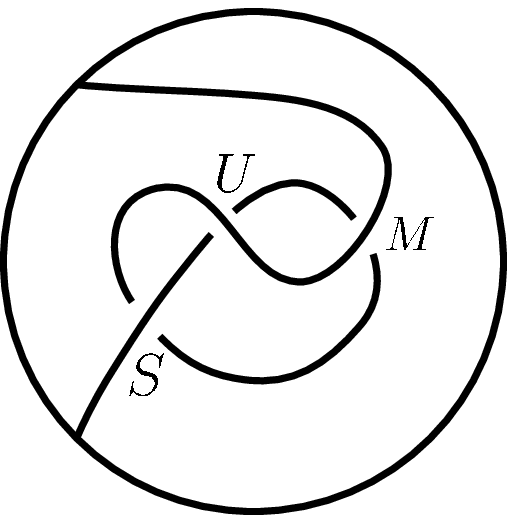}}
      \end{center}
   \caption{}
  \label{lemma2-proof25}
\end{figure} 

%

First suppose that $N$ is on $s(\tilde a,M)$. If $r(\tilde a,M)$ has mixed crossings then we have $T(-2)$. 
Therefore we have $r(\tilde a,M)$ has no mixed crossings. 
Suppose that $N$ is on $MU$ of $s(\tilde a,M)$. 
Then after changing some mixed crossings on $MU$ of $s(\tilde a,M)$ except $N$ and deforming $MU$ 
of $s(\tilde a,M)$ away from $s(\tilde a,S)$ we have the tangle $T(-2)$ by the use of Lemma \ref{hook-tangle}. 
Therefore we may suppose that $N$ is on $SM$. 
Then we may suppose that $\tilde a$ is over $\tilde b$ at $N$, otherwise we can apply 
Lemma \ref{diagram-reducing-lemma} to $r(\tilde T,S)$.
Let $P$ be the first mixed crossing of $\tilde b$ with $MU$ of $s(\tilde a,M)$. 
Then we have the tangle $T_1$ as illustrated in Fig. \ref{lemma2-proof26} (a), 
or we have the situation illustrated in Fig. \ref{lemma2-proof26} (b). 
Then we give over/under crossing information so that the third dotted line is under everything, 
the first dotted line is under everything except the third dotted line, and the second dotted 
line is over everything. Note that we can first deform the third dotted line such that it does not 
intersect the first dotted line as illustrated in Fig. \ref{lemma2-proof26} (c) and we have the tangle $T_1$.

\begin{figure}[htbp]
      \begin{center}
\scalebox{0.65}{\includegraphics*{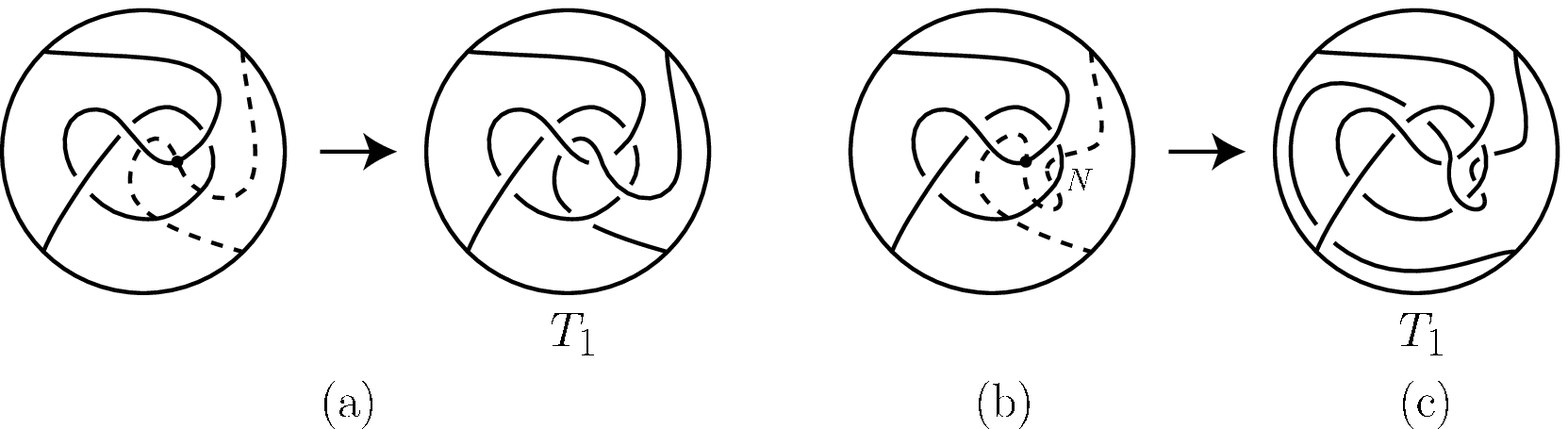}}
      \end{center}
   \caption{}
  \label{lemma2-proof26}
\end{figure} 

%

Thus we have that $N$ is not on $s(\tilde a,M)$. Now suppose that $N$ is on $s(\tilde a,S)$. Then we have one of the situations illustrated in Fig. \ref{lemma2-proof27} and we have $T(-2)$ or $T_1$. Note that in the situation illustrated in Fig. \ref{lemma2-proof27} (c) there are just two mixed crossings of $r(\tilde T,M)$.

\begin{figure}[htbp]
      \begin{center}
\scalebox{0.65}{\includegraphics*{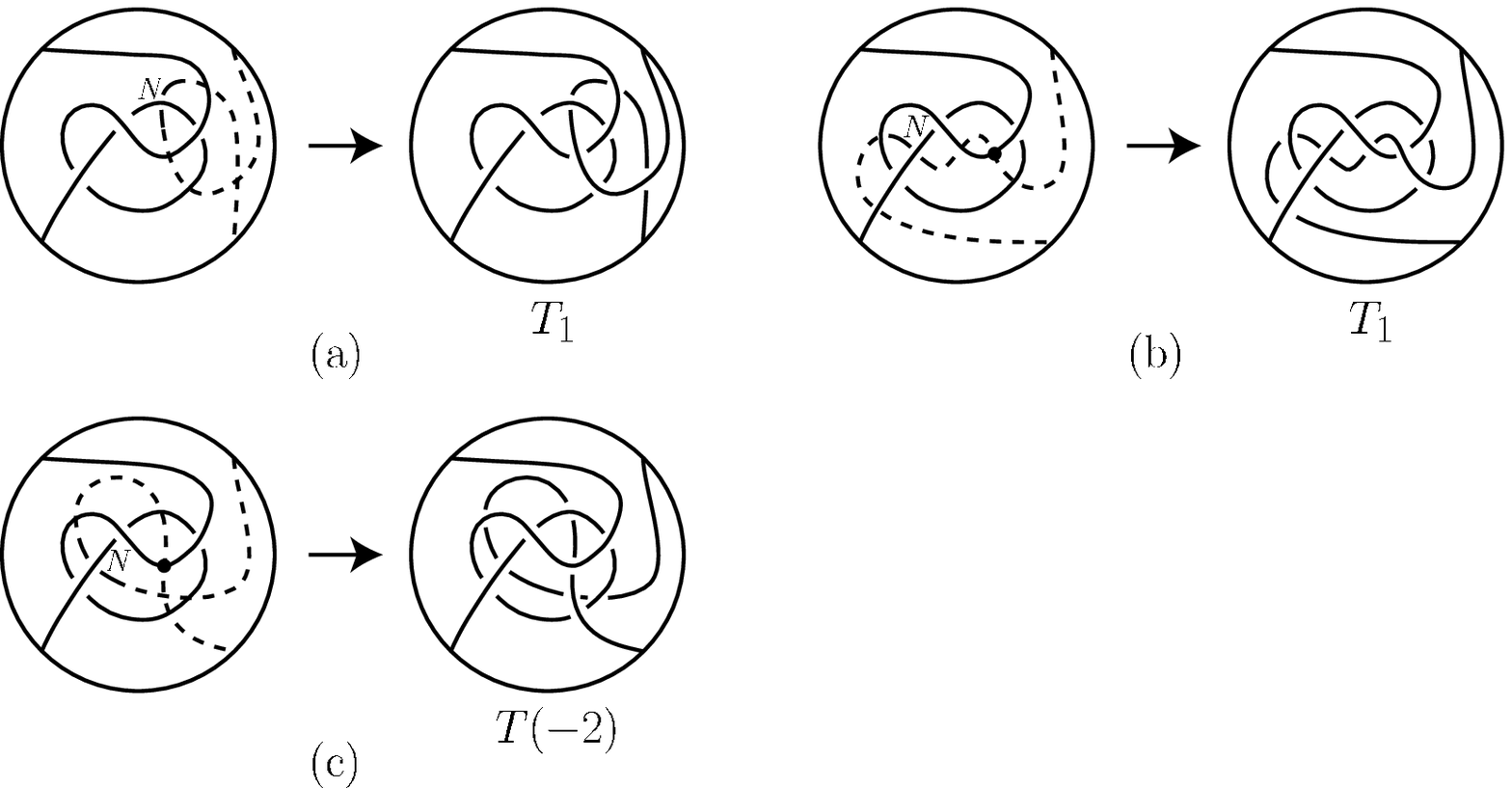}}
      \end{center}
   \caption{}
  \label{lemma2-proof27}
\end{figure} 

%

Thus we have that $N$ is not on $s(\tilde a,M)\cup s(\tilde a,S)$. Then there are two successive positive mixed crossings on $s(\tilde a,S)$.
If they are not on $MU$ then we have $T(-2)$ by \lq\lq over and under technique''. If they are as illustrated in Fig. \ref{lemma2-proof28} (a), (b), (c) or (d) then we have $T(-2)$ or $T_1$. Then we assume that there are no such situations. In the situation illustrated in Fig. \ref{lemma2-proof29} (a) we either have $T_1$, or if the third dotted line do not intersects $SA_\infty$ then we deform $A_0MUS$ first and then we have $T(-2)$ as illustrated. If the third dotted line intersects $SA_\infty$ then we have the situation illustrated in Fig. \ref{lemma2-proof29} (b) and we have $T(-2)$. In the case illustrated in Fig. \ref{lemma2-proof29} (c) we apply Lemma \ref{almost-positive-tangle-lemma} to $r(\tilde T,M)$ and we have $T(-2)$.

\begin{figure}[htbp]
      \begin{center}
\scalebox{0.65}{\includegraphics*{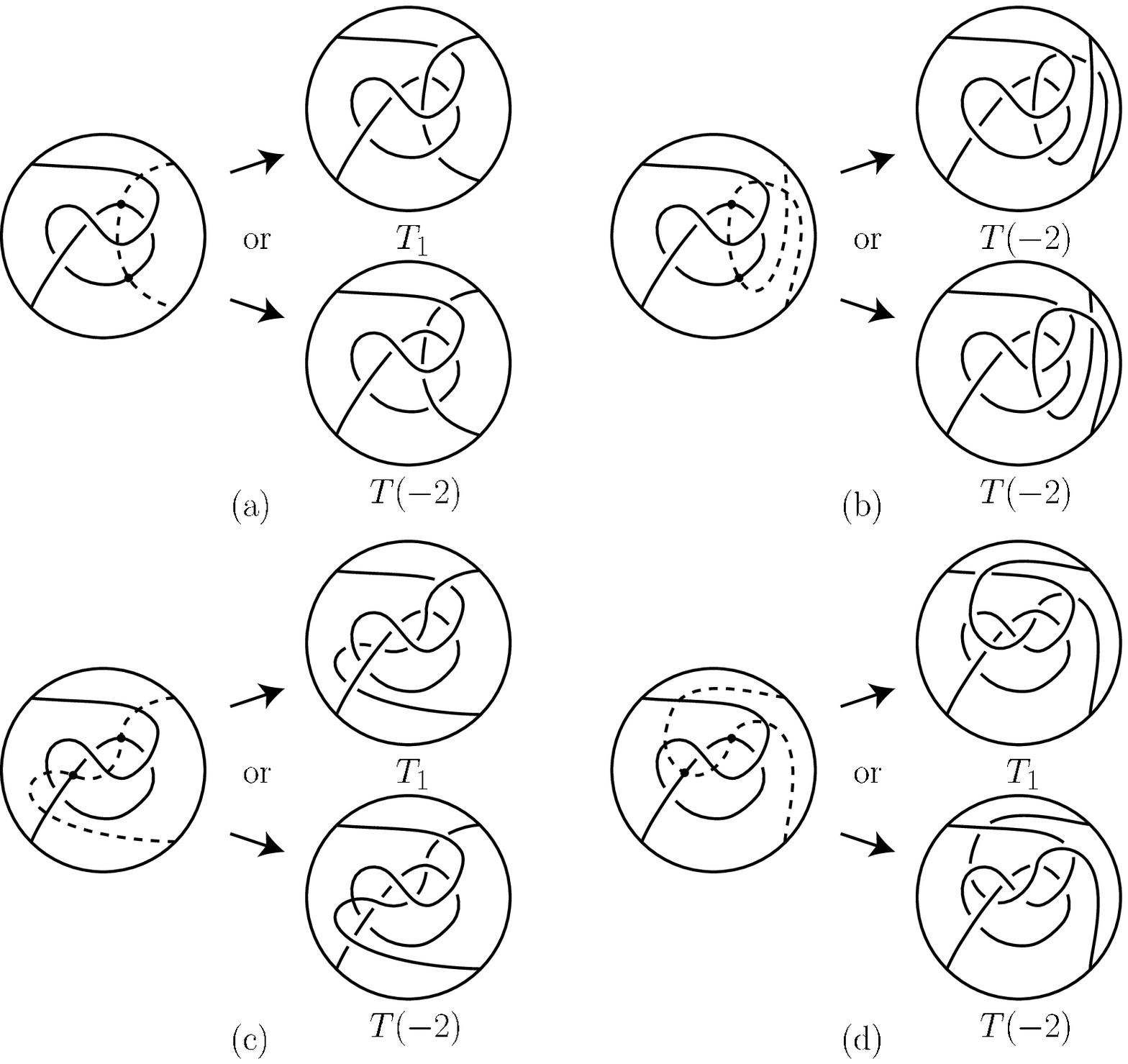}}
      \end{center}
   \caption{}
  \label{lemma2-proof28}
\end{figure} 

%

%
\begin{figure}[htbp]
      \begin{center}
\scalebox{0.65}{\includegraphics*{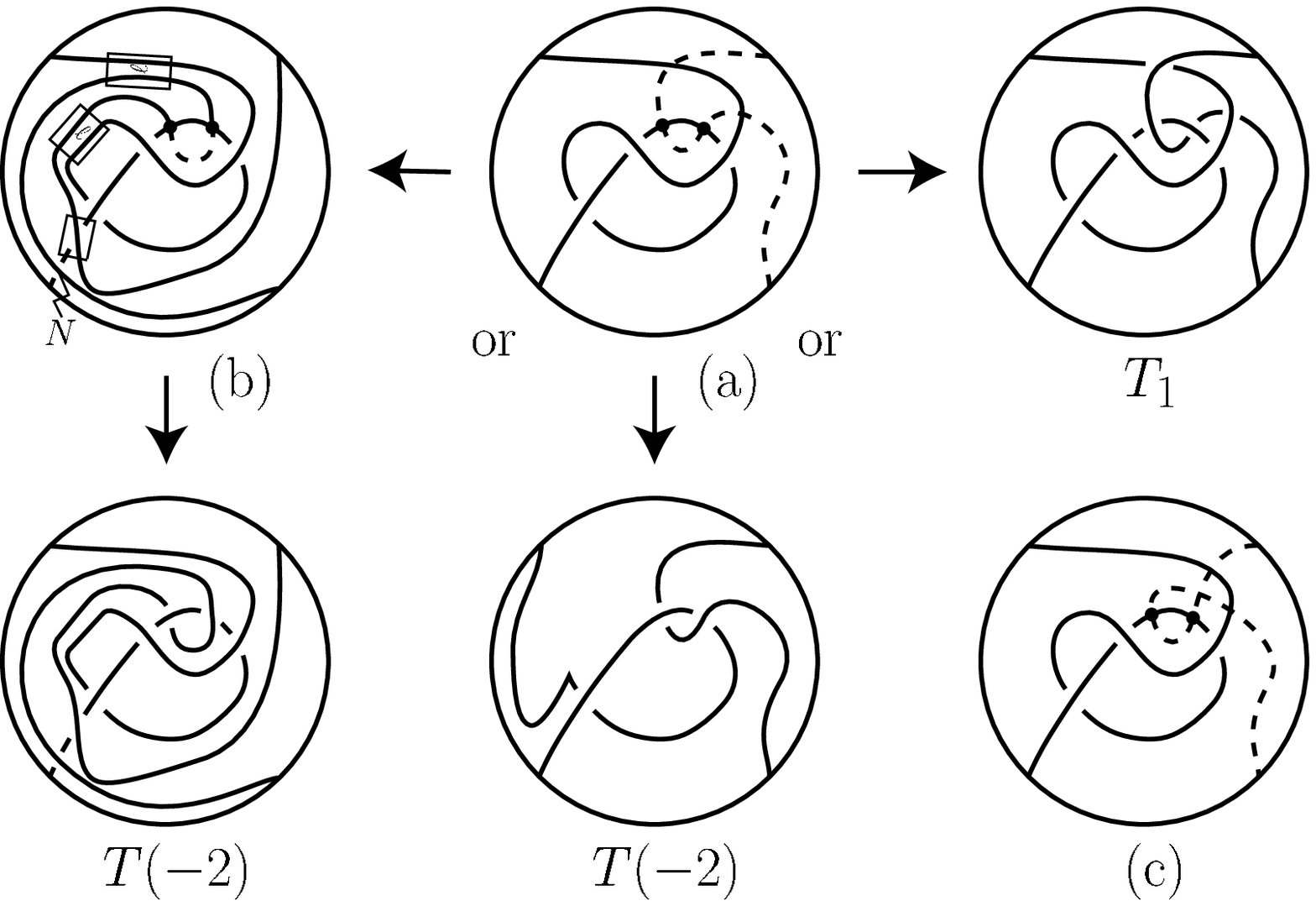}}
      \end{center}
   \caption{}
  \label{lemma2-proof29}
\end{figure} 

%

\item[Case 1.3.3.]
$q=s=1$.

First we consider the case that $\tilde a=\tilde a_{2+}$ in Fig. \ref{lemma2-proof17}.
Then we have $M=V$. By Lemma \ref{diagram-reducing-lemma} we have $\tilde T\geq r(\tilde T,M)$.
Note that, because $\tilde T$ is R2-reduced, the 2-gon $UMU$ has mixed crossings. 
Then by Lemma \ref{almost-positive-tangle-lemma} we have that $r(\tilde T,M)\geq T(-2)$.

Next we consider the case that $\tilde a=\tilde a_{2-}$ in Fig. \ref{lemma2-proof17}.
Then we have $M=U$. See Fig. \ref{lemma2-proof30}.

\begin{figure}[htbp]
      \begin{center}
\scalebox{0.65}{\includegraphics*{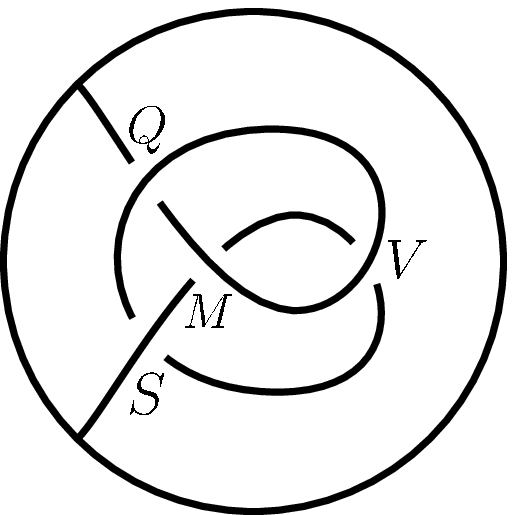}}
      \end{center}
   \caption{}
  \label{lemma2-proof30}
\end{figure} 

%

Note that, because $\tilde T$ is R2-reduced, the 2-gon $MVM$ has mixed crossings. 
Therefore we have that both $s(\tilde a,Q)$ and $s(\tilde a,S)$ have mixed crossings.
Suppose that $N$ is on $r(r(\tilde a,Q),S)$. By Lemma \ref{diagram-reducing-lemma} we have $\tilde T\geq r(\tilde T,Q)$ and $\tilde T\geq r(\tilde T,S)$. Then by Lemma \ref{almost-positive-tangle-lemma} we have $r(\tilde T,Q)\geq T(-2)$. Therefore we have that $N$ is not on $r(r(\tilde a,Q),S)$. Then up to horizontal symmetry we may suppose that $N$ is on $s(\tilde a,S)$. By Lemma \ref{diagram-reducing-lemma} we have $\tilde T\geq r(\tilde T,Q)$.
If $\tilde a$ is under $\tilde b$ at $N$ then by Lemma \ref{diagram-reducing-lemma} we have $\tilde T\geq r(\tilde T,S)$. Then since $r(\tilde T,S)$ is a positive diagram we have $r(\tilde T,S)\geq T(-2)$ by Lemma \ref{hook-tangle}.
Therefore we have that $\tilde a$ is over $\tilde b$ at $N$. We have by Lemma \ref{almost-positive-tangle-lemma} that $r(\tilde T,Q)$ is an R1-augmentation of the diagram $\tilde T_{1+}$ in Fig. \ref{almost-positive-tangles}. 
Now we consider the position of two successive positive mixed crossings on $s(\tilde a,Q)$. If they are not on $MV$ then we have $T(-2)$ by \lq\lq over and under technique''. Suppose that one of them is on $QM$ and the other is on $MV$. Since $s(\tilde a,A)$ has just two mixed crossings we have that this case reduces to the previous case. Suppose that one of them is on $MV$ and the other is on $VQ$. Then we have $T(-2)$ or $T_5$ as illustrated in Fig. \ref{lemma2-proof31} (a) and (b). Finally suppose that no former cases occur and both of them are on $MV$. Then we consider the position of the first mixed crossing of $\tilde b$ with $s(\tilde a,Q)$,  the last mixed crossing of $\tilde b$ with $s(\tilde a,Q)$ and the two mixed crossings of $\tilde b$ with $s(\tilde a,S)$. Then we have $T(-2)$ or $T_1$ as illustrated in Fig. \ref{lemma2-proof31} (c).

\begin{figure}[htbp]
      \begin{center}
\scalebox{0.65}{\includegraphics*{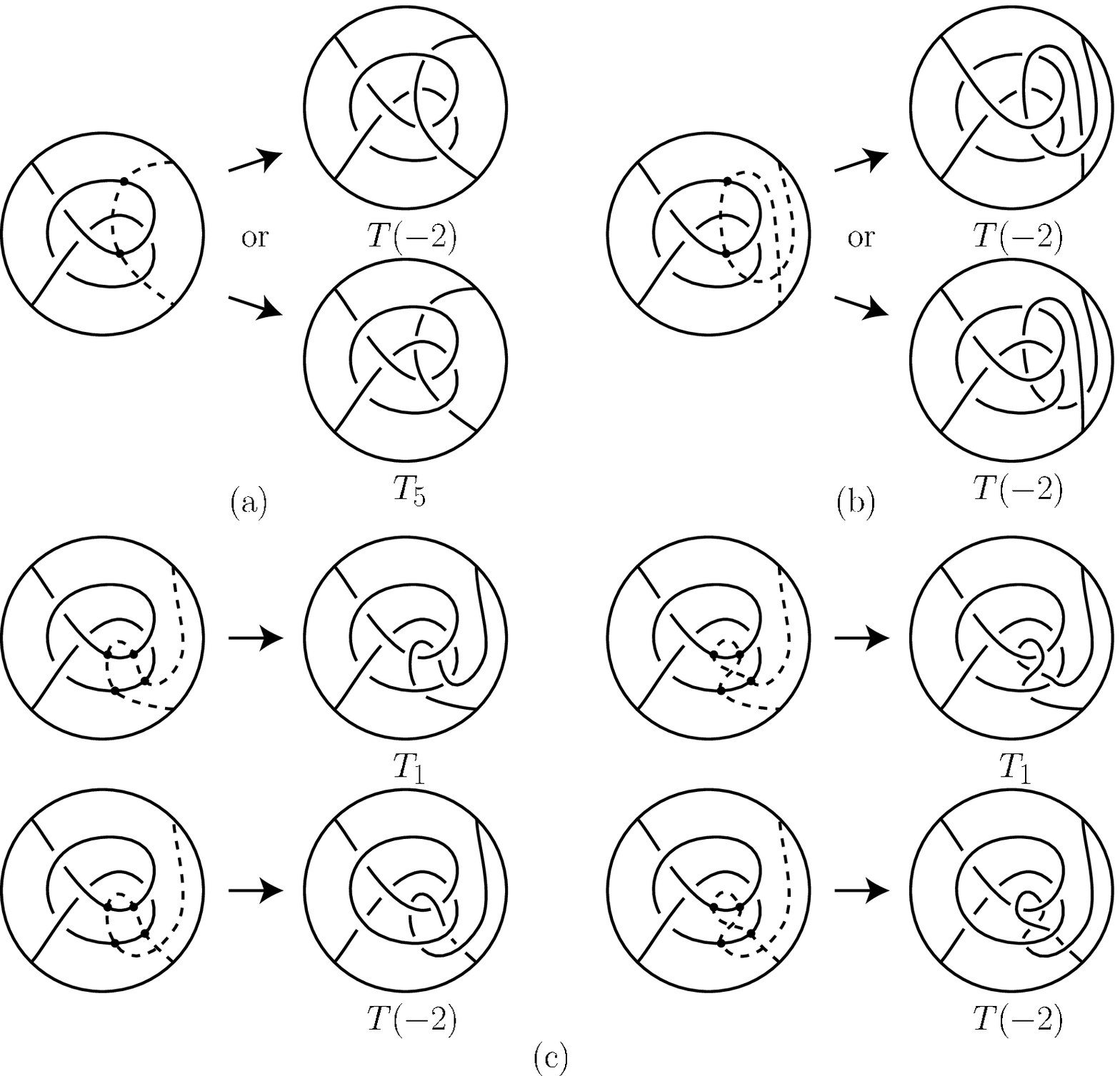}}
      \end{center}
   \caption{}
  \label{lemma2-proof31}
\end{figure} 

%

\item[Case 2.]
$\tilde a$ is almost trivial.

By Lemma \ref{almost-positive-tangle-lemma} we may suppose that $r(\tilde T,M)$ is up to horizontal symmetry, an 
R1-augmentation of one of the diagrams in Fig. \ref{lemma2-proof32}. Note that because  
 $\tilde T$ is R2-reduced, $r(\tilde a,M)$ has no R1-residuals.

\begin{figure}[htbp]
      \begin{center}
\scalebox{0.65}{\includegraphics*{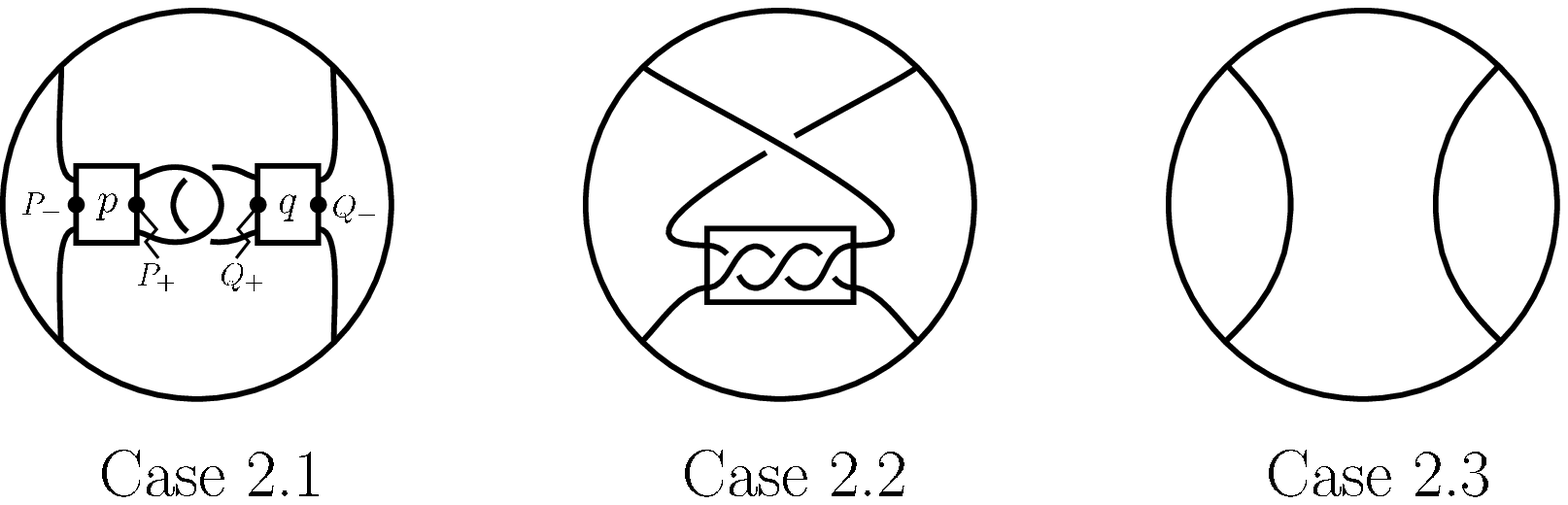}}
      \end{center}
   \caption{}
  \label{lemma2-proof32}
\end{figure} 

%

\item[Case 2.1.]

Suppose that $B_0Q_+\cup Q_+B_\infty$ has mixed crossings. Then by applying Lemma \ref{hook-tangle} to $r(\tilde T,Q_+)$ we have the tangle $T(-2)$. Therefore we have that $B_0Q_+\cup Q_+B_\infty$ has no mixed crossings. Then by flyping we may suppose that $q=0$.
In the following we consider the case $p\leq1$. By that argument it is easily seen that in the case $p\geq2$ we have the tangle $T(-2)$.

First suppose $p=1$.
We consider the position of $M$ on $\tilde a$ and where $\tilde b$ $s(\tilde a,M)$ has intersection with $\tilde b$.
A dotted arrow in Fig. \ref{lemma2-proof33} implies that $M$ is on its initial point and $s(\tilde a,M)$ has intersection with the part of $\tilde b$ where the terminal point is there. For example Arrow (1) describes the case that $M$ is on $UP$ and $s(\tilde a,M)$ has intersection with $UN$ of $\tilde b$. We denote this case by Case (1).
Because $\tilde T$ is R2-reduced, there are only five possible cases illustrated in Fig. \ref{lemma2-proof33}. 
Note that Case (3) and Case (5) may occur simultaneously. In Cases (1), (2) and (3) we have 
the tangle $T(-2)$ by Lemma \ref{hook-tangle} as illustrated in Fig. \ref{lemma2-proof34}.
Thus we have that these cases do not happen. In Case (4) we have $T(-2)$ by Lemma \ref{vertical-trefoil-tangle-lemma} 
unless $\tilde T$ is a diagram $\tilde T_{4+}$ in Fig. \ref{diagrams2} as illustrated in Fig. \ref{lemma2-proof34}. In Case (5) we have $T(-2)$ by Lemma \ref{vertical-trefoil-tangle-lemma} unless $\tilde T$ is a diagram $\tilde T_{6+}$ in Fig. \ref{diagrams2} as illustrated in Fig. \ref{lemma2-proof34}.

\begin{figure}[htbp]
      \begin{center}
\scalebox{0.65}{\includegraphics*{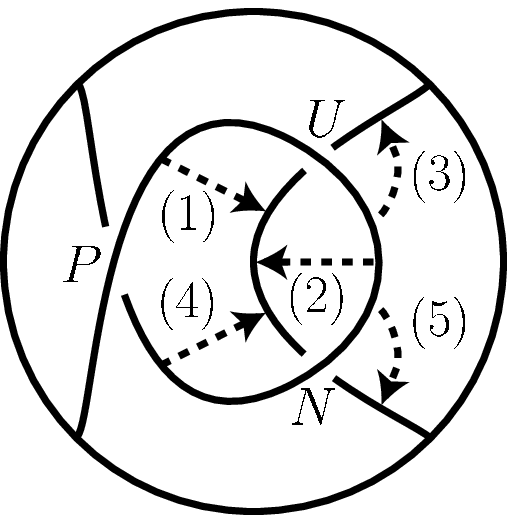}}
      \end{center}
   \caption{}
  \label{lemma2-proof33}
\end{figure} 

%

%
\begin{figure}[htbp]
      \begin{center}
\scalebox{0.65}{\includegraphics*{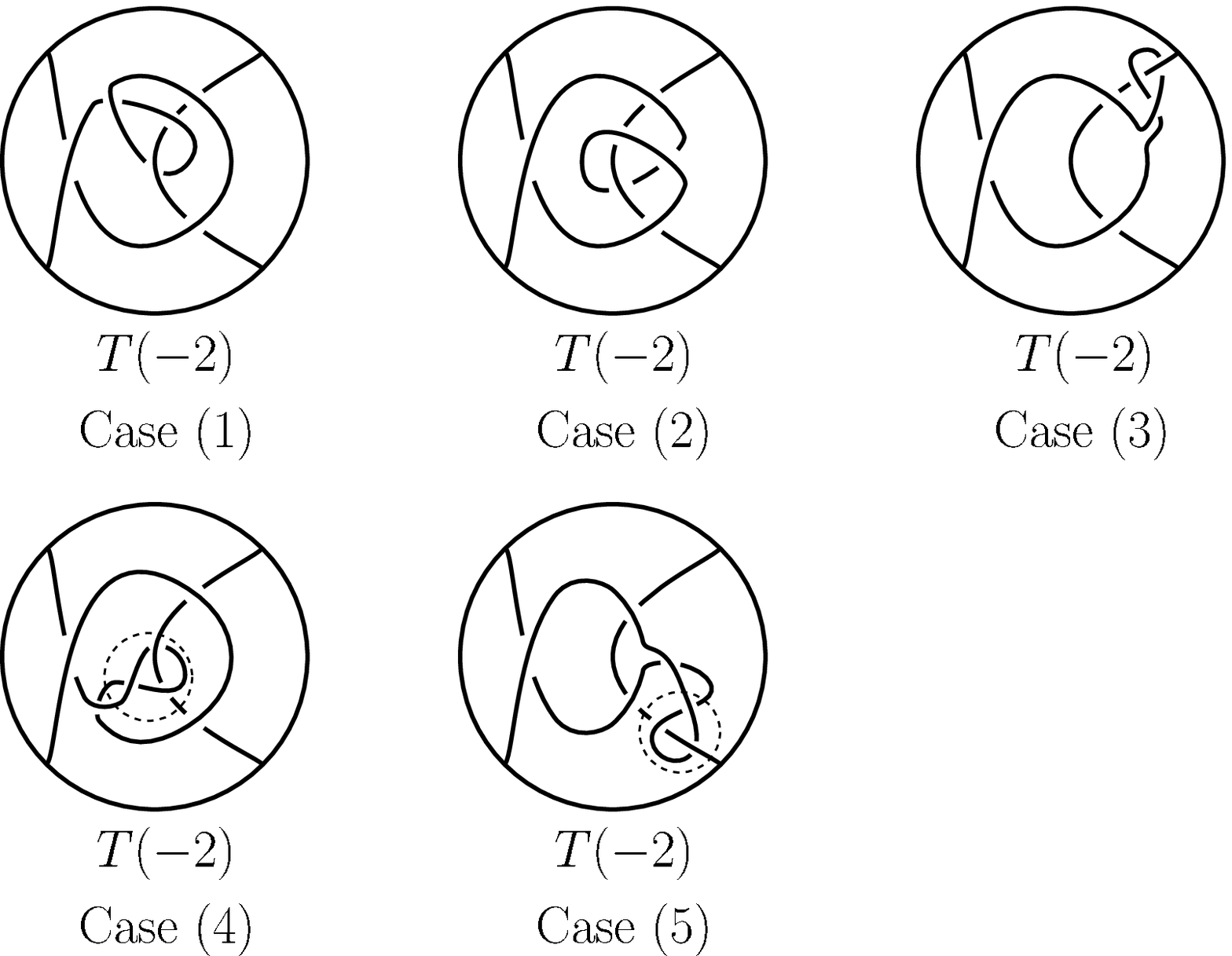}}
      \end{center}
   \caption{}
  \label{lemma2-proof34}
\end{figure} 

%

Next suppose $p=0$.
Then we have the cases illustrated in Fig. \ref{lemma2-proof35}. We have $T(-2)$ in Cases (1) and (2). Cases (3), (4) and (5) correspond the diagrams $\tilde T_{3+}$, $\tilde T_{2++}$ and $\tilde T_{2-+}$ in Fig. \ref{diagrams2} respectively.

\item[Case 2.2.]

In the cases illustrated in Fig. \ref{lemma2-proof35} we have the tangle $T(-2)$ or $T_2$ as illustrated in Fig. \ref{lemma2-proof36}.
Here we denote by Case (i)$'$ a subcase of Case (i). The proof for Case (i)$'$ is entirely analogous to that of Case (i) after removing extra crossings.
Other cases are illustrated in Fig. \ref{lemma2-proof37}. Then we have in Cases (6), (7), (8), (9) and (10) the diagram $\tilde T_{2-+}$ in Fig. \ref{diagrams2} and in Cases (11), (12), (13), (14) and (15) the diagram $\tilde T_{2++}$ in Fig. \ref{diagrams2}.

\begin{figure}[htbp]
      \begin{center}
\scalebox{0.7}{\includegraphics*{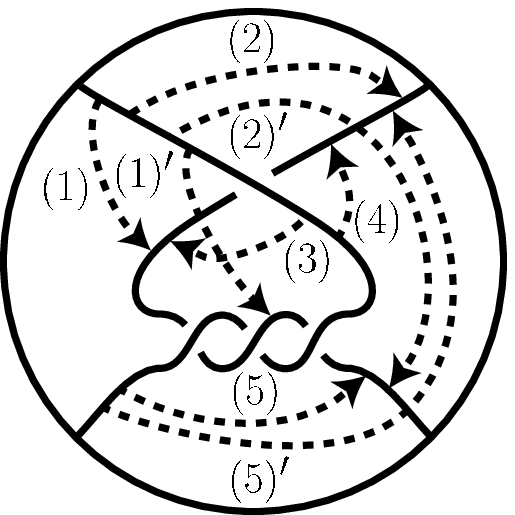}}
      \end{center}
   \caption{}
  \label{lemma2-proof35}
\end{figure} 

%

%
\begin{figure}[htbp]
      \begin{center}
\scalebox{0.65}{\includegraphics*{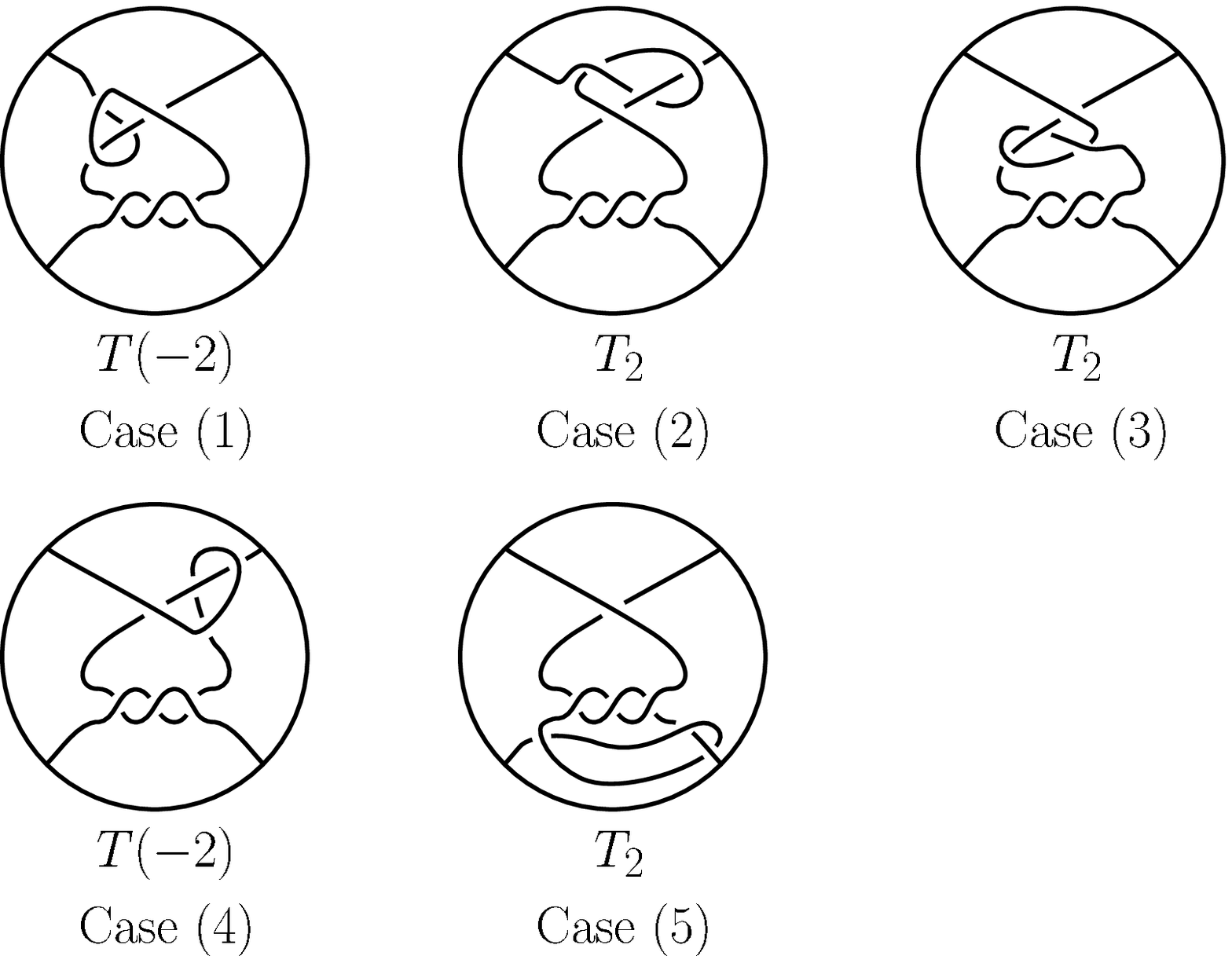}}
      \end{center}
   \caption{}
  \label{lemma2-proof36}
\end{figure} 

%

%
\begin{figure}[htbp]
      \begin{center}
\scalebox{0.7}{\includegraphics*{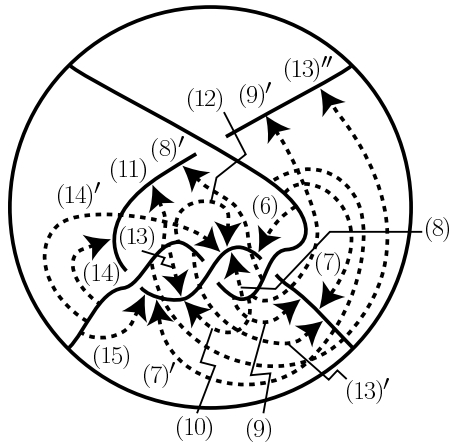}}
      \end{center}
   \caption{Case 2.2 of the proof of Lemma 4.2}
  \label{lemma2-proof37}
\end{figure} 

%

\item[Case 2.3.]
In this case we have the diagram $\tilde T_{1}$ in Fig. \ref{diagrams2}.

\end{enumerate}

This completes the proof of Lemma \ref{lemma2}.
\end{proof}


\begin{Lemma}\label{sublemma2}
Let $\hat T$ be a 2-string tangle projection with vertical connection. 
Suppose that $\hat T$ has no self-crossings.
If $\hat T$ is not a projection of the tangle $R(T(3,1))$ illustrated in Fig. \ref{lemma3tangles} then $\hat T=\hat T(2n)$ (see Fig. \ref{tangle-notation}) for some non-negative integer $n$.
\end{Lemma}

\begin{proof}
The proof here is given by contradiction. Suppose that $\hat T$ is a 2-string tangle projection with vertical connection with $2n$ mixed crossings and no self-crossings that is not equal to $\hat T(2n)$ such that it is not a projection of the tangle $R(T(3,1))$. Suppose that $n$ is smallest among such tangle projections. 
It is clear that $\hat T$ has a 2-gon as $\hat T_1$ illustrated in Fig. \ref{R2-projections}. Let $\hat T'$ be a tangle projection obtained from $\hat T$ by replacing the 2-gon by a pair of parallel arcs as $\hat T_2$ illustrated in Fig. \ref{R2-projections}. Then by Lemma \ref{sublemma1} we have that $\hat T'=\hat T(2(n-1))$. However it is easy to check that all tangle projections without mixed crossings obtained from $\hat T(2(n-1))$ by replacing a pair of parallel arcs by a 2-gon except $\hat T(2n)$ is a projection of $R(T(3,1))$. See Fig. \ref{sublemma2-proof}. This completes the proof.
\end{proof}

\begin{figure}[htbp]
      \begin{center}
\scalebox{0.65}{\includegraphics*{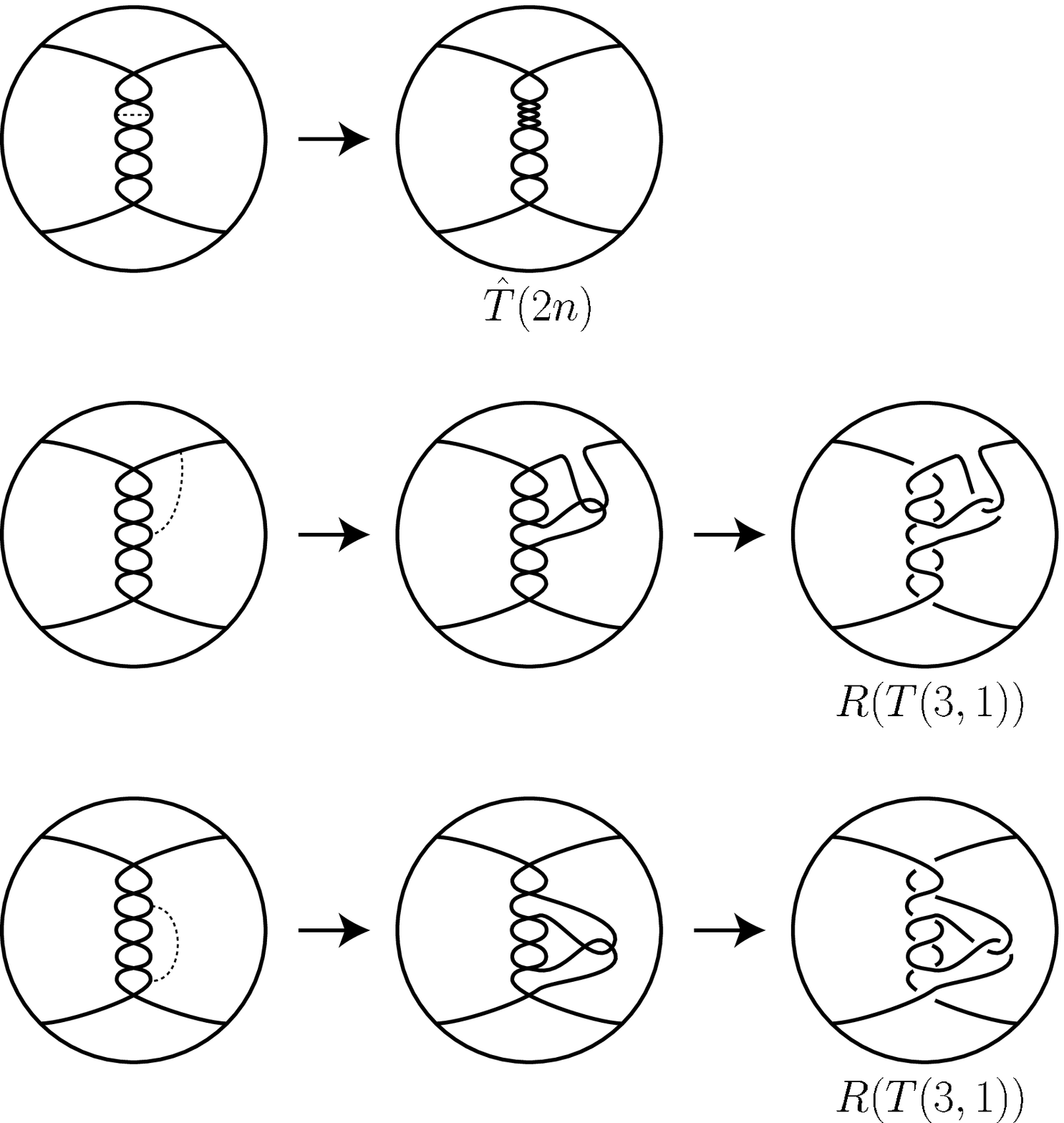}}
      \end{center}
   \caption{}
  \label{sublemma2-proof}
\end{figure} 

%

\begin{Lemma}\label{lemma3}
Let $\hat T$ be a reduced 2-string tangle projection with vertical connection. If $\hat T$ is not
a projection of any of the tangles in Fig. \ref{lemma3tangles} then $\hat T$ is one of the projections
illustrated in Fig. \ref{lemma3projections}
\end{Lemma}

\begin{figure}[htbp]
      \begin{center}
\scalebox{0.65}{\includegraphics*{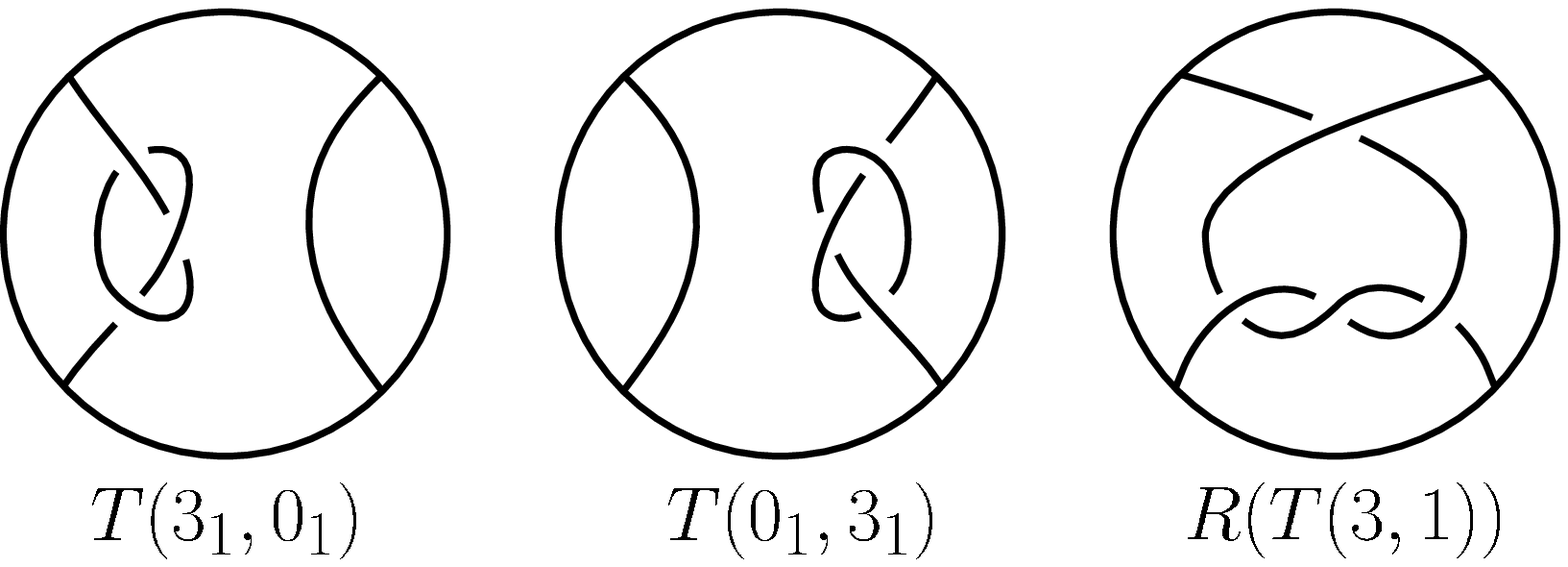}}
      \end{center}
   \caption{}
  \label{lemma3tangles}
\end{figure} 

%

%
\begin{figure}[htbp]
      \begin{center}
\scalebox{0.65}{\includegraphics*{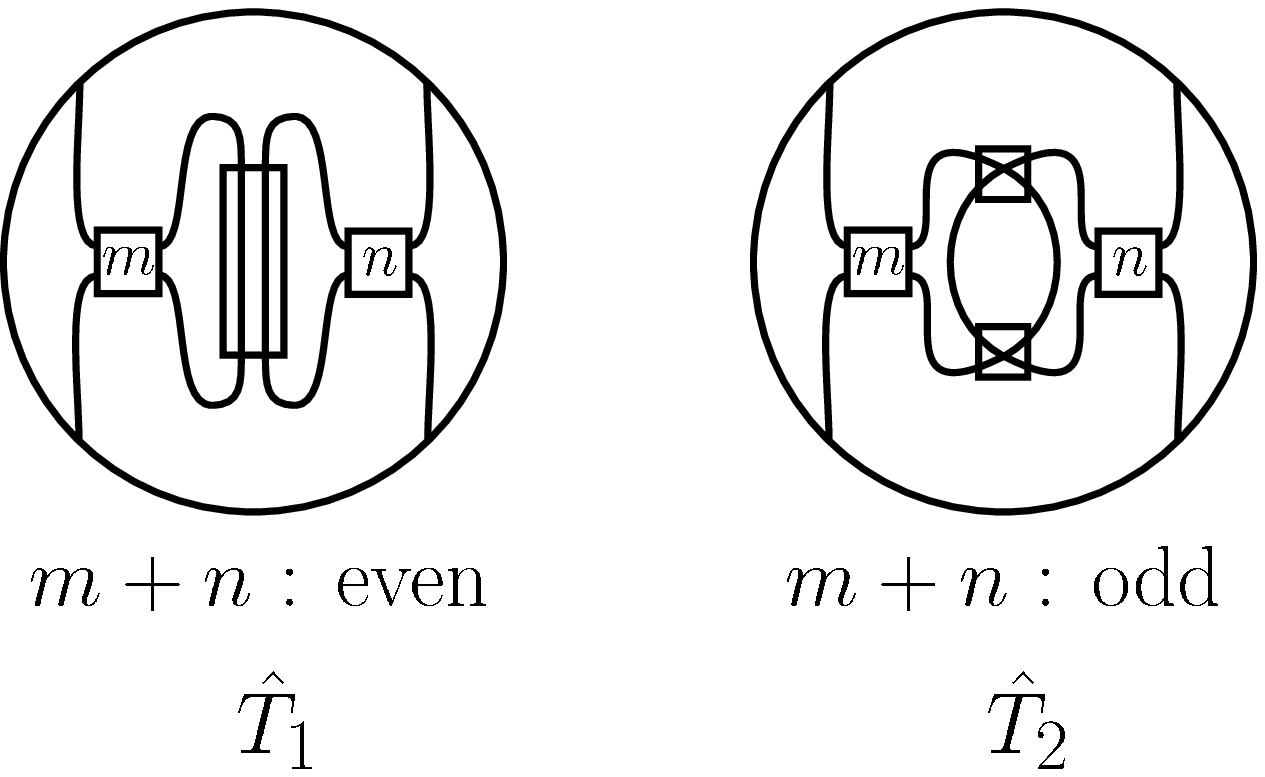}}
      \end{center}
   \caption{}
  \label{lemma3projections}
\end{figure} 

%

\begin{proof}
We have that both $\hat a$ and $\hat b$ are almost trivial.

First suppose that $\hat a$ has self-crossings. 
If there is a self-crossing of $\hat a$ which is not rightmost then we have $R(T(3,1))$ using 
Lemma \ref{reducing-lemma} as illustrated in Fig. \ref{lemma3-proof1}.

\begin{figure}[htbp]
      \begin{center}
\scalebox{0.65}{\includegraphics*{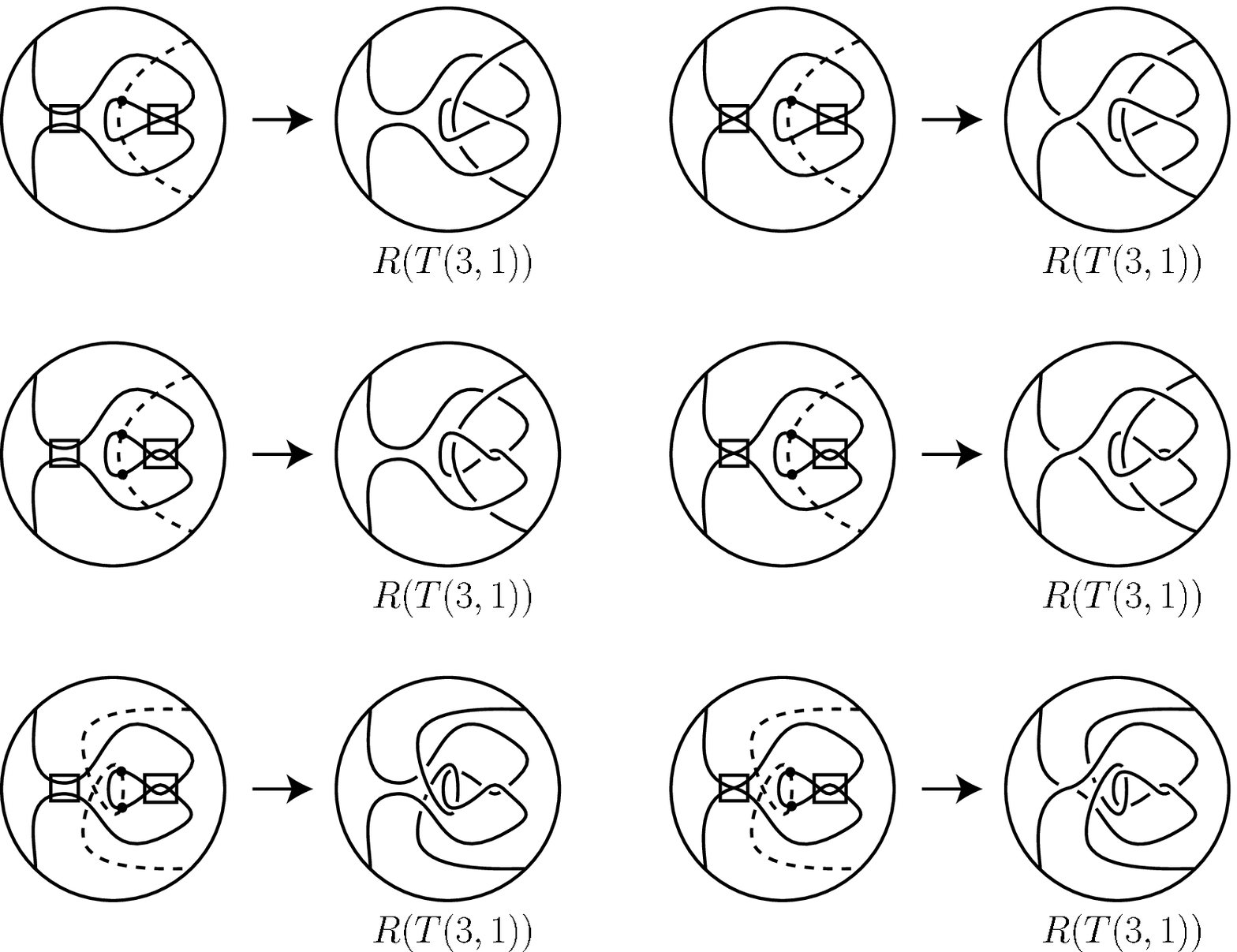}}
      \end{center}
   \caption{}
  \label{lemma3-proof1}
\end{figure} 

%

Therefore all mixed crossings on $\hat a$ are rightmost. Let $B_{2i+1}$ be a mixed crossing with 
maximal multiplicity among all mixed crossings on $\hat a$. Then we have that $B_{2i+2}$ also has the maximal multiplicity.
We consider the position of successive mixed crossings $B_{2i}$, $B_{2i+1}$ and $B_{2i+2}$, or $B_{2i+1}$, $B_{2i+2}$ and $B_{2i+3}$. If they are positioned as illustrated in Fig. \ref{lemma3-proof2}, where only the positions of $B_{2i}$, $B_{2i+1}$ and $B_{2i+2}$ are illustrated, then we have the tangle $R(T(3,1))$. Up to horizontal symmetry we have the tangle $R(T(3,1))$ for the positions of $B_{2i+1}$, $B_{2i+2}$ and $B_{2i+3}$ corresponding to that of $B_{2i}$, $B_{2i+1}$ illustrated in Fig. \ref{lemma3-proof2}. Note that in Fig. \ref{lemma3-proof2} the position of $B_{2i}$ is not precisely described. Namely the multiplicity of $B_{2i}$ on $\hat a$ may be greater than that illustrated in Fig. \ref{lemma3-proof2}. However by the argument described in Fig. \ref{lemma2-proof6} we have the same result for these cases.

\begin{figure}[htbp]
      \begin{center}
\scalebox{0.65}{\includegraphics*{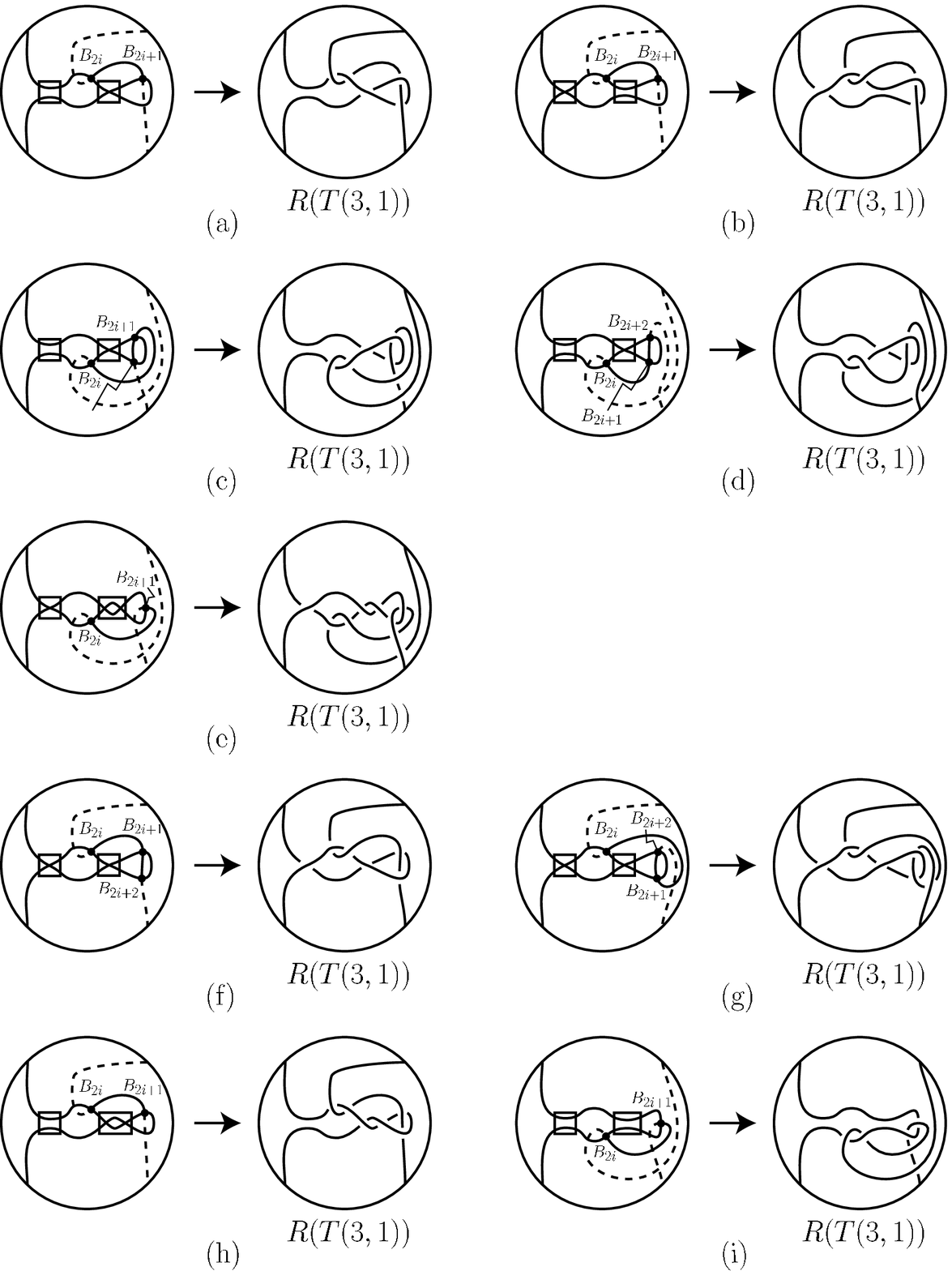}}
      \end{center}
   \caption{}
  \label{lemma3-proof2}
\end{figure} 

%

Note that all cases that the multiplicity $m(B_{2i+1})$ is odd are illustrated in Fig. \ref{lemma3-proof2} (a), (b), (c), (d) and (e). 
Suppose that $B_{2j}$, $B_{2i+1}$ and $B_{2i+2}$ are positioned as illustrated in Fig. \ref{lemma3-proof3} (a) where $j\leq i$ and $B_{k}$ are related to $B_{2i+1}$ on $\hat a$ for $2j< k\leq 2i$. Let $P$ be the root of $B_{2i+1}$ on $\hat a$. Then we have that $B_{2j}B_{2i+1}$ and $A_0P\cup PA_\infty$ intersect only at $B_{2j}$. Therefore after deforming $B_0B_{2j}$ as illustrated in Fig. \ref{lemma3-proof3} (a) we have the tangle $R(T(3,1))$. 
Suppose that $B_{2i}$ and $B_{2i+1}$ are positioned as illustrated in Fig. \ref{lemma3-proof3} (b). 
Then either we have the case illustrated in Fig. \ref{lemma3-proof3} (a) or the case that is the horizontal
symmetry of the case illustrated in Fig. \ref{lemma3-proof2} (f) or (i), or we have the case illustrated in Fig. \ref{lemma3-proof3} (c). In any case we 
have $R(T(3,1))$.

\begin{figure}[htbp]
      \begin{center}
\scalebox{0.65}{\includegraphics*{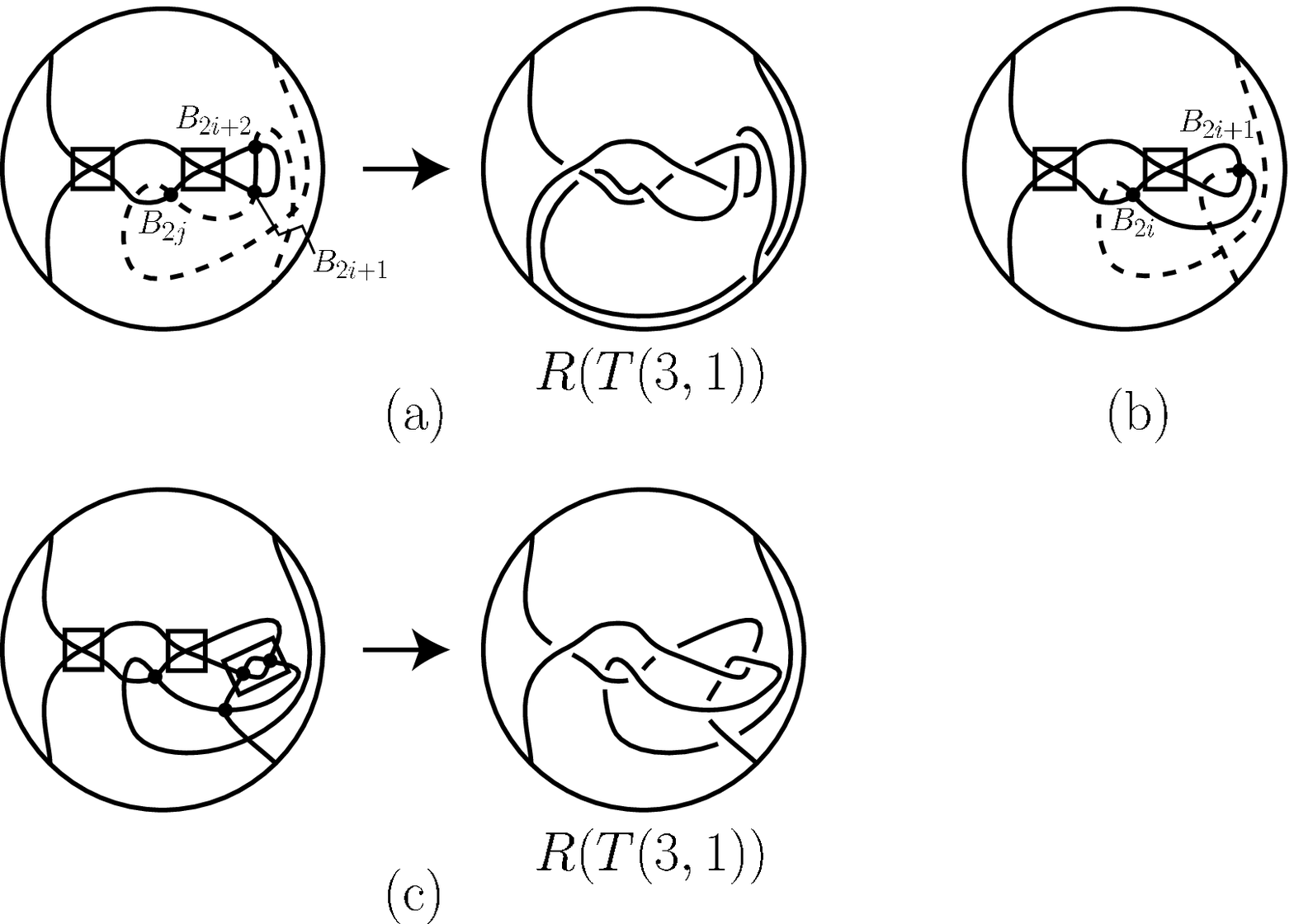}}
      \end{center}
   \caption{}
  \label{lemma3-proof3}
\end{figure} 

%

Therefore we have that all mixed crossings are related on $\hat a$. Similarly we have that all mixed crossings are related on $\hat b$. Therefore we have the situation that there is a subtangle projection $\hat T_{0}$ of $\hat T$ as illustrated in Fig. \ref{lemma3-proof4} (a) where $m$ and $n$ are non-negative integers such that $\hat T_0$ has no self-crossings.

First suppose that $m+n$ is odd. If $\hat T_{0}$ is a
projection of the tangle $T(4)$ then we have $R(T(3,1))$ as illustrated in Fig. \ref{lemma3-proof4} (b). Then we have the conclusion by applying Lemma \ref{-4-lemma} to $\hat T_0$. 

Next suppose that $m+n$ is even. If $\hat T_{0}$ is a projection of the tangle $R(T(3,1))$ then we clearly have that $\hat T$ is also a projection of $R(T(3,1))$. Then we have the conclusion by applying Lemma \ref{sublemma2} to $\hat T_0$. This completes the proof.
\end{proof}

\begin{figure}[htbp]
      \begin{center}
\scalebox{0.65}{\includegraphics*{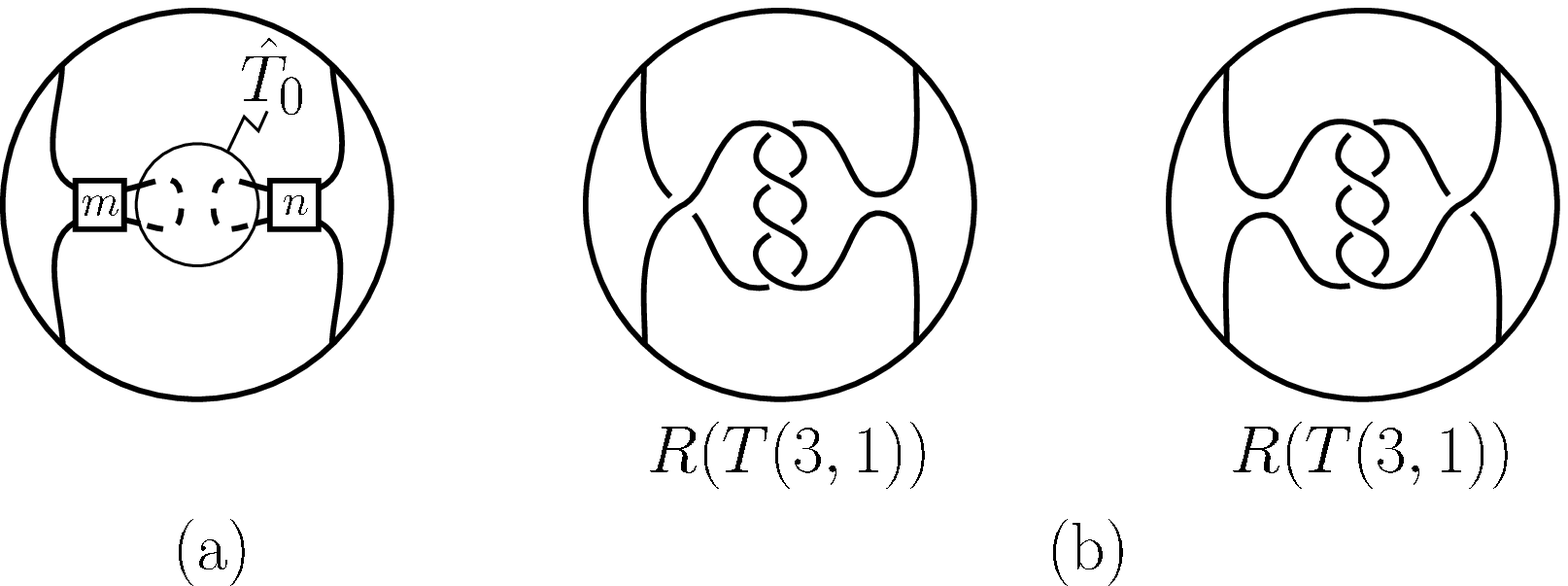}}
      \end{center}
   \caption{}
  \label{lemma3-proof4}
\end{figure} 

%


\begin{Lemma}\label{sublemma3}
Let $\hat T$ be a 2-string tangle projection with X-connection. 
Suppose that $\hat T$ has no self-crossings.
Suppose that $\hat T$ is not a projection of the tangle $T(1/2,-3)$ illustrated in Fig. \ref{lemma4tangles}. 
Then there are odd numbers $n_1,n_2,\cdots,n_k$ such that $\hat T=R(\hat T(n_1,n_2,\cdots,n_k))$ (Fig. \ref{sublemma3projection}).
\end{Lemma}

\begin{figure}[htbp]
      \begin{center}
\scalebox{0.45}{\includegraphics*{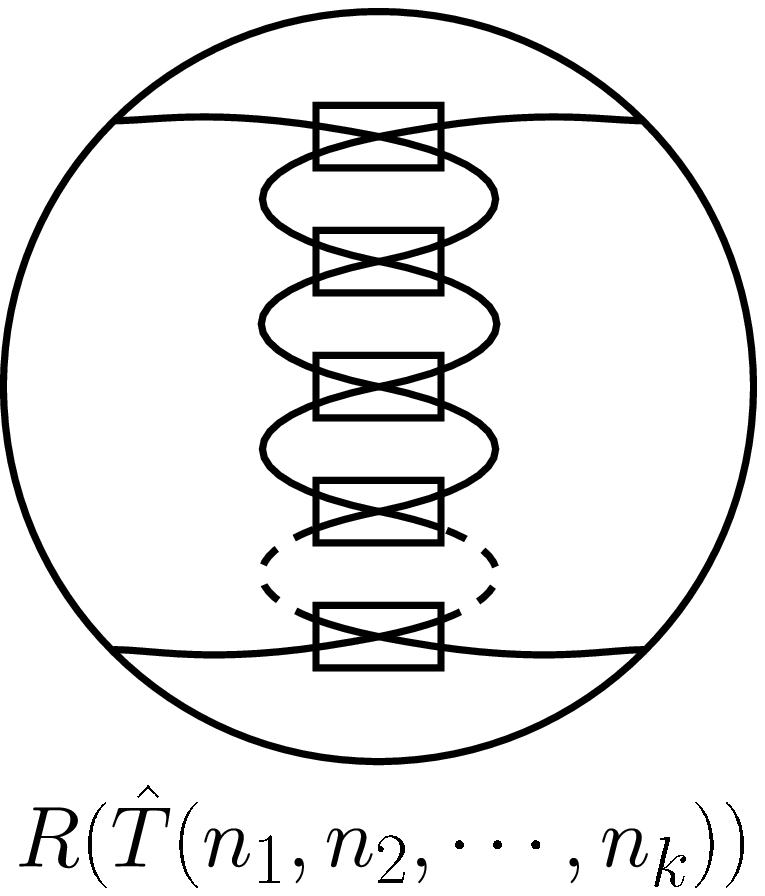}}
      \end{center}
   \caption{}
  \label{sublemma3projection}
\end{figure} 

%

\begin{proof}
The proof here is given by contradiction. Suppose that $\hat T$ is a 2-string tangle projection with X-connection with $n$ mixed crossings and no self-crossings that is not equal to any $R(\hat T(n_1,n_2,\cdots,n_k))$ such that it is not a projection of the tangle $T(1/2,-3)$. Suppose that $n$ is smallest among such tangle projections. 
It is clear that $\hat T$ has a 2-gon as $\hat T_1$ illustrated in Fig. \ref{R2-projections}. Let $\hat T'$ be a tangle projection obtained from $\hat T$ by replacing the 2-gon by a pair of parallel arcs as $\hat T_2$ illustrated in Fig. \ref{R2-projections}. Then by Lemma \ref{sublemma1} we have that $\hat T'=R(\hat T(n_1,n_2,\cdots,n_k))$ for some odd numbers $n_1,n_2,\cdots,n_k$. However it is easy to check that a tangle projection without mixed crossings obtained from $R(\hat T(n_1,n_2,\cdots,n_k))$ by replacing a pair of parallel arcs by a 2-gon is either a projection of $T(1/2,-3)$, or equal to $R(\hat T(n'_1,n'_2,\cdots,n'_{k'}))$ for some odd numbers $n'_1,n'_2,\cdots,n'_{k'}$. See Fig. \ref{sublemma3-proof}. In Fig. \ref{sublemma3-proof} some typical cases are illustrated. Any other case is essentially the same as one of the typical cases. In Fig. \ref{sublemma3-proof} (a) some typical cases that will produce some $R(\hat T(n'_1,n'_2,\cdots,n'_{k'}))$ are illustrated.
This completes the proof.
\end{proof}

\begin{figure}[htbp]
      \begin{center}
\scalebox{0.65}{\includegraphics*{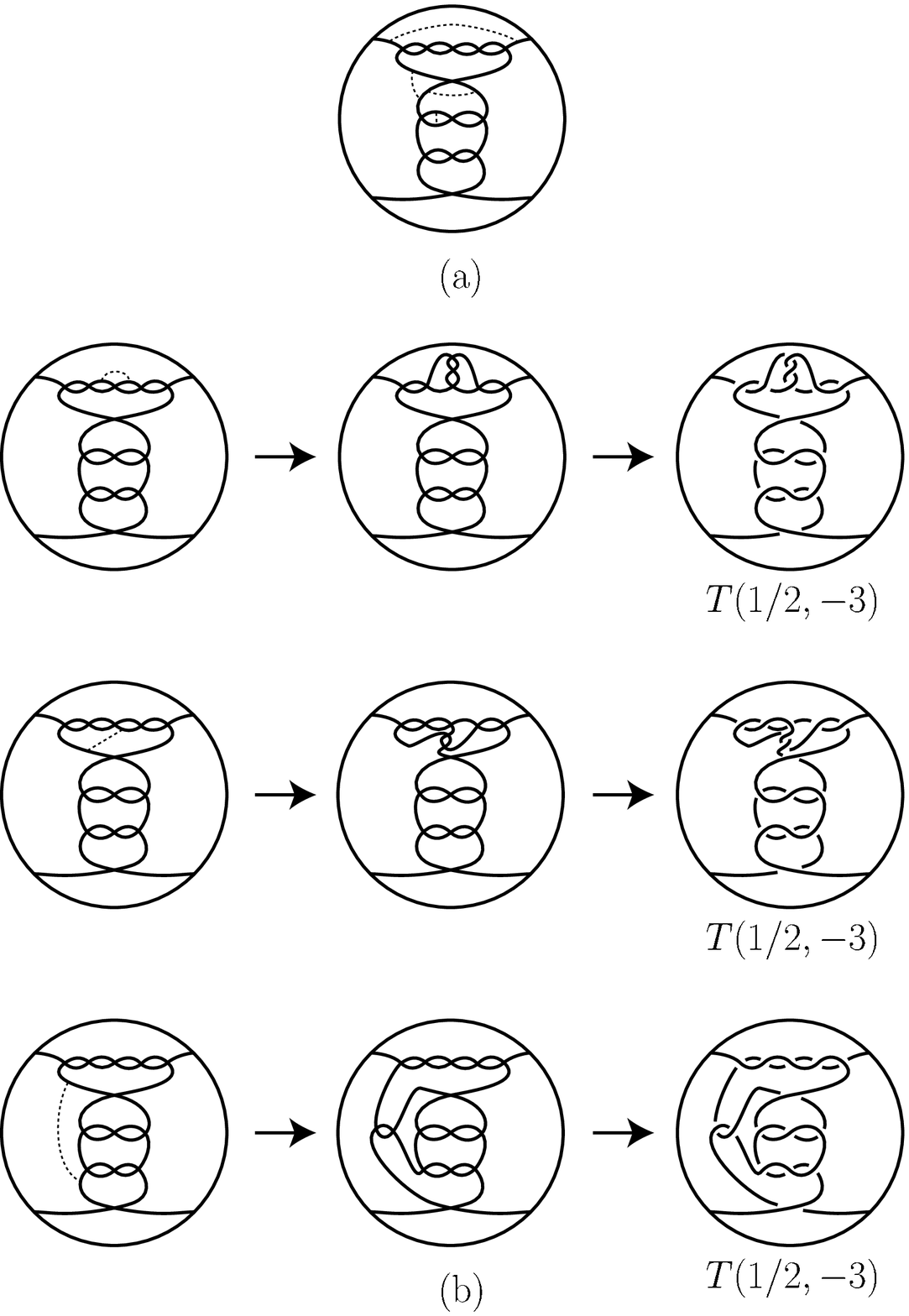}}
      \end{center}
   \caption{}
  \label{sublemma3-proof}
\end{figure} 

%


\begin{Lemma}\label{lemma4}
Let $\hat T$ be a reduced 2-string projection with X-connection which is not a projection of any 
of the tangles in Fig. \ref{lemma4tangles}. 
Then $\hat T$ is (a flype of) (the vertical and/or horizontal symmetry of) one of the tangle projections in Fig. \ref{lemma4projections}.
\end{Lemma}

\begin{figure}[htbp]
      \begin{center}
\scalebox{0.65}{\includegraphics*{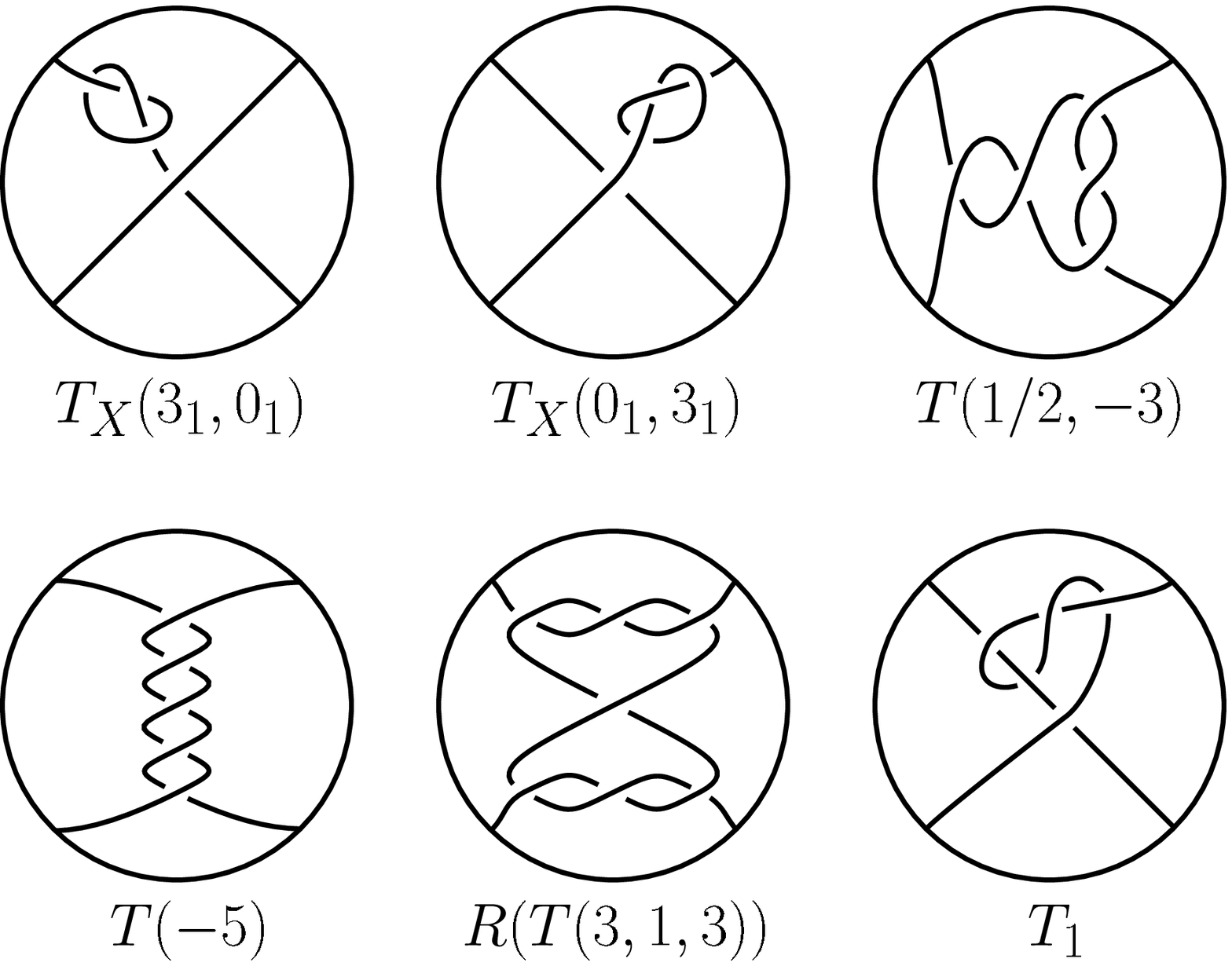}}
      \end{center}
   \caption{}
  \label{lemma4tangles}
\end{figure} 

%

%
\begin{figure}[htbp]
      \begin{center}
\scalebox{0.65}{\includegraphics*{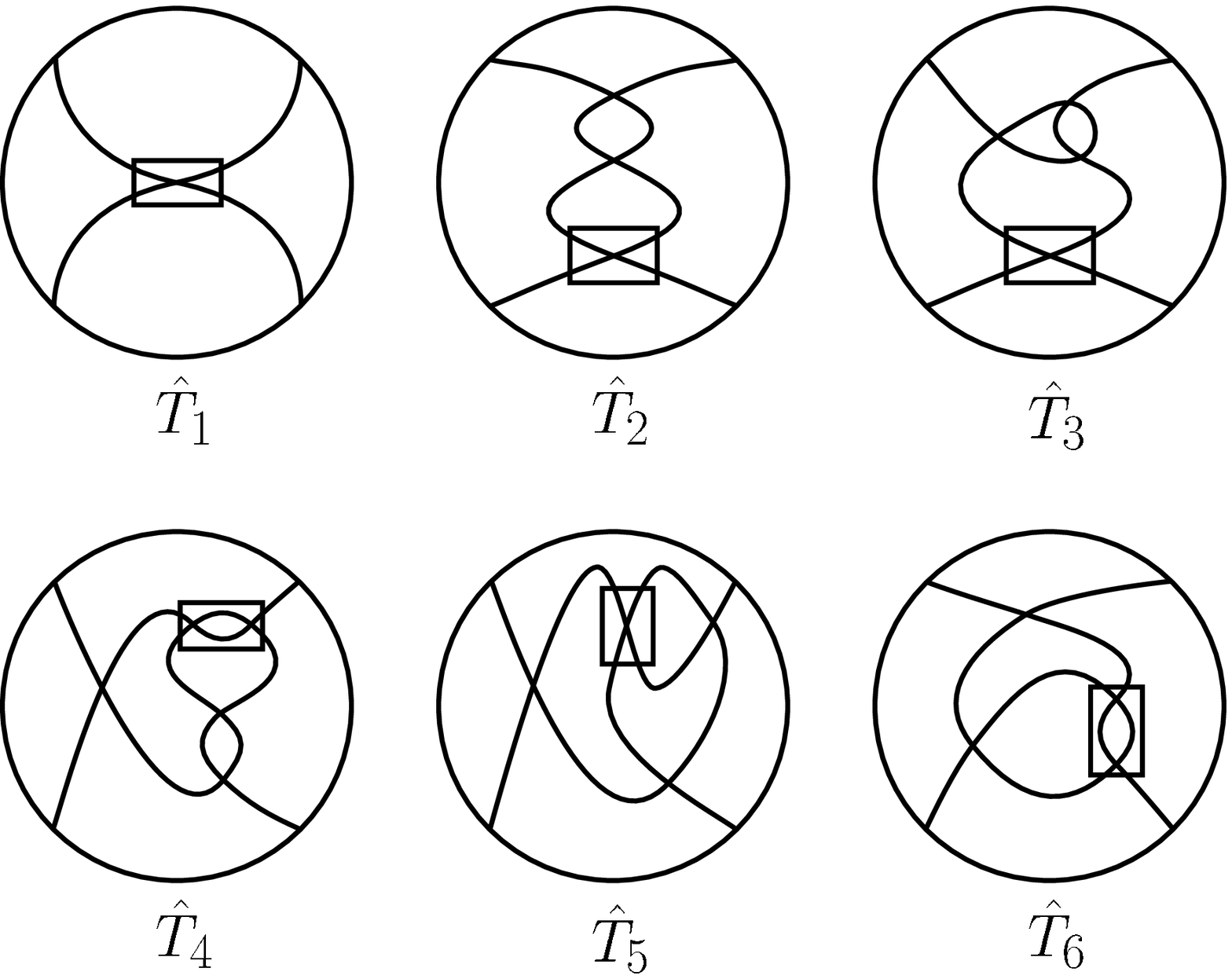}}
      \end{center}
   \caption{}
  \label{lemma4projections}
\end{figure} 

%

\begin{proof}
By Lemma \ref{sublemma3} and by the fact that $\hat T$ is not a projection of $T(1/2,-3)$, $T(-5)$ or $R(T(3,1,3))$ we have that the core $\hat T'=\hat a'\cup\hat b'$ is (a flype of) a projection in Fig. \ref{lemma4proof1}.

\begin{figure}[htbp]
      \begin{center}
\scalebox{0.65}{\includegraphics*{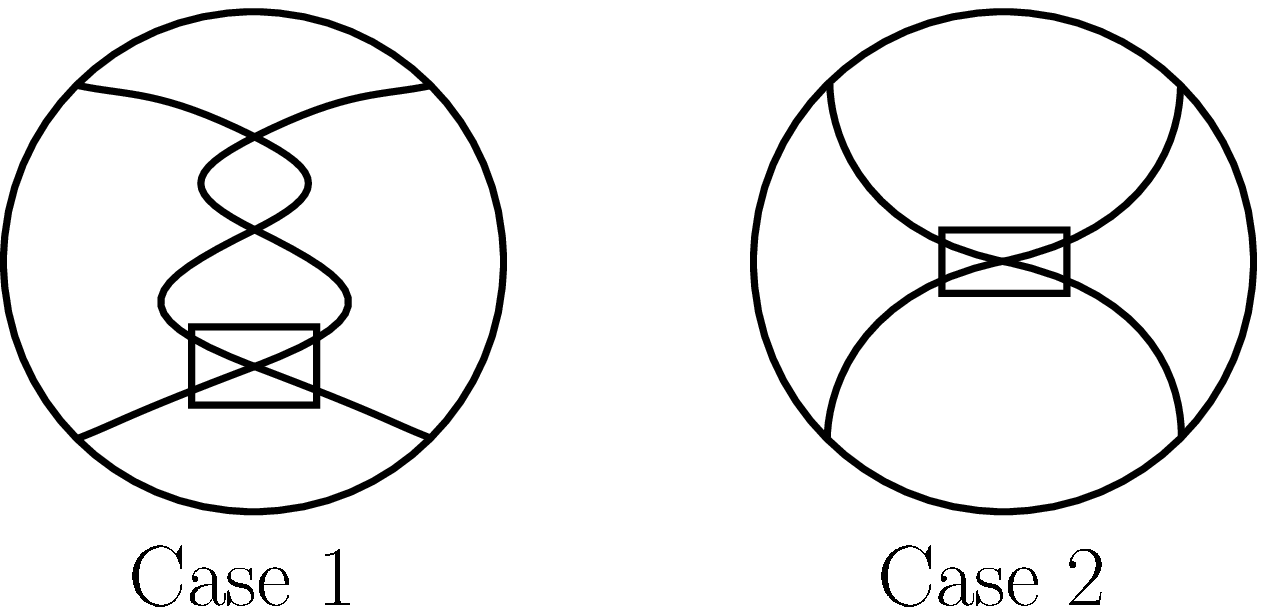}}
      \end{center}
   \caption{}
  \label{lemma4proof1}
\end{figure} 

%

\begin{enumerate}
\item[Case 1.]
If $\hat T$ is not equal to (a flype of) $\hat T_2$ then  $\hat T$ has self-crossing.
Then using Lemma \ref{hook-tangle} we have the tangle $T(-5)$ as typically illustrated in 
Fig. \ref{lemma4proof2} and we have the conclusion. In Fig. \ref{lemma4proof2} a dotted line expresses that there is a 
self-crossing, say $P$ of $\hat T$ on the core $\hat T'$ such that $P$ is on one end of the dotted line and $s(\hat T,P)$ has an intersection with a part of $\hat T$ where the other end is on. Up to symmetry and flyping Fig. \ref{lemma4proof2} illustrates all cases.

\begin{figure}[htbp]
      \begin{center}
\scalebox{0.65}{\includegraphics*{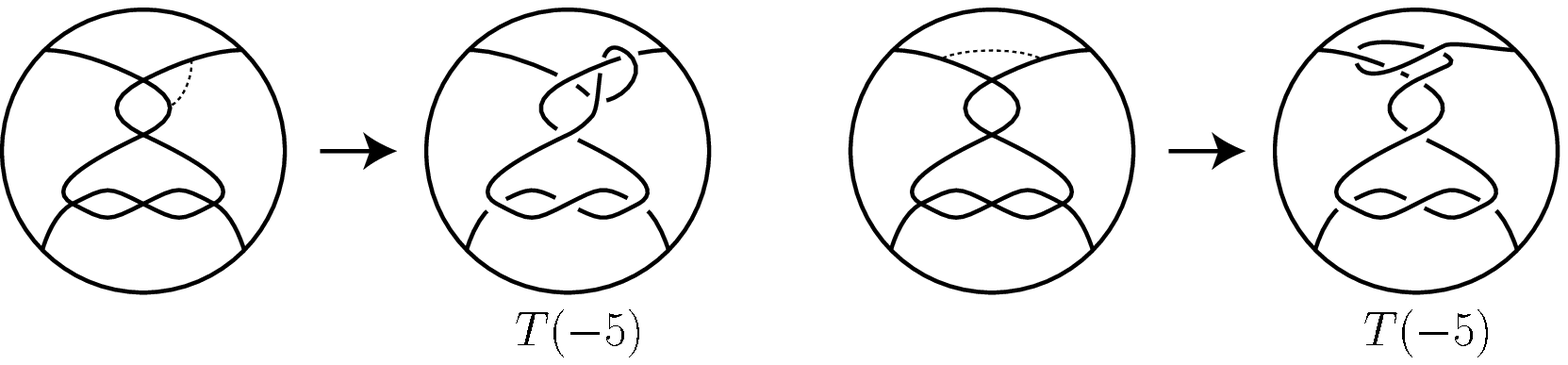}}
      \end{center}
   \caption{}
  \label{lemma4proof2}
\end{figure} 

%

\item[Case 2.1]
$\hat T'$ has three or more mixed crossings.

If $\hat T$ is not equal to (a flype of) $\hat T_3$ then using Lemma \ref{hook-tangle} and Lemma \ref{vertical-trefoil-tangle-lemma} we
have the tangle $T(1/2,-3)$, $R(T(3,1,3))$ or $T_1$ as illustrated in Fig. \ref{lemma4proof3}.

\begin{figure}[htbp]
      \begin{center}
\scalebox{0.65}{\includegraphics*{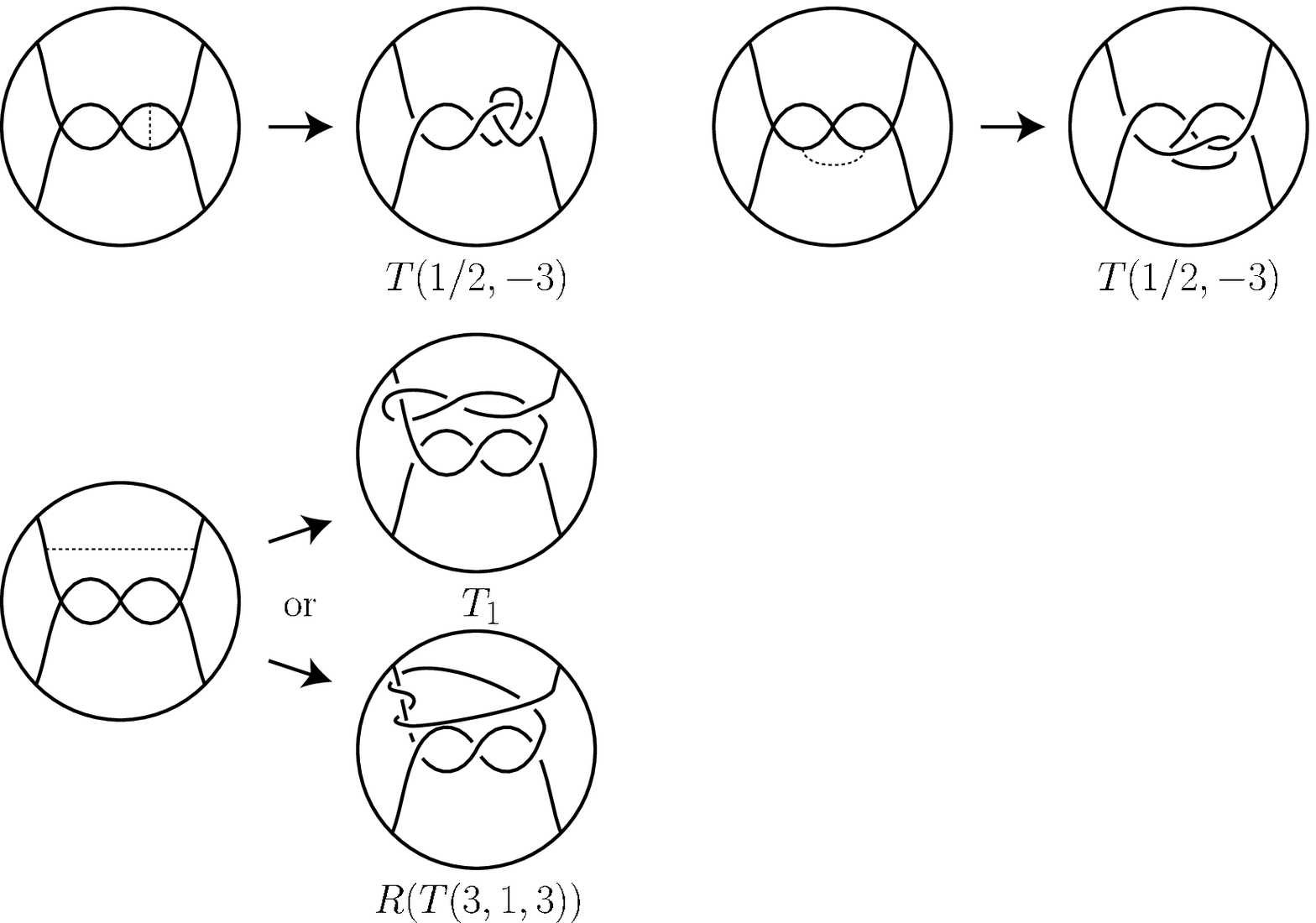}}
      \end{center}
   \caption{}
  \label{lemma4proof3}
\end{figure} 

%

\item[Case 2.2.]
$\hat T'$ has just one mixed crossing.

First suppose that there is a self-crossing $P$ on the core $\hat T'$ such that $r(\hat T,P)$ has just one mixed crossing.
Then by we apply Lemma \ref{lemma1} to a certain subtangle projection of $\hat T$ and have either $\hat T$ is a projection of $T(-5)$, $T(1/2,-3)$, $R(T(3,1,3))$ or $T_1$, or $\hat T$ is equal to (a flype of) $\hat T_4$, $\hat T_5$ or $\hat T_6$. See Fig. \ref{lemma4proof4}.

\begin{figure}[htbp]
      \begin{center}
\scalebox{0.65}{\includegraphics*{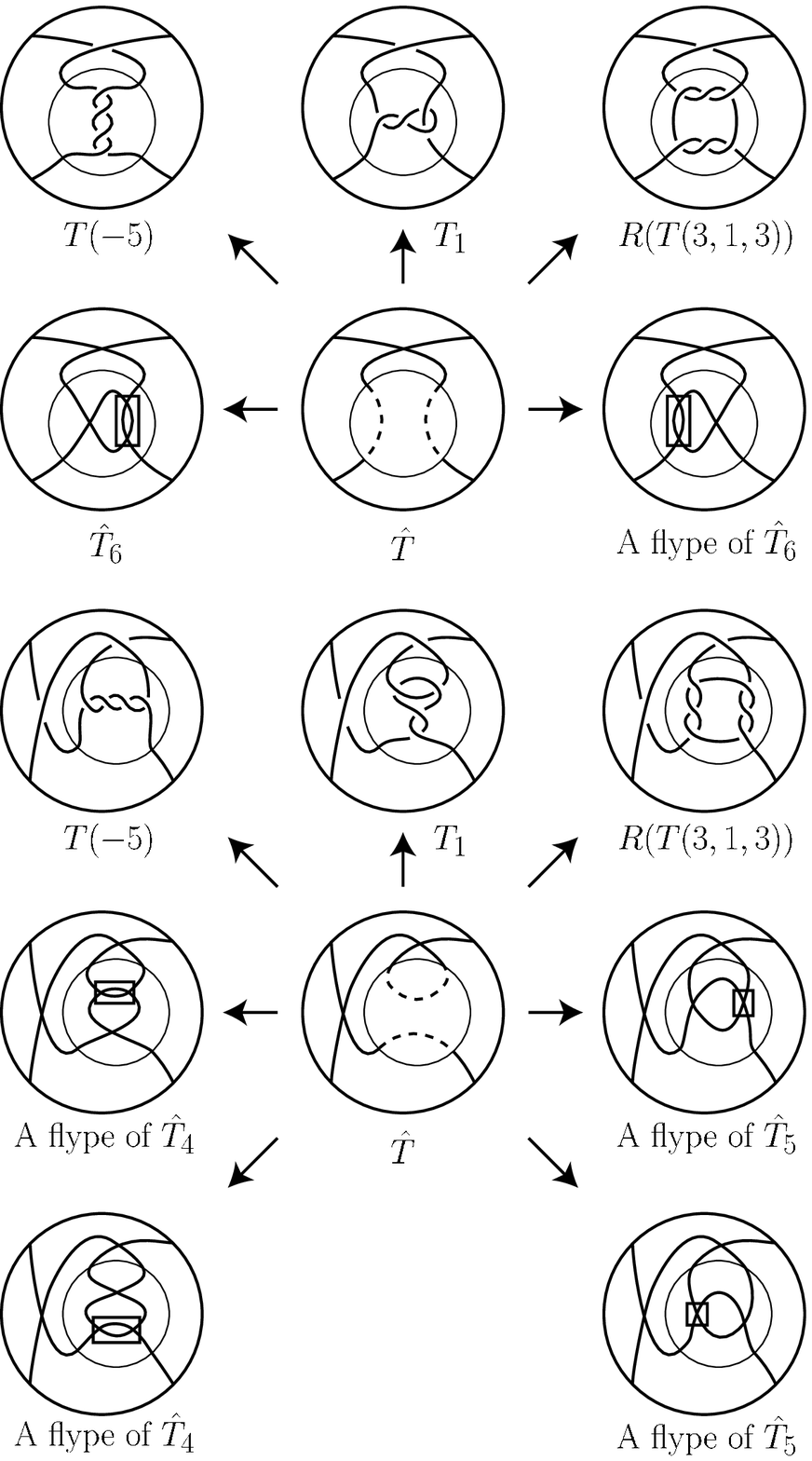}}
      \end{center}
   \caption{}
  \label{lemma4proof4}
\end{figure} 

%

Next suppose that for any self-crossing $P$ on the core $\hat T'$ $r(\hat T,P)$ has three or more mixed crossings.
Then we take self-crossings $P_1,P_2,\cdots,P_n$ on the core $\hat T'$ such that $r(r(\cdots r(\hat T,P_1),P_2),\cdots,P_{n-1})$ has three or more mixed crossings and $r(r(\cdots r(\hat T,P_1),P_2),\cdots,P_n)$ has just one mixed crossing. Set $\hat S_0=r(r(\cdots r(\hat T,P_1),P_2),\cdots,P_{n-2})$ and $\hat S=r(r(\cdots r(\hat T,P_1),P_2),\cdots,P_{n-1})$.
Then we have that $\hat S=r(\hat S_0,P_{n-1})$ is equal to (a flype of) $\hat T_{4}$, $\hat T_{5}$ or $\hat T_{6}$. 
Then by applying Lemma \ref{hook-tangle} we have the tangle $T(-5)$, $T(1/2,-3)$, or $T_1$ as typically illustrated in Fig. \ref{lemma4proof5},  \ref{lemma4proof6} and \ref{lemma4proof7}.
Note that in Fig. \ref{lemma4proof5}, \ref{lemma4proof6} and \ref{lemma4proof7} a dotted line express the situation that $P_{n-1}$ is at one end and $s(\hat S_0,P_{n-1})$ has intersection with the part of $\hat S$ where the other end is on. All typical cases up to flyping are illustrated in Fig. \ref{lemma4proof5}, \ref{lemma4proof6} and \ref{lemma4proof7} and other cases are entirely analogous.
This completes the proof.
\end{enumerate}
\end{proof}

\begin{figure}[htbp]
      \begin{center}
\scalebox{0.65}{\includegraphics*{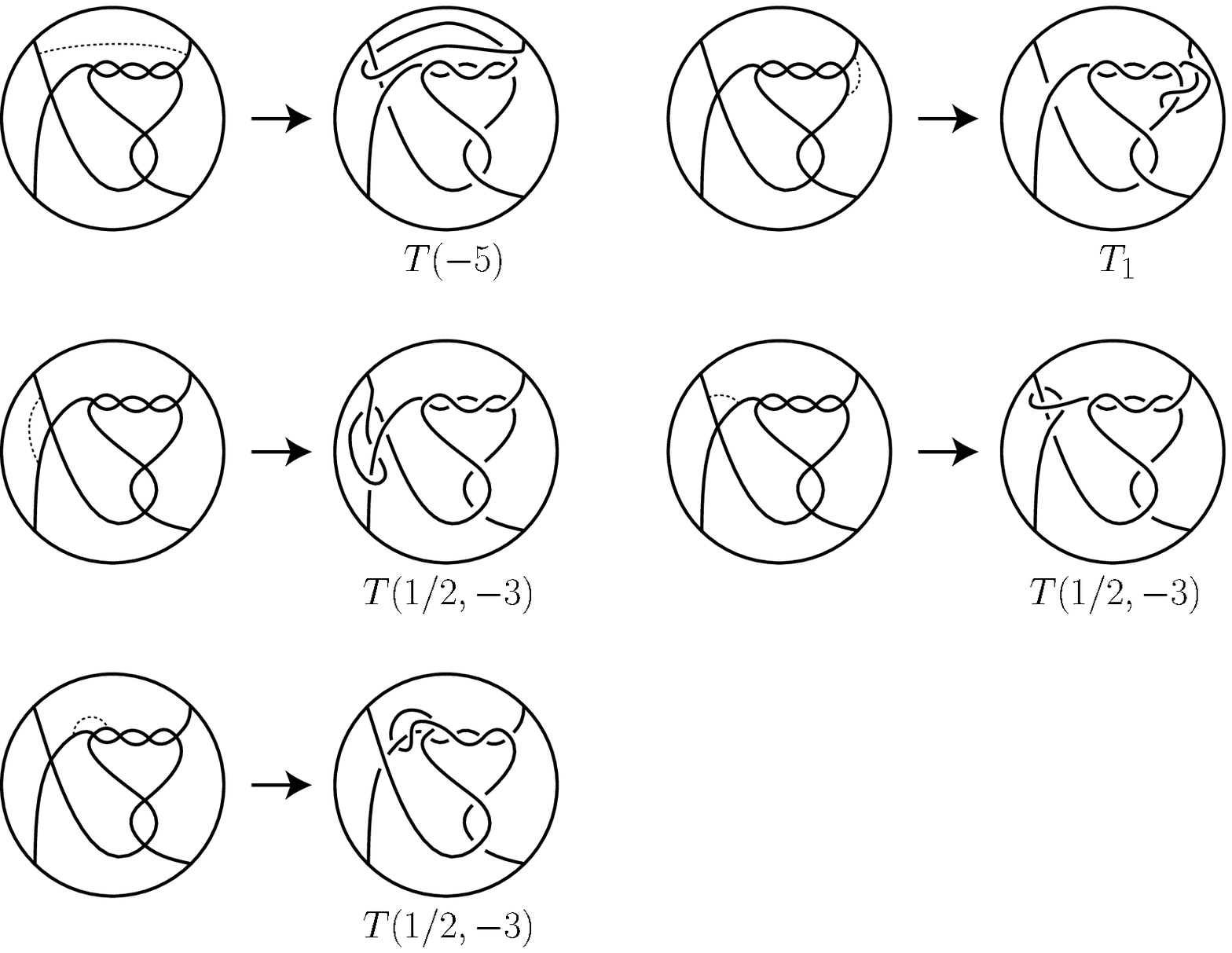}}
      \end{center}
   \caption{}
  \label{lemma4proof5}
\end{figure} 

%

%
\begin{figure}[htbp]
      \begin{center}
\scalebox{0.65}{\includegraphics*{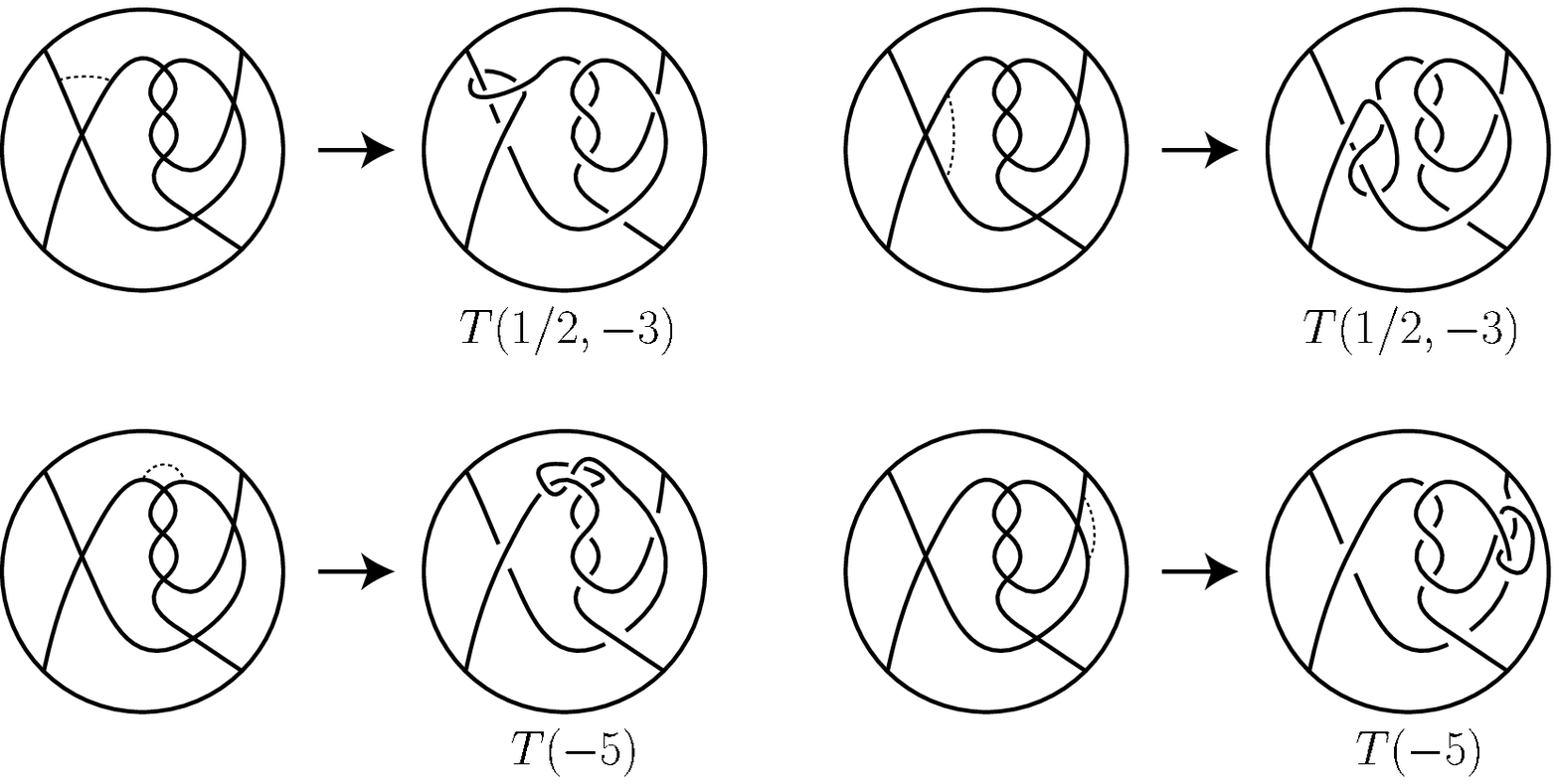}}
      \end{center}
   \caption{}
  \label{lemma4proof6}
\end{figure} 

%

%
\begin{figure}[htbp]
      \begin{center}
\scalebox{0.65}{\includegraphics*{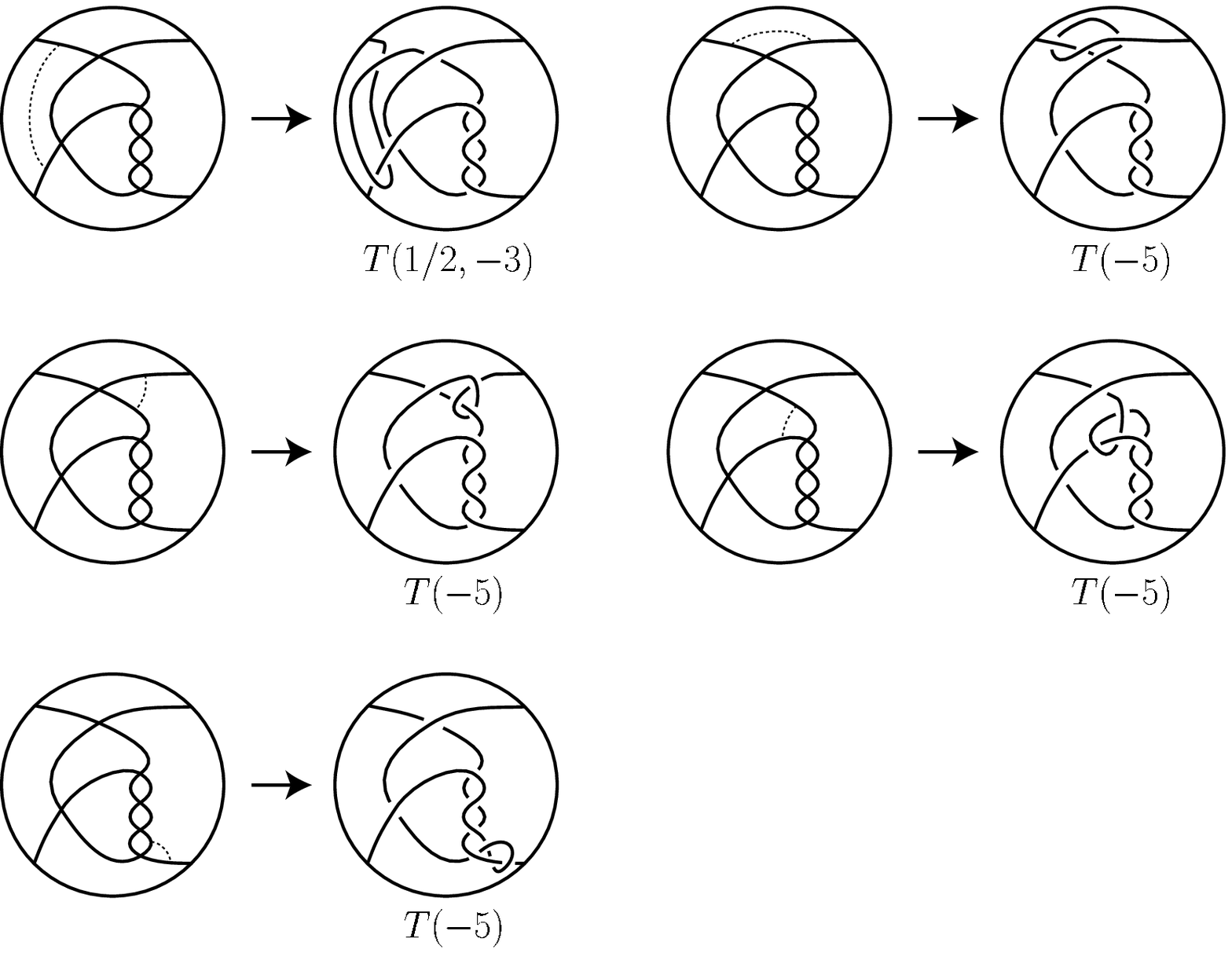}}
      \end{center}
   \caption{}
  \label{lemma4proof7}
\end{figure} 

%

\begin{Lemma}\label{sublemma4}
Let $\tilde T$ be an almost positive 2-string tangle diagram with X-connection.
Suppose that $\tilde T$ has no self-crossings and $\tilde T$ is not greater than or 
equal to the tangle $T(-3)$ nor $T(1/3)$ (Fig. \ref{lemma5tangles}).
Then the underlying projection $\hat T$ of $\tilde T$ is equal to $\hat T(1)$, $\hat T(3)$ or $\hat T(1/3)$.
\end{Lemma}

\begin{proof}
We give the  proof by the contradiction. Suppose that there is an almost positive 2-string 
tangle diagram $\tilde T$ with X-connection without self-crossings that is not greater than 
or equal to the tangle $T(-3)$ nor $T(1/3)$ that contains five or more mixed crossings.
We take such $\tilde T$ with minimal mixed crossings.
It is clear that a tangle with X-connection without self-crossings with three or more mixed crossings has a 2-gon.
Thus we have that $\tilde T$ has a 2-gon.
If one of the two crossings of the 2-gon is the negative crossing then we have a positive diagram $\tilde T_0$ from $\tilde T$ by applying the second Reidemeister move. Then $\tilde T_0$ represents the same tangle as $\tilde T$ and has three or more mixed crossings. Then by Lemma \ref{one-three-lemma} we have that $\tilde T_0$ is greater than or equal to $T(-3)$ or $T(1/3)$. Thus we have that this case does not happen.
Therefore we have that the crossings of the 2-gon are both positive.
Let $\tilde T_1$ be a diagram obtained from $\tilde T$ by replacing that 2-gon by a pair of parallel arcs. 
It is clear that $\tilde T$ is greater than or equal to $\tilde T_1$. Therefore we have that $\tilde T_1$ is not greater than or equal to the tangle $T(-3)$ nor $T(1/3)$. By the minimality of $\tilde T$ we have that $\tilde T_1$ has just three mixed crossings.
Therefore we have that the underlying projection of $\tilde T_1$ is $\hat T(3)$ or $\hat T(1/3)$ and by considering the position of the negative crossing we have that $\tilde T_1$ is one of the six tangle diagrams whose underlying projection is $\hat T(3)$ or $\hat T(1/3)$.
But then it is easy to check that any tangle diagram obtained from $\tilde T_1$ by replacing a pair of parallel arcs by a 2-gon with positive crossings is a diagram of $T(-3)$ of $T(1/3)$. This is a contradiction.
\end{proof}

\begin{Lemma}\label{lemma5}
Let $\tilde T$ be a prime R2-reduced almost positive 2-string tangle diagram with X-connection.
Suppose that the negative crossing of $\tilde T$ is a mixed crossing and $\tilde T$ is not
greater than or equal to any of the tangles in Fig. \ref{lemma5tangles}.
Then $\tilde T$ is one of the diagrams illustrated in Fig. \ref{lemma5diagrams}.
\end{Lemma}

\begin{figure}[htbp]
      \begin{center}
\scalebox{0.65}{\includegraphics*{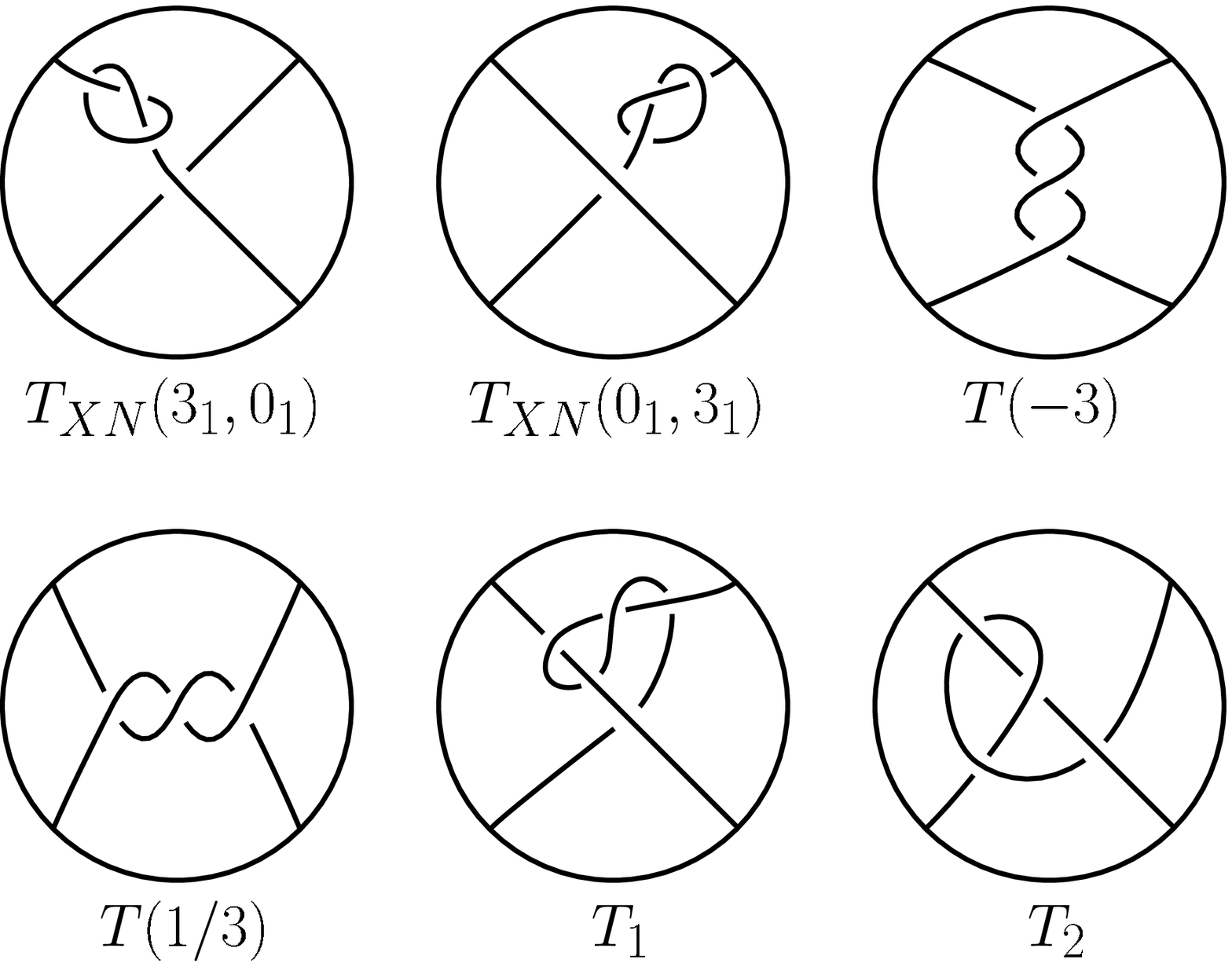}}
      \end{center}
   \caption{}
  \label{lemma5tangles}
\end{figure} 

%

%
\begin{figure}[htbp]
      \begin{center}
\scalebox{0.65}{\includegraphics*{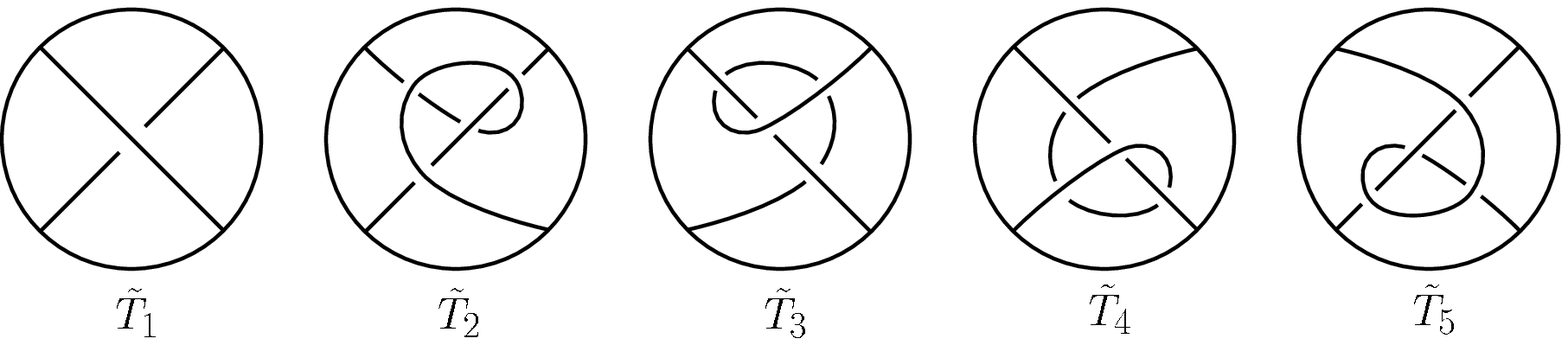}}
      \end{center}
   \caption{}
  \label{lemma5diagrams}
\end{figure} 

%

\begin{proof}
We consider the following cases.

\begin{enumerate}
\item[Case 1.]
There is a self-crossing $P$ such that $r(\tilde T,P)$ is a positive diagram.
Suppose that there are three or more mixed crossings of $r(\tilde T,P)$.
Then by Lemma \ref{one-three-lemma} we have $T(1/3)$ or $T(-3)$.
Thus we have that $r(\tilde T,P)$ has just one mixed crossing.
Then by Lemma \ref{almost-positive-tangle-lemma} we have the tangle $T(-3)$ or $T(1/3)$ as illustrated in Fig. \ref{lemma5proof1}.

\begin{figure}[htbp]
      \begin{center}
\scalebox{0.65}{\includegraphics*{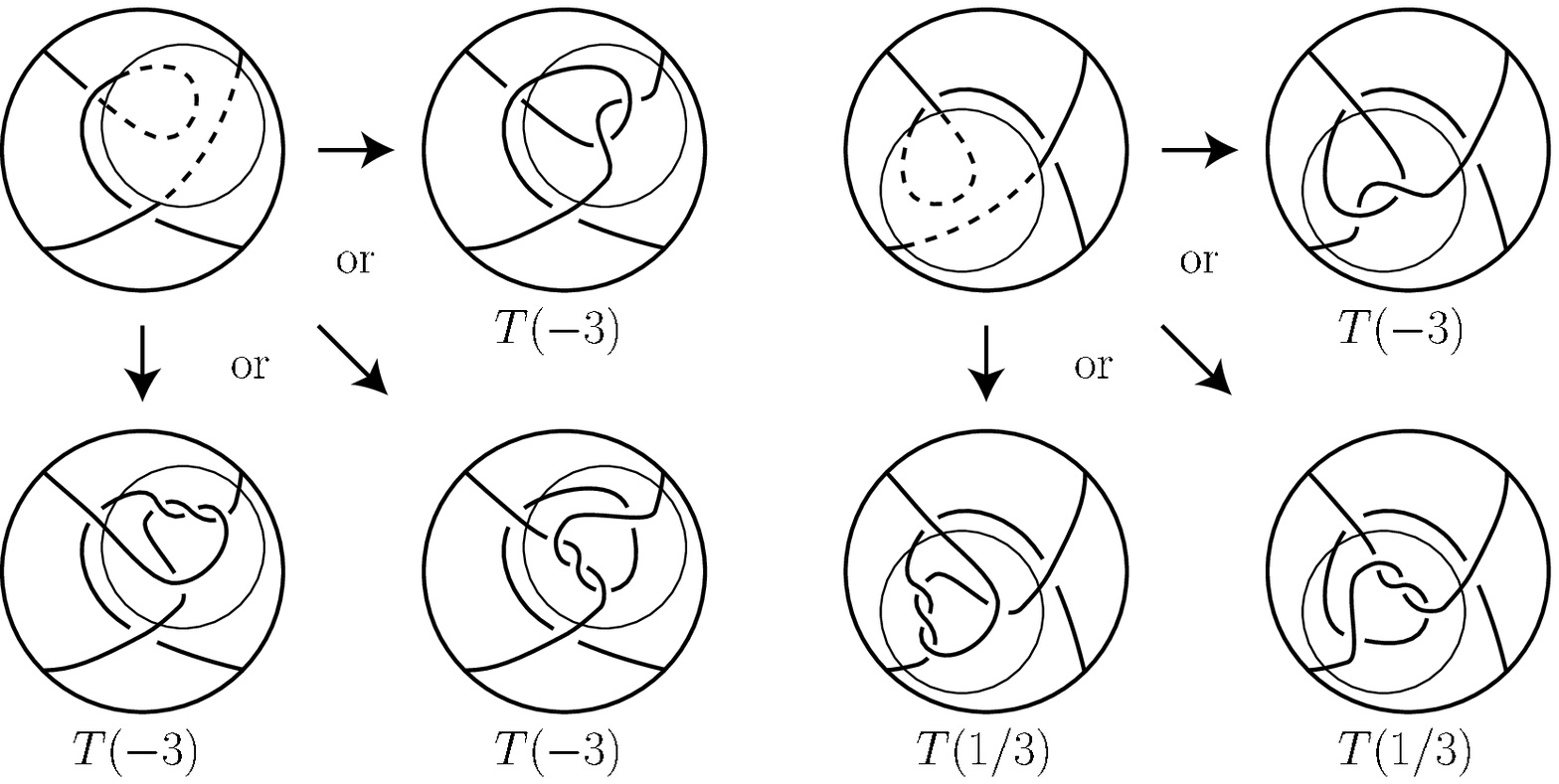}}
      \end{center}
   \caption{}
  \label{lemma5proof1}
\end{figure} 

%

\item[Case 2.]
$\tilde T$ has no self-crossings.

In this case we have $\tilde T$ is equal to $\tilde T_1$ by Lemma \ref{sublemma4}.

\item[Case 3.]
Otherwise.

By Case 2 we have that the spine $\tilde T'$ of $\tilde T$ is equal to $\tilde T_1$.
Then by applying Lemma \ref{vertical-trefoil-tangle-lemma} or Lemma \ref{hook-tangle} as illustrated in Fig. \ref{lemma5proof2} we have the result. 
This completes the proof.
\end{enumerate}
\end{proof}

\begin{figure}[htbp]
      \begin{center}
\scalebox{0.65}{\includegraphics*{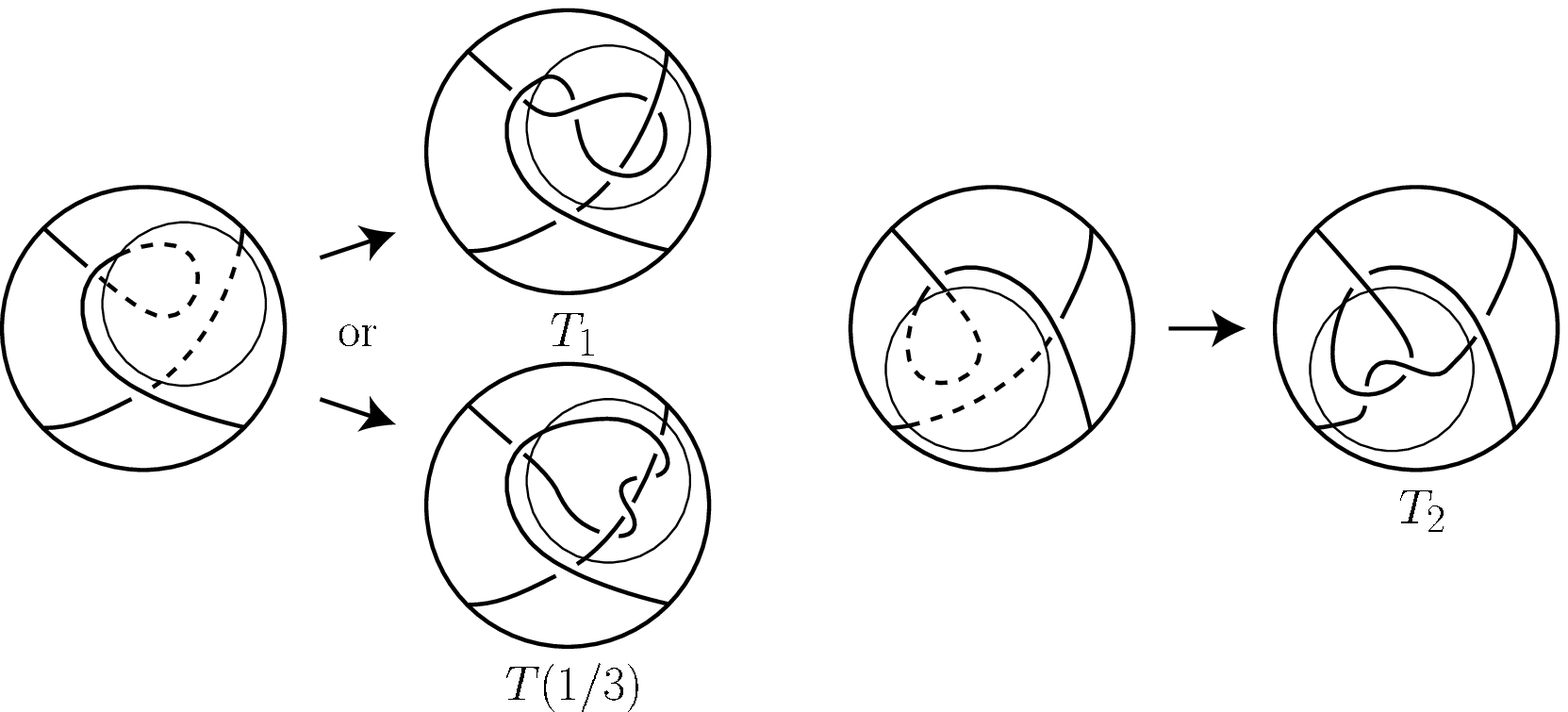}}
      \end{center}
   \caption{}
  \label{lemma5proof2}
\end{figure} 

%


\begin{Lemma}\label{lemma6}
Let $\tilde L=\tilde\ell_1\cup\tilde\ell_2$ be a 2-almost positive link diagram of a 2-component link $L$ such that the negative crossings are mixed crossings. Suppose that $\tilde L$ has four or more mixed crossings. Then $L$ is greater than or equal to the 2-component trivial link.
\end{Lemma}

\begin{proof}
Let $M$ and $N$ be the negative mixed crossings. Let $\tilde a$ be the arc on $\tilde\ell_1$ from $M$ to $N$ and $\tilde b$ the arc on $\tilde\ell_1$ from $N$ to $M$. Let $\tilde c$ be the arc on $\tilde\ell_2$ from $M$ to $N$ and $\tilde d$ the arc on $\tilde\ell_2$ from $N$ to $M$. Suppose that there is a self-crossing of $\tilde\ell_1$ 
between $\tilde a$ and $\tilde b$. Then by Lemma \ref{diagram-reducing-lemma} we have the result. Similarly if there is a 
self-crossing of $\tilde\ell_2$ between $\tilde c$ and $\tilde d$ then we have the result. Therefore we may suppose that there are no such self-crossings. We may suppose without loss of generality that $\tilde a$ and $\tilde c$ has a positive mixed crossing. Using Lemma \ref{diagram-reducing-lemma} we may suppose that $\tilde d$ is a simple arc. First suppose that $\tilde a$ has a positive mixed crossing on $\tilde d$. Then we have the trivial link by changing the crossings without changing $M$ and $N$ as illustrated in Fig \ref{lemma6proof1}. Next suppose that $\tilde a$ has no mixed crossings on $\tilde d$ except $M$ and $N$. Then we change the crossings of $\tilde L$ so that $\tilde b$ is parallel to $\tilde d$ and apply Lemma \ref{hook-tangle} on the tangle diagram that is complementary to a neighborhood of $\tilde d$ and we have the trivial link. See Fig. \ref{lemma6proof2}.
\end{proof}

\begin{figure}[htbp]
      \begin{center}
\scalebox{0.65}{\includegraphics*{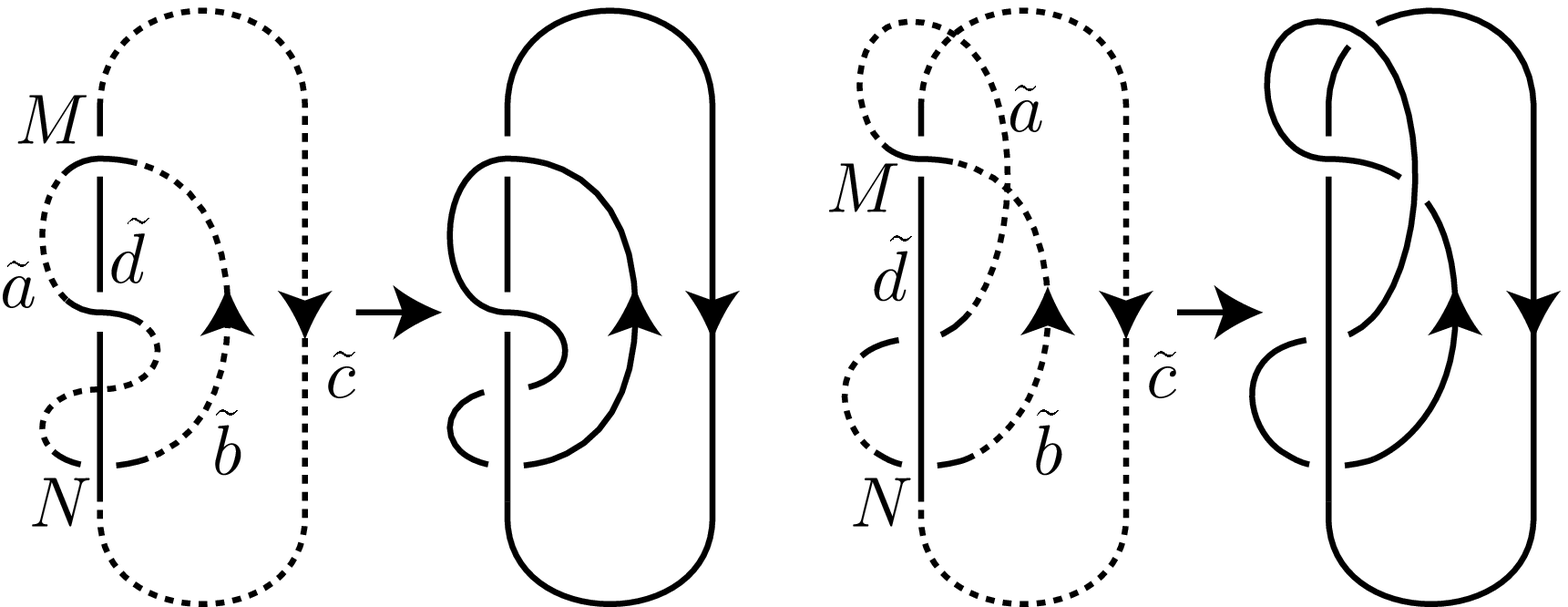}}
      \end{center}
   \caption{}
  \label{lemma6proof1}
\end{figure} 

%

%
\begin{figure}[htbp]
      \begin{center}
\scalebox{0.65}{\includegraphics*{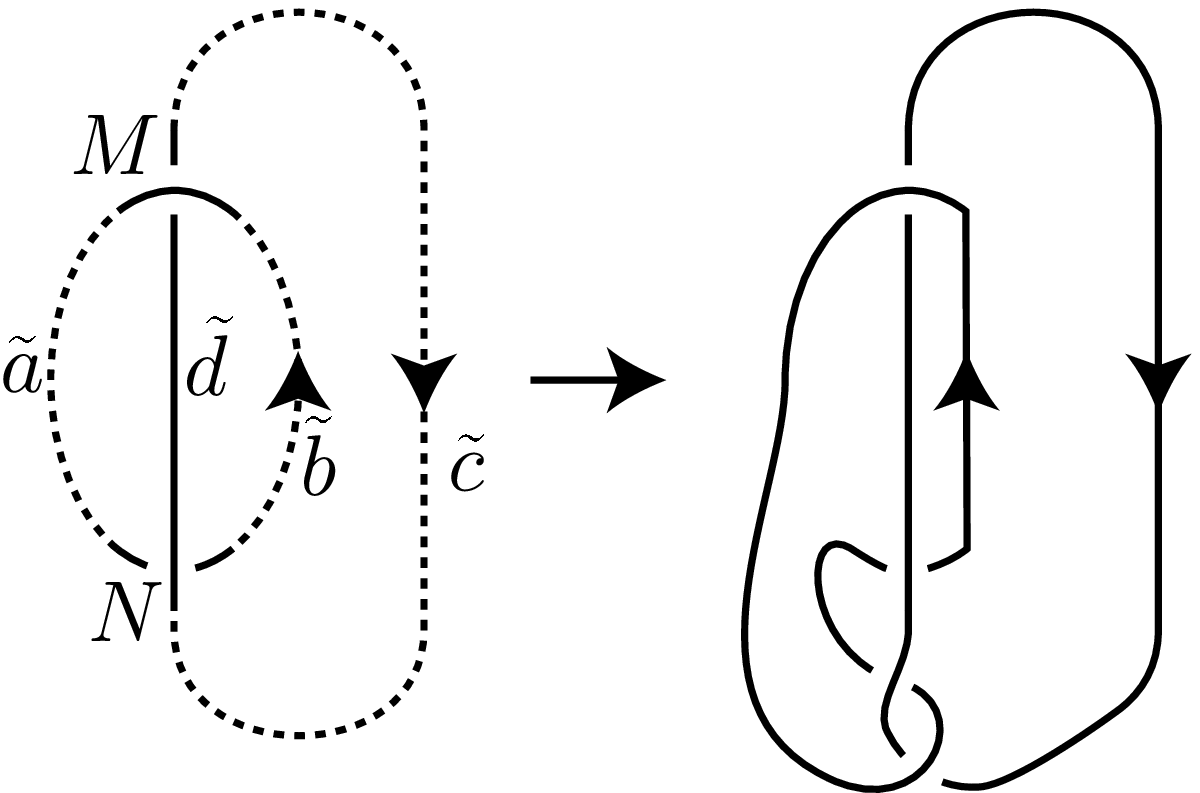}}
      \end{center}
   \caption{}
  \label{lemma6proof2}
\end{figure} 

%

\begin{Lemma}\label{lemma7}
Let $\tilde L=\tilde\ell_1\cup\tilde\ell_2\cup\tilde\ell_3$ be a 2-almost positive diagram of a 3-component link $L$.
Suppose that $\tilde\ell_1\cap\tilde\ell_2$ has the two negative mixed crossings and has no other mixed crossing.
Suppose that $\tilde\ell_3$ has mixed crossings with both of $\tilde\ell_1$ and $\tilde\ell_2$.
Then $L$ is greater than or equal to the link of Fig. \ref{6_2-Whitehead-3-comp} (d).
\end{Lemma}

\begin{proof}
If we find a part as illustrated in Fig. \ref{lemma7proof} (a) in $\tilde L$ then we have the link of Fig. \ref{6_2-Whitehead-3-comp} (d).
If there are no such parts then we have one of the parts illustrated in Fig. \ref{lemma7proof} (b). Then we have the link of Fig. \ref{6_2-Whitehead-3-comp} (d) or the link $L_0$ in Fig. \ref{lemma7proof} (b) that is greater than or equal to the link of Fig. \ref{6_2-Whitehead-3-comp} (d) as illustrated in Fig. \ref{lemma7proof} (c).
\end{proof}

\begin{figure}[htbp]
      \begin{center}
\scalebox{0.65}{\includegraphics*{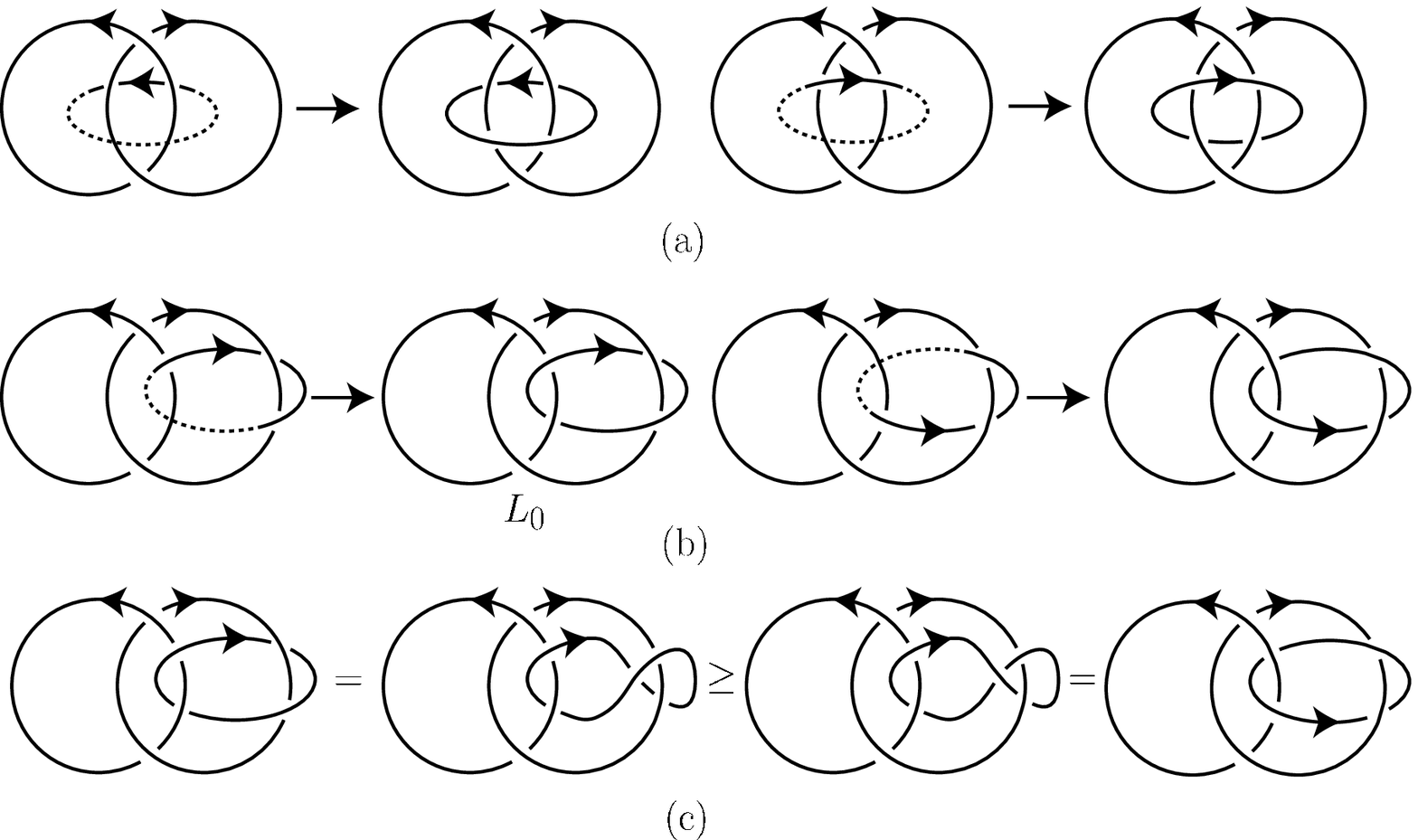}}
      \end{center}
   \caption{}
  \label{lemma7proof}
\end{figure} 

%


\begin{Theorem}\label{2-almost-positive-theorem}
Let $\tilde L$ be a 2-almost positive diagram of a link $L$.
Then

\begin{enumerate}
\item[(1)]
$L\geq$ right-handed trefoil knot (plus trivial components) or

\item[(2)]
$L\geq$ $6_2$ (Fig. \ref{6_2-Whitehead-3-comp} (a)) (plus trivial components) or

\item[(3)]
$L\geq$ right-handed Hopf link (plus trivial components) or

\item[(4)]
$L\geq$ disjoint or connected sum of right-handed trefoil knot and left handed Hopf link
(plus trivial components) or

\item[(5)]
$L\geq$ Whitehead link (Fig. \ref{6_2-Whitehead-3-comp} (b)) (plus trivial components) or

\item[(6)]
$L\geq$ disjoint and/or connected sum of two right-handed Hopf links and a left-handed
Hopf link (plus trivial components) or

\item[(7)]
$L\geq$ disjoint or connected sum of $(2,4)$-torus link (Fig. \ref{6_2-Whitehead-3-comp} (c)) and a left-handed Hopf link (plus trivial components) or

\item[(8)]
$L\geq$ the link of Fig. \ref{6_2-Whitehead-3-comp} (d) (plus trivial components) or

\item[(9)]
$\tilde L$ is a diagram obtained from the diagrams in Fig. \ref{almost-positive-trivial} and 
Fig. \ref{2ap-trivial} and their 
reversals on ${\mathbb S}^2=R^2\cup \infty$ (reflection in $y$ axis)
 by performing, possibly, diagram-disjoint sum operation, diagram-connected sum operation, and first and second 
Reidemeister moves which increase the number of crossings, or

\item[(10)]
$\tilde L$ is a diagram obtained from $\tilde L_1$, $\tilde L_2$ or $\tilde L_3$ of Fig. \ref{2-almost-positive-diagrams} 
by performing, possibly, first Reidemeister moves which increase the number of crossings and adding trivial circles, or

\item[(11)]
$\tilde L$ is a diagram obtained from $\tilde L_3$ of Fig. \ref{2-almost-positive-diagrams} and $\tilde T_1$ or $\tilde T_2$ of Fig. \ref{torus-link-diagrams} 
by, possibly,  performing first Reidemeister moves which increase the number of crossings, 
diagram-disjoint sum operation or diagram-connected sum operation (plus almost trivial components).
\end{enumerate}
\end{Theorem}

\begin{figure}[htbp]
      \begin{center}
\scalebox{0.65}{\includegraphics*{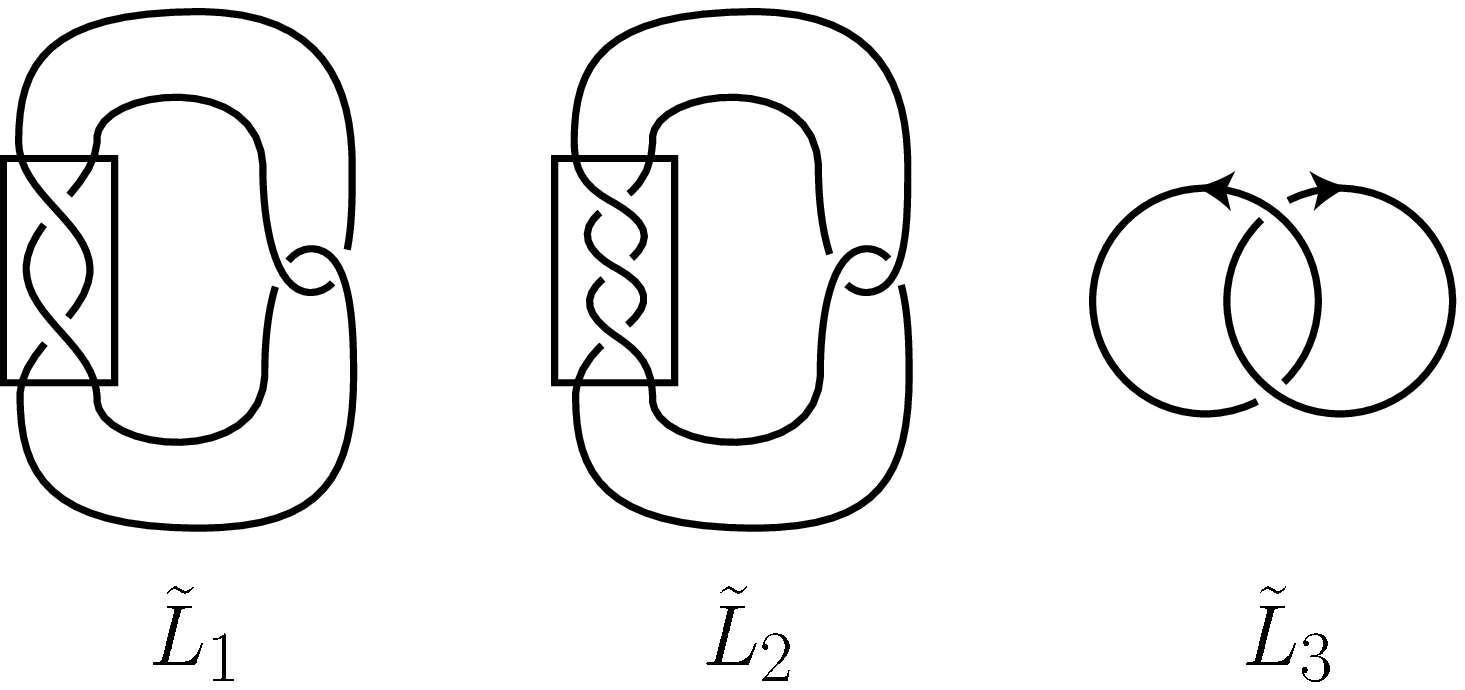}}
      \end{center}
   \caption{}
  \label{2-almost-positive-diagrams}
\end{figure} 

%

%
\begin{figure}[htbp]
      \begin{center}
\scalebox{0.65}{\includegraphics*{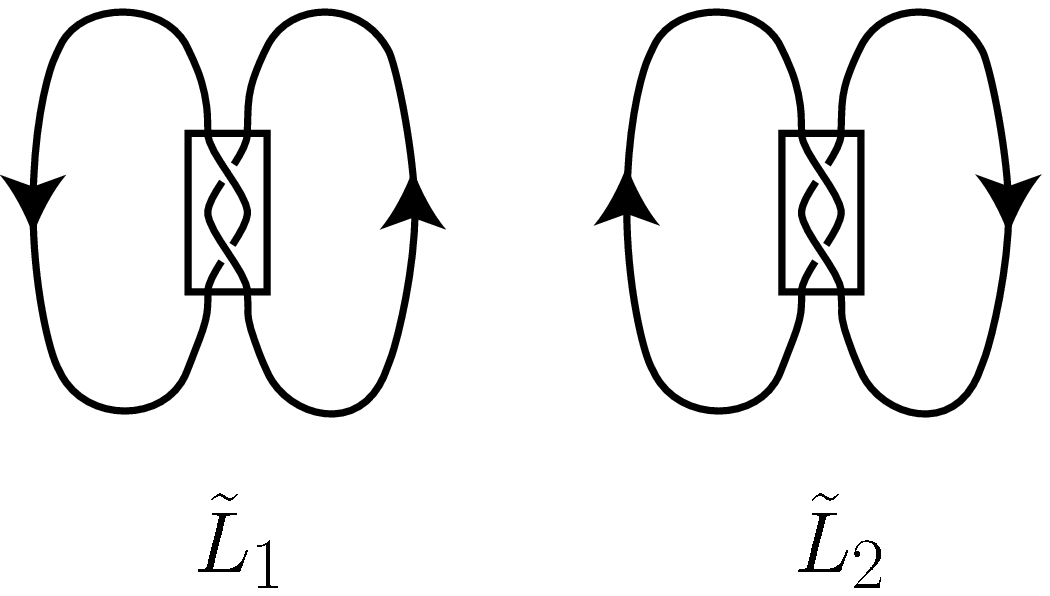}}
      \end{center}
   \caption{}
  \label{torus-link-diagrams}
\end{figure} 

%

\begin{proof}
We decompose $\tilde L$ into prime factors. If the two negative crossings
belong to different prime factors then the result follows by Theorem \ref{almost-positive-theorem}.
Therefore it is sufficient to consider the case that $\tilde L$ is prime and has just two negative crossings, say $M$ and $N$.
Then the case that $\tilde L$ is not prime immediately follows.

\begin{enumerate}
\item[Case 1.]
$M$ and $N$ are self-crossings of a component, say $\tilde\ell_{1}$ of $\tilde L$.

We first note that by applying Lemma \ref{diagram-reducing-lemma} we have that $\tilde\ell_1$ is greater than or equal to a trivial knot.
Therefore if $\tilde L$ has other components then we have $L\geq$ right-handed Hopf link.
Hence we may suppose $\tilde L=\tilde\ell_{1}$.
We have the following three cases according to the position of $M$ and $N$ as illustrated in Fig. \ref{2-ap-theorem-proof1}

\begin{figure}[htbp]
      \begin{center}
\scalebox{0.61}{\includegraphics*{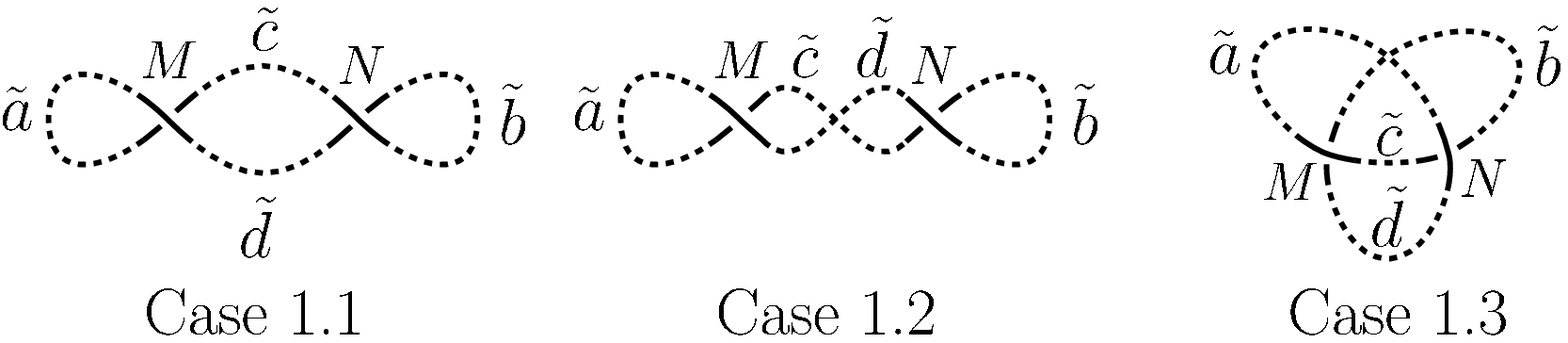}}
      \end{center}
   \caption{}
  \label{2-ap-theorem-proof1}
\end{figure} 

%

First suppose $(\tilde a\cup\tilde b)\cap (\tilde c\cup\tilde d)\neq
\emptyset$. 
Let $P$ be a crossing of $(\tilde a\cup\tilde b)\cap (\tilde c\cup\tilde d)$. Let $\tilde T$ be the complementary tangle diagram of $\tilde L$ at $P$. 
Then $\tilde T$ is a prime R2-reduced 2-string tangle diagram with vertical connection with just two
negative crossings, one is a mixed crossing and the other is a self-crossing.
Note that the $X_+$-closures of the tangles $T(-2),T(3_1,0_1),T(0_1,3_1)$ and $T_2$ in Fig. \ref{tangles2} are right-handed trefoil knots and the $X_+$-closures of the tangles $T_1,T_3$ and $T_4$ in Fig. \ref{tangles2} are $6_2$. The $X_+$-closures of the tangles $T_5$ and $T_6$ in Fig. \ref{tangles2} are greater than or equal to the $X_-$-closures of the tangles $T_5$ and $T_6$ in Fig. \ref{tangles2} respectively and these knots are $6_2$.
Note also that the $X_+$-closures of the tangle diagrams $\tilde T_1,\tilde T_{2++},\tilde T_{2+-},\tilde T_{2-+},\tilde T_{2--},\tilde T_{3+}$ and $\tilde T_{3-}$ in Fig. \ref{diagrams2} are not R2-reduced. The $X_+$-closures of the tangle diagrams $\tilde T_{4+}$ and $\tilde T_{5-}$ in Fig. \ref{diagrams2} are equal to the link diagram $\tilde T_2$ of Fig. \ref{2ap-trivial}. The $X_+$-closures of the tangle diagrams $\tilde T_{4-}$ and $\tilde T_{5+}$ in Fig. \ref{diagrams2} are equal to the reversal of the link diagram $\tilde T_2$ of Fig. \ref{2ap-trivial}. The $X_+$-closures of the tangle diagrams $\tilde T_{6+}$ and $\tilde T_{7-}$ in Fig. \ref{diagrams2} are equal to the link diagram $\tilde T_1$ of Fig. \ref{2ap-trivial}. The $X_+$-closures of the tangle diagrams $\tilde T_{6-}$ and $\tilde T_{7+}$ in Fig. \ref{diagrams2} are equal to the reversal of the link diagram $\tilde T_1$ of Fig. \ref{2ap-trivial}.
Then by Lemma \ref{lemma2} we have $L\geq$ right-handed trefoil knot or $L\geq$ $6_2$ if
and only if $\tilde L$ is not equal to the link diagrams 
$\tilde T_1$ or $\tilde T_2$ of Fig. \ref{2ap-trivial} or their reversals (reflections in $y$ axis).

Next suppose $(\tilde a\cup\tilde b)\cap(\tilde c\cup\tilde d)=\emptyset$. Choose disjoint disks 
$D_1$ and $D_2$ on ${\mathbb S}^2$ such that $D_1$ contains $\tilde a$ and 
$\tilde b$ and does not contain $M$ and $N$, $D_{2}$ contains $\tilde c$ and $\tilde d$ and
does not contain $M$ and $N$. See Fig. \ref{2-ap-theorem-proof1-2}.
In Case 1.1 we apply Lemma \ref{lemma3} to $D_{1}$ and Lemma \ref{vertical-trefoil-tangle-lemma} to $D_{2}$. 
Note that if we have the tangle $T(3_1,0_1)$ or $T(0_1,3_1)$ on $D_1$ then by trivializing the tangle on $D_2$ we have the right-handed trefoil knot.
If we have the tangle $R(T(-3,-1))$ on $D_1$ then by trivializing the tangle on $D_2$ we have $6_2$.
If we have the tangle $T(-1,-2)$ on $D_2$ then by trivializing the tangle on $D_1$ we have the right-handed trefoil knot.
Then taking the fact that $\tilde T$ is R2 reduced into account we have that $L\geq$ right-handed trefoil knot or $L\geq$ $6_2$ if and only if $\tilde L$ is not $\tilde T_1$ in Fig. \ref{2-almost-positive-diagrams}. 
In Case 1.2 we have the right-handed trefoil knot by applying Lemma \ref{one-three-lemma} to $D_2$ and trivializing the tangle on $D_1$.

In Case 1.3 we apply Lemma \ref{lemma4} to $D_1$ and Lemma \ref{lemma1} to $D_2$. 
If we have one of the tangles in Fig. \ref{lemma4tangles} except $T(1/2,-3)$ on $D_1$ then by trivializing the tangle on $D_2$ we have the right-handed trefoil knot. If we have the tangle $T(1/2,-3)$ on $D_1$ then by trivializing the tangle on $D_2$ we have $6_2$. 
If we have one of the tangles in Fig. \ref{tangles1} on $D_2$ then by taking the positive tangle with one crossing on $D_1$ we have the right-handed trefoil knot. Therefore, because $\tilde L$ is R2-reduced, 
the diagram on $D_1$ is a positive diagram whose underlying projection is one of the projections 
$\hat T_1,\hat T_3,\hat T_4$ and $\hat T_5$ in Fig. \ref{lemma4projections} or their flypes. 
Also because $\tilde L$ is R2-reduced, we have  that the diagram on $D_2$ is a positive diagram 
whose underlying projection is $\hat T(1,2n)$ or $\hat T(2n,1)$ in Fig. \ref{projections1} unless it is a trivial diagram. Suppose that it is not a trivial diagram. Then by Lemma \ref{hook-tangle} and Lemma \ref{one-three-lemma} we have the right-handed trefoil knot unless the underlying projection of the diagram on $D_1$ is $\hat T_1$ in Fig. \ref{lemma4projections}. See Fig. \ref{2-ap-theorem-proof2}. If the underlying projection of the diagram on $D_1$ is $\hat T_1$ in Fig. \ref{lemma4projections} then we either have the right-handed trefoil knot or the diagram $\tilde T_{9}$ in Fig. \ref{2ap-trivial}. See Fig. \ref{2-ap-theorem-proof3}.
Suppose that the diagram on $D_2$ is trivial. Then we have that $\tilde L$ is (a reversal of) $\tilde T_3$, $\tilde T_4$, $\tilde T_5$, $\tilde T_6$, $\tilde T_7$, $\tilde T_8$ or $\tilde T_{10}$ in Fig. \ref{2ap-trivial} or $\tilde L_2$ in Fig. \ref{2-almost-positive-diagrams}.

\begin{figure}[htbp]
      \begin{center}
\scalebox{0.61}{\includegraphics*{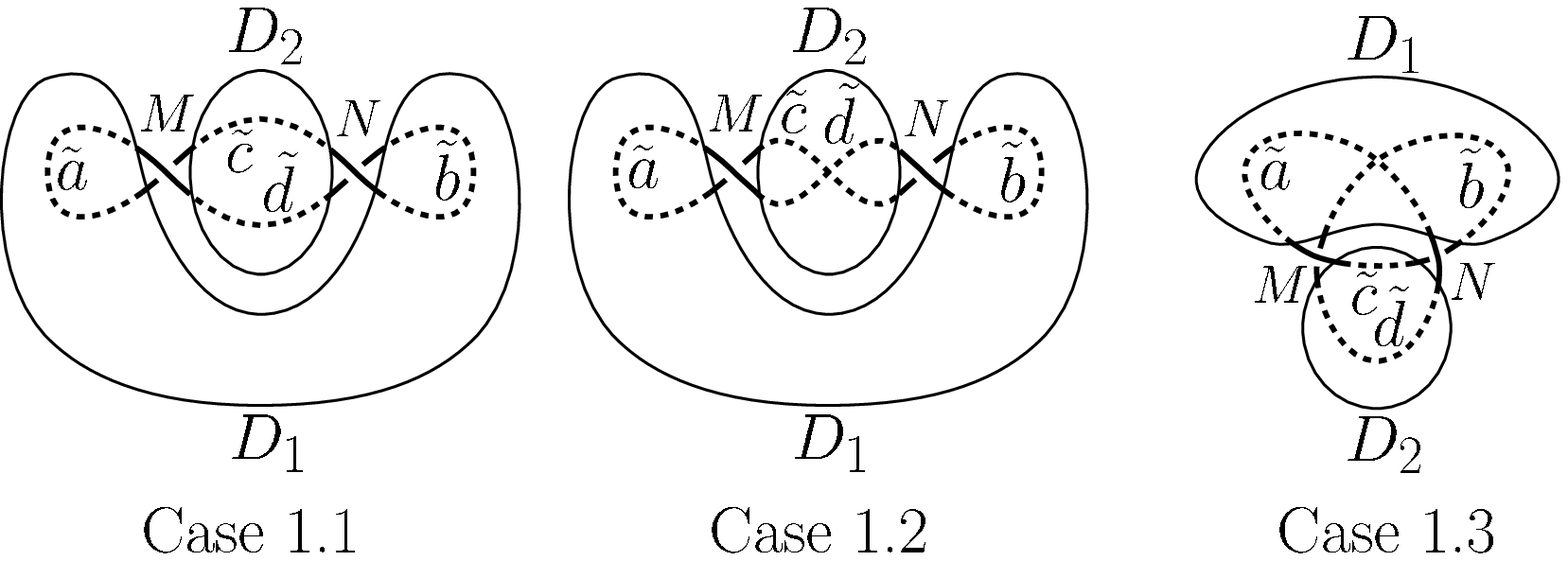}}
      \end{center}
   \caption{}
  \label{2-ap-theorem-proof1-2}
\end{figure} 

%

%
\begin{figure}[htbp]
      \begin{center}
\scalebox{0.61}{\includegraphics*{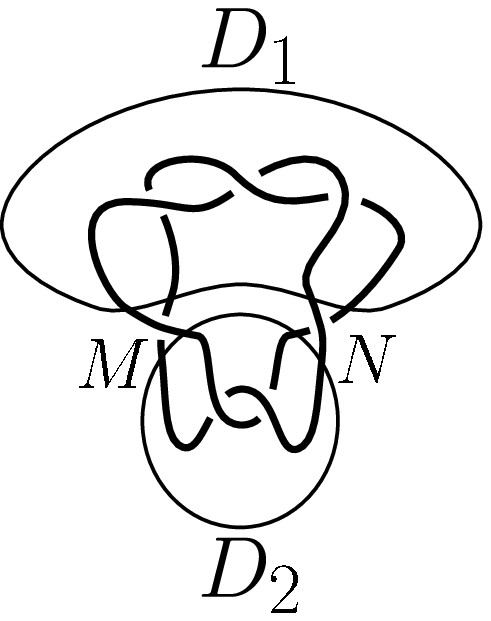}}
      \end{center}
   \caption{}
  \label{2-ap-theorem-proof2}
\end{figure} 

%

%
\begin{figure}[htbp]
      \begin{center}
\scalebox{0.61}{\includegraphics*{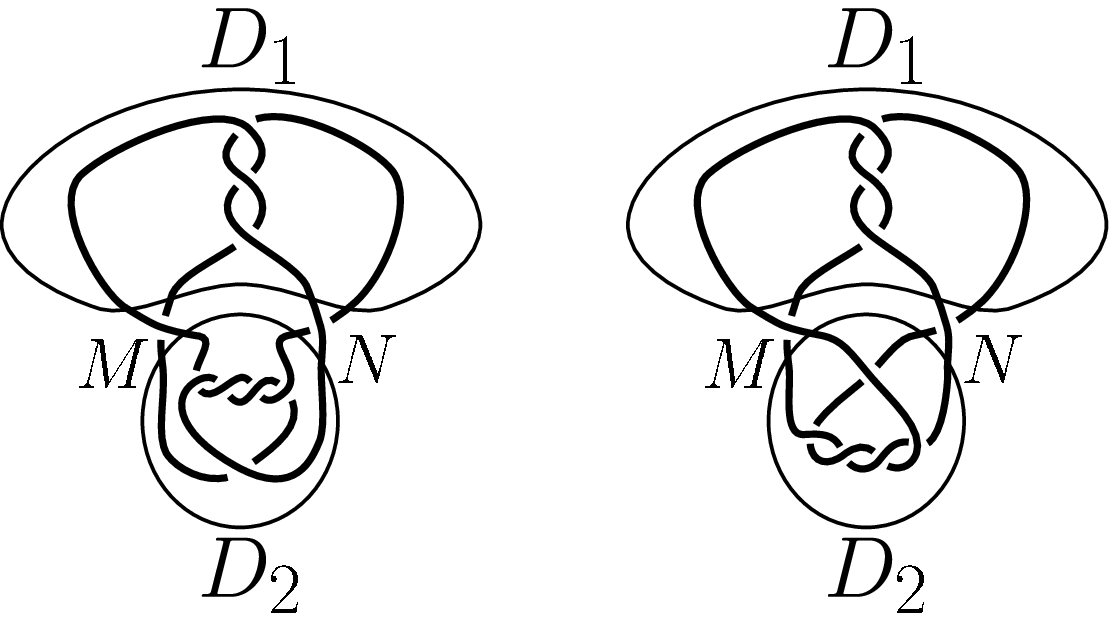}}
      \end{center}
   \caption{}
  \label{2-ap-theorem-proof3}
\end{figure} 

%

\item[Case 2.]
$M$ and $N$ are self-crossings of different components of $\tilde L$.

In this case we easily have $L\geq$ right-handed Hopf link (plus trivial components).

\item[Case 3.]
One of $M$ and $N$, say $M$, is a self-crossing of a component, say
$\tilde\ell_{1}$, and $N$ is a mixed crossing of components, say
$\tilde\ell_{i}$ and $\tilde\ell_{j}$.

Then we have either $L\geq$ right-handed Hopf link (plus trivial components) or $\tilde L=\tilde\ell_{i}\cup\tilde\ell_{j}$.
If $\tilde L=\tilde\ell_{i}\cup\tilde\ell_{j}$ then choose a small disk as illustrated in Fig. \ref{2-ap-theorem-proof4} 
and apply Lemma \ref{lemma2} to the tangle diagram on the complementary disk on ${\mathbb S}^2$. 
Note that because  $\tilde L$ is R2-reduced, we 
have no tangle diagrams in Fig. \ref{diagrams2}. 
Then we have $L\geq$ right-handed trefoil knot and an unknot or $L\geq$ right-handed Hopf link or $L\geq$ Whitehead link.

\begin{figure}[htbp]
      \begin{center}
\scalebox{0.65}{\includegraphics*{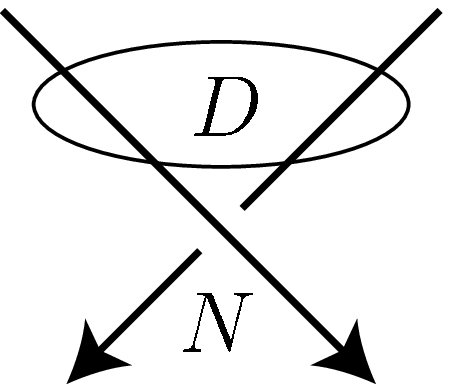}}
      \end{center}
   \caption{}
  \label{2-ap-theorem-proof4}
\end{figure} 

%

\item[Case 4.]
Both $M$ and $N$ are mixed crossings.

First suppose that $\tilde L$ has four or more components. Then using Lemma \ref{lemma6} we have (3) or (6).
Suppose that $\tilde L$ has three components, say $\tilde\ell_1,\tilde\ell_2$ and $\tilde\ell_3$. Suppose that $M$ and $N$ are mixed crossings of different pair of components. We may suppose without loss of generality that $M$ is a mixed crossing between $\tilde\ell_1$ and $\tilde\ell_2$ and $N$ is  between $\tilde\ell_2$ and $\tilde\ell_3$. Suppose that $\tilde\ell_1$ and  $\tilde\ell_2$ has four or more mixed crossings between them. Then by applying Lemma \ref{one-three-lemma} to a complementary tangle diagram of $\tilde\ell_1\cup\tilde\ell_2$ at $M$ we have that the diagram $\tilde\ell_1\cup\tilde\ell_2$ is greater than or equal to the right-handed Hopf link. Then we have (3). Therefore we have that $\tilde\ell_1$ and  $\tilde\ell_2$  have just two mixed crossings between them. Then by the primeness of $\tilde L$ we have that $\tilde\ell_1$ and $\tilde\ell_3$ has mixed crossings between them. Note that there is a mixed crossing, say $P$ between $\tilde\ell_1$ and $\tilde\ell_3$ such that $\tilde\ell_1$ is over $\tilde\ell_3$ at $P$, and there is another mixed crossing, say $Q$ between $\tilde\ell_1$ and $\tilde\ell_3$ such that $\tilde\ell_1$ is under $\tilde\ell_3$ at $Q$. If $\tilde\ell_2$ is both over $\tilde\ell_1$ at $M$ and over $\tilde\ell_3$ at $N$, or both under $\tilde\ell_1$ at $M$ and under $\tilde\ell_3$ at $N$, then we have (3). Suppose that $\tilde\ell_2$ is over $\tilde\ell_1$ at $M$ and under $\tilde\ell_3$ at $N$. Then we change crossings of $\tilde L$ so that $\tilde\ell_1$ is under everything except at $P$, and $\tilde\ell_3$ is over everything except at $P$. Then without changing $M$ and $N$ we have a right-handed Hopf link. The other case is similar.
Next suppose that $M$ and $N$ are between the same pair of components, say $\tilde\ell_1$ and $\tilde\ell_2$. If $\tilde\ell_1$ and $\tilde\ell_2$ have four or more mixed crossings between them then we have (3) by Lemma \ref{lemma6}. Thus we have that $\tilde\ell_1$ and $\tilde\ell_2$ have no other mixed crossings between them. Then by the primeness of $\tilde L$ and by Lemma \ref{lemma7} we have (8).
Suppose that $\tilde L$ has just two components. Then using Lemma \ref{lemma5} to the complementary tangle diagram of $M$ we have (3), (4) or
(5) if $\tilde L$ is not $\tilde T_{11}$ of Fig. \ref{2ap-trivial} nor $\tilde L_3$ of Fig. \ref{2-almost-positive-diagrams}.
This completes the proof.
\end{enumerate}
\end{proof}

Theorem \ref{2ap1}, Theorem \ref{2ap2} and Corollary \ref{2ap3} are immediate corollaries of Theorem \ref{2-almost-positive-theorem}.

\section{3-almost positive knots}

\begin{Theorem}\label{3ap1-1}
Let $K$ be a 3-almost positive knot. 
Then either $K\geq$ trivial knot or $K$ is the left-handed trefoil knot (plus positive knots as connected summands). 
\end{Theorem}

\begin{proof}
Let $\tilde K$ be a diagram of $K$ with three negative crossings $N_1$, $N_2$ and $N_3$. 
Suppose for example that a loop from $N_1$ to $N_1$ passes through at most one of $N_2$ and $N_3$, 
or passes through both of them as over-crossing, or passes through both of them as under-crossing. 
Then by Lemma \ref{diagram-reducing-lemma} we have that $K$ is greater than or equal to the trivial knot. 
Thus we are left with the case in which every loop from $N_i$ to $N_i$ passes through two other 
negative crossings, one as over-crossing the other as under-crossing. This situation is 
illustrated in Fig. \ref{3ap-proof1}.

\begin{figure}[htbp]
      \begin{center}
\scalebox{0.58}{\includegraphics*{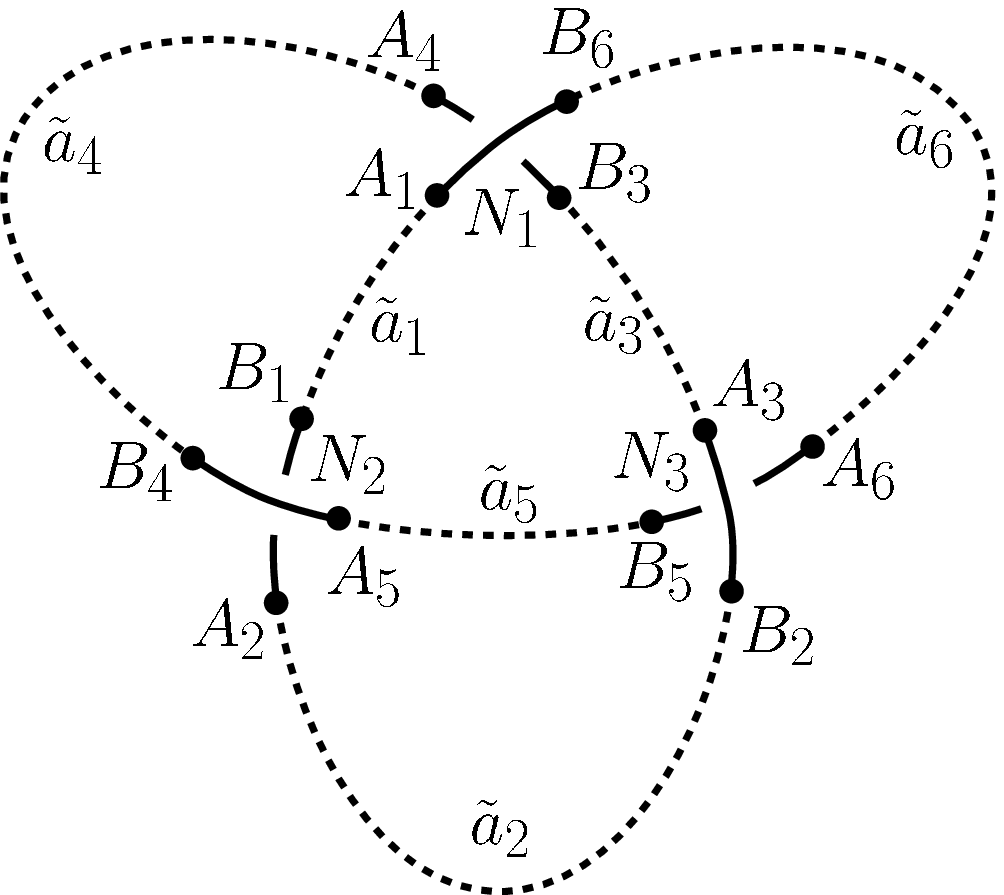}}
      \end{center}
   \caption{}
  \label{3ap-proof1}
\end{figure} 

%

We name the points as in Fig. \ref{3ap-proof1}. Let $\tilde a_{i}=A_{i}B_{i}$. We consider the suffix modulo six. 
If $\tilde a_{i}\cap\tilde a_{i+1}$ is nonempty for some $i$, then it reduces to the case with two negative crossings 
as any crossing of $\tilde a_{i}\cap\tilde a_{i+1}$ can serve as a crossing $P$ from Lemma \ref{diagram-reducing-lemma}. 
Suppose $\tilde a'_{i}\cap(\tilde a_{i+2}\cup\tilde a_{i-2})$ is nonempty (the spine $\tilde a'_{i}$ of 
$\tilde a_{i}$ has been defined in Section 2.1). We may assume without loss of generality 
that  $\tilde a'_{i}\cap\tilde a_{i+2}$ is nonempty. Let $P$ be the first crossing starting from $A_{i+2}$ 
with $\tilde a_{i}$. We reduce $A_{i+2}P$ and $\tilde a_{i+1}$ to become simple arcs
 using Lemma \ref{diagram-reducing-lemma}. Then we pull down $PB_{i+2}A_{i+3}B_{i+3}$,
pull up $\tilde a_{i+4}$ and we have a knot with no negative crossings as illustrated in 
Fig. \ref{3ap-proof2} and the case is completed. 

\begin{figure}[htbp]
      \begin{center}
\scalebox{0.58}{\includegraphics*{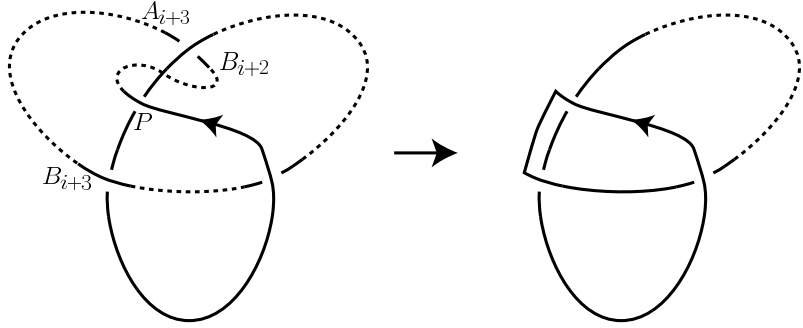}}
      \end{center}
   \caption{Reduction in the case of $\tilde a'_{i}\cap\tilde a_{i+2}\neq \emptyset$}
  \label{3ap-proof2}
\end{figure} 

%

Now suppose $\tilde a'_{i}\cap \tilde a_{i+3}$ is nonempty. 
We simplify $\tilde a_{i+1}$ and $\tilde a_{i+2}$. 
Let $P$ be a crossing of $\tilde a'_{i}\cup \tilde a_{i+3}$. 
We pull down $A_{i+3}P$ and pull up $PB_{i+3}A_{i+4}B_{i+4}$. 
Then we have a diagram with only one negative crossing as illustrated in Fig. \ref{3ap-proof3}.

\begin{figure}[htbp]
      \begin{center}
\scalebox{0.58}{\includegraphics*{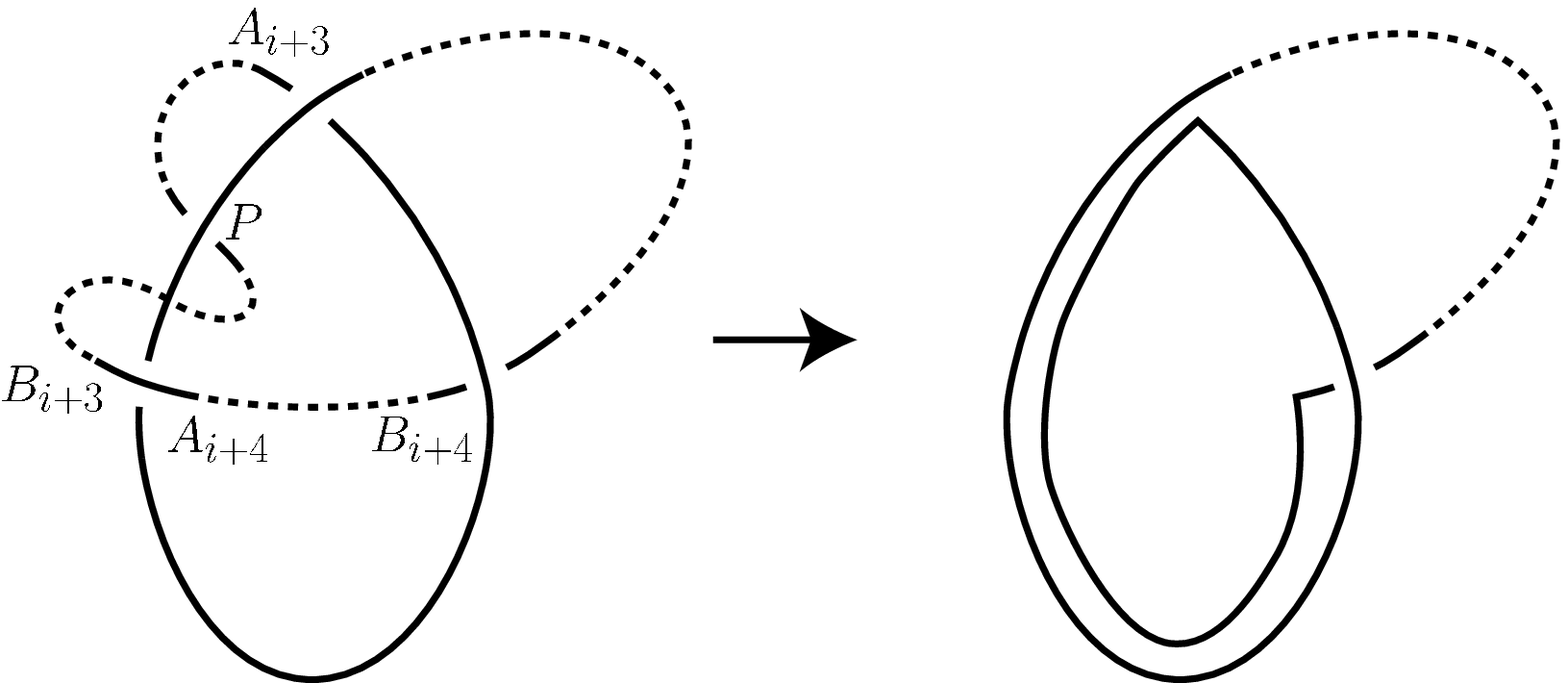}}
      \end{center}
   \caption{}
  \label{3ap-proof3}
\end{figure} 

%

Therefore we may suppose that no $\tilde a'_{i}$ has crossings with other $\tilde a_{j}$'s. 
Then we have that the union of all $\tilde a'_{i}$ and $B_iA_{i+1}$ forms a standard diagram of a left-handed trefoil knot.
Now suppose $\tilde a_{i}\cap\tilde a_{j}\neq\emptyset$. Then the result follows by Lemma \ref{hook-tangle} as illustrated in Fig. \ref{3ap-proof4}.

\begin{figure}[htbp]
      \begin{center}
\scalebox{0.58}{\includegraphics*{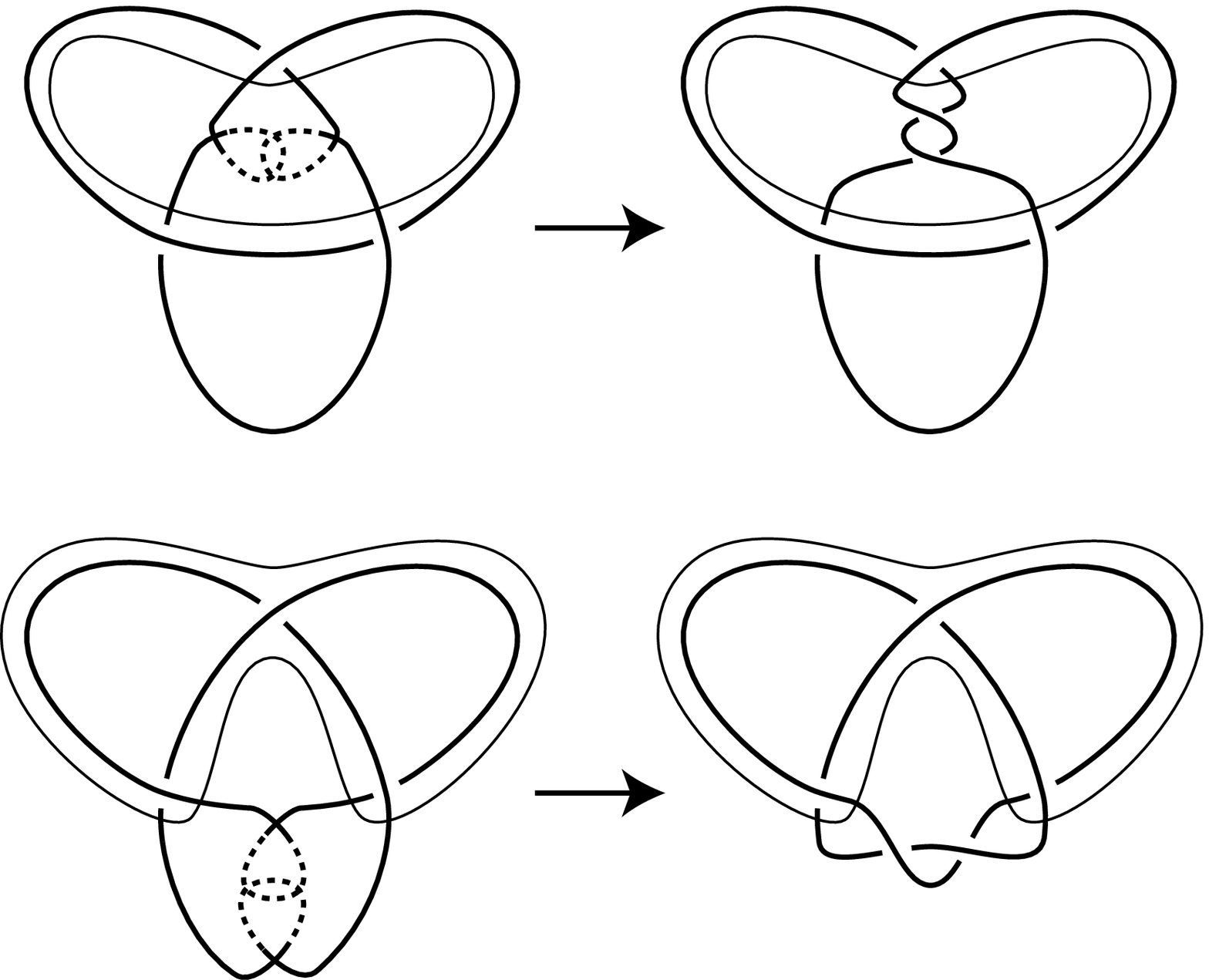}}
      \end{center}
   \caption{}
  \label{3ap-proof4}
\end{figure} 

%

Therefore $\tilde a_{i}\cap\tilde a_{j}$ is empty for all $i\neq j$.
Then the knot $K$ is the left-handed trefoil knot (possibly with some positive knots connected summed). This completes the proof.
\end{proof}

Corollary \ref{3ap2} follows immediately from Theorem \ref{3ap1} and the fact that the connected sum of the left-handed trefoil knot and a nontrivial positive knot has non-positive signature.

We do not know whether any knot with the left-handed trefoil knot as a connected sum summand can dominate 
the trivial knot. We think it is highly unlikely.





\section{Corollaries}\label{Section 6}
In the previous sections we have been applying our ``domination" technique mainly to 
the classical (Trotter-Murasugi) signature, $\sigma (L)$. Here we give several other 
applications of the theory including amphicheirality, sliceness, bounds on 
Jones polynomial, and Tristram-Levine signatures.

\begin{Corollary}\label{corollary1}
\begin{enumerate}
\item[(a)]
A nontrivial almost positive link is not amphicheiral.

\item[(b)]
2-almost positive link $L$ without a trivial component is amphicheiral if and only if either

\begin{enumerate}
\item[(i)]
$L$ is a connected or disjoint sum of right-handed and left-handed
Hopf links, or

\item[(ii)]
$L$ is the figure eight knot.

\end{enumerate}
\end{enumerate}

\end{Corollary}

\begin{proof}
If $L$ is amphicheiral then the signature $\sigma (L)=0$, so (a) follows from 
Corollary \ref{ap4}. To prove (b), we  apply Corollary \ref{2ap3}. 
The cases (1) and (3) of Corollary \ref{2ap3} have $\sigma (L)=0$. An amphicheiral link $L$ has the global
linking number $lk(L)=0$. Among the links described in Corollary \ref{2ap3} only connected or disjoint sums
of right and left-handed Hopf links (plus possibly trivial components) have $lk(L)=0$, and they
are amphicheiral. For a twist knot it is well known by the classification of 2-bridge knots \cite{Sch}
that it is amphicheiral if and only if it is the trivial knot or the figure eight knot.
\end{proof}

\begin{Corollary}\label{corollary2}
\begin{enumerate}
\item[(a)]
A nontrivial almost positive link is not a slice link.

\item[(b)]
A  nontrivial 2-almost positive link is a (smoothly) slice link if and only if
it is the Stevedore knot (possibly with additional trivial components).
\end{enumerate}
\end{Corollary}

\begin{proof}
If $L$ is a slice link then $\sigma (L)=0$ and any pair of components
have linking number zero. Therefore (a) follows from Corollary 1.7 and (b) from 
Corollary 1.10 by observing that a twist knot is smoothly slice  if and only if 
it is the trivial knot or the stevedore knot (Casson and Gordon \cite{Ca-G}).
\end{proof}

Recall that the Jones polynomial \cite{Jo} of a link $L$ is an element of
the ring of the Laurent polynomials in variable $\sqrt t$. That is
$V_L(t)\in Z[{\sqrt t}^{\pm 1}]$ satisfying the skein relation:
$$ t^{-1}V_{L_+}(t) -tV_{L_-}(t)=(\sqrt t - \frac{1}{\sqrt t})V_{L_0}(t)$$
and normalized to be $1$ at the trivial knot.

Let $d_{min} V_{L}(t)$ denote the lowest power of $t$ in  $V_{L}(t)$. Notice that
 $V_{L}(t)\in{\rm Z}[t^{\pm 1}]$ if the number of components of $L$,
$\mu (L)$, is odd and $\sqrt {t} V_{L}(t)\in{\rm Z}[t^{\pm 1}]$ if
$\mu (L)$ is even, so the inequalities in Corollary \ref{corollary3} (a)-(c) are sharp
in the case of a knot.

\begin{Corollary}\label{corollary3}
Let $L$ be a nontrivial nonsplit link, then
\begin{enumerate}
\item[(a)]
If $L$ is positive then $d_{min}
V_{L}(t)\geq 3/2$ unless
\begin{enumerate}
\item[(i)]
$L$ is a $(2,2k)$-torus link with anti-parallel orientation of components (Fig. \ref{torus-link2}) 
(denote it $T^{an}_{2,2k}$) then $d_{min}
V_{L}(t)=1/2$ or,

\item[(ii)]
$L=T^{an}_{2,2k}\#T^{an}_{2,2k}$; then $d_{min}
V_{L}(t)=1$ or,

\item[(iii)]
$L$ is a pretzel knot $L(p_{1},p_{2},p_{3})$ of Fig. \ref{pretzel-knot} or a 3-component
pretzel link $L(q_{1},q_{2},q_{3})$ of Fig. \ref{three-three} (c); then  $d_{min}
V_{L}(t)=1$.

\end{enumerate}

\item[(b)]
If $L$ is an almost positive link then $d_{min}
V_{L}(t)\geq 1/2$.

\item[(c)]
If $L$ is a 2-almost positive link then $d_{min}
V_{L}(t)\geq -3/2$ unless either

\begin{enumerate}
\item[(i)]
$L$
is a left-handed Hopf link (denoted $H_{-}$) then $d_{min}
V_{L}(t)=-5/2$ or,

\item[(ii)]
$L=H_{-}\#T^{an}_{2,2k}$ then $d_{min}
V_{L}(t)=-2$ or,

\item[(iii)]
$L$ is a twist knot for which $d_{min}
V_{L}(t)=-2$.
\end{enumerate}

\item[(d)]
If $K$ is a 3-almost positive knot then $d_{min}
V_{L}(t)\geq -3$ except the left-handed trefoil knot for which $d_{min}
V_{L}(t)=-4$.
\end{enumerate}
\end{Corollary}

\begin{proof}
K.~Murasugi showed (\cite{M-3} Theorem 13.3) that if $D$ is a diagram of 
a nonsplit link $L$ with $c_{-}(D)$  negative crossings then 
$$d_{min} V_{L}(t)\geq -c_{-}(D)-\frac{1}{2}\sigma (L).$$

Murasugi's inequality and Corollary \ref{3ap2} (a) give Corollary \ref{corollary3} (d). 
Murasugi's inequality and Corollary \ref{2ap3} give Corollary \ref{corollary3} (c) 
(there is no need to perform calculations for exceptional cases (i)-(iii) 
because for alternating diagrams (or their
connected sums) Murasugi's inequality becomes the equality (\cite{M-3}). 
Murasugi's inequality and Corollary \ref{p3} suffice to prove (a). 
In the case (b) Murasugi's inequality and the fact that
$\sigma (L)<0$ are not sufficient (we get only $d_{min}
V_{L}(t)\geq -1/2$; or in the case of knots $d_{min} V_{L}(t)\geq 0$) 
and to improve this we would have to show that 
Murasugi's inequality is a strict inequality in our case (an almost positive diagram with 
alternating connected summands has a nugatory crossing). We choose, however, a different
method which is of interest on its own  and generalizes Corollary \ref{corollary3} (a) and (b).

\begin{Theorem}\label{6-4}
Let $D$ be a diagram of an oriented link $L$, then

\begin{enumerate}
\item[(a)]
If $D$ is positive then 

$$d_{min}
V_{L}(t)= \frac{1}{2}(c(D)-s(D)+1)$$

where $c(D)$ is the number of crossings of $D$ and $s(D)$ is the number
of Seifert circles of $D$. In particular, if $L$ is a knot then $d_{min}V_{L}(t)= g(D)$, where 
$g(D)$ is the genus of the Seifert surface obtained from $D$ by the Seifert algorithm\footnote{After 
the first version of the paper was written, Kronheimer and Mrowka \cite{KM} and then Rasmussen \cite{Ras}
 proved that $g(D)$ is equal to genus of the knot $L$ as well as a slice genus of $L$.}. 

\item[(b)]
If $D$ has one negative crossing say $P$ then either
\begin{enumerate}
\item[(i)]
$P$ is a singular crossing (i.e. there are no other crossings joining
the same, as $P$, Seifert circles of $D$) then

$$d_{min}
V_{L}(t)= \frac{1}{2}(c(D)-s(D)+1)$$

or

\item[(ii)]
$P$ is not a singular crossing and then

$$d_{min}
V_{L}(t)= \frac{1}{2}(c(D)-s(D)+1)-1 = \frac{1}{2}(c_{+}(D)-c_{-}(D)-s(D)+1)$$

\end{enumerate}
\end{enumerate}
In particular Corollary \ref{corollary3} (b) follows.
\end{Theorem}
\end{proof}
To prove Theorem \ref{6-4} we need the following technical Lemma.

\begin{Lemma}\label{5.5}
Let $D$ be an oriented link diagram and $F=P_{1},P_{2},\ldots ,P_{k}$ 
a family of different crossings of $D$.

\begin{enumerate}
\item[(a)]
Assume that $A$-splittings 
$ \left(
\begin{minipage}{49pt}
\scalebox{0.2}{\includegraphics{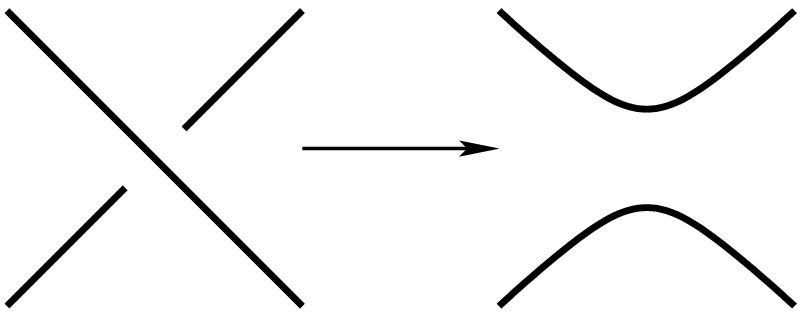}}
\end{minipage}
\right) $
 of crossings $P_{1},P_{2},\ldots ,P_{k}$
produce a diagram $D'$ such that two arcs obtained from the splitting of $P_{i}$, for any $i$,
are on different connected components of $D'$ considered as a graph (we say that $D$ is $+$-adequate 
with respect to crossings $F$). Then

$$4d_{min}
V_{D}(t)= 4d_{min}
V_{D'}(t) + 3(c_{+}(D)-c_{-}(D) - (c_{+}(D') -c_{-}(D')) - c(F)$$

where  $c(F)$  is the number of crossings in the family $F$.

\item[(b)]
Assume that all crossings in $F$ are positive and $D'$ is obtained from $D$
by smoothing 
$ \left( 
\begin{minipage}{49pt}
\scalebox{0.2}{\includegraphics{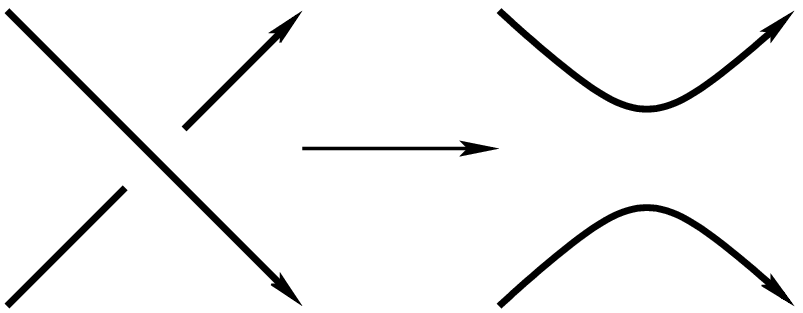}}
\end{minipage}
\right) $
 all crossings of $F$, and for any $i$ two arcs obtained from 
the smoothing of $P_{i}$ are on different connected components of $D'$
considered as a graph. Then 

$$d_{min}
V_{D}(t)= d_{min}
V_{D'}(t) + \frac{1}{2}c(F)$$

\item[(c)] (compare \cite{M-P-1})\ \ 
Let $D=D_{1}*D_{2}$ (planar star (Murasugi) product) 
and $D_{2}$ be a positive diagram. Then

$$d_{min}
V_{D}(t)= d_{min}
V_{D_{1}}(t)+d_{min}
V_{D_{2}}(t)=d_{min}
V_{D_{1}}(t)+\frac{1}{2}(c(D_{2})-s(D_{2})+1).$$

\end{enumerate}
\end{Lemma}

\begin{proof}
We use the Kauffman bracket approach to $V_{L}(t)$.\\
Recall that the Kauffman bracket polynomial of a link diagram,
$\langle  D \rangle \in Z[A^{\pm 1}]$, satisfies
the Kauffman bracket skein relations \cite{Ka-2}:
\[
\langle  D_+ \rangle = A \langle  D_0 \rangle +
A^{-1}\langle  D_{\infty} \rangle,\ \langle  D \sqcup O \rangle = (-A^2 - A^{-2})\langle  D \rangle,
\]
and is normalized to be $1$ at the crossingless diagram of the unknot.\ 
The Jones polynomial $V_L(t) \in Z[t^{\pm \frac{1}{2}}]$ can be obtained from
the Kauffman bracket polynomial of any oriented diagram $D$ of a link $L$ by putting $t=A^{-4}$ in
$V_L(t) = (-A^3)^{-w(D)}\langle  D\rangle$, where $w(D)$ is the writhe or Tait number of an oriented
diagram $D$ (that is, $w(D)=c_+(D)-c_-(D) = {\sum}_p {\rm sgn} (p)$ where the sum is taken over all
crossings $p$ of oriented diagram~$D$).

\begin{enumerate}
 \item[(a)]
We can repeat essentially the proofs of [K], \cite{L-T} and [T] (see for example
the proof of Proposition 1 in \cite{L-T} about the maximal and minimal terms of 
the bracket $<D>$ of an adequate diagram $D$). In particular, we get immediately that 
$A$-smoothing along $F$ contributes to the maximal term in $<D>$, so  
$d_{max \ A}<D> = d_{max \ A}<D'> + c(F)$, and thus $-4d_{min \ t}V_D(t) + 3w(D) = 
-4d_{min \ t}V_{D'}(t) + 3w(D') + c(F)$.

\item[(b)]
Observe that a smoothing at a positive crossing is the same as A-splitting
the crossing so (b) follows from (a) (Figure 6.1 illustrates this fact).
\ \\
\begin{figure}[htbp]
      \begin{center}
\scalebox{0.5}{\includegraphics{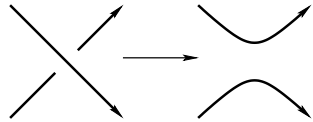}}
      \end{center}
   \caption{$A$-splitting agrees with smoothing for a positive crossing}
  \label{smoothing-A}
\end{figure}

\item[(c)]
The family of crossings of $D_{2}$ satisfies the assumptions of (b) so (c)
follows from (b).

\end{enumerate}
\end{proof}

PROOF OF THEOREM \ref{6-4}:

\begin{enumerate}
\item[(a)]
We apply Lemma \ref{5.5} (a) to family $F$ of all crossings of $D$, then $D'$ is
a collection of all Seifert circles of $D$ so $d_{min}V_{D'}(t) =
1/2(-s(D)+1)$. Therefore $d_{min}V_{D}(t)=1/2(c(D)-s(D)+1).$

\item[(b)]
\begin{enumerate}
\item[(i)]
$P$ is a singular crossing and $F$ a family of all other crossings of $D$.
Let $D_{+}^{P}$ be a positive diagram obtained from $D$ by changing the crossing
$P$ from negative to positive. Finally let $D'$  (resp. $(D_{+}^{P})'$)
denote the diagram obtained from $D$ (resp. $D_{+}^{P}$) by smoothing
all crossings from family $F$. Notice that $D'$ represents the link ambient isotopic
to $(D_{+}^{P})'$ and in both cases we can use Lemma \ref{5.5} (b) and family $F$
(see Fig. 6.2).

\ \\
\begin{figure}[htbp]
      \begin{center}
\scalebox{0.5}{\includegraphics{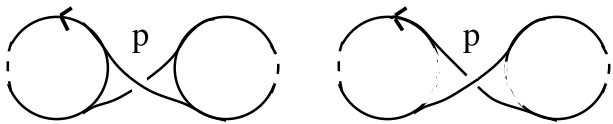}}
      \end{center}
   \caption{$P$ is a nugatory crossing in $D'$ and $(D^P_+)'$}
  \label{Figure 6.2}
\end{figure}


Therefore

$$d_{min}
V_{D}(t)= d_{min}
V_{D'}(t)+\frac{1}{2}c(F)=d_{min}
V_{(D_{+}^{P})'}(t)+\frac{1}{2}c(F)=d_{min}
V_{D_{+}^{P}}(t)$$

Furthermore $D_{+}^{P}$ is a positive diagram so by part (a) of Theorem \ref{6-4}

$$d_{min}
V_{D_{+}^{P}}(t)=\frac{1}{2}(c(D_{+}^{P})-s(D_{+}^{P})-1).$$

The formula from (b)(i) holds because $c(D_{+}^{P})=c(D)$ and
$s(D_{+}^{P})=s(D)$.

\item[(ii)]
Let $Q_{1},Q_{2},\ldots ,Q_{n}$ be all other crossings of $D$ which join
the same Seifert circles at $P$. Let $F$ be the set of crossings of $D$
different from $P$ and  $Q_{1},Q_{2},\ldots ,Q_{n}$. Furthermore let
$D_{0}=D_{00}^{PQ_{1}}$ denote the diagram obtained from $D$ by smoothing
crossings $P$ and $Q_{1}$. Finally let $D'$ (resp. $D'_{0}$) denote the 
diagram obtained from $D$ (resp. $D_{0}$) by smoothing all crossings from
family $F$. Notice that $D'$ represents a link isotopic to $D'_{0}$ and in
both cases we can apply Lemma \ref{5.5} (b) using family $F$ (see Fig. 6.3).

\ \\
\begin{figure}[htbp]
      \begin{center}
\scalebox{0.4}{\includegraphics{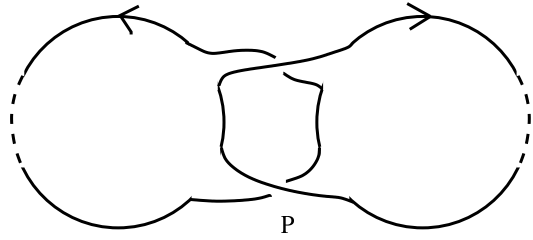}}
      \end{center}
   \caption{$P$ is a negative crossing which is non-singular and cancels in $D'$}
  \label{smoothing-A}
\end{figure}


Therefore (as in Case (i)), $d_{min}V_{D}(t)= d_{min}V_{D_{0}}(t)$.
Also $D_{0}$ is a positive diagram so by part (a) of Theorem 5.5,

$$d_{min}V_{D_{0}}(t)=\frac{1}{2}(c(D_{0})-s(D_{0})-1).$$

Because $c(D)=c(D_{0})+2$, $s(D)=s(D_{0})$, we get the formula from 
Theorem \ref{6-4} (b)(ii).
\end{enumerate}

Now assume that $D$ is a connected diagram, then $c(D)-s(D)+1\geq 0$,
and the equality holds if and only if
 all crossings are nugatory so if $L$ is nontrivial
then $c(D)-s(D)+1\geq 1$. Now assume additionally that $P$ is the only negative crossing of $D$,
$P$ is not singular and $D$ has no nugatory crossings.  If
 \begin{enumerate}
\item[(1)]
$c(D)-s(D)+1=1$ then $D$ looks as in Fig. 6.4 (i) so it represents a split link.

\item[(2)]
$c(D)-s(D)+1=2$ then $D$ looks as in Fig. 6.4 (ii) or (iii) so it represents the trivial knot or a split link.

Corollary \ref{corollary3} (b) follows.
\end{enumerate}
\end{enumerate}

\ \\
\begin{figure}[htbp]
      \begin{center}
\scalebox{0.6}{\includegraphics{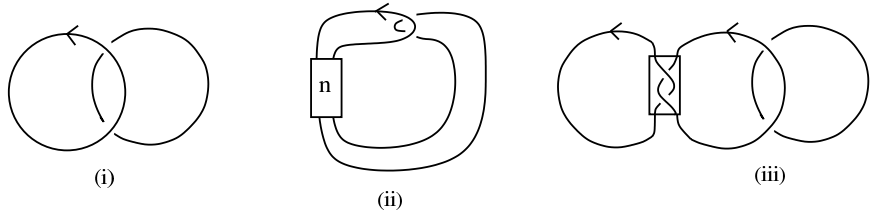}}
      \end{center}
   \caption{$c(D)-s(D)+1=1$ or $c(D)-s(D)+1=2$ with a non-singular negative crossing}
  \label{smoothing-A}
\end{figure}

\ \\
\begin{Remark}\label{Remark 6.6}
If we consider the Seifert graph associated to $D$, say $\Gamma (D)$
(see \cite{Cr-Mo} or \cite{M-P-2}) then $c(D)-s(D)+1$ is the first Betti 
number of $\Gamma (D)$ (called cyclomatic number in graph theory). 
Our method is to analyze diagrams such that $\Gamma (D)$
has small cyclomatic number. In essence we can then extend 
Corollary \ref{corollary3} (c) to $d_{min}V_{L}(t)\geq 1/2$ 
unless... But the list of exceptions would be rather
long so not very interesting (unless we will try to use this 
to determine which links are 2-almost positive).
\end{Remark}
\ \\
\begin{figure}[htbp]
      \begin{center}
\scalebox{0.5}{\includegraphics{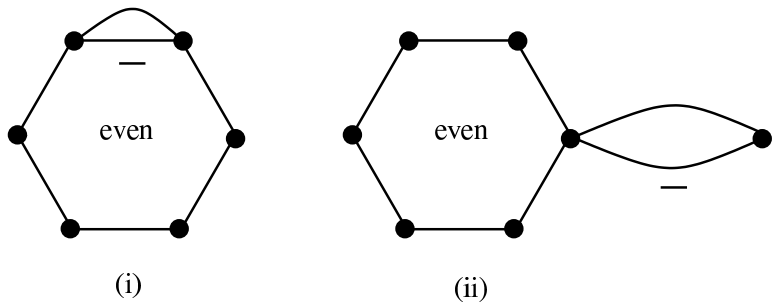}}
      \end{center}
   \caption{The Seifert graph in (i) is the graph of Fig. 6.4 (ii) for $n$ odd. The Seifert graph 
in (ii) is the graph of Fig. 6.4 (ii) for $n$ even or of Fig. 6.4 (iii)}
  \label{Seifert graphs}
\end{figure}

\begin{Corollary}\label{5.7}
A nontrivial positive link $L$ has unknotting number one if and only if 
$L$ is a right-handed Hopf link (plus, possibly, trivial components) 
or $L$ is a (positive)
twist knot ( possibly with additional trivial components).
\end{Corollary}

\begin{proof}
A Hopf link and twist knots have unknotting number one, we will show 
the converse.
Let $u(L)$ be the unknotting number of $L$. 
Then $2u(L)\geq |\sigma (L)|$ \cite{M-1} and
clearly $u(L)\geq\sum_{i<j}|lk(\ell_{i},\ell_{j})|$ 
where the sum is taken over all different pairs
of components of $L$. Therefore by Corollary 1.9 we are left 
with pretzel knots $L(p_{1},p_{2},p_{3})$ of Fig. \ref{pretzel-knot}. 
These pretzel knots bound genus one Seifert
surfaces in a natural manner. Since genus one unknotting 
number one knots are doubled knots \cite{Ko}, \cite{S-T} 
and pretzel knots are simple \cite{Ko}, 
we are left with twist knots.
\end{proof}

In the next Theorem we use our domination theorems to show, generalizing \cite{P-1}, 
that Tristram-Levine signature satisfies some inequalities for $k$-almost positive links.

Let $\sigma_{\psi}(L)$ be the Tristram-Levine signature of $L$; we assume $|\psi|=1$. 
We use the convention that $\sigma_{\psi}(L)$ is the signature of the Hermitian matrix 
$H=\bar\psi A_L+\psi A_{L}^{T}$, where $A_L$ is a Seifert matrix of $L$.
Because for  $Re (\psi) \geq 0$ we have $\sigma_{\psi}(L_+) \leq \sigma_{\psi}(L_-)$ 
(\cite{P-Tr-1,P-2,P-3})   
we can use domination results: Theorems 2.15, 3.2, 4.11 and 5.1 to bound  
 Tristram-Levine signature of positive, almost positive, 2-almost positive and 3-almost positive 
links. In particular, we have. 
\begin{Theorem}\label{Theorem 6.8} 
(i) (\cite{P-1}): If $L$ is a nontrivial positive link and $Re (\psi) > \frac{1}{2}$, 
then $\sigma_{\psi}(L) < 0 $.\\
(ii)  If $K$ is a nontrivial positive knot and $Re (\psi) \geq 0 $ then $\sigma_{\psi}(K) \leq \sigma_{\psi}(\bar 5_1)$ 
or $K$ is a connected sum of some pretzel knots of type
$L(2k_1+1,2k_2+1,2k_3+1)$, where $k_i \geq 0$; see Figure \ref{pretzel-knot}.
 Furthermore we have 
(compare e.g. Examples 5.16, and 5.14 of \cite{P-3}):
 $$ \sigma_{\psi}(\bar 5_1)= 
\left\{
\begin{array}{ccc}
-4 & if & Re(\psi) > \frac{1+\sqrt 5}{4} \\
-3 & if & Re(\psi) = Re (e^{\pi i/5})= \frac{1+\sqrt 5}{4} \approx 0.809... \\
-2 & if & Re (e^{2\pi i/5}) \leq Re(\psi) < Re (e^{\pi i/5}) \\
-1 & if & Re(\psi) = Re (e^{2\pi i/5}) = \frac{\sqrt 5 -1}{4} \approx 0.309...\\
0  & if & 0 \leq Re(\psi) < Re (e^{2\pi i/5}).
\end{array}\right. $$
For $1+k_1+k_2+k_3 + k_1k_2+k_1k_3+k_2k_3 > 0$,
$$\sigma_{\psi}(L(2k_1+1,2k_2+1,2k_3+1))=
\left\{
\begin{array}{ccc}
-2 & if & Re(\psi) > \frac{1}{2\sqrt {1+k_1+k_2+k_3 + k_1k_2+k_1k_3+k_2k_3}} \\
-1 & if & Re(\psi) = \frac{1}{2\sqrt {1+k_1+k_2+k_3 + k_1k_2+k_1k_3+k_2k_3}} \\
 0 & if & 0 \leq Re(\psi) < \frac{1}{2\sqrt {1+k_1+k_2+k_3 + k_1k_2+k_1k_3+k_2k_3}}.
\end{array}\right. $$
In the case that $k_1=k_2=k_3=0$ represents the right handed trefoil knot, $\bar 3_1$.\\
(iii) If $L$ is a nontrivial almost positive link and $Re (\psi) \geq 0$ then \\
$\sigma_{\psi}(L) \leq \sigma_{\psi}(\bar 3_1)$ or $\sigma_{\psi}(L) \leq \sigma_{\psi}(H_+)$,  
and if $L$ has a nontrivial component and $Re(\psi) > \frac{1}{2}$, we have  $\sigma_{\psi}(L) \leq -2$

(iv) If $K$ is a nontrivial 2-almost positive knot and $Re (\psi) \geq 0$ then, \\
$\sigma_{\psi}(K) \leq \sigma_{\psi}(\bar 3_1)$ or $\sigma_{\psi}(K) \leq \sigma_{\psi}(6_2)$ or 
$K$ is a twist knot with a negative clasp (see Figure 1.3). We have the Tristram-Levin signature 
equal to zero in the last case, furthermore for $ \sigma_{\psi}(6_2)$ we 
have\footnote{In the convention of \cite{Gor,Ch-L} 
one defines the Tristram-Levine signature function of variable $\xi$ ($|\xi|=1$) as 
$\sigma_L(\xi)=\sigma((1-\bar\xi) A + (1-\xi) A^T))$. For $Re (\psi) \geq 0$, one has 
$\sigma_{\psi}(L) = \sigma_L(\xi)$, where $\xi=-\psi^2$. 
In {\it knotinfo} Web page \cite{Ch-L}, the parameter $s$ 
satisfying $\xi=e^{\pi i s}$ is used. In particular, $\sigma_{6_2}(\xi)=-1$ 
for $Re (\xi) = \frac{3-\sqrt 5}{4}=1-\cos (\pi/5) \approx 0.191 $, and $s\approx 0.44$; 
compare\cite{P-3}. }:
$$\sigma_{\psi}(6_2)=
\left\{
\begin{array}{ccc}
-2 & if & Re(\psi) > \frac{1}{2}\sqrt{\frac{1+\sqrt 5}{2}} \\
-1 & if & Re(\psi) = \frac{1}{2}\sqrt{\frac{1+\sqrt 5}{2}} \approx 0.636... \\
 0 & if & 0 \leq Re(\psi) < \frac{1}{2}\sqrt{\frac{1+\sqrt 5}{2}}.
\end{array}\right. $$

(v) If $K$ is a 3-almost positive knot, and $Re (\psi) \geq 0$ then \\
$\sigma_{\psi}(K) \leq 0$ or $K$ is the left handed trefoil knot.
\end{Theorem}
In \cite{P-2} we observe that domination results can be applied to any signature-like invariant $\sigma$ 
satisfying Trotter type inequalities. In 1987 or 1991 we could only speculate about existence 
of such invariants\footnote{Khovanov homology and Heegaard-Floer homology allow construction 
of such signatures; e.g. Rasmussen invariant\cite{Ras}.} \cite{P-1,P-Tr-3}.

Finally, we  prove Theorem 1.13(b) following Cochran and  Gompf paper \cite{Co-G}, using  Theorem 1.8, 
and the computation by S.~Akbulut (compare Problem 4.2 of \cite{Kir}).\\
{\bf Theorem 1.13 (b)}. If $K$ is a $2$-almost positive knot different from a twist knot with 
a negative clasp then $K(1/n)$ (i.e. $1/n$ surgery on $K$, $n>0$) is a homology 3-sphere 
that does not bound a compact, smooth homology $4$-ball. 
Furthermore, $K(1/n)$ has a nontrivial Floer homology.
\begin{proof} 
It is proved in  \cite{Co-Li,Co-G} that if $K_1 \geq K_2$ then $K_1 \succeq K_2$. Furthermore, 
Cochran and Gompf prove  in \cite{Co-G} that if 
$K_1 \succeq K_2$ and $K_2(1/1)$ is a homology sphere which bounds NSPD (non-standard positive definite) 
homology sphere, then for any $n>0$, the homology sphere $K_1(1/n)$ bounds NSPD homology sphere. Then, 
using Donaldson result they conclude that $K_1(1/n)$ does not  bound a compact, smooth homology 
$4$-ball\footnote{Homology $4$-ball glued along $K(1/n)$ to an NSPD homology sphere would produce 
a closed smooth $4$-manifold with NSPD intersection form, contradicting Donaldson Theorem.}. 
The fact that $\bar 3_1(1/1)$ bounds NSPD homology sphere is an important tool in \cite{Co-G}. 
Because our generalization of Cochran-Gompf result uses Theorem 1.8 we need to show that $6_2(1/1)$ 
bounds NSPD homology sphere. The proof of this fact was provided to us by S.~Akbulut 
(in letter sent February 11, 1992) in which he demonstrated (using Kirby calculus) that 
 $6_2(1/1)$ bounds a homology sphere with $E_8$ intersection form.
Furthermore, according to \cite{Co-G}: {\it Andreas Floer pointed out to us that if a homology sphere 
$\Sigma$ bounds NSPD, then the Floer homology $I_*(\Sigma)$ is non-trivial, because Donaldson's Theorem 
will hold for $X$ with boundary $\Sigma$ as long as $I_*(\Sigma)$ is trivial.}
\end{proof}

\section{ Acknowledgements}

We would like to thank Selman Akbulut for providing us with the proof that $6_2(\frac{1}{1})$ 
bounds a homology sphere with $E_8$ intersection form.


\section{After Twenty Years}
The first version of this paper was written in 1990, almost 19 years before this 
arXiv version of the paper appears. Many results 
of the work (which has several abstracts published \cite{P-1,P-2,T-3,T-4} and was available from 
the authors but 
was not published (and it was the time before arXiv), has been rediscovered and/or 
generalized. We should mention here the series of papers by A.~Stoimenov \cite{St-1,St-2,St-3,St-4,St-5}, 
and  papers by T.~Nakamura \cite{Nak}, L.~Rudolph \cite{R-2}, and M.~Hirasawa \cite{Hir}. 
Furthermore Ozawa proved that positive diagrams of a 
composite knots are visually composite \cite{Oza} (see also \cite{Cr}).


\begin{center}
\begin{tabular}{l  @{\hspace{2.5 cm}} l}
                            &                         \\
Kouki Taniyama        &     J\'ozef~H. Przytycki\\
e-mail:\  taniyama@waseda.jp      &     e-mail:\ przytyck@gwu.edu\\
Waseda University and GWU &    George Washington University                            
\end{tabular}
\end{center}

\end{document}